\documentclass[12pt]{amsart}

\makeatletter
\def\section{\@startsection{section}{1}%
  \z@{1.2\linespacing\@plus\linespacing}{.5\linespacing}%
  {\normalfont\scshape\centering}}
\def\subsection{\@startsection{subsection}{2}%
  \z@{0.9\linespacing\@plus.7\linespacing}{-.5em}%
  {\normalfont\bfseries}}
\makeatother

\usepackage{amsmath}
\usepackage{amsfonts}
\usepackage{amsbsy}
\usepackage{amssymb}
\usepackage{times}
\usepackage{mathrsfs}
\usepackage{exscale}
\usepackage{graphicx}
\usepackage[colorlinks,citecolor=rouge,linkcolor=rouge,
            bookmarksopen,
            bookmarksnumbered
           ]{hyperref}
\usepackage[capitalize,nameinlink]{cleveref}[0.19]
\crefformat{equation}{\textup{#2(#1)#3}}

\usepackage{times}
\ifx\figforTeXisloaded\relax \else\global\let\figforTeXisloaded=\relax\fi
\catcode`\@=11
\ifx\ctr@ln@m\undefined\else%
    \immediate\write16{*** Fig4TeX WARNING : \string\ctr@ln@m\space already defined.}\fi
\def\ctr@ln@m#1{\ifx#1\undefined\else%
    \immediate\write16{*** Fig4TeX WARNING : \string#1 already defined.}\fi}
\ctr@ln@m\ctr@ld@f
\def\ctr@ld@f#1#2{\ctr@ln@m#2#1#2}
\ctr@ld@f\def\ctr@ln@w#1#2{\ctr@ln@m#2\csname#1\endcsname#2}
{\catcode`\/=0 \catcode`/\=12 /ctr@ld@f/gdef/BS@{\}}
\ctr@ld@f\def\ctr@lcsn@m#1{\expandafter\ifx\csname#1\endcsname\relax\else%
    \immediate\write16{*** Fig4TeX WARNING : \BS@\expandafter\string#1\space already defined.}\fi}
\ctr@ld@f\edef\colonc@tcode{\the\catcode`\:}
\ctr@ld@f\edef\semicolonc@tcode{\the\catcode`\;}
\ctr@ld@f\def\t@stc@tcodech@nge{{\let\c@tcodech@nged=\z@%
    \ifnum\colonc@tcode=\the\catcode`\:\else\let\c@tcodech@nged=\@ne\fi%
    \ifnum\semicolonc@tcode=\the\catcode`\;\else\let\c@tcodech@nged=\@ne\fi%
    \ifx\c@tcodech@nged\@ne%
    \immediate\write16{}
    \immediate\write16{!!!=============================================================!!!}
    \immediate\write16{ Fig4TeX WARNING :}
    \immediate\write16{ The category code of some characters has been changed, which will}
    \immediate\write16{ result in an error (message "Runaway argument?").}
    \immediate\write16{ This probably comes from another package that changed the category}
    \immediate\write16{ code after Fig4TeX was loaded. If that proves to be exact, the}
    \immediate\write16{ solution is to exchange the loading commands on top of your file}
    \immediate\write16{ so that Fig4TeX is loaded last. For example, in LaTeX, we should}
    \immediate\write16{ say :}
    \immediate\write16{\BS@ usepackage[french]{babel}}
    \immediate\write16{\BS@ usepackage{fig4tex}}
    \immediate\write16{!!!=============================================================!!!}
    \immediate\write16{}
    \fi}}
\ctr@ld@f\def\FigforTeX{F\kern-.05em i\kern-.05em g\kern-.1em\raise-.14em\hbox{4}\kern-.19em\TeX}
\ctr@ln@w{newdimen}\epsil@n\epsil@n=0.00005pt
\ctr@ln@w{newdimen}\Cepsil@n\Cepsil@n=0.005pt
\ctr@ln@w{newdimen}\dcq@\dcq@=254pt
\ctr@ln@w{newdimen}\PI@\PI@=3.141592pt
\ctr@ln@w{newdimen}\DemiPI@deg\DemiPI@deg=90pt
\ctr@ln@w{newdimen}\PI@deg\PI@deg=180pt
\ctr@ln@w{newdimen}\DePI@deg\DePI@deg=360pt
\ctr@ld@f\chardef\t@n=10
\ctr@ld@f\chardef\c@nt=100
\ctr@ld@f\chardef\@lxxiv=74
\ctr@ld@f\chardef\@xci=91
\ctr@ld@f\mathchardef\@nMnCQn=9949
\ctr@ld@f\chardef\@vi=6
\ctr@ld@f\chardef\@xxx=30
\ctr@ld@f\chardef\@lvi=56
\ctr@ld@f\chardef\@@lxxi=71
\ctr@ld@f\chardef\@lxxxv=85
\ctr@ld@f\mathchardef\@@mmmmlxviii=4068
\ctr@ld@f\mathchardef\@ccclx=360
\ctr@ld@f\mathchardef\@dccxx=720
\ctr@ln@w{newcount}\p@rtent \ctr@ln@w{newcount}\f@ctech \ctr@ln@w{newcount}\result@tent
\ctr@ln@w{newdimen}\v@lmin \ctr@ln@w{newdimen}\v@lmax \ctr@ln@w{newdimen}\v@leur
\ctr@ln@w{newdimen}\result@t\ctr@ln@w{newdimen}\result@@t
\ctr@ln@w{newdimen}\mili@u \ctr@ln@w{newdimen}\c@rre \ctr@ln@w{newdimen}\delt@
\ctr@ld@f\def\degT@rd{0.017453 }  % pi/180
\ctr@ld@f\def\rdT@deg{57.295779 } % 180/pi
\ctr@ln@m\v@leurseule
{\catcode`p=12 \catcode`t=12 \gdef\v@leurseule#1pt{#1}}
\ctr@ld@f\def\repdecn@mb#1{\expandafter\v@leurseule\the#1\space}
\ctr@ld@f\def\arct@n#1(#2,#3){{\v@lmin=#2\v@lmax=#3%
    \maxim@m{\mili@u}{-\v@lmin}{\v@lmin}\maxim@m{\c@rre}{-\v@lmax}{\v@lmax}%
    \delt@=\mili@u\m@ech\mili@u%
    \ifdim\c@rre>\@nMnCQn\mili@u\divide\v@lmax\tw@\c@lATAN\v@leur(\z@,\v@lmax)% DY > 9949 DX
    \else%
    \maxim@m{\mili@u}{-\v@lmin}{\v@lmin}\maxim@m{\c@rre}{-\v@lmax}{\v@lmax}%
    \m@ech\c@rre%
    \ifdim\mili@u>\@nMnCQn\c@rre\divide\v@lmin\tw@% DX > 9949 DY
    \maxim@m{\mili@u}{-\v@lmin}{\v@lmin}\c@lATAN\v@leur(\mili@u,\z@)%
    \else\c@lATAN\v@leur(\delt@,\v@lmax)\fi\fi%
    \ifdim\v@lmin<\z@\v@leur=-\v@leur\ifdim\v@lmax<\z@\advance\v@leur-\PI@%
    \else\advance\v@leur\PI@\fi\fi%
    \global\result@t=\v@leur}#1=\result@t}
\ctr@ld@f\def\m@ech#1{\ifdim#1>1.646pt\divide\mili@u\t@n\divide\c@rre\t@n\m@ech#1\fi}
\ctr@ld@f\def\c@lATAN#1(#2,#3){{\v@lmin=#2\v@lmax=#3\v@leur=\z@\delt@=\tw@ pt%
    \un@iter{0.785398}{\v@lmax<}%
    \un@iter{0.463648}{\v@lmax<}%
    \un@iter{0.244979}{\v@lmax<}%
    \un@iter{0.124355}{\v@lmax<}%
    \un@iter{0.062419}{\v@lmax<}%
    \un@iter{0.031240}{\v@lmax<}%
    \un@iter{0.015624}{\v@lmax<}%
    \un@iter{0.007812}{\v@lmax<}%
    \un@iter{0.003906}{\v@lmax<}%
    \un@iter{0.001953}{\v@lmax<}%
    \un@iter{0.000976}{\v@lmax<}%
    \un@iter{0.000488}{\v@lmax<}%
    \un@iter{0.000244}{\v@lmax<}%
    \un@iter{0.000122}{\v@lmax<}%
    \un@iter{0.000061}{\v@lmax<}%
    \un@iter{0.000030}{\v@lmax<}%
    \un@iter{0.000015}{\v@lmax<}%
    \global\result@t=\v@leur}#1=\result@t}
\ctr@ld@f\def\un@iter#1#2{%
    \divide\delt@\tw@\edef\dpmn@{\repdecn@mb{\delt@}}%
    \mili@u=\v@lmin%
    \ifdim#2\z@%
      \advance\v@lmin-\dpmn@\v@lmax\advance\v@lmax\dpmn@\mili@u%
      \advance\v@leur-#1pt%
    \else%
      \advance\v@lmin\dpmn@\v@lmax\advance\v@lmax-\dpmn@\mili@u%
      \advance\v@leur#1pt%
    \fi}
\ctr@ld@f\def\c@ssin#1#2#3{\expandafter\ifx\csname COS@\number#3\endcsname\relax\c@lCS{#3pt}%
    \expandafter\xdef\csname COS@\number#3\endcsname{\repdecn@mb\result@t}%
    \expandafter\xdef\csname SIN@\number#3\endcsname{\repdecn@mb\result@@t}\fi%
    \edef#1{\csname COS@\number#3\endcsname}\edef#2{\csname SIN@\number#3\endcsname}}
\ctr@ld@f\def\c@lCS#1{{\mili@u=#1\p@rtent=\@ne%
    \relax\ifdim\mili@u<\z@\red@ng<-\else\red@ng>+\fi\f@ctech=\p@rtent%
    \relax\ifdim\mili@u<\z@\mili@u=-\mili@u\f@ctech=-\f@ctech\fi\c@@lCS}}
\ctr@ld@f\def\c@@lCS{\v@lmin=\mili@u\c@rre=-\mili@u\advance\c@rre\DemiPI@deg\v@lmax=\c@rre%
    \mili@u\@@lxxi\mili@u\divide\mili@u\@@mmmmlxviii%
    \edef\v@larg{\repdecn@mb{\mili@u}}\mili@u=-\v@larg\mili@u%
    \edef\v@lmxde{\repdecn@mb{\mili@u}}%
    \c@rre\@@lxxi\c@rre\divide\c@rre\@@mmmmlxviii%
    \edef\v@largC{\repdecn@mb{\c@rre}}\c@rre=-\v@largC\c@rre%
    \edef\v@lmxdeC{\repdecn@mb{\c@rre}}%
    \fctc@s\mili@u\v@lmin\global\result@t\p@rtent\v@leur%
    \let\t@mp=\v@larg\let\v@larg=\v@largC\let\v@largC=\t@mp%
    \let\t@mp=\v@lmxde\let\v@lmxde=\v@lmxdeC\let\v@lmxdeC=\t@mp%
    \fctc@s\c@rre\v@lmax\global\result@@t\f@ctech\v@leur}
\ctr@ld@f\def\fctc@s#1#2{\v@leur=#1\relax\ifdim#2<\@lxxxv\p@\cosser@h\else\sinser@t\fi}
\ctr@ld@f\def\cosser@h{\advance\v@leur\@lvi\p@\divide\v@leur\@lvi%
    \v@leur=\v@lmxde\v@leur\advance\v@leur\@xxx\p@%
    \v@leur=\v@lmxde\v@leur\advance\v@leur\@ccclx\p@%
    \v@leur=\v@lmxde\v@leur\advance\v@leur\@dccxx\p@\divide\v@leur\@dccxx}
\ctr@ld@f\def\sinser@t{\v@leur=\v@lmxdeC\p@\advance\v@leur\@vi\p@%
    \v@leur=\v@largC\v@leur\divide\v@leur\@vi}
\ctr@ld@f\def\red@ng#1#2{\relax\ifdim\mili@u#1#2\DemiPI@deg\advance\mili@u#2-\PI@deg%
    \p@rtent=-\p@rtent\red@ng#1#2\fi}
\ctr@ld@f\def\pr@c@lCS#1#2#3{\ctr@lcsn@m{COS@\number#3 }%
    \expandafter\xdef\csname COS@\number#3\endcsname{#1}%
    \expandafter\xdef\csname SIN@\number#3\endcsname{#2}}
\pr@c@lCS{1}{0}{0}
\pr@c@lCS{0.7071}{0.7071}{45}\pr@c@lCS{0.7071}{-0.7071}{-45}
\pr@c@lCS{0}{1}{90}          \pr@c@lCS{0}{-1}{-90}
\pr@c@lCS{-1}{0}{180}        \pr@c@lCS{-1}{0}{-180}
\pr@c@lCS{0}{-1}{270}        \pr@c@lCS{0}{1}{-270}
\ctr@ld@f\def\invers@#1#2{{\v@leur=#2\maxim@m{\v@lmax}{-\v@leur}{\v@leur}%
    \f@ctech=\@ne\m@inv@rs%
    \multiply\v@leur\f@ctech\edef\v@lv@leur{\repdecn@mb{\v@leur}}%
    \p@rtentiere{\p@rtent}{\v@leur}\v@lmin=\p@\divide\v@lmin\p@rtent%
    \inv@rs@\multiply\v@lmax\f@ctech\global\result@t=\v@lmax}#1=\result@t}
\ctr@ld@f\def\m@inv@rs{\ifdim\v@lmax<\p@\multiply\v@lmax\t@n\multiply\f@ctech\t@n\m@inv@rs\fi}
\ctr@ld@f\def\inv@rs@{\v@lmax=-\v@lmin\v@lmax=\v@lv@leur\v@lmax%
    \advance\v@lmax\tw@ pt\v@lmax=\repdecn@mb{\v@lmin}\v@lmax%
    \delt@=\v@lmax\advance\delt@-\v@lmin\ifdim\delt@<\z@\delt@=-\delt@\fi%
    \ifdim\delt@>\epsil@n\v@lmin=\v@lmax\inv@rs@\fi}
\ctr@ld@f\def\minim@m#1#2#3{\relax\ifdim#2<#3#1=#2\else#1=#3\fi}
\ctr@ld@f\def\maxim@m#1#2#3{\relax\ifdim#2>#3#1=#2\else#1=#3\fi}
\ctr@ld@f\def\p@rtentiere#1#2{#1=#2\divide#1by65536 }
\ctr@ld@f\def\r@undint#1#2{{\v@leur=#2\divide\v@leur\t@n\p@rtentiere{\p@rtent}{\v@leur}%
    \v@leur=\p@rtent pt\global\result@t=\t@n\v@leur}#1=\result@t}
\ctr@ld@f\def\sqrt@#1#2{{\v@leur=#2%
    \minim@m{\v@lmin}{\p@}{\v@leur}\maxim@m{\v@lmax}{\p@}{\v@leur}%
    \f@ctech=\@ne\m@sqrt@\sqrt@@%
    \mili@u=\v@lmin\advance\mili@u\v@lmax\divide\mili@u\tw@\multiply\mili@u\f@ctech%
    \global\result@t=\mili@u}#1=\result@t}
\ctr@ld@f\def\m@sqrt@{\ifdim\v@leur>\dcq@\divide\v@leur\c@nt\v@lmax=\v@leur%
    \multiply\f@ctech\t@n\m@sqrt@\fi}
\ctr@ld@f\def\sqrt@@{\mili@u=\v@lmin\advance\mili@u\v@lmax\divide\mili@u\tw@%
    \c@rre=\repdecn@mb{\mili@u}\mili@u%
    \ifdim\c@rre<\v@leur\v@lmin=\mili@u\else\v@lmax=\mili@u\fi%
    \delt@=\v@lmax\advance\delt@-\v@lmin\ifdim\delt@>\epsil@n\sqrt@@\fi}
\ctr@ld@f\def\extrairelepremi@r#1\de#2{\expandafter\lepremi@r#2@#1#2}
\ctr@ld@f\def\lepremi@r#1,#2@#3#4{\def#3{#1}\def#4{#2}\ignorespaces}
\ctr@ld@f\def\@cfor#1:=#2\do#3{%
  \edef\@fortemp{#2}%
  \ifx\@fortemp\empty\else\@cforloop#2,\@nil,\@nil\@@#1{#3}\fi}
\ctr@ln@m\@nextwhile
\ctr@ld@f\def\@cforloop#1,#2\@@#3#4{%
  \def#3{#1}%
  \ifx#3\Fig@nnil\let\@nextwhile=\Fig@fornoop\else#4\relax\let\@nextwhile=\@cforloop\fi%
  \@nextwhile#2\@@#3{#4}}

\ctr@ld@f\def\@ecfor#1:=#2\do#3{%
  \def\@@cfor{\@cfor#1:=}%
  \edef\@@@cfor{#2}%
  \expandafter\@@cfor\@@@cfor\do{#3}}
\ctr@ld@f\def\Fig@nnil{\@nil}
\ctr@ld@f\def\Fig@fornoop#1\@@#2#3{}
\ctr@ln@m\list@@rg
\ctr@ld@f\def\trtlis@rg#1#2{\def\list@@rg{#1}%
    \@ecfor\p@rv@l:=\list@@rg\do{\expandafter#2\p@rv@l|}}
\ctr@ln@w{newbox}\b@xvisu
\ctr@ln@w{newtoks}\let@xte
\ctr@ln@w{newif}\ifitis@K
\ctr@ln@w{newcount}\s@mme
\ctr@ln@w{newcount}\l@mbd@un \ctr@ln@w{newcount}\l@mbd@de
\ctr@ln@w{newcount}\superc@ntr@l\superc@ntr@l=\@ne        % Controle impose
\ctr@ln@w{newcount}\typec@ntr@l\typec@ntr@l=\superc@ntr@l % Controle souhaite
\ctr@ln@w{newdimen}\v@lX  \ctr@ln@w{newdimen}\v@lY  \ctr@ln@w{newdimen}\v@lZ
\ctr@ln@w{newdimen}\v@lXa \ctr@ln@w{newdimen}\v@lYa \ctr@ln@w{newdimen}\v@lZa
\ctr@ln@w{newdimen}\unit@\unit@=\p@ % Initialisation a la valeur par defaut.
\ctr@ld@f\def\unit@util{pt}
\ctr@ld@f\def\ptT@ptps{0.996264 }
\ctr@ld@f\def\ptpsT@pt{1.00375 }
\ctr@ld@f\def\ptT@unit@{1} % Initialisation correspondant a la valeur par defaut de \unit@
\ctr@ld@f\def\setunit@#1{\def\unit@util{#1}\setunit@@#1:\invers@{\result@t}{\unit@}%
    \edef\ptT@unit@{\repdecn@mb\result@t}}
\ctr@ld@f\def\setunit@@#1#2:{\ifcat#1a\unit@=\@ne#1#2\else\unit@=#1#2\fi}
\ctr@ld@f\def\d@fm@cdim#1#2{{\v@leur=#2\v@leur=\ptT@unit@\v@leur\xdef#1{\repdecn@mb\v@leur}}}
\ctr@ln@w{newif}\ifBdingB@x\BdingB@xtrue
\ctr@ln@w{newdimen}\c@@rdXmin \ctr@ln@w{newdimen}\c@@rdYmin  % Dimensions de la BoundingBox
\ctr@ln@w{newdimen}\c@@rdXmax \ctr@ln@w{newdimen}\c@@rdYmax
\ctr@ld@f\def\b@undb@x#1#2{\ifBdingB@x%
    \relax\ifdim#1<\c@@rdXmin\global\c@@rdXmin=#1\fi%
    \relax\ifdim#2<\c@@rdYmin\global\c@@rdYmin=#2\fi%
    \relax\ifdim#1>\c@@rdXmax\global\c@@rdXmax=#1\fi%
    \relax\ifdim#2>\c@@rdYmax\global\c@@rdYmax=#2\fi\fi}
\ctr@ld@f\def\b@undb@xP#1{{\Figg@tXY{#1}\b@undb@x{\v@lX}{\v@lY}}}
\ctr@ld@f\def\ellBB@x#1;#2,#3(#4,#5,#6){{\s@uvc@ntr@l\et@tellBB@x%
    \setc@ntr@l{2}\figptell-2::#1;#2,#3(#4,#6)\b@undb@xP{-2}%
    \figptell-2::#1;#2,#3(#5,#6)\b@undb@xP{-2}%
    \c@ssin{\C@}{\S@}{#6}\v@lmin=\C@ pt\v@lmax=\S@ pt%
    \mili@u=#3\v@lmin\delt@=#2\v@lmax\arct@n\v@leur(\delt@,\mili@u)%
    \mili@u=-#3\v@lmax\delt@=#2\v@lmin\arct@n\c@rre(\delt@,\mili@u)%
    \v@leur=\rdT@deg\v@leur\advance\v@leur-\DePI@deg%
    \c@rre=\rdT@deg\c@rre\advance\c@rre-\DePI@deg%
    \v@lmin=#4pt\v@lmax=#5pt%
    \loop\ifdim\v@leur<\v@lmax\ifdim\v@leur>\v@lmin%
    \edef\@ngle{\repdecn@mb\v@leur}\figptell-2::#1;#2,#3(\@ngle,#6)%
    \b@undb@xP{-2}\fi\advance\v@leur\PI@deg\repeat%
    \loop\ifdim\c@rre<\v@lmax\ifdim\c@rre>\v@lmin%
    \edef\@ngle{\repdecn@mb\c@rre}\figptell-2::#1;#2,#3(\@ngle,#6)%
    \b@undb@xP{-2}\fi\advance\c@rre\PI@deg\repeat%
    \resetc@ntr@l\et@tellBB@x}\ignorespaces}
\ctr@ld@f\def\initb@undb@x{\c@@rdXmin=\maxdimen\c@@rdYmin=\maxdimen%
    \c@@rdXmax=-\maxdimen\c@@rdYmax=-\maxdimen}
\ctr@ld@f\def\c@ntr@lnum#1{%
    \relax\ifnum\typec@ntr@l=\@ne%
    \ifnum#1<\z@%
    \immediate\write16{*** Forbidden point number (#1). Abort.}\end\fi\fi%
    \set@bjc@de{#1}}
\ctr@ln@m\objc@de
\ctr@ld@f\def\set@bjc@de#1{\edef\objc@de{@BJ\ifnum#1<\z@ M\romannumeral-#1\else\romannumeral#1\fi}}
\s@mme=\m@ne\loop\ifnum\s@mme>-19
  \set@bjc@de{\s@mme}\ctr@lcsn@m\objc@de\ctr@lcsn@m{\objc@de T}
\advance\s@mme\m@ne\repeat
\s@mme=\@ne\loop\ifnum\s@mme<6
  \set@bjc@de{\s@mme}\ctr@lcsn@m\objc@de\ctr@lcsn@m{\objc@de T}
\advance\s@mme\@ne\repeat
\ctr@ld@f\def\setc@ntr@l#1{\ifnum\superc@ntr@l>#1\typec@ntr@l=\superc@ntr@l%
    \else\typec@ntr@l=#1\fi}
\ctr@ld@f\def\resetc@ntr@l#1{\global\superc@ntr@l=#1\setc@ntr@l{#1}}
\ctr@ld@f\def\s@uvc@ntr@l#1{\edef#1{\the\superc@ntr@l}}
\ctr@ln@m\c@lproscal
\ctr@ld@f\def\c@lproscalDD#1[#2,#3]{{\Figg@tXY{#2}%
    \edef\Xu@{\repdecn@mb{\v@lX}}\edef\Yu@{\repdecn@mb{\v@lY}}\Figg@tXY{#3}%
    \global\result@t=\Xu@\v@lX\global\advance\result@t\Yu@\v@lY}#1=\result@t}
\ctr@ld@f\def\c@lproscalTD#1[#2,#3]{{\Figg@tXY{#2}\edef\Xu@{\repdecn@mb{\v@lX}}%
    \edef\Yu@{\repdecn@mb{\v@lY}}\edef\Zu@{\repdecn@mb{\v@lZ}}%
    \Figg@tXY{#3}\global\result@t=\Xu@\v@lX\global\advance\result@t\Yu@\v@lY%
    \global\advance\result@t\Zu@\v@lZ}#1=\result@t}
\ctr@ld@f\def\c@lprovec#1{%
    \det@rmC\v@lZa(\v@lX,\v@lY,\v@lmin,\v@lmax)%
    \det@rmC\v@lXa(\v@lY,\v@lZ,\v@lmax,\v@leur)%
    \det@rmC\v@lYa(\v@lZ,\v@lX,\v@leur,\v@lmin)%
    \Figv@ctCreg#1(\v@lXa,\v@lYa,\v@lZa)}
\ctr@ld@f\def\det@rm#1[#2,#3]{{\Figg@tXY{#2}\Figg@tXYa{#3}%
    \delt@=\repdecn@mb{\v@lX}\v@lYa\advance\delt@-\repdecn@mb{\v@lY}\v@lXa%
    \global\result@t=\delt@}#1=\result@t}
\ctr@ld@f\def\det@rmC#1(#2,#3,#4,#5){{\global\result@t=\repdecn@mb{#2}#5%
    \global\advance\result@t-\repdecn@mb{#3}#4}#1=\result@t}
\ctr@ld@f\def\getredf@ctDD#1(#2,#3){{\maxim@m{\v@lXa}{-#2}{#2}\maxim@m{\v@lYa}{-#3}{#3}%
    \maxim@m{\v@lXa}{\v@lXa}{\v@lYa}% \v@lXa = ||X||inf
    \ifdim\v@lXa>\@xci pt\divide\v@lXa\@xci%
    \p@rtentiere{\p@rtent}{\v@lXa}\advance\p@rtent\@ne\else\p@rtent=\@ne\fi%
    \global\result@tent=\p@rtent}#1=\result@tent\ignorespaces}
\ctr@ld@f\def\getredf@ctTD#1(#2,#3,#4){{\maxim@m{\v@lXa}{-#2}{#2}\maxim@m{\v@lYa}{-#3}{#3}%
    \maxim@m{\v@lZa}{-#4}{#4}\maxim@m{\v@lXa}{\v@lXa}{\v@lYa}%
    \maxim@m{\v@lXa}{\v@lXa}{\v@lZa}% \v@lXa = ||X||inf
    \ifdim\v@lXa>\@lxxiv pt\divide\v@lXa\@lxxiv%
    \p@rtentiere{\p@rtent}{\v@lXa}\advance\p@rtent\@ne\else\p@rtent=\@ne\fi%
    \global\result@tent=\p@rtent}#1=\result@tent\ignorespaces}
\ctr@ld@f\def\FigptintercircB@zDD#1:#2:#3,#4[#5,#6,#7,#8]{{\s@uvc@ntr@l\et@tfigptintercircB@zDD%
    \setc@ntr@l{2}\figvectPDD-1[#5,#8]\Figg@tXY{-1}\getredf@ctDD\f@ctech(\v@lX,\v@lY)%
    \mili@u=#4\unit@\divide\mili@u\f@ctech\c@rre=\repdecn@mb{\mili@u}\mili@u%
    \figptBezierDD-5::#3[#5,#6,#7,#8]%
    \v@lmin=#3\p@\v@lmax=\v@lmin\advance\v@lmax0.1\p@%
    \loop\edef\T@{\repdecn@mb{\v@lmax}}\figptBezierDD-2::\T@[#5,#6,#7,#8]%
    \figvectPDD-1[-5,-2]\n@rmeucCDD{\delt@}{-1}\ifdim\delt@<\c@rre\v@lmin=\v@lmax%
    \advance\v@lmax0.1\p@\repeat%
    \loop\mili@u=\v@lmin\advance\mili@u\v@lmax%
    \divide\mili@u\tw@\edef\T@{\repdecn@mb{\mili@u}}\figptBezierDD-2::\T@[#5,#6,#7,#8]%
    \figvectPDD-1[-5,-2]\n@rmeucCDD{\delt@}{-1}\ifdim\delt@>\c@rre\v@lmax=\mili@u%
    \else\v@lmin=\mili@u\fi\v@leur=\v@lmax\advance\v@leur-\v@lmin%
    \ifdim\v@leur>\epsil@n\repeat\figptcopyDD#1:#2/-2/%
    \resetc@ntr@l\et@tfigptintercircB@zDD}\ignorespaces}
\ctr@ln@m\figptinterlines
\ctr@ld@f\def\inters@cDD#1:#2[#3,#4;#5,#6]{{\s@uvc@ntr@l\et@tinters@cDD%
    \setc@ntr@l{2}\vecunit@{-1}{#4}\vecunit@{-2}{#6}%
    \Figg@tXY{-1}\setc@ntr@l{1}\Figg@tXYa{#3}%
    \edef\A@{\repdecn@mb{\v@lX}}\edef\B@{\repdecn@mb{\v@lY}}%
    \v@lmin=\B@\v@lXa\advance\v@lmin-\A@\v@lYa%
    \Figg@tXYa{#5}\setc@ntr@l{2}\Figg@tXY{-2}%
    \edef\C@{\repdecn@mb{\v@lX}}\edef\D@{\repdecn@mb{\v@lY}}%
    \v@lmax=\D@\v@lXa\advance\v@lmax-\C@\v@lYa%
    \delt@=\A@\v@lY\advance\delt@-\B@\v@lX%
    \invers@{\v@leur}{\delt@}\edef\v@ldelta{\repdecn@mb{\v@leur}}%
    \v@lXa=\A@\v@lmax\advance\v@lXa-\C@\v@lmin%
    \v@lYa=\B@\v@lmax\advance\v@lYa-\D@\v@lmin%
    \v@lXa=\v@ldelta\v@lXa\v@lYa=\v@ldelta\v@lYa%
    \setc@ntr@l{1}\Figp@intregDD#1:{#2}(\v@lXa,\v@lYa)%
    \resetc@ntr@l\et@tinters@cDD}\ignorespaces}
\ctr@ld@f\def\inters@cTD#1:#2[#3,#4;#5,#6]{{\s@uvc@ntr@l\et@tinters@cTD%
    \setc@ntr@l{2}\figvectNVTD-1[#4,#6]\figvectNVTD-2[#6,-1]\figvectPTD-1[#3,#5]%
    \r@pPSTD\v@leur[-2,-1,#4]\edef\v@lcoef{\repdecn@mb{\v@leur}}%
    \figpttraTD#1:{#2}=#3/\v@lcoef,#4/\resetc@ntr@l\et@tinters@cTD}\ignorespaces}
\ctr@ld@f\def\r@pPSTD#1[#2,#3,#4]{{\Figg@tXY{#2}\edef\Xu@{\repdecn@mb{\v@lX}}%
    \edef\Yu@{\repdecn@mb{\v@lY}}\edef\Zu@{\repdecn@mb{\v@lZ}}%
    \Figg@tXY{#3}\v@lmin=\Xu@\v@lX\advance\v@lmin\Yu@\v@lY\advance\v@lmin\Zu@\v@lZ%
    \Figg@tXY{#4}\v@lmax=\Xu@\v@lX\advance\v@lmax\Yu@\v@lY\advance\v@lmax\Zu@\v@lZ%
    \invers@{\v@leur}{\v@lmax}\global\result@t=\repdecn@mb{\v@leur}\v@lmin}%
    #1=\result@t}
\ctr@ln@m\n@rminf
\ctr@ld@f\def\n@rminfDD#1#2{{\Figg@tXY{#2}\maxim@m{\v@lX}{\v@lX}{-\v@lX}%
    \maxim@m{\v@lY}{\v@lY}{-\v@lY}\maxim@m{\global\result@t}{\v@lX}{\v@lY}}%
    #1=\result@t}
\ctr@ld@f\def\n@rminfTD#1#2{{\Figg@tXY{#2}\maxim@m{\v@lX}{\v@lX}{-\v@lX}%
    \maxim@m{\v@lY}{\v@lY}{-\v@lY}\maxim@m{\v@lZ}{\v@lZ}{-\v@lZ}%
    \maxim@m{\v@lX}{\v@lX}{\v@lY}\maxim@m{\global\result@t}{\v@lX}{\v@lZ}}%
    #1=\result@t}
\ctr@ld@f\def\n@rmeucCDD#1#2{\Figg@tXY{#2}\divide\v@lX\f@ctech\divide\v@lY\f@ctech%
    #1=\repdecn@mb{\v@lX}\v@lX\v@lX=\repdecn@mb{\v@lY}\v@lY\advance#1\v@lX}
\ctr@ld@f\def\n@rmeucCTD#1#2{\Figg@tXY{#2}%
    \divide\v@lX\f@ctech\divide\v@lY\f@ctech\divide\v@lZ\f@ctech%
    #1=\repdecn@mb{\v@lX}\v@lX\v@lX=\repdecn@mb{\v@lY}\v@lY\advance#1\v@lX%
    \v@lX=\repdecn@mb{\v@lZ}\v@lZ\advance#1\v@lX}
\ctr@ln@m\n@rmeucSV
\ctr@ld@f\def\n@rmeucSVDD#1#2{{\Figg@tXY{#2}%
    \v@lXa=\repdecn@mb{\v@lX}\v@lX\v@lYa=\repdecn@mb{\v@lY}\v@lY%
    \advance\v@lXa\v@lYa\sqrt@{\global\result@t}{\v@lXa}}#1=\result@t}
\ctr@ld@f\def\n@rmeucSVTD#1#2{{\Figg@tXY{#2}\v@lXa=\repdecn@mb{\v@lX}\v@lX%
    \v@lYa=\repdecn@mb{\v@lY}\v@lY\v@lZa=\repdecn@mb{\v@lZ}\v@lZ%
    \advance\v@lXa\v@lYa\advance\v@lXa\v@lZa\sqrt@{\global\result@t}{\v@lXa}}#1=\result@t}
\ctr@ln@m\n@rmeuc
\ctr@ld@f\def\n@rmeucDD#1#2{{\Figg@tXY{#2}\getredf@ctDD\f@ctech(\v@lX,\v@lY)%
    \divide\v@lX\f@ctech\divide\v@lY\f@ctech%
    \v@lXa=\repdecn@mb{\v@lX}\v@lX\v@lYa=\repdecn@mb{\v@lY}\v@lY%
    \advance\v@lXa\v@lYa\sqrt@{\global\result@t}{\v@lXa}%
    \global\multiply\result@t\f@ctech}#1=\result@t}
\ctr@ld@f\def\n@rmeucTD#1#2{{\Figg@tXY{#2}\getredf@ctTD\f@ctech(\v@lX,\v@lY,\v@lZ)%
    \divide\v@lX\f@ctech\divide\v@lY\f@ctech\divide\v@lZ\f@ctech%
    \v@lXa=\repdecn@mb{\v@lX}\v@lX%
    \v@lYa=\repdecn@mb{\v@lY}\v@lY\v@lZa=\repdecn@mb{\v@lZ}\v@lZ%
    \advance\v@lXa\v@lYa\advance\v@lXa\v@lZa\sqrt@{\global\result@t}{\v@lXa}%
    \global\multiply\result@t\f@ctech}#1=\result@t}
\ctr@ln@m\vecunit@
\ctr@ld@f\def\vecunit@DD#1#2{{\Figg@tXY{#2}\getredf@ctDD\f@ctech(\v@lX,\v@lY)%
    \divide\v@lX\f@ctech\divide\v@lY\f@ctech%
    \Figv@ctCreg#1(\v@lX,\v@lY)\n@rmeucSV{\v@lYa}{#1}%
    \invers@{\v@lXa}{\v@lYa}\edef\v@lv@lXa{\repdecn@mb{\v@lXa}}%
    \v@lX=\v@lv@lXa\v@lX\v@lY=\v@lv@lXa\v@lY%
    \Figv@ctCreg#1(\v@lX,\v@lY)\multiply\v@lYa\f@ctech\global\result@t=\v@lYa}}
\ctr@ld@f\def\vecunit@TD#1#2{{\Figg@tXY{#2}\getredf@ctTD\f@ctech(\v@lX,\v@lY,\v@lZ)%
    \divide\v@lX\f@ctech\divide\v@lY\f@ctech\divide\v@lZ\f@ctech%
    \Figv@ctCreg#1(\v@lX,\v@lY,\v@lZ)\n@rmeucSV{\v@lYa}{#1}%
    \invers@{\v@lXa}{\v@lYa}\edef\v@lv@lXa{\repdecn@mb{\v@lXa}}%
    \v@lX=\v@lv@lXa\v@lX\v@lY=\v@lv@lXa\v@lY\v@lZ=\v@lv@lXa\v@lZ%
    \Figv@ctCreg#1(\v@lX,\v@lY,\v@lZ)\multiply\v@lYa\f@ctech\global\result@t=\v@lYa}}
\ctr@ld@f\def\vecunitC@TD[#1,#2]{\Figg@tXYa{#1}\Figg@tXY{#2}%
    \advance\v@lX-\v@lXa\advance\v@lY-\v@lYa\advance\v@lZ-\v@lZa\c@lvecunitTD}
\ctr@ld@f\def\vecunitCV@TD#1{\Figg@tXY{#1}\c@lvecunitTD}
\ctr@ld@f\def\c@lvecunitTD{\getredf@ctTD\f@ctech(\v@lX,\v@lY,\v@lZ)%
    \divide\v@lX\f@ctech\divide\v@lY\f@ctech\divide\v@lZ\f@ctech%
    \v@lXa=\repdecn@mb{\v@lX}\v@lX%
    \v@lYa=\repdecn@mb{\v@lY}\v@lY\v@lZa=\repdecn@mb{\v@lZ}\v@lZ%
    \advance\v@lXa\v@lYa\advance\v@lXa\v@lZa\sqrt@{\v@lYa}{\v@lXa}%
    \invers@{\v@lXa}{\v@lYa}\edef\v@lv@lXa{\repdecn@mb{\v@lXa}}%
    \v@lX=\v@lv@lXa\v@lX\v@lY=\v@lv@lXa\v@lY\v@lZ=\v@lv@lXa\v@lZ}
\ctr@ln@m\figgetangle
\ctr@ld@f\def\figgetangleDD#1[#2,#3,#4]{\ifps@cri{\s@uvc@ntr@l\et@tfiggetangleDD\setc@ntr@l{2}%
    \figvectPDD-1[#2,#3]\figvectPDD-2[#2,#4]\vecunit@{-1}{-1}%
    \c@lproscalDD\delt@[-2,-1]\figvectNVDD-1[-1]\c@lproscalDD\v@leur[-2,-1]%
    \arct@n\v@lmax(\delt@,\v@leur)\v@lmax=\rdT@deg\v@lmax%
    \ifdim\v@lmax<\z@\advance\v@lmax\DePI@deg\fi\xdef#1{\repdecn@mb{\v@lmax}}%
    \resetc@ntr@l\et@tfiggetangleDD}\ignorespaces\fi}
\ctr@ld@f\def\figgetangleTD#1[#2,#3,#4,#5]{\ifps@cri{\s@uvc@ntr@l\et@tfiggetangleTD\setc@ntr@l{2}%
    \figvectPTD-1[#2,#3]\figvectPTD-2[#2,#5]\figvectNVTD-3[-1,-2]%
    \figvectPTD-2[#2,#4]\figvectNVTD-4[-3,-1]%
    \vecunit@{-1}{-1}\c@lproscalTD\delt@[-2,-1]\c@lproscalTD\v@leur[-2,-4]%
    \arct@n\v@lmax(\delt@,\v@leur)\v@lmax=\rdT@deg\v@lmax%
    \ifdim\v@lmax<\z@\advance\v@lmax\DePI@deg\fi\xdef#1{\repdecn@mb{\v@lmax}}%
    \resetc@ntr@l\et@tfiggetangleTD}\ignorespaces\fi}    
\ctr@ld@f\def\figgetdist#1[#2,#3]{\ifps@cri{\s@uvc@ntr@l\et@tfiggetdist\setc@ntr@l{2}%
    \figvectP-1[#2,#3]\n@rmeuc{\v@lX}{-1}\v@lX=\ptT@unit@\v@lX\xdef#1{\repdecn@mb{\v@lX}}%
    \resetc@ntr@l\et@tfiggetdist}\ignorespaces\fi}
\ctr@ld@f\def\Figg@tT#1{\c@ntr@lnum{#1}%
    {\expandafter\expandafter\expandafter\extr@ctT\csname\objc@de\endcsname:%
     \ifnum\B@@ltxt=\z@\ptn@me{#1}\else\csname\objc@de T\endcsname\fi}}
\ctr@ld@f\def\extr@ctT#1,#2,#3/#4:{\def\B@@ltxt{#3}}
\ctr@ld@f\def\Figg@tXY#1{\c@ntr@lnum{#1}%
    \expandafter\expandafter\expandafter\extr@ctC\csname\objc@de\endcsname:}
\ctr@ln@m\extr@ctC
\ctr@ld@f\def\extr@ctCDD#1/#2,#3,#4:{\v@lX=#2\v@lY=#3}
\ctr@ld@f\def\extr@ctCTD#1/#2,#3,#4:{\v@lX=#2\v@lY=#3\v@lZ=#4}
\ctr@ld@f\def\Figg@tXYa#1{\c@ntr@lnum{#1}%
    \expandafter\expandafter\expandafter\extr@ctCa\csname\objc@de\endcsname:}
\ctr@ln@m\extr@ctCa
\ctr@ld@f\def\extr@ctCaDD#1/#2,#3,#4:{\v@lXa=#2\v@lYa=#3}
\ctr@ld@f\def\extr@ctCaTD#1/#2,#3,#4:{\v@lXa=#2\v@lYa=#3\v@lZa=#4}
\ctr@ln@m\t@xt@
\ctr@ld@f\def\figinit#1{\t@stc@tcodech@nge\initpr@lim\Figinit@#1,:\initpss@ttings\ignorespaces}
\ctr@ld@f\def\Figinit@#1,#2:{\setunit@{#1}\def\t@xt@{#2}\ifx\t@xt@\empty\else\Figinit@@#2:\fi}
\ctr@ld@f\def\Figinit@@#1#2:{\if#12 \else\Figs@tproj{#1}\initTD@\fi}
\ctr@ln@w{newif}\ifTr@isDim
\ctr@ld@f\def\UnD@fined{UNDEFINED}
\ctr@ld@f\def\ifundefined#1{\expandafter\ifx\csname#1\endcsname\relax}
\ctr@ln@m\@utoFN
\ctr@ln@m\@utoFInDone
\ctr@ln@m\disob@unit
\ctr@ld@f\def\initpr@lim{\initb@undb@x\figsetmark{}\figsetptname{$A_{##1}$}\def\Sc@leFact{1}%
    \initDD@\figsetroundcoord{yes}\ps@critrue\expandafter\setupd@te\defaultupdate:%
    \edef\disob@unit{\UnD@fined}\edef\t@rgetpt{\UnD@fined}\gdef\@utoFInDone{1}\gdef\@utoFN{0}}
\ctr@ld@f\def\initDD@{\Tr@isDimfalse%
    \ifPDFm@ke%
     \let\Ps@rcerc=\Ps@rcercBz%
     \let\Ps@rell=\Ps@rellBz%
    \fi
    \let\c@lDCUn=\c@lDCUnDD%
    \let\c@lDCDeux=\c@lDCDeuxDD%
    \let\c@ldefproj=\relax%
    \let\c@lproscal=\c@lproscalDD%
    \let\c@lprojSP=\relax%
    \let\extr@ctC=\extr@ctCDD%
    \let\extr@ctCa=\extr@ctCaDD%
    \let\extr@ctCF=\extr@ctCFDD%
    \let\Figp@intreg=\Figp@intregDD%
    \let\Figpts@xes=\Figpts@xesDD%
    \let\n@rmeucSV=\n@rmeucSVDD\let\n@rmeuc=\n@rmeucDD\let\n@rminf=\n@rminfDD%
    \let\pr@dMatV=\pr@dMatVDD%
    \let\ps@xes=\ps@xesDD%
    \let\vecunit@=\vecunit@DD%
    \let\figcoord=\figcoordDD%
    \let\figgetangle=\figgetangleDD%
    \let\figpt=\figptDD%
    \let\figptBezier=\figptBezierDD%
    \let\figptbary=\figptbaryDD%
    \let\figptcirc=\figptcircDD%
    \let\figptcircumcenter=\figptcircumcenterDD%
    \let\figptcopy=\figptcopyDD%
    \let\figptcurvcenter=\figptcurvcenterDD%
    \let\figptell=\figptellDD%
    \let\figptendnormal=\figptendnormalDD%
    \let\figptinterlineplane=\figptinterlineplaneDD%
    \let\figptinterlines=\inters@cDD%
    \let\figptorthocenter=\figptorthocenterDD%
    \let\figptorthoprojline=\figptorthoprojlineDD%
    \let\figptorthoprojplane=\figptorthoprojplaneDD%
    \let\figptrot=\figptrotDD%
    \let\figptscontrol=\figptscontrolDD%
    \let\figptsintercirc=\figptsintercircDD%
    \let\figptsinterlinell=\figptsinterlinellDD%
    \let\figptsorthoprojline=\figptsorthoprojlineDD%
    \let\figptorthoprojplane=\figptorthoprojplaneDD%
    \let\figptsrot=\figptsrotDD%
    \let\figptssym=\figptssymDD%
    \let\figptstra=\figptstraDD%
    \let\figptsym=\figptsymDD%
    \let\figpttraC=\figpttraCDD%
    \let\figpttra=\figpttraDD%
    \let\figptvisilimSL=\figptvisilimSLDD%
    \let\figsetobdist=\figsetobdistDD%
    \let\figsettarget=\figsettargetDD%
    \let\figsetview=\figsetviewDD%
    \let\figvectDBezier=\figvectDBezierDD%
    \let\figvectN=\figvectNDD%
    \let\figvectNV=\figvectNVDD%
    \let\figvectP=\figvectPDD%
    \let\figvectU=\figvectUDD%
    \let\psarccircP=\psarccircPDD%
    \let\psarccirc=\psarccircDD%
    \let\psarcell=\psarcellDD%
    \let\psarcellPA=\psarcellPADD%
    \let\psarrowBezier=\psarrowBezierDD%
    \let\psarrowcircP=\psarrowcircPDD%
    \let\psarrowcirc=\psarrowcircDD%
    \let\psarrowhead=\psarrowheadDD%
    \let\psarrow=\psarrowDD%
    \let\psBezier=\psBezierDD%
    \let\pscirc=\pscircDD%
    \let\pscurve=\pscurveDD%
    \let\psnormal=\psnormalDD%
    }
\ctr@ld@f\def\initTD@{\Tr@isDimtrue\initb@undb@xTD\newt@rgetptfalse\newdis@bfalse%
    \let\c@lDCUn=\c@lDCUnTD%
    \let\c@lDCDeux=\c@lDCDeuxTD%
    \let\c@ldefproj=\c@ldefprojTD%
    \let\c@lproscal=\c@lproscalTD%
    \let\extr@ctC=\extr@ctCTD%
    \let\extr@ctCa=\extr@ctCaTD%
    \let\extr@ctCF=\extr@ctCFTD%
    \let\Figp@intreg=\Figp@intregTD%
    \let\Figpts@xes=\Figpts@xesTD%
    \let\n@rmeucSV=\n@rmeucSVTD\let\n@rmeuc=\n@rmeucTD\let\n@rminf=\n@rminfTD%
    \let\pr@dMatV=\pr@dMatVTD%
    \let\ps@xes=\ps@xesTD%
    \let\vecunit@=\vecunit@TD%
    \let\figcoord=\figcoordTD%
    \let\figgetangle=\figgetangleTD%
    \let\figpt=\figptTD%
    \let\figptBezier=\figptBezierTD%
    \let\figptbary=\figptbaryTD%
    \let\figptcirc=\figptcircTD%
    \let\figptcircumcenter=\figptcircumcenterTD%
    \let\figptcopy=\figptcopyTD%
    \let\figptcurvcenter=\figptcurvcenterTD%
    \let\figptinterlineplane=\figptinterlineplaneTD%
    \let\figptinterlines=\inters@cTD%
    \let\figptorthocenter=\figptorthocenterTD%
    \let\figptorthoprojline=\figptorthoprojlineTD%
    \let\figptorthoprojplane=\figptorthoprojplaneTD%
    \let\figptrot=\figptrotTD%
    \let\figptscontrol=\figptscontrolTD%
    \let\figptsintercirc=\figptsintercircTD%
    \let\figptsorthoprojline=\figptsorthoprojlineTD%
    \let\figptsorthoprojplane=\figptsorthoprojplaneTD%
    \let\figptsrot=\figptsrotTD%
    \let\figptssym=\figptssymTD%
    \let\figptstra=\figptstraTD%
    \let\figptsym=\figptsymTD%
    \let\figpttraC=\figpttraCTD%
    \let\figpttra=\figpttraTD%
    \let\figptvisilimSL=\figptvisilimSLTD%
    \let\figsetobdist=\figsetobdistTD%
    \let\figsettarget=\figsettargetTD%
    \let\figsetview=\figsetviewTD%
    \let\figvectDBezier=\figvectDBezierTD%
    \let\figvectN=\figvectNTD%
    \let\figvectNV=\figvectNVTD%
    \let\figvectP=\figvectPTD%
    \let\figvectU=\figvectUTD%
    \let\psarccircP=\psarccircPTD%
    \let\psarccirc=\psarccircTD%
    \let\psarcell=\psarcellTD%
    \let\psarcellPA=\psarcellPATD%
    \let\psarrowBezier=\psarrowBezierTD%
    \let\psarrowcircP=\psarrowcircPTD%
    \let\psarrowcirc=\psarrowcircTD%
    \let\psarrowhead=\psarrowheadTD%
    \let\psarrow=\psarrowTD%
    \let\psBezier=\psBezierTD%
    \let\pscirc=\pscircTD%
    \let\pscurve=\pscurveTD%
    }
\ctr@ld@f\def\un@v@ilable#1{\immediate\write16{*** The macro #1 is not available in the current context.}}
\ctr@ld@f\def\figinsert#1{{\def\t@xt@{#1}\relax%
    \ifx\t@xt@\empty\ifnum\@utoFInDone>\z@\Figinsert@\DefGIfilen@me,:\fi%
    \else\expandafter\FiginsertNu@#1 :\fi}\ignorespaces}
\ctr@ld@f\def\FiginsertNu@#1 #2:{\def\t@xt@{#1}\relax\ifx\t@xt@\empty\def\t@xt@{#2}%
    \ifx\t@xt@\empty\ifnum\@utoFInDone>\z@\Figinsert@\DefGIfilen@me,:\fi%
    \else\FiginsertNu@#2:\fi\else\expandafter\FiginsertNd@#1 #2:\fi}
\ctr@ld@f\def\FiginsertNd@#1#2:{\ifcat#1a\Figinsert@#1#2,:\else%
    \ifnum\@utoFInDone>\z@\Figinsert@\DefGIfilen@me,#1#2,:\fi\fi}
\ctr@ln@m\Sc@leFact
\ctr@ld@f\def\Figinsert@#1,#2:{\def\t@xt@{#2}\ifx\t@xt@\empty\xdef\Sc@leFact{1}\else%
    \X@rgdeux@#2\xdef\Sc@leFact{\@rgdeux}\fi%
    \Figdisc@rdLTS{#1}{\t@xt@}\@psfgetbb{\t@xt@}%
    \v@lX=\@psfllx\p@\v@lX=\ptpsT@pt\v@lX\v@lX=\Sc@leFact\v@lX%
    \v@lY=\@psflly\p@\v@lY=\ptpsT@pt\v@lY\v@lY=\Sc@leFact\v@lY%
    \b@undb@x{\v@lX}{\v@lY}%
    \v@lX=\@psfurx\p@\v@lX=\ptpsT@pt\v@lX\v@lX=\Sc@leFact\v@lX%
    \v@lY=\@psfury\p@\v@lY=\ptpsT@pt\v@lY\v@lY=\Sc@leFact\v@lY%
    \b@undb@x{\v@lX}{\v@lY}%
    \ifPDFm@ke\Figinclud@PDF{\t@xt@}{\Sc@leFact}\else%
    \v@lX=\c@nt pt\v@lX=\Sc@leFact\v@lX\edef\F@ct{\repdecn@mb{\v@lX}}%
    \ifx\TeXturesonMacOSltX\special{postscriptfile #1 vscale=\F@ct\space hscale=\F@ct}%
    \else\includegraphics{#1}\fi\fi%
    \message{[\t@xt@]}\ignorespaces}
\ctr@ld@f\def\Figdisc@rdLTS#1#2{\expandafter\Figdisc@rdLTS@#1 :#2}
\ctr@ld@f\def\Figdisc@rdLTS@#1 #2:#3{\def#3{#1}\relax\ifx#3\empty\expandafter\Figdisc@rdLTS@#2:#3\fi}
\ctr@ld@f\def\figinsertE#1{\FiginsertE@#1,:\ignorespaces}
\ctr@ld@f\def\FiginsertE@#1,#2:{{\def\t@xt@{#2}\ifx\t@xt@\empty\xdef\Sc@leFact{1}\else%
    \X@rgdeux@#2\xdef\Sc@leFact{\@rgdeux}\fi%
    \Figdisc@rdLTS{#1}{\t@xt@}\pdfximage{\t@xt@}%
    \setbox\Gb@x=\hbox{\pdfrefximage\pdflastximage}%
    \v@lX=\z@\v@lY=-\Sc@leFact\dp\Gb@x\b@undb@x{\v@lX}{\v@lY}%
    \advance\v@lX\Sc@leFact\wd\Gb@x\advance\v@lY\Sc@leFact\dp\Gb@x%
    \advance\v@lY\Sc@leFact\ht\Gb@x\b@undb@x{\v@lX}{\v@lY}%
    \v@lX=\Sc@leFact\wd\Gb@x\pdfximage width \v@lX {\t@xt@}%
    \rlap{\pdfrefximage\pdflastximage}\message{[\t@xt@]}}\ignorespaces}
\ctr@ld@f\def\X@rgdeux@#1,{\edef\@rgdeux{#1}}
\ctr@ln@m\figpt
\ctr@ld@f\def\figptDD#1:#2(#3,#4){\ifps@cri\c@ntr@lnum{#1}%
    {\v@lX=#3\unit@\v@lY=#4\unit@\Fig@dmpt{#2}{\z@}}\ignorespaces\fi}
\ctr@ld@f\def\Fig@dmpt#1#2{\def\t@xt@{#1}\ifx\t@xt@\empty\def\B@@ltxt{\z@}%
    \else\expandafter\gdef\csname\objc@de T\endcsname{#1}\def\B@@ltxt{\@ne}\fi%
    \expandafter\xdef\csname\objc@de\endcsname{\ifitis@vect@r\C@dCl@svect%
    \else\C@dCl@spt\fi,\z@,\B@@ltxt/\the\v@lX,\the\v@lY,#2}}
\ctr@ld@f\def\C@dCl@spt{P}
\ctr@ld@f\def\C@dCl@svect{V}
\ctr@ln@m\c@@rdYZ
\ctr@ln@m\c@@rdY
\ctr@ld@f\def\figptTD#1:#2(#3,#4){\ifps@cri\c@ntr@lnum{#1}%
    \def\c@@rdYZ{#4,0,0}\extrairelepremi@r\c@@rdY\de\c@@rdYZ%
    \extrairelepremi@r\c@@rdZ\de\c@@rdYZ%
    {\v@lX=#3\unit@\v@lY=\c@@rdY\unit@\v@lZ=\c@@rdZ\unit@\Fig@dmpt{#2}{\the\v@lZ}%
    \b@undb@xTD{\v@lX}{\v@lY}{\v@lZ}}\ignorespaces\fi}
\ctr@ln@m\Figp@intreg
\ctr@ld@f\def\Figp@intregDD#1:#2(#3,#4){\c@ntr@lnum{#1}%
    {\result@t=#4\v@lX=#3\v@lY=\result@t\Fig@dmpt{#2}{\z@}}\ignorespaces}
\ctr@ld@f\def\Figp@intregTD#1:#2(#3,#4){\c@ntr@lnum{#1}%
    \def\c@@rdYZ{#4,\z@,\z@}\extrairelepremi@r\c@@rdY\de\c@@rdYZ%
    \extrairelepremi@r\c@@rdZ\de\c@@rdYZ%
    {\v@lX=#3\v@lY=\c@@rdY\v@lZ=\c@@rdZ\Fig@dmpt{#2}{\the\v@lZ}%
    \b@undb@xTD{\v@lX}{\v@lY}{\v@lZ}}\ignorespaces}
\ctr@ln@m\figptBezier
\ctr@ld@f\def\figptBezierDD#1:#2:#3[#4,#5,#6,#7]{\ifps@cri{\s@uvc@ntr@l\et@tfigptBezierDD%
    \FigptBezier@#3[#4,#5,#6,#7]\Figp@intregDD#1:{#2}(\v@lX,\v@lY)%
    \resetc@ntr@l\et@tfigptBezierDD}\ignorespaces\fi}
\ctr@ld@f\def\figptBezierTD#1:#2:#3[#4,#5,#6,#7]{\ifps@cri{\s@uvc@ntr@l\et@tfigptBezierTD%
    \FigptBezier@#3[#4,#5,#6,#7]\Figp@intregTD#1:{#2}(\v@lX,\v@lY,\v@lZ)%
    \resetc@ntr@l\et@tfigptBezierTD}\ignorespaces\fi}
\ctr@ld@f\def\FigptBezier@#1[#2,#3,#4,#5]{\setc@ntr@l{2}%
    \edef\T@{#1}\v@leur=\p@\advance\v@leur-#1pt\edef\UNmT@{\repdecn@mb{\v@leur}}%
    \figptcopy-4:/#2/\figptcopy-3:/#3/\figptcopy-2:/#4/\figptcopy-1:/#5/%
    \l@mbd@un=-4 \l@mbd@de=-\thr@@\p@rtent=\m@ne\c@lDecast%
    \l@mbd@un=-4 \l@mbd@de=-\thr@@\p@rtent=-\tw@\c@lDecast%
    \l@mbd@un=-4 \l@mbd@de=-\thr@@\p@rtent=-\thr@@\c@lDecast\Figg@tXY{-4}}
\ctr@ln@m\c@lDCUn
\ctr@ld@f\def\c@lDCUnDD#1#2{\Figg@tXY{#1}\v@lX=\UNmT@\v@lX\v@lY=\UNmT@\v@lY%
    \Figg@tXYa{#2}\advance\v@lX\T@\v@lXa\advance\v@lY\T@\v@lYa%
    \Figp@intregDD#1:(\v@lX,\v@lY)}
\ctr@ld@f\def\c@lDCUnTD#1#2{\Figg@tXY{#1}\v@lX=\UNmT@\v@lX\v@lY=\UNmT@\v@lY\v@lZ=\UNmT@\v@lZ%
    \Figg@tXYa{#2}\advance\v@lX\T@\v@lXa\advance\v@lY\T@\v@lYa\advance\v@lZ\T@\v@lZa%
    \Figp@intregTD#1:(\v@lX,\v@lY,\v@lZ)}
\ctr@ld@f\def\c@lDecast{\relax\ifnum\l@mbd@un<\p@rtent\c@lDCUn{\l@mbd@un}{\l@mbd@de}%
    \advance\l@mbd@un\@ne\advance\l@mbd@de\@ne\c@lDecast\fi}
\ctr@ld@f\def\figptmap#1:#2=#3/#4/#5/{\ifps@cri{\s@uvc@ntr@l\et@tfigptmap%
    \setc@ntr@l{2}\figvectP-1[#4,#3]\Figg@tXY{-1}%
    \pr@dMatV/#5/\figpttra#1:{#2}=#4/1,-1/%
    \resetc@ntr@l\et@tfigptmap}\ignorespaces\fi}
\ctr@ln@m\pr@dMatV
\ctr@ld@f\def\pr@dMatVDD/#1,#2;#3,#4/{\v@lXa=#1\v@lX\advance\v@lXa#2\v@lY%
    \v@lYa=#3\v@lX\advance\v@lYa#4\v@lY\Figv@ctCreg-1(\v@lXa,\v@lYa)}
\ctr@ld@f\def\pr@dMatVTD/#1,#2,#3;#4,#5,#6;#7,#8,#9/{%
    \v@lXa=#1\v@lX\advance\v@lXa#2\v@lY\advance\v@lXa#3\v@lZ%
    \v@lYa=#4\v@lX\advance\v@lYa#5\v@lY\advance\v@lYa#6\v@lZ%
    \v@lZa=#7\v@lX\advance\v@lZa#8\v@lY\advance\v@lZa#9\v@lZ%
    \Figv@ctCreg-1(\v@lXa,\v@lYa,\v@lZa)}
\ctr@ln@m\figptbary
\ctr@ld@f\def\figptbaryDD#1:#2[#3;#4]{\ifps@cri{\edef\list@num{#3}\extrairelepremi@r\p@int\de\list@num%
    \s@mme=\z@\@ecfor\c@ef:=#4\do{\advance\s@mme\c@ef}%
    \edef\listec@ef{#4,0}\extrairelepremi@r\c@ef\de\listec@ef%
    \Figg@tXY{\p@int}\divide\v@lX\s@mme\divide\v@lY\s@mme%
    \multiply\v@lX\c@ef\multiply\v@lY\c@ef%
    \@ecfor\p@int:=\list@num\do{\extrairelepremi@r\c@ef\de\listec@ef%
           \Figg@tXYa{\p@int}\divide\v@lXa\s@mme\divide\v@lYa\s@mme%
           \multiply\v@lXa\c@ef\multiply\v@lYa\c@ef%
           \advance\v@lX\v@lXa\advance\v@lY\v@lYa}%
    \Figp@intregDD#1:{#2}(\v@lX,\v@lY)}\ignorespaces\fi}
\ctr@ld@f\def\figptbaryTD#1:#2[#3;#4]{\ifps@cri{\edef\list@num{#3}\extrairelepremi@r\p@int\de\list@num%
    \s@mme=\z@\@ecfor\c@ef:=#4\do{\advance\s@mme\c@ef}%
    \edef\listec@ef{#4,0}\extrairelepremi@r\c@ef\de\listec@ef%
    \Figg@tXY{\p@int}\divide\v@lX\s@mme\divide\v@lY\s@mme\divide\v@lZ\s@mme%
    \multiply\v@lX\c@ef\multiply\v@lY\c@ef\multiply\v@lZ\c@ef%
    \@ecfor\p@int:=\list@num\do{\extrairelepremi@r\c@ef\de\listec@ef%
           \Figg@tXYa{\p@int}\divide\v@lXa\s@mme\divide\v@lYa\s@mme\divide\v@lZa\s@mme%
           \multiply\v@lXa\c@ef\multiply\v@lYa\c@ef\multiply\v@lZa\c@ef%
           \advance\v@lX\v@lXa\advance\v@lY\v@lYa\advance\v@lZ\v@lZa}%
    \Figp@intregTD#1:{#2}(\v@lX,\v@lY,\v@lZ)}\ignorespaces\fi}
\ctr@ld@f\def\figptbaryR#1:#2[#3;#4]{\ifps@cri{%
    \v@leur=\z@\@ecfor\c@ef:=#4\do{\maxim@m{\v@lmax}{\c@ef pt}{-\c@ef pt}%
    \ifdim\v@lmax>\v@leur\v@leur=\v@lmax\fi}%
    \ifdim\v@leur<\p@\f@ctech=\@M\else\ifdim\v@leur<\t@n\p@\f@ctech=\@m\else%
    \ifdim\v@leur<\c@nt\p@\f@ctech=\c@nt\else\ifdim\v@leur<\@m\p@\f@ctech=\t@n\else%
    \f@ctech=\@ne\fi\fi\fi\fi%
    \def\listec@ef{0}%
    \@ecfor\c@ef:=#4\do{\sc@lec@nvRI{\c@ef pt}\edef\listec@ef{\listec@ef,\the\s@mme}}%
    \extrairelepremi@r\c@ef\de\listec@ef\figptbary#1:#2[#3;\listec@ef]}\ignorespaces\fi}
\ctr@ld@f\def\sc@lec@nvRI#1{\v@leur=#1\p@rtentiere{\s@mme}{\v@leur}\advance\v@leur-\s@mme\p@%
    \multiply\v@leur\f@ctech\p@rtentiere{\p@rtent}{\v@leur}%
    \multiply\s@mme\f@ctech\advance\s@mme\p@rtent}
\ctr@ln@m\figptcirc
\ctr@ld@f\def\figptcircDD#1:#2:#3;#4(#5){\ifps@cri{\s@uvc@ntr@l\et@tfigptcircDD%
    \c@lptellDD#1:{#2}:#3;#4,#4(#5)\resetc@ntr@l\et@tfigptcircDD}\ignorespaces\fi}
\ctr@ld@f\def\figptcircTD#1:#2:#3,#4,#5;#6(#7){\ifps@cri{\s@uvc@ntr@l\et@tfigptcircTD%
    \setc@ntr@l{2}\c@lExtAxes#3,#4,#5(#6)\figptellP#1:{#2}:#3,-4,-5(#7)%
    \resetc@ntr@l\et@tfigptcircTD}\ignorespaces\fi}
\ctr@ln@m\figptcircumcenter
\ctr@ld@f\def\figptcircumcenterDD#1:#2[#3,#4,#5]{\ifps@cri{\s@uvc@ntr@l\et@tfigptcircumcenterDD%
    \setc@ntr@l{2}\figvectNDD-5[#3,#4]\figptbaryDD-3:[#3,#4;1,1]%
                  \figvectNDD-6[#4,#5]\figptbaryDD-4:[#4,#5;1,1]%
    \resetc@ntr@l{2}\inters@cDD#1:{#2}[-3,-5;-4,-6]%
    \resetc@ntr@l\et@tfigptcircumcenterDD}\ignorespaces\fi}
\ctr@ld@f\def\figptcircumcenterTD#1:#2[#3,#4,#5]{\ifps@cri{\s@uvc@ntr@l\et@tfigptcircumcenterTD%
    \setc@ntr@l{2}\figvectNTD-1[#3,#4,#5]%
    \figvectPTD-3[#3,#4]\figvectNVTD-5[-1,-3]\figptbaryTD-3:[#3,#4;1,1]%
    \figvectPTD-4[#4,#5]\figvectNVTD-6[-1,-4]\figptbaryTD-4:[#4,#5;1,1]%
    \resetc@ntr@l{2}\inters@cTD#1:{#2}[-3,-5;-4,-6]%
    \resetc@ntr@l\et@tfigptcircumcenterTD}\ignorespaces\fi}
\ctr@ln@m\figptcopy
\ctr@ld@f\def\figptcopyDD#1:#2/#3/{\ifps@cri{\Figg@tXY{#3}%
    \Figp@intregDD#1:{#2}(\v@lX,\v@lY)}\ignorespaces\fi}
\ctr@ld@f\def\figptcopyTD#1:#2/#3/{\ifps@cri{\Figg@tXY{#3}%
    \Figp@intregTD#1:{#2}(\v@lX,\v@lY,\v@lZ)}\ignorespaces\fi}
\ctr@ln@m\figptcurvcenter
\ctr@ld@f\def\figptcurvcenterDD#1:#2:#3[#4,#5,#6,#7]{\ifps@cri{\s@uvc@ntr@l\et@tfigptcurvcenterDD%
    \setc@ntr@l{2}\c@lcurvradDD#3[#4,#5,#6,#7]\edef\Sprim@{\repdecn@mb{\result@t}}%
    \figptBezierDD-1::#3[#4,#5,#6,#7]\figpttraDD#1:{#2}=-1/\Sprim@,-5/%
    \resetc@ntr@l\et@tfigptcurvcenterDD}\ignorespaces\fi}
\ctr@ld@f\def\figptcurvcenterTD#1:#2:#3[#4,#5,#6,#7]{\ifps@cri{\s@uvc@ntr@l\et@tfigptcurvcenterTD%
    \setc@ntr@l{2}\figvectDBezierTD -5:1,#3[#4,#5,#6,#7]%
    \figvectDBezierTD -6:2,#3[#4,#5,#6,#7]\vecunit@TD{-5}{-5}%
    \edef\Sprim@{\repdecn@mb{\result@t}}\figvectNVTD-1[-6,-5]%
    \figvectNVTD-5[-5,-1]\c@lproscalTD\v@leur[-6,-5]%
    \invers@{\v@leur}{\v@leur}\v@leur=\Sprim@\v@leur\v@leur=\Sprim@\v@leur%
    \figptBezierTD-1::#3[#4,#5,#6,#7]\edef\Sprim@{\repdecn@mb{\v@leur}}%
    \figpttraTD#1:{#2}=-1/\Sprim@,-5/\resetc@ntr@l\et@tfigptcurvcenterTD}\ignorespaces\fi}
\ctr@ld@f\def\c@lcurvradDD#1[#2,#3,#4,#5]{{\figvectDBezierDD -5:1,#1[#2,#3,#4,#5]%
    \figvectDBezierDD -6:2,#1[#2,#3,#4,#5]\vecunit@DD{-5}{-5}%
    \edef\Sprim@{\repdecn@mb{\result@t}}\figvectNVDD-5[-5]\c@lproscalDD\v@leur[-6,-5]%
    \invers@{\v@leur}{\v@leur}\v@leur=\Sprim@\v@leur\v@leur=\Sprim@\v@leur%
    \global\result@t=\v@leur}}
\ctr@ln@m\figptell
\ctr@ld@f\def\figptellDD#1:#2:#3;#4,#5(#6,#7){\ifps@cri{\s@uvc@ntr@l\et@tfigptell%
    \c@lptellDD#1::#3;#4,#5(#6)\figptrotDD#1:{#2}=#1/#3,#7/%
    \resetc@ntr@l\et@tfigptell}\ignorespaces\fi}
\ctr@ld@f\def\c@lptellDD#1:#2:#3;#4,#5(#6){\c@ssin{\C@}{\S@}{#6}\v@lmin=\C@ pt\v@lmax=\S@ pt%
    \v@lmin=#4\v@lmin\v@lmax=#5\v@lmax%
    \edef\Xc@mp{\repdecn@mb{\v@lmin}}\edef\Yc@mp{\repdecn@mb{\v@lmax}}%
    \setc@ntr@l{2}\figvectC-1(\Xc@mp,\Yc@mp)\figpttraDD#1:{#2}=#3/1,-1/}
\ctr@ld@f\def\figptellP#1:#2:#3,#4,#5(#6){\ifps@cri{\s@uvc@ntr@l\et@tfigptellP%
    \setc@ntr@l{2}\figvectP-1[#3,#4]\figvectP-2[#3,#5]%
    \v@leur=#6pt\c@lptellP{#3}{-1}{-2}\figptcopy#1:{#2}/-3/%
    \resetc@ntr@l\et@tfigptellP}\ignorespaces\fi}
\ctr@ln@m\@ngle
\ctr@ld@f\def\c@lptellP#1#2#3{\edef\@ngle{\repdecn@mb\v@leur}\c@ssin{\C@}{\S@}{\@ngle}%
    \figpttra-3:=#1/\C@,#2/\figpttra-3:=-3/\S@,#3/}
\ctr@ln@m\figptendnormal
\ctr@ld@f\def\figptendnormalDD#1:#2:#3,#4[#5,#6]{\ifps@cri{\s@uvc@ntr@l\et@tfigptendnormal%
    \Figg@tXYa{#5}\Figg@tXY{#6}%
    \advance\v@lX-\v@lXa\advance\v@lY-\v@lYa%
    \setc@ntr@l{2}\Figv@ctCreg-1(\v@lX,\v@lY)\vecunit@{-1}{-1}\Figg@tXY{-1}%
    \delt@=#3\unit@\maxim@m{\delt@}{\delt@}{-\delt@}\edef\l@ngueur{\repdecn@mb{\delt@}}%
    \v@lX=\l@ngueur\v@lX\v@lY=\l@ngueur\v@lY%
    \delt@=\p@\advance\delt@-#4pt\edef\l@ngueur{\repdecn@mb{\delt@}}%
    \figptbaryR-1:[#5,#6;#4,\l@ngueur]\Figg@tXYa{-1}%
    \advance\v@lXa\v@lY\advance\v@lYa-\v@lX%
    \setc@ntr@l{1}\Figp@intregDD#1:{#2}(\v@lXa,\v@lYa)\resetc@ntr@l\et@tfigptendnormal}%
    \ignorespaces\fi}
\ctr@ld@f\def\figptexcenter#1:#2[#3,#4,#5]{\ifps@cri{\let@xte={-}%
    \Figptexinsc@nter#1:#2[#3,#4,#5]}\ignorespaces\fi}
\ctr@ld@f\def\figptincenter#1:#2[#3,#4,#5]{\ifps@cri{\let@xte={}%
    \Figptexinsc@nter#1:#2[#3,#4,#5]}\ignorespaces\fi}
\ctr@ld@f% pour compatibilite avec anciennes versions
\ctr@ld@f\def\Figptexinsc@nter#1:#2[#3,#4,#5]{%
    \figgetdist\LA@[#4,#5]\figgetdist\LB@[#3,#5]\figgetdist\LC@[#3,#4]%
    \figptbaryR#1:{#2}[#3,#4,#5;\the\let@xte\LA@,\LB@,\LC@]}
\ctr@ln@m\figptinterlineplane
\ctr@ld@f\def\figptinterlineplaneDD{\un@v@ilable{figptinterlineplane}}
\ctr@ld@f\def\figptinterlineplaneTD#1:#2[#3,#4;#5,#6]{\ifps@cri{\s@uvc@ntr@l\et@tfigptinterlineplane%
    \setc@ntr@l{2}\figvectPTD-1[#3,#5]\vecunit@TD{-2}{#6}%
    \r@pPSTD\v@leur[-2,-1,#4]\edef\v@lcoef{\repdecn@mb{\v@leur}}%
    \figpttraTD#1:{#2}=#3/\v@lcoef,#4/\resetc@ntr@l\et@tfigptinterlineplane}\ignorespaces\fi}
\ctr@ln@m\figptorthocenter
\ctr@ld@f\def\figptorthocenterDD#1:#2[#3,#4,#5]{\ifps@cri{\s@uvc@ntr@l\et@tfigptorthocenterDD%
    \setc@ntr@l{2}\figvectNDD-3[#3,#4]\figvectNDD-4[#4,#5]%
    \resetc@ntr@l{2}\inters@cDD#1:{#2}[#5,-3;#3,-4]%
    \resetc@ntr@l\et@tfigptorthocenterDD}\ignorespaces\fi}
\ctr@ld@f\def\figptorthocenterTD#1:#2[#3,#4,#5]{\ifps@cri{\s@uvc@ntr@l\et@tfigptorthocenterTD%
    \setc@ntr@l{2}\figvectNTD-1[#3,#4,#5]%
    \figvectPTD-2[#3,#4]\figvectNVTD-3[-1,-2]%
    \figvectPTD-2[#4,#5]\figvectNVTD-4[-1,-2]%
    \resetc@ntr@l{2}\inters@cTD#1:{#2}[#5,-3;#3,-4]%
    \resetc@ntr@l\et@tfigptorthocenterTD}\ignorespaces\fi}
\ctr@ln@m\figptorthoprojline
\ctr@ld@f\def\figptorthoprojlineDD#1:#2=#3/#4,#5/{\ifps@cri{\s@uvc@ntr@l\et@tfigptorthoprojlineDD%
    \setc@ntr@l{2}\figvectPDD-3[#4,#5]\figvectNVDD-4[-3]\resetc@ntr@l{2}%
    \inters@cDD#1:{#2}[#3,-4;#4,-3]\resetc@ntr@l\et@tfigptorthoprojlineDD}\ignorespaces\fi}
\ctr@ld@f\def\figptorthoprojlineTD#1:#2=#3/#4,#5/{\ifps@cri{\s@uvc@ntr@l\et@tfigptorthoprojlineTD%
    \setc@ntr@l{2}\figvectPTD-1[#4,#3]\figvectPTD-2[#4,#5]\vecunit@TD{-2}{-2}%
    \c@lproscalTD\v@leur[-1,-2]\edef\v@lcoef{\repdecn@mb{\v@leur}}%
    \figpttraTD#1:{#2}=#4/\v@lcoef,-2/\resetc@ntr@l\et@tfigptorthoprojlineTD}\ignorespaces\fi}
\ctr@ln@m\figptorthoprojplane
\ctr@ld@f\def\figptorthoprojplaneDD{\un@v@ilable{figptorthoprojplane}}
\ctr@ld@f\def\figptorthoprojplaneTD#1:#2=#3/#4,#5/{\ifps@cri{\s@uvc@ntr@l\et@tfigptorthoprojplane%
    \setc@ntr@l{2}\figvectPTD-1[#3,#4]\vecunit@TD{-2}{#5}%
    \c@lproscalTD\v@leur[-1,-2]\edef\v@lcoef{\repdecn@mb{\v@leur}}%
    \figpttraTD#1:{#2}=#3/\v@lcoef,-2/\resetc@ntr@l\et@tfigptorthoprojplane}\ignorespaces\fi}
\ctr@ld@f\def\figpthom#1:#2=#3/#4,#5/{\ifps@cri{\s@uvc@ntr@l\et@tfigpthom%
    \setc@ntr@l{2}\figvectP-1[#4,#3]\figpttra#1:{#2}=#4/#5,-1/%
    \resetc@ntr@l\et@tfigpthom}\ignorespaces\fi}
\ctr@ln@m\figptrot
\ctr@ld@f\def\figptrotDD#1:#2=#3/#4,#5/{\ifps@cri{\s@uvc@ntr@l\et@tfigptrotDD%
    \c@ssin{\C@}{\S@}{#5}\setc@ntr@l{2}\figvectPDD-1[#4,#3]\Figg@tXY{-1}%
    \v@lXa=\C@\v@lX\advance\v@lXa-\S@\v@lY%
    \v@lYa=\S@\v@lX\advance\v@lYa\C@\v@lY%
    \Figv@ctCreg-1(\v@lXa,\v@lYa)\figpttraDD#1:{#2}=#4/1,-1/%
    \resetc@ntr@l\et@tfigptrotDD}\ignorespaces\fi}
\ctr@ld@f\def\figptrotTD#1:#2=#3/#4,#5,#6/{\ifps@cri{\s@uvc@ntr@l\et@tfigptrotTD%
    \c@ssin{\C@}{\S@}{#5}%
    \setc@ntr@l{2}\figptorthoprojplaneTD-3:=#4/#3,#6/\figvectPTD-2[-3,#3]%
    \n@rmeucTD\v@leur{-2}\ifdim\v@leur<\Cepsil@n\Figg@tXYa{#3}\else%
    \edef\v@lcoef{\repdecn@mb{\v@leur}}\figvectNVTD-1[#6,-2]%
    \Figg@tXYa{-1}\v@lXa=\v@lcoef\v@lXa\v@lYa=\v@lcoef\v@lYa\v@lZa=\v@lcoef\v@lZa%
    \v@lXa=\S@\v@lXa\v@lYa=\S@\v@lYa\v@lZa=\S@\v@lZa\Figg@tXY{-2}%
    \advance\v@lXa\C@\v@lX\advance\v@lYa\C@\v@lY\advance\v@lZa\C@\v@lZ%
    \Figg@tXY{-3}\advance\v@lXa\v@lX\advance\v@lYa\v@lY\advance\v@lZa\v@lZ\fi%
    \Figp@intregTD#1:{#2}(\v@lXa,\v@lYa,\v@lZa)\resetc@ntr@l\et@tfigptrotTD}\ignorespaces\fi}
\ctr@ln@m\figptsym
\ctr@ld@f\def\figptsymDD#1:#2=#3/#4,#5/{\ifps@cri{\s@uvc@ntr@l\et@tfigptsymDD%
    \resetc@ntr@l{2}\figptorthoprojlineDD-5:=#3/#4,#5/\figvectPDD-2[#3,-5]%
    \figpttraDD#1:{#2}=#3/2,-2/\resetc@ntr@l\et@tfigptsymDD}\ignorespaces\fi}
\ctr@ld@f\def\figptsymTD#1:#2=#3/#4,#5/{\ifps@cri{\s@uvc@ntr@l\et@tfigptsymTD%
    \resetc@ntr@l{2}\figptorthoprojplaneTD-3:=#3/#4,#5/\figvectPTD-2[#3,-3]%
    \figpttraTD#1:{#2}=#3/2,-2/\resetc@ntr@l\et@tfigptsymTD}\ignorespaces\fi}
\ctr@ln@m\figpttra
\ctr@ld@f\def\figpttraDD#1:#2=#3/#4,#5/{\ifps@cri{\Figg@tXYa{#5}\v@lXa=#4\v@lXa\v@lYa=#4\v@lYa%
    \Figg@tXY{#3}\advance\v@lX\v@lXa\advance\v@lY\v@lYa%
    \Figp@intregDD#1:{#2}(\v@lX,\v@lY)}\ignorespaces\fi}
\ctr@ld@f\def\figpttraTD#1:#2=#3/#4,#5/{\ifps@cri{\Figg@tXYa{#5}\v@lXa=#4\v@lXa\v@lYa=#4\v@lYa%
    \v@lZa=#4\v@lZa\Figg@tXY{#3}\advance\v@lX\v@lXa\advance\v@lY\v@lYa%
    \advance\v@lZ\v@lZa\Figp@intregTD#1:{#2}(\v@lX,\v@lY,\v@lZ)}\ignorespaces\fi}
\ctr@ln@m\figpttraC
\ctr@ld@f\def\figpttraCDD#1:#2=#3/#4,#5/{\ifps@cri{\v@lXa=#4\unit@\v@lYa=#5\unit@%
    \Figg@tXY{#3}\advance\v@lX\v@lXa\advance\v@lY\v@lYa%
    \Figp@intregDD#1:{#2}(\v@lX,\v@lY)}\ignorespaces\fi}
\ctr@ld@f\def\figpttraCTD#1:#2=#3/#4,#5,#6/{\ifps@cri{\v@lXa=#4\unit@\v@lYa=#5\unit@\v@lZa=#6\unit@%
    \Figg@tXY{#3}\advance\v@lX\v@lXa\advance\v@lY\v@lYa\advance\v@lZ\v@lZa%
    \Figp@intregTD#1:{#2}(\v@lX,\v@lY,\v@lZ)}\ignorespaces\fi}
\ctr@ld@f\def\figptsaxes#1:#2(#3){\ifps@cri{\an@lys@xes#3,:\ifx\t@xt@\empty%
    \ifTr@isDim\Figpts@xes#1:#2(0,#3,0,#3,0,#3)\else\Figpts@xes#1:#2(0,#3,0,#3)\fi%
    \else\Figpts@xes#1:#2(#3)\fi}\ignorespaces\fi}
\ctr@ln@m\Figpts@xes
\ctr@ld@f\def\Figpts@xesDD#1:#2(#3,#4,#5,#6){%
    \s@mme=#1\figpttraC\the\s@mme:$x$=#2/#4,0/%
    \advance\s@mme\@ne\figpttraC\the\s@mme:$y$=#2/0,#6/}
\ctr@ld@f\def\Figpts@xesTD#1:#2(#3,#4,#5,#6,#7,#8){%
    \s@mme=#1\figpttraC\the\s@mme:$x$=#2/#4,0,0/%
    \advance\s@mme\@ne\figpttraC\the\s@mme:$y$=#2/0,#6,0/%
    \advance\s@mme\@ne\figpttraC\the\s@mme:$z$=#2/0,0,#8/}
\ctr@ld@f\def\figptsmap#1=#2/#3/#4/{\ifps@cri{\s@uvc@ntr@l\et@tfigptsmap%
    \setc@ntr@l{2}\def\list@num{#2}\s@mme=#1%
    \@ecfor\p@int:=\list@num\do{\figvectP-1[#3,\p@int]\Figg@tXY{-1}%
    \pr@dMatV/#4/\figpttra\the\s@mme:=#3/1,-1/\advance\s@mme\@ne}%
    \resetc@ntr@l\et@tfigptsmap}\ignorespaces\fi}
\ctr@ln@m\figptscontrol
\ctr@ld@f\def\figptscontrolDD#1[#2,#3,#4,#5]{\ifps@cri{\s@uvc@ntr@l\et@tfigptscontrolDD\setc@ntr@l{2}%
    \v@lX=\z@\v@lY=\z@\Figtr@nptDD{-5}{#2}\Figtr@nptDD{2}{#5}%
    \divide\v@lX\@vi\divide\v@lY\@vi%
    \Figtr@nptDD{3}{#3}\Figtr@nptDD{-1.5}{#4}\Figp@intregDD-1:(\v@lX,\v@lY)%
    \v@lX=\z@\v@lY=\z@\Figtr@nptDD{2}{#2}\Figtr@nptDD{-5}{#5}%
    \divide\v@lX\@vi\divide\v@lY\@vi\Figtr@nptDD{-1.5}{#3}\Figtr@nptDD{3}{#4}%
    \s@mme=#1\advance\s@mme\@ne\Figp@intregDD\the\s@mme:(\v@lX,\v@lY)%
    \figptcopyDD#1:/-1/\resetc@ntr@l\et@tfigptscontrolDD}\ignorespaces\fi}
\ctr@ld@f\def\figptscontrolTD#1[#2,#3,#4,#5]{\ifps@cri{\s@uvc@ntr@l\et@tfigptscontrolTD\setc@ntr@l{2}%
    \v@lX=\z@\v@lY=\z@\v@lZ=\z@\Figtr@nptTD{-5}{#2}\Figtr@nptTD{2}{#5}%
    \divide\v@lX\@vi\divide\v@lY\@vi\divide\v@lZ\@vi%
    \Figtr@nptTD{3}{#3}\Figtr@nptTD{-1.5}{#4}\Figp@intregTD-1:(\v@lX,\v@lY,\v@lZ)%
    \v@lX=\z@\v@lY=\z@\v@lZ=\z@\Figtr@nptTD{2}{#2}\Figtr@nptTD{-5}{#5}%
    \divide\v@lX\@vi\divide\v@lY\@vi\divide\v@lZ\@vi\Figtr@nptTD{-1.5}{#3}\Figtr@nptTD{3}{#4}%
    \s@mme=#1\advance\s@mme\@ne\Figp@intregTD\the\s@mme:(\v@lX,\v@lY,\v@lZ)%
    \figptcopyTD#1:/-1/\resetc@ntr@l\et@tfigptscontrolTD}\ignorespaces\fi}
\ctr@ld@f\def\Figtr@nptDD#1#2{\Figg@tXYa{#2}\v@lXa=#1\v@lXa\v@lYa=#1\v@lYa%
    \advance\v@lX\v@lXa\advance\v@lY\v@lYa}
\ctr@ld@f\def\Figtr@nptTD#1#2{\Figg@tXYa{#2}\v@lXa=#1\v@lXa\v@lYa=#1\v@lYa\v@lZa=#1\v@lZa%
    \advance\v@lX\v@lXa\advance\v@lY\v@lYa\advance\v@lZ\v@lZa}
\ctr@ld@f\def\figptscontrolcurve#1,#2[#3]{\ifps@cri{\s@uvc@ntr@l\et@tfigptscontrolcurve%
    \def\list@num{#3}\extrairelepremi@r\Ak@\de\list@num%
    \extrairelepremi@r\Ai@\de\list@num\extrairelepremi@r\Aj@\de\list@num%
    \s@mme=#1\figptcopy\the\s@mme:/\Ai@/%
    \setc@ntr@l{2}\figvectP -1[\Ak@,\Aj@]%
    \@ecfor\Ak@:=\list@num\do{\advance\s@mme\@ne\figpttra\the\s@mme:=\Ai@/\curv@roundness,-1/%
       \figvectP -1[\Ai@,\Ak@]\advance\s@mme\@ne\figpttra\the\s@mme:=\Aj@/-\curv@roundness,-1/%
       \advance\s@mme\@ne\figptcopy\the\s@mme:/\Aj@/%
       \edef\Ai@{\Aj@}\edef\Aj@{\Ak@}}\advance\s@mme-#1\divide\s@mme\thr@@%
       \xdef#2{\the\s@mme}%
    \resetc@ntr@l\et@tfigptscontrolcurve}\ignorespaces\fi}
\ctr@ln@m\figptsintercirc
\ctr@ld@f\def\figptsintercircDD#1[#2,#3;#4,#5]{\ifps@cri{\s@uvc@ntr@l\et@tfigptsintercircDD%
    \setc@ntr@l{2}\let\c@lNVintc=\c@lNVintcDD\Figptsintercirc@#1[#2,#3;#4,#5]%    
    \resetc@ntr@l\et@tfigptsintercircDD}\ignorespaces\fi}
\ctr@ld@f\def\figptsintercircTD#1[#2,#3;#4,#5;#6]{\ifps@cri{\s@uvc@ntr@l\et@tfigptsintercircTD%
    \setc@ntr@l{2}\let\c@lNVintc=\c@lNVintcTD\vecunitC@TD[#2,#6]%
    \Figv@ctCreg-3(\v@lX,\v@lY,\v@lZ)\Figptsintercirc@#1[#2,#3;#4,#5]%
    \resetc@ntr@l\et@tfigptsintercircTD}\ignorespaces\fi}
\ctr@ld@f\def\Figptsintercirc@#1[#2,#3;#4,#5]{\figvectP-1[#2,#4]%
    \vecunit@{-1}{-1}\delt@=\result@t\f@ctech=\result@tent%
    \s@mme=#1\advance\s@mme\@ne\figptcopy#1:/#2/\figptcopy\the\s@mme:/#4/%
    \ifdim\delt@=\z@\else%
    \v@lmin=#3\unit@\v@lmax=#5\unit@\v@leur=\v@lmin\advance\v@leur\v@lmax%
    \ifdim\v@leur>\delt@%
    \v@leur=\v@lmin\advance\v@leur-\v@lmax\maxim@m{\v@leur}{\v@leur}{-\v@leur}%
    \ifdim\v@leur<\delt@%
    \divide\v@lmin\f@ctech\divide\v@lmax\f@ctech\divide\delt@\f@ctech%
    \v@lmin=\repdecn@mb{\v@lmin}\v@lmin\v@lmax=\repdecn@mb{\v@lmax}\v@lmax%
    \invers@{\v@leur}{\delt@}\advance\v@lmax-\v@lmin%
    \v@lmax=-\repdecn@mb{\v@leur}\v@lmax\advance\delt@\v@lmax\delt@=.5\delt@%
    \v@lmax=\delt@\multiply\v@lmax\f@ctech%
    \edef\t@ille{\repdecn@mb{\v@lmax}}\figpttra-2:=#2/\t@ille,-1/%
    \delt@=\repdecn@mb{\delt@}\delt@\advance\v@lmin-\delt@%
    \sqrt@{\v@leur}{\v@lmin}\multiply\v@leur\f@ctech\edef\t@ille{\repdecn@mb{\v@leur}}%
    \c@lNVintc\figpttra#1:=-2/-\t@ille,-1/\figpttra\the\s@mme:=-2/\t@ille,-1/\fi\fi\fi}
\ctr@ld@f\def\c@lNVintcDD{\Figg@tXY{-1}\Figv@ctCreg-1(-\v@lY,\v@lX)} % <=> \figvectNVDD-1[-1]
\ctr@ld@f\def\c@lNVintcTD{{\Figg@tXY{-3}\v@lmin=\v@lX\v@lmax=\v@lY\v@leur=\v@lZ%
    \Figg@tXY{-1}\c@lprovec{-3}\vecunit@{-3}{-3}% <=> \figvectNVTD-3[-1,-3]\vecunit@{-3}{-3}
    \Figg@tXY{-1}\v@lmin=\v@lX\v@lmax=\v@lY%
    \v@leur=\v@lZ\Figg@tXY{-3}\c@lprovec{-1}}} % <=> \figvectNVTD-1[-3,-1]
\ctr@ln@m\figptsinterlinell
\ctr@ld@f\def\figptsinterlinellDD#1[#2,#3,#4,#5;#6,#7]{\ifps@cri{\s@uvc@ntr@l\et@tfigptsinterlinellDD%
    \figptcopy#1:/#6/\s@mme=#1\advance\s@mme\@ne\figptcopy\the\s@mme:/#7/%
    \v@lmin=#3\unit@\v@lmax=#4\unit@% a, b
    \setc@ntr@l{2}\figptbaryDD-4:[#6,#7;1,1]\figptsrotDD-3=-4,#7/#2,-#5/% D et rotation
    \Figg@tXY{-3}\Figg@tXYa{#2}\advance\v@lX-\v@lXa\advance\v@lY-\v@lYa% alpha, beta
    \figvectP-1[-3,-2]\Figg@tXYa{-1}\figvectP-3[-4,#7]\Figptsint@rLE{#1}% u1, u2
    \resetc@ntr@l\et@tfigptsinterlinellDD}\ignorespaces\fi}
\ctr@ld@f\def\figptsinterlinellP#1[#2,#3,#4;#5,#6]{\ifps@cri{\s@uvc@ntr@l\et@tfigptsinterlinellP%
    \figptcopy#1:/#5/\s@mme=#1\advance\s@mme\@ne\figptcopy\the\s@mme:/#6/\setc@ntr@l{2}%
    \figvectP-1[#2,#3]\vecunit@{-1}{-1}\v@lmin=\result@t% a
    \figvectP-2[#2,#4]\vecunit@{-2}{-2}\v@lmax=\result@t% b
    \figptbary-4:[#5,#6;1,1]% D
    \figvectP-3[#2,-4]\c@lproscal\v@lX[-3,-1]\c@lproscal\v@lY[-3,-2]% alpha, beta
    \figvectP-3[-4,#6]\c@lproscal\v@lXa[-3,-1]\c@lproscal\v@lYa[-3,-2]% u1, u2
    \Figptsint@rLE{#1}\resetc@ntr@l\et@tfigptsinterlinellP}\ignorespaces\fi}
\ctr@ld@f\def\Figptsint@rLE#1{%
    \getredf@ctDD\f@ctech(\v@lmin,\v@lmax)%
    \getredf@ctDD\p@rtent(\v@lX,\v@lY)\ifnum\p@rtent>\f@ctech\f@ctech=\p@rtent\fi%
    \getredf@ctDD\p@rtent(\v@lXa,\v@lYa)\ifnum\p@rtent>\f@ctech\f@ctech=\p@rtent\fi%
    \divide\v@lmin\f@ctech\divide\v@lmax\f@ctech\divide\v@lX\f@ctech\divide\v@lY\f@ctech%
    \divide\v@lXa\f@ctech\divide\v@lYa\f@ctech%
    \c@rre=\repdecn@mb\v@lXa\v@lmax\mili@u=\repdecn@mb\v@lYa\v@lmin%
    \getredf@ctDD\f@ctech(\c@rre,\mili@u)%
    \c@rre=\repdecn@mb\v@lX\v@lmax\mili@u=\repdecn@mb\v@lY\v@lmin%
    \getredf@ctDD\p@rtent(\c@rre,\mili@u)\ifnum\p@rtent>\f@ctech\f@ctech=\p@rtent\fi%
    \divide\v@lmin\f@ctech\divide\v@lmax\f@ctech\divide\v@lX\f@ctech\divide\v@lY\f@ctech%
    \divide\v@lXa\f@ctech\divide\v@lYa\f@ctech%
    \v@lmin=\repdecn@mb{\v@lmin}\v@lmin\v@lmax=\repdecn@mb{\v@lmax}\v@lmax%
    \edef\G@xde{\repdecn@mb\v@lmin}\edef\P@xde{\repdecn@mb\v@lmax}%
    \c@rre=-\v@lmax\v@leur=\repdecn@mb\v@lY\v@lY\advance\c@rre\v@leur\c@rre=\G@xde\c@rre%
    \v@leur=\repdecn@mb\v@lX\v@lX\v@leur=\P@xde\v@leur\advance\c@rre\v@leur% C
    \v@lmin=\repdecn@mb\v@lYa\v@lmin\v@lmax=\repdecn@mb\v@lXa\v@lmax%
    \mili@u=\repdecn@mb\v@lX\v@lmax\advance\mili@u\repdecn@mb\v@lY\v@lmin% B
    \v@lmax=\repdecn@mb\v@lXa\v@lmax\advance\v@lmax\repdecn@mb\v@lYa\v@lmin% A
    \ifdim\v@lmax>\epsil@n%
    \maxim@m{\v@leur}{\c@rre}{-\c@rre}\maxim@m{\v@lmin}{\mili@u}{-\mili@u}%
    \maxim@m{\v@leur}{\v@leur}{\v@lmin}\maxim@m{\v@lmin}{\v@lmax}{-\v@lmax}%
    \maxim@m{\v@leur}{\v@leur}{\v@lmin}\p@rtentiere{\p@rtent}{\v@leur}\advance\p@rtent\@ne%
    \divide\c@rre\p@rtent\divide\mili@u\p@rtent\divide\v@lmax\p@rtent%
    \delt@=\repdecn@mb{\mili@u}\mili@u\v@leur=\repdecn@mb{\v@lmax}\c@rre%
    \advance\delt@-\v@leur\ifdim\delt@<\z@\else\sqrt@\delt@\delt@%
    \invers@\v@lmax\v@lmax\edef\Uns@rAp{\repdecn@mb\v@lmax}%
    \v@leur=-\mili@u\advance\v@leur-\delt@\v@leur=\Uns@rAp\v@leur%
    \edef\t@ille{\repdecn@mb\v@leur}\figpttra#1:=-4/\t@ille,-3/\s@mme=#1\advance\s@mme\@ne%
    \v@leur=-\mili@u\advance\v@leur\delt@\v@leur=\Uns@rAp\v@leur%
    \edef\t@ille{\repdecn@mb\v@leur}\figpttra\the\s@mme:=-4/\t@ille,-3/\fi\fi}
\ctr@ln@m\figptsorthoprojline
\ctr@ld@f\def\figptsorthoprojlineDD#1=#2/#3,#4/{\ifps@cri{\s@uvc@ntr@l\et@tfigptsorthoprojlineDD%
    \setc@ntr@l{2}\figvectPDD-3[#3,#4]\figvectNVDD-4[-3]\resetc@ntr@l{2}%
    \def\list@num{#2}\s@mme=#1\@ecfor\p@int:=\list@num\do{%
    \inters@cDD\the\s@mme:[\p@int,-4;#3,-3]\advance\s@mme\@ne}%
    \resetc@ntr@l\et@tfigptsorthoprojlineDD}\ignorespaces\fi}
\ctr@ld@f\def\figptsorthoprojlineTD#1=#2/#3,#4/{\ifps@cri{\s@uvc@ntr@l\et@tfigptsorthoprojlineTD%
    \setc@ntr@l{2}\figvectPTD-2[#3,#4]\vecunit@TD{-2}{-2}%
    \def\list@num{#2}\s@mme=#1\@ecfor\p@int:=\list@num\do{%
    \figvectPTD-1[#3,\p@int]\c@lproscalTD\v@leur[-1,-2]%
    \edef\v@lcoef{\repdecn@mb{\v@leur}}\figpttraTD\the\s@mme:=#3/\v@lcoef,-2/%
    \advance\s@mme\@ne}\resetc@ntr@l\et@tfigptsorthoprojlineTD}\ignorespaces\fi}
\ctr@ln@m\figptsorthoprojplane
\ctr@ld@f\def\figptsorthoprojplaneDD{\un@v@ilable{figptsorthoprojplane}}
\ctr@ld@f\def\figptsorthoprojplaneTD#1=#2/#3,#4/{\ifps@cri{\s@uvc@ntr@l\et@tfigptsorthoprojplane%
    \setc@ntr@l{2}\vecunit@TD{-2}{#4}%
    \def\list@num{#2}\s@mme=#1\@ecfor\p@int:=\list@num\do{\figvectPTD-1[\p@int,#3]%
    \c@lproscalTD\v@leur[-1,-2]\edef\v@lcoef{\repdecn@mb{\v@leur}}%
    \figpttraTD\the\s@mme:=\p@int/\v@lcoef,-2/\advance\s@mme\@ne}%
    \resetc@ntr@l\et@tfigptsorthoprojplane}\ignorespaces\fi}
\ctr@ld@f\def\figptshom#1=#2/#3,#4/{\ifps@cri{\s@uvc@ntr@l\et@tfigptshom%
    \setc@ntr@l{2}\def\list@num{#2}\s@mme=#1%
    \@ecfor\p@int:=\list@num\do{\figvectP-1[#3,\p@int]%
    \figpttra\the\s@mme:=#3/#4,-1/\advance\s@mme\@ne}%
    \resetc@ntr@l\et@tfigptshom}\ignorespaces\fi}
\ctr@ln@m\figptsrot
\ctr@ld@f\def\figptsrotDD#1=#2/#3,#4/{\ifps@cri{\s@uvc@ntr@l\et@tfigptsrotDD%
    \c@ssin{\C@}{\S@}{#4}\setc@ntr@l{2}\def\list@num{#2}\s@mme=#1%
    \@ecfor\p@int:=\list@num\do{\figvectPDD-1[#3,\p@int]\Figg@tXY{-1}%
    \v@lXa=\C@\v@lX\advance\v@lXa-\S@\v@lY%
    \v@lYa=\S@\v@lX\advance\v@lYa\C@\v@lY%
    \Figv@ctCreg-1(\v@lXa,\v@lYa)\figpttraDD\the\s@mme:=#3/1,-1/\advance\s@mme\@ne}%
    \resetc@ntr@l\et@tfigptsrotDD}\ignorespaces\fi}
\ctr@ld@f\def\figptsrotTD#1=#2/#3,#4,#5/{\ifps@cri{\s@uvc@ntr@l\et@tfigptsrotTD%
    \c@ssin{\C@}{\S@}{#4}%
    \setc@ntr@l{2}\def\list@num{#2}\s@mme=#1%
    \@ecfor\p@int:=\list@num\do{\figptorthoprojplaneTD-3:=#3/\p@int,#5/%
    \figvectPTD-2[-3,\p@int]%
    \figvectNVTD-1[#5,-2]\n@rmeucTD\v@leur{-2}\edef\v@lcoef{\repdecn@mb{\v@leur}}%
    \Figg@tXYa{-1}\v@lXa=\v@lcoef\v@lXa\v@lYa=\v@lcoef\v@lYa\v@lZa=\v@lcoef\v@lZa%
    \v@lXa=\S@\v@lXa\v@lYa=\S@\v@lYa\v@lZa=\S@\v@lZa\Figg@tXY{-2}%
    \advance\v@lXa\C@\v@lX\advance\v@lYa\C@\v@lY\advance\v@lZa\C@\v@lZ%
    \Figg@tXY{-3}\advance\v@lXa\v@lX\advance\v@lYa\v@lY\advance\v@lZa\v@lZ%
    \Figp@intregTD\the\s@mme:(\v@lXa,\v@lYa,\v@lZa)\advance\s@mme\@ne}%
    \resetc@ntr@l\et@tfigptsrotTD}\ignorespaces\fi}
\ctr@ln@m\figptssym
\ctr@ld@f\def\figptssymDD#1=#2/#3,#4/{\ifps@cri{\s@uvc@ntr@l\et@tfigptssymDD%
    \setc@ntr@l{2}\figvectPDD-3[#3,#4]\Figg@tXY{-3}\Figv@ctCreg-4(-\v@lY,\v@lX)%
    \resetc@ntr@l{2}\def\list@num{#2}\s@mme=#1%
    \@ecfor\p@int:=\list@num\do{\inters@cDD-5:[#3,-3;\p@int,-4]\figvectPDD-2[\p@int,-5]%
    \figpttraDD\the\s@mme:=\p@int/2,-2/\advance\s@mme\@ne}%
    \resetc@ntr@l\et@tfigptssymDD}\ignorespaces\fi}
\ctr@ld@f\def\figptssymTD#1=#2/#3,#4/{\ifps@cri{\s@uvc@ntr@l\et@tfigptssymTD%
    \setc@ntr@l{2}\vecunit@TD{-2}{#4}\def\list@num{#2}\s@mme=#1%
    \@ecfor\p@int:=\list@num\do{\figvectPTD-1[\p@int,#3]%
    \c@lproscalTD\v@leur[-1,-2]\v@leur=2\v@leur\edef\v@lcoef{\repdecn@mb{\v@leur}}%
    \figpttraTD\the\s@mme:=\p@int/\v@lcoef,-2/\advance\s@mme\@ne}%
    \resetc@ntr@l\et@tfigptssymTD}\ignorespaces\fi}
\ctr@ln@m\figptstra
\ctr@ld@f\def\figptstraDD#1=#2/#3,#4/{\ifps@cri{\Figg@tXYa{#4}\v@lXa=#3\v@lXa\v@lYa=#3\v@lYa%
    \def\list@num{#2}\s@mme=#1\@ecfor\p@int:=\list@num\do{\Figg@tXY{\p@int}%
    \advance\v@lX\v@lXa\advance\v@lY\v@lYa%
    \Figp@intregDD\the\s@mme:(\v@lX,\v@lY)\advance\s@mme\@ne}}\ignorespaces\fi}
\ctr@ld@f\def\figptstraTD#1=#2/#3,#4/{\ifps@cri{\Figg@tXYa{#4}\v@lXa=#3\v@lXa\v@lYa=#3\v@lYa%
    \v@lZa=#3\v@lZa\def\list@num{#2}\s@mme=#1\@ecfor\p@int:=\list@num\do{\Figg@tXY{\p@int}%
    \advance\v@lX\v@lXa\advance\v@lY\v@lYa\advance\v@lZ\v@lZa%
    \Figp@intregTD\the\s@mme:(\v@lX,\v@lY,\v@lZ)\advance\s@mme\@ne}}\ignorespaces\fi}
\ctr@ln@m\figptvisilimSL
\ctr@ld@f\def\figptvisilimSLDD{\un@v@ilable{figptvisilimSL}}
\ctr@ld@f\def\figptvisilimSLTD#1:#2[#3,#4;#5,#6]{\ifps@cri{\s@uvc@ntr@l\et@tfigptvisilimSLTD%
    \setc@ntr@l{2}\figvectP-1[#3,#4]\n@rminf{\delt@}{-1}%
    \ifcase\curr@ntproj\v@lX=\cxa@\p@\v@lY=-\p@\v@lZ=\cxb@\p@% Proj cav
    \Figv@ctCreg-2(\v@lX,\v@lY,\v@lZ)\figvectP-3[#5,#6]\figvectNV-1[-2,-3]%
    \or\figvectP-1[#5,#6]\vecunitCV@TD{-1}\v@lmin=\v@lX\v@lmax=\v@lY% Proj ortho
    \v@leur=\v@lZ\v@lX=\cza@\p@\v@lY=\czb@\p@\v@lZ=\czc@\p@\c@lprovec{-1}%
    \or\c@ley@pt{-2}\figvectN-1[#5,#6,-2]\fi% Proj rea
    \edef\Ai@{#3}\edef\Aj@{#4}\figvectP-2[#5,\Ai@]\c@lproscal\v@leur[-1,-2]%
    \ifdim\v@leur>\z@\p@rtent=\@ne\else\p@rtent=\m@ne\fi%
    \figvectP-2[#5,\Aj@]\c@lproscal\v@leur[-1,-2]%
    \ifdim\p@rtent\v@leur>\z@\figptcopy#1:#2/#3/%
    \message{*** \BS@ figptvisilimSL: points are on the same side.}\else%
    \figptcopy-3:/#3/\figptcopy-4:/#4/%
    \loop\figptbary-5:[-3,-4;1,1]\figvectP-2[#5,-5]\c@lproscal\v@leur[-1,-2]%
    \ifdim\p@rtent\v@leur>\z@\figptcopy-3:/-5/\else\figptcopy-4:/-5/\fi%
    \divide\delt@\tw@\ifdim\delt@>\epsil@n\repeat%
    \figptbary#1:#2[-3,-4;1,1]\fi\resetc@ntr@l\et@tfigptvisilimSLTD}\ignorespaces\fi}
\ctr@ld@f\def\c@ley@pt#1{\t@stp@r\ifitis@K\v@lX=\cza@\p@\v@lY=\czb@\p@\v@lZ=\czc@\p@%
    \Figv@ctCreg-1(\v@lX,\v@lY,\v@lZ)\Figp@intreg-2:(\wd\Bt@rget,\ht\Bt@rget,\dp\Bt@rget)%
    \figpttra#1:=-2/-\disob@intern,-1/\else\end\fi}
\ctr@ld@f\def\t@stp@r{\itis@Ktrue\ifnewt@rgetpt\else\itis@Kfalse%
    \message{*** \BS@ figptvisilimXX: target point undefined.}\fi\ifnewdis@b\else%
    \itis@Kfalse\message{*** \BS@ figptvisilimXX: observation distance undefined.}\fi%
    \ifitis@K\else\message{*** This macro must be called after \BS@ psbeginfig or after
    having set the missing parameter(s) with \BS@ figset proj()}\fi}
\ctr@ld@f\def\figscan#1(#2,#3){{\s@uvc@ntr@l\et@tfigscan\@psfgetbb{#1}\if@psfbbfound\else%
    \def\@psfllx{0}\def\@psflly{20}\def\@psfurx{540}\def\@psfury{640}\fi\figscan@{#2}{#3}%
    \resetc@ntr@l\et@tfigscan}\ignorespaces}
\ctr@ld@f\def\figscan@#1#2{%
    \unit@=\@ne bp\setc@ntr@l{2}\figsetmark{}%
    \def\minst@p{20pt}%
    \v@lX=\@psfllx\p@\v@lX=\Sc@leFact\v@lX\r@undint\v@lX\v@lX%
    \v@lY=\@psflly\p@\v@lY=\Sc@leFact\v@lY\ifdim\v@lY>\z@\r@undint\v@lY\v@lY\fi%
    \delt@=\@psfury\p@\delt@=\Sc@leFact\delt@%
    \advance\delt@-\v@lY\v@lXa=\@psfurx\p@\v@lXa=\Sc@leFact\v@lXa\v@leur=\minst@p%
    \edef\valv@lY{\repdecn@mb{\v@lY}}\edef\LgTr@it{\the\delt@}%
    \loop\ifdim\v@lX<\v@lXa\edef\valv@lX{\repdecn@mb{\v@lX}}%
    \figptDD -1:(\valv@lX,\valv@lY)\figwriten -1:\hbox{\vrule height\LgTr@it}(0)%
    \ifdim\v@leur<\minst@p\else\figsetmark{\raise-8bp\hbox{$\scriptscriptstyle\triangle$}}%
    \figwrites -1:\@ffichnb{0}{\valv@lX}(6)\v@leur=\z@\figsetmark{}\fi%
    \advance\v@leur#1pt\advance\v@lX#1pt\repeat%
    \def\minst@p{10pt}%
    \v@lX=\@psfllx\p@\v@lX=\Sc@leFact\v@lX\ifdim\v@lX>\z@\r@undint\v@lX\v@lX\fi%
    \v@lY=\@psflly\p@\v@lY=\Sc@leFact\v@lY\r@undint\v@lY\v@lY%
    \delt@=\@psfurx\p@\delt@=\Sc@leFact\delt@%
    \advance\delt@-\v@lX\v@lYa=\@psfury\p@\v@lYa=\Sc@leFact\v@lYa\v@leur=\minst@p%
    \edef\valv@lX{\repdecn@mb{\v@lX}}\edef\LgTr@it{\the\delt@}%
    \loop\ifdim\v@lY<\v@lYa\edef\valv@lY{\repdecn@mb{\v@lY}}%
    \figptDD -1:(\valv@lX,\valv@lY)\figwritee -1:\vbox{\hrule width\LgTr@it}(0)%
    \ifdim\v@leur<\minst@p\else\figsetmark{$\triangleright$\kern4bp}%
    \figwritew -1:\@ffichnb{0}{\valv@lY}(6)\v@leur=\z@\figsetmark{}\fi%
    \advance\v@leur#2pt\advance\v@lY#2pt\repeat}
\ctr@ld@f
\ctr@ld@f\def\figscan@E#1(#2,#3){{\s@uvc@ntr@l\et@tfigscan@E%
    \Figdisc@rdLTS{#1}{\t@xt@}\pdfximage{\t@xt@}%
    \setbox\Gb@x=\hbox{\pdfrefximage\pdflastximage}%
    \edef\@psfllx{0}\v@lY=-\dp\Gb@x\edef\@psflly{\repdecn@mb{\v@lY}}%
    \edef\@psfurx{\repdecn@mb{\wd\Gb@x}}%
    \v@lY=\dp\Gb@x\advance\v@lY\ht\Gb@x\edef\@psfury{\repdecn@mb{\v@lY}}%
    \figscan@{#2}{#3}\resetc@ntr@l\et@tfigscan@E}\ignorespaces}
\ctr@ld@f\def\figshowpts[#1,#2]{{\figsetmark{$\bullet$}\figsetptname{\bf ##1}%
    \p@rtent=#2\relax\ifnum\p@rtent<\z@\p@rtent=\z@\fi%
    \s@mme=#1\relax\ifnum\s@mme<\z@\s@mme=\z@\fi%
    \loop\ifnum\s@mme<\p@rtent\pt@rvect{\s@mme}%
    \ifitis@K\figwriten{\the\s@mme}:(4pt)\fi\advance\s@mme\@ne\repeat%
    \pt@rvect{\s@mme}\ifitis@K\figwriten{\the\s@mme}:(4pt)\fi}\ignorespaces}
\ctr@ld@f\def\pt@rvect#1{\set@bjc@de{#1}%
    \expandafter\expandafter\expandafter\inqpt@rvec\csname\objc@de\endcsname:}
\ctr@ld@f\def\inqpt@rvec#1#2:{\if#1\C@dCl@spt\itis@Ktrue\else\itis@Kfalse\fi}
\ctr@ld@f\def\figshowsettings{{%
    \immediate\write16{====================================================================}%
    \immediate\write16{ Current settings about:}%
    \immediate\write16{ --- GENERAL ---}%
    \immediate\write16{Scale factor and Unit = \unit@util\space (\the\unit@)
     \space -> \BS@ figinit{ScaleFactorUnit}}%
    \immediate\write16{Update mode = \ifpsupdatem@de yes\else no\fi
     \space-> \BS@ psset(update=yes/no) or \BS@ pssetdefault(update=yes/no)}%
    \immediate\write16{ --- PRINTING ---}%
    \immediate\write16{Implicit point name = \ptn@me{i} \space-> \BS@ figsetptname{Name}}%
    \immediate\write16{Point marker = \the\c@nsymb \space -> \BS@ figsetmark{Mark}}%
    \immediate\write16{Print rounded coordinates = \ifr@undcoord yes\else no\fi
     \space-> \BS@ figsetroundcoord{yes/no}}%
    \immediate\write16{ --- GRAPHICAL (general) ---}%
    \immediate\write16{First-level (or primary) settings:}%
    \immediate\write16{ Color = \curr@ntcolor \space-> \BS@ psset(color=ColorDefinition)}%
    \immediate\write16{ Filling mode = \iffillm@de yes\else no\fi
     \space-> \BS@ psset(fillmode=yes/no)}%
    \immediate\write16{ Line join = \curr@ntjoin \space-> \BS@ psset(join=miter/round/bevel)}%
    \immediate\write16{ Line style = \curr@ntdash \space-> \BS@ psset(dash=Index/Pattern)}%
    \immediate\write16{ Line width = \curr@ntwidth
     \space-> \BS@ psset(width=real in PostScript units)}%
    \immediate\write16{Second-level (or secondary) settings:}%
    \immediate\write16{ Color = \sec@ndcolor \space-> \BS@ psset second(color=ColorDefinition)}%
    \immediate\write16{ Line style = \curr@ntseconddash
     \space-> \BS@ psset second(dash=Index/Pattern)}%
    \immediate\write16{ Line width = \curr@ntsecondwidth
     \space-> \BS@ psset second(width=real in PostScript units)}%
    \immediate\write16{Third-level (or ternary) settings:}%
    \immediate\write16{ Color = \th@rdcolor \space-> \BS@ psset third(color=ColorDefinition)}%
    \immediate\write16{ --- GRAPHICAL (specific) ---}%
    \immediate\write16{Arrow-head:}%
    \immediate\write16{ (half-)Angle = \@rrowheadangle
     \space-> \BS@ psset arrowhead(angle=real in degrees)}%
    \immediate\write16{ Filling mode = \if@rrowhfill yes\else no\fi
     \space-> \BS@ psset arrowhead(fillmode=yes/no)}%
    \immediate\write16{ "Outside" = \if@rrowhout yes\else no\fi
     \space-> \BS@ psset arrowhead(out=yes/no)}%
    \immediate\write16{ Length = \@rrowheadlength
     \if@rrowratio\space(not active)\else\space(active)\fi
     \space-> \BS@ psset arrowhead(length=real in user coord.)}%
    \immediate\write16{ Ratio = \@rrowheadratio
     \if@rrowratio\space(active)\else\space(not active)\fi
     \space-> \BS@ psset arrowhead(ratio=real in [0,1])}%
    \immediate\write16{Curve: Roundness = \curv@roundness
     \space-> \BS@ psset curve(roundness=real in [0,0.5])}%
    \immediate\write16{Mesh: Diagonal = \c@ntrolmesh
     \space-> \BS@ psset mesh(diag=integer in {-1,0,1})}%
    \immediate\write16{Flow chart:}%
    \immediate\write16{ Arrow position = \@rrowp@s
     \space-> \BS@ psset flowchart(arrowposition=real in [0,1])}%
    \immediate\write16{ Arrow reference point = \ifcase\@rrowr@fpt start\else end\fi
     \space-> \BS@ psset flowchart(arrowrefpt = start/end)}%     
    \immediate\write16{ Line type = \ifcase\fclin@typ@ curve\else polygon\fi
     \space-> \BS@ psset flowchart(line=polygon/curve)}%
    \immediate\write16{ Padding = (\Xp@dd, \Yp@dd)
     \space-> \BS@ psset flowchart(padding = real in user coord.)}%
    \immediate\write16{\space\space\space\space(or
     \BS@ psset flowchart(xpadding=real, ypadding=real) )}%
    \immediate\write16{ Radius = \fclin@r@d
     \space-> \BS@ psset flowchart(radius=positive real in user coord.)}%
    \immediate\write16{ Shape = \fcsh@pe
     \space-> \BS@ psset flowchart(shape = rectangle, ellipse or lozenge)}%
    \immediate\write16{ Thickness = \thickn@ss
     \space-> \BS@ psset flowchart(thickness = real in user coord.)}%
    \ifTr@isDim%
    \immediate\write16{ --- 3D to 2D PROJECTION ---}%
    \immediate\write16{Projection : \typ@proj \space-> \BS@ figinit{ScaleFactorUnit, ProjType}}%
    \immediate\write16{Longitude (psi) = \v@lPsi \space-> \BS@ figset proj(psi=real in degrees)}%
    \ifcase\curr@ntproj\immediate\write16{Depth coeff. (Lambda)
     \space = \v@lTheta \space-> \BS@ figset proj(lambda=real in [0,1])}%
    \else\immediate\write16{Latitude (theta)
     \space = \v@lTheta \space-> \BS@ figset proj(theta=real in degrees)}%
    \fi%
    \ifnum\curr@ntproj=\tw@%
    \immediate\write16{Observation distance = \disob@unit
     \space-> \BS@ figset proj(dist=real in user coord.)}%
    \immediate\write16{Target point = \t@rgetpt \space-> \BS@ figset proj(targetpt=pt number)}%
     \v@lX=\ptT@unit@\wd\Bt@rget\v@lY=\ptT@unit@\ht\Bt@rget\v@lZ=\ptT@unit@\dp\Bt@rget%
    \immediate\write16{ Its coordinates are
     (\repdecn@mb{\v@lX}, \repdecn@mb{\v@lY}, \repdecn@mb{\v@lZ})}%
    \fi%
    \fi%
    \immediate\write16{====================================================================}%
    \ignorespaces}}
\ctr@ln@w{newif}\ifitis@vect@r
\ctr@ld@f\def\figvectC#1(#2,#3){{\itis@vect@rtrue\figpt#1:(#2,#3)}\ignorespaces}
\ctr@ld@f\def\Figv@ctCreg#1(#2,#3){{\itis@vect@rtrue\Figp@intreg#1:(#2,#3)}\ignorespaces}
\ctr@ln@m\figvectDBezier
\ctr@ld@f\def\figvectDBezierDD#1:#2,#3[#4,#5,#6,#7]{\ifps@cri{\s@uvc@ntr@l\et@tfigvectDBezierDD%
    \FigvectDBezier@#2,#3[#4,#5,#6,#7]\v@lX=\c@ef\v@lX\v@lY=\c@ef\v@lY%
    \Figv@ctCreg#1(\v@lX,\v@lY)\resetc@ntr@l\et@tfigvectDBezierDD}\ignorespaces\fi}
\ctr@ld@f\def\figvectDBezierTD#1:#2,#3[#4,#5,#6,#7]{\ifps@cri{\s@uvc@ntr@l\et@tfigvectDBezierTD%
    \FigvectDBezier@#2,#3[#4,#5,#6,#7]\v@lX=\c@ef\v@lX\v@lY=\c@ef\v@lY\v@lZ=\c@ef\v@lZ%
    \Figv@ctCreg#1(\v@lX,\v@lY,\v@lZ)\resetc@ntr@l\et@tfigvectDBezierTD}\ignorespaces\fi}
\ctr@ld@f\def\FigvectDBezier@#1,#2[#3,#4,#5,#6]{\setc@ntr@l{2}%
    \edef\T@{#2}\v@leur=\p@\advance\v@leur-#2pt\edef\UNmT@{\repdecn@mb{\v@leur}}%
    \ifnum#1=\tw@\def\c@ef{6}\else\def\c@ef{3}\fi%
    \figptcopy-4:/#3/\figptcopy-3:/#4/\figptcopy-2:/#5/\figptcopy-1:/#6/%
    \l@mbd@un=-4 \l@mbd@de=-\thr@@\p@rtent=\m@ne\c@lDecast%
    \ifnum#1=\tw@\c@lDCDeux{-4}{-3}\c@lDCDeux{-3}{-2}\c@lDCDeux{-4}{-3}\else%
    \l@mbd@un=-4 \l@mbd@de=-\thr@@\p@rtent=-\tw@\c@lDecast%
    \c@lDCDeux{-4}{-3}\fi\Figg@tXY{-4}}
\ctr@ln@m\c@lDCDeux
\ctr@ld@f\def\c@lDCDeuxDD#1#2{\Figg@tXY{#2}\Figg@tXYa{#1}%
    \advance\v@lX-\v@lXa\advance\v@lY-\v@lYa\Figp@intregDD#1:(\v@lX,\v@lY)}
\ctr@ld@f\def\c@lDCDeuxTD#1#2{\Figg@tXY{#2}\Figg@tXYa{#1}\advance\v@lX-\v@lXa%
    \advance\v@lY-\v@lYa\advance\v@lZ-\v@lZa\Figp@intregTD#1:(\v@lX,\v@lY,\v@lZ)}
\ctr@ln@m\figvectN
\ctr@ld@f\def\figvectNDD#1[#2,#3]{\ifps@cri{\Figg@tXYa{#2}\Figg@tXY{#3}%
    \advance\v@lX-\v@lXa\advance\v@lY-\v@lYa%
    \Figv@ctCreg#1(-\v@lY,\v@lX)}\ignorespaces\fi}
\ctr@ld@f\def\figvectNTD#1[#2,#3,#4]{\ifps@cri{\vecunitC@TD[#2,#4]\v@lmin=\v@lX\v@lmax=\v@lY%
    \v@leur=\v@lZ\vecunitC@TD[#2,#3]\c@lprovec{#1}}\ignorespaces\fi}
\ctr@ln@m\figvectNV
\ctr@ld@f\def\figvectNVDD#1[#2]{\ifps@cri{\Figg@tXY{#2}\Figv@ctCreg#1(-\v@lY,\v@lX)}\ignorespaces\fi}
\ctr@ld@f\def\figvectNVTD#1[#2,#3]{\ifps@cri{\vecunitCV@TD{#3}\v@lmin=\v@lX\v@lmax=\v@lY%
    \v@leur=\v@lZ\vecunitCV@TD{#2}\c@lprovec{#1}}\ignorespaces\fi}
\ctr@ln@m\figvectP
\ctr@ld@f\def\figvectPDD#1[#2,#3]{\ifps@cri{\Figg@tXYa{#2}\Figg@tXY{#3}%
    \advance\v@lX-\v@lXa\advance\v@lY-\v@lYa%
    \Figv@ctCreg#1(\v@lX,\v@lY)}\ignorespaces\fi}
\ctr@ld@f\def\figvectPTD#1[#2,#3]{\ifps@cri{\Figg@tXYa{#2}\Figg@tXY{#3}%
    \advance\v@lX-\v@lXa\advance\v@lY-\v@lYa\advance\v@lZ-\v@lZa%
    \Figv@ctCreg#1(\v@lX,\v@lY,\v@lZ)}\ignorespaces\fi}
\ctr@ln@m\figvectU
\ctr@ld@f\def\figvectUDD#1[#2]{\ifps@cri{\n@rmeuc\v@leur{#2}\invers@\v@leur\v@leur%
    \delt@=\repdecn@mb{\v@leur}\unit@\edef\v@ldelt@{\repdecn@mb{\delt@}}%
    \Figg@tXY{#2}\v@lX=\v@ldelt@\v@lX\v@lY=\v@ldelt@\v@lY%
    \Figv@ctCreg#1(\v@lX,\v@lY)}\ignorespaces\fi}
\ctr@ld@f\def\figvectUTD#1[#2]{\ifps@cri{\n@rmeuc\v@leur{#2}\invers@\v@leur\v@leur%
    \delt@=\repdecn@mb{\v@leur}\unit@\edef\v@ldelt@{\repdecn@mb{\delt@}}%
    \Figg@tXY{#2}\v@lX=\v@ldelt@\v@lX\v@lY=\v@ldelt@\v@lY\v@lZ=\v@ldelt@\v@lZ%
    \Figv@ctCreg#1(\v@lX,\v@lY,\v@lZ)}\ignorespaces\fi}
\ctr@ld@f\def\figvisu#1#2#3{\c@ldefproj\initb@undb@x\xdef\figforTeXFigno{\figforTeXnextFigno}%
    \s@mme=\figforTeXnextFigno\advance\s@mme\@ne\xdef\figforTeXnextFigno{\number\s@mme}%
    \setbox\b@xvisu=\hbox{\ifnum\@utoFN>\z@\figinsert{}\gdef\@utoFInDone{0}\fi\ignorespaces#3}%
    \gdef\@utoFInDone{1}\gdef\@utoFN{0}%
    \v@lXa=-\c@@rdYmin\v@lYa=\c@@rdYmax\advance\v@lYa-\c@@rdYmin%
    \v@lX=\c@@rdXmax\advance\v@lX-\c@@rdXmin%
    \setbox#1=\hbox{#2}\v@lY=-\v@lX\maxim@m{\v@lX}{\v@lX}{\wd#1}%
    \advance\v@lY\v@lX\divide\v@lY\tw@\advance\v@lY-\c@@rdXmin%
    \setbox#1=\vbox{\parindent0mm\hsize=\v@lX\vskip\v@lYa%
    \rlap{\hskip\v@lY\smash{\raise\v@lXa\box\b@xvisu}}%
    \def\t@xt@{#2}\ifx\t@xt@\empty\else\medskip\centerline{#2}\fi}\wd#1=\v@lX}
\ctr@ld@f\def\figDecrementFigno{{\xdef\figforTeXnextFigno{\figforTeXFigno}%
    \s@mme=\figforTeXFigno\advance\s@mme\m@ne\xdef\figforTeXFigno{\number\s@mme}}}
\ctr@ln@w{newbox}\Bt@rget\setbox\Bt@rget=\null
\ctr@ln@w{newbox}\BminTD@\setbox\BminTD@=\null
\ctr@ln@w{newbox}\BmaxTD@\setbox\BmaxTD@=\null
\ctr@ln@w{newif}\ifnewt@rgetpt\ctr@ln@w{newif}\ifnewdis@b
\ctr@ld@f\def\b@undb@xTD#1#2#3{%
    \relax\ifdim#1<\wd\BminTD@\global\wd\BminTD@=#1\fi%
    \relax\ifdim#2<\ht\BminTD@\global\ht\BminTD@=#2\fi%
    \relax\ifdim#3<\dp\BminTD@\global\dp\BminTD@=#3\fi%
    \relax\ifdim#1>\wd\BmaxTD@\global\wd\BmaxTD@=#1\fi%
    \relax\ifdim#2>\ht\BmaxTD@\global\ht\BmaxTD@=#2\fi%
    \relax\ifdim#3>\dp\BmaxTD@\global\dp\BmaxTD@=#3\fi}
\ctr@ld@f\def\c@ldefdisob{{\ifdim\wd\BminTD@<\maxdimen\v@leur=\wd\BmaxTD@\advance\v@leur-\wd\BminTD@%
    \delt@=\ht\BmaxTD@\advance\delt@-\ht\BminTD@\maxim@m{\v@leur}{\v@leur}{\delt@}%
    \delt@=\dp\BmaxTD@\advance\delt@-\dp\BminTD@\maxim@m{\v@leur}{\v@leur}{\delt@}%
    \v@leur=5\v@leur\else\v@leur=800pt\fi\c@ldefdisob@{\v@leur}}}
\ctr@ln@m\disob@intern
\ctr@ln@m\disob@
\ctr@ln@m\divf@ctproj
\ctr@ld@f\def\c@ldefdisob@#1{{\v@leur=#1\ifdim\v@leur<\p@\v@leur=800pt\fi%
    \xdef\disob@intern{\repdecn@mb{\v@leur}}%
    \delt@=\ptT@unit@\v@leur\xdef\disob@unit{\repdecn@mb{\delt@}}%
    \f@ctech=\@ne\loop\ifdim\v@leur>\t@n pt\divide\v@leur\t@n\multiply\f@ctech\t@n\repeat%
    \xdef\disob@{\repdecn@mb{\v@leur}}\xdef\divf@ctproj{\the\f@ctech}}%
    \global\newdis@btrue}
\ctr@ln@m\t@rgetpt
\ctr@ld@f\def\c@ldeft@rgetpt{\newt@rgetpttrue\def\t@rgetpt{CenterBoundBox}{%
    \delt@=\wd\BmaxTD@\advance\delt@-\wd\BminTD@\divide\delt@\tw@%
    \v@leur=\wd\BminTD@\advance\v@leur\delt@\global\wd\Bt@rget=\v@leur%
    \delt@=\ht\BmaxTD@\advance\delt@-\ht\BminTD@\divide\delt@\tw@%
    \v@leur=\ht\BminTD@\advance\v@leur\delt@\global\ht\Bt@rget=\v@leur%
    \delt@=\dp\BmaxTD@\advance\delt@-\dp\BminTD@\divide\delt@\tw@%
    \v@leur=\dp\BminTD@\advance\v@leur\delt@\global\dp\Bt@rget=\v@leur}}
\ctr@ln@m\c@ldefproj
\ctr@ld@f\def\c@ldefprojTD{\ifnewt@rgetpt\else\c@ldeft@rgetpt\fi\ifnewdis@b\else\c@ldefdisob\fi}
\ctr@ld@f\def\c@lprojcav{% Projection cavaliere : X = x + y L cos t, Y = z + y L sin t
    \v@lZa=\cxa@\v@lY\advance\v@lX\v@lZa%
    \v@lZa=\cxb@\v@lY\v@lY=\v@lZ\advance\v@lY\v@lZa\ignorespaces}
\ctr@ln@m\v@lcoef
\ctr@ld@f\def\c@lprojrea{% Projection realiste
    \advance\v@lX-\wd\Bt@rget\advance\v@lY-\ht\Bt@rget\advance\v@lZ-\dp\Bt@rget%
    \v@lZa=\cza@\v@lX\advance\v@lZa\czb@\v@lY\advance\v@lZa\czc@\v@lZ%
    \divide\v@lZa\divf@ctproj\advance\v@lZa\disob@ pt\invers@{\v@lZa}{\v@lZa}%
    \v@lZa=\disob@\v@lZa\edef\v@lcoef{\repdecn@mb{\v@lZa}}%
    \v@lXa=\cxa@\v@lX\advance\v@lXa\cxb@\v@lY\v@lXa=\v@lcoef\v@lXa%
    \v@lY=\cyb@\v@lY\advance\v@lY\cya@\v@lX\advance\v@lY\cyc@\v@lZ%
    \v@lY=\v@lcoef\v@lY\v@lX=\v@lXa\ignorespaces}
\ctr@ld@f\def\c@lprojort{% Projection orthogonale
    \v@lXa=\cxa@\v@lX\advance\v@lXa\cxb@\v@lY%
    \v@lY=\cyb@\v@lY\advance\v@lY\cya@\v@lX\advance\v@lY\cyc@\v@lZ%
    \v@lX=\v@lXa\ignorespaces}
\ctr@ld@f\def\Figptpr@j#1:#2/#3/{{\Figg@tXY{#3}\superc@lprojSP%
    \Figp@intregDD#1:{#2}(\v@lX,\v@lY)}\ignorespaces}
\ctr@ln@m\figsetobdist
\ctr@ld@f\def\figsetobdistDD{\un@v@ilable{figsetobdist}}
\ctr@ld@f\def\figsetobdistTD(#1){{\ifcurr@ntPS%
    \immediate\write16{*** \BS@ figsetobdist is ignored inside a
     \BS@ psbeginfig-\BS@ psendfig block.}%
    \else\v@leur=#1\unit@\c@ldefdisob@{\v@leur}\fi}\ignorespaces}
\ctr@ln@m\c@lprojSP
\ctr@ln@m\curr@ntproj
\ctr@ln@m\typ@proj
\ctr@ln@m\superc@lprojSP
\ctr@ld@f\def\Figs@tproj#1{%
    \if#13 \d@faultproj\else\if#1c\d@faultproj%
    \else\if#1o\xdef\curr@ntproj{1}\xdef\typ@proj{orthogonal}%
         \figsetviewTD(\def@ultpsi,\def@ulttheta)%
         \global\let\c@lprojSP=\c@lprojort\global\let\superc@lprojSP=\c@lprojort%
    \else\if#1r\xdef\curr@ntproj{2}\xdef\typ@proj{realistic}%
         \figsetviewTD(\def@ultpsi,\def@ulttheta)%
         \global\let\c@lprojSP=\c@lprojrea\global\let\superc@lprojSP=\c@lprojrea%
    \else\d@faultproj\message{*** Unknown projection. Cavalier projection assumed.}%
    \fi\fi\fi\fi}
\ctr@ld@f\def\d@faultproj{\xdef\curr@ntproj{0}\xdef\typ@proj{cavalier}\figsetviewTD(\def@ultpsi,0.5)%
         \global\let\c@lprojSP=\c@lprojcav\global\let\superc@lprojSP=\c@lprojcav}
\ctr@ln@m\figsettarget
\ctr@ld@f\def\figsettargetDD{\un@v@ilable{figsettarget}}
\ctr@ld@f\def\figsettargetTD[#1]{{\ifcurr@ntPS%
    \immediate\write16{*** \BS@ figsettarget is ignored inside a
     \BS@ psbeginfig-\BS@ psendfig block.}%
    \else\global\newt@rgetpttrue\xdef\t@rgetpt{#1}\Figg@tXY{#1}\global\wd\Bt@rget=\v@lX%
    \global\ht\Bt@rget=\v@lY\global\dp\Bt@rget=\v@lZ\fi}\ignorespaces}
\ctr@ln@m\figsetview
\ctr@ld@f\def\figsetviewDD{\un@v@ilable{figsetview}}
\ctr@ld@f\def\figsetviewTD(#1){\ifcurr@ntPS%
     \immediate\write16{*** \BS@ figsetview is ignored inside a
     \BS@ psbeginfig-\BS@ psendfig block.}\else\Figsetview@#1,:\fi\ignorespaces}
\ctr@ld@f\def\Figsetview@#1,#2:{{\xdef\v@lPsi{#1}\def\t@xt@{#2}%
    \ifx\t@xt@\empty\def\@rgdeux{\v@lTheta}\else\X@rgdeux@#2\fi%
    \c@ssin{\costhet@}{\sinthet@}{#1}\v@lmin=\costhet@ pt\v@lmax=\sinthet@ pt%
    \ifcase\curr@ntproj%
    \v@leur=\@rgdeux\v@lmin\xdef\cxa@{\repdecn@mb{\v@leur}}%
    \v@leur=\@rgdeux\v@lmax\xdef\cxb@{\repdecn@mb{\v@leur}}\v@leur=\@rgdeux pt%
    \relax\ifdim\v@leur>\p@\message{*** Lambda too large ! See \BS@ figset proj() !}\fi%
    \else%
    \v@lmax=-\v@lmax\xdef\cxa@{\repdecn@mb{\v@lmax}}\xdef\cxb@{\costhet@}%
    \ifx\t@xt@\empty\edef\@rgdeux{\def@ulttheta}\fi\c@ssin{\C@}{\S@}{\@rgdeux}%
    \v@lmax=-\S@ pt%
    \v@leur=\v@lmax\v@leur=\costhet@\v@leur\xdef\cya@{\repdecn@mb{\v@leur}}%
    \v@leur=\v@lmax\v@leur=\sinthet@\v@leur\xdef\cyb@{\repdecn@mb{\v@leur}}%
    \xdef\cyc@{\C@}\v@lmin=-\C@ pt%
    \v@leur=\v@lmin\v@leur=\costhet@\v@leur\xdef\cza@{\repdecn@mb{\v@leur}}%
    \v@leur=\v@lmin\v@leur=\sinthet@\v@leur\xdef\czb@{\repdecn@mb{\v@leur}}%
    \xdef\czc@{\repdecn@mb{\v@lmax}}\fi%
    \xdef\v@lTheta{\@rgdeux}}}
\ctr@ld@f\def\def@ultpsi{40}
\ctr@ld@f\def\def@ulttheta{25}
\ctr@ln@m\l@debut
\ctr@ln@m\n@mref
\ctr@ld@f\def\figset#1(#2){\def\t@xt@{#1}\ifx\t@xt@\empty\trtlis@rg{#2}{\Figsetwr@te}% write
    \else\keln@mde#1|%
    \def\n@mref{pr}\ifx\l@debut\n@mref\ifcurr@ntPS% projection
     \immediate\write16{*** \BS@ figset proj(...) is ignored inside a
     \BS@ psbeginfig-\BS@ psendfig block.}\else\trtlis@rg{#2}{\Figsetpr@j}\fi\else%
    \def\n@mref{wr}\ifx\l@debut\n@mref\trtlis@rg{#2}{\Figsetwr@te}\else% write
    \immediate\write16{*** Unknown keyword: \BS@ figset #1(...)}%
    \fi\fi\fi\ignorespaces}
\ctr@ld@f\def\Figsetpr@j#1=#2|{\keln@mtr#1|%
    \def\n@mref{dep}\ifx\l@debut\n@mref\Figsetd@p{#2}\else% depth (lambda)
    \def\n@mref{dis}\ifx\l@debut\n@mref%
     \ifnum\curr@ntproj=\tw@\figsetobdist(#2)\else\Figset@rr\fi\else% dist
    \def\n@mref{lam}\ifx\l@debut\n@mref\Figsetd@p{#2}\else% depth (lambda)
    \def\n@mref{lat}\ifx\l@debut\n@mref\Figsetth@{#2}\else% latitude (theta)
    \def\n@mref{lon}\ifx\l@debut\n@mref\figsetview(#2)\else% longitude (psi)
    \def\n@mref{psi}\ifx\l@debut\n@mref\figsetview(#2)\else% longitude (psi)
    \def\n@mref{tar}\ifx\l@debut\n@mref%
     \ifnum\curr@ntproj=\tw@\figsettarget[#2]\else\Figset@rr\fi\else% target point
    \def\n@mref{the}\ifx\l@debut\n@mref\Figsetth@{#2}\else% latitude (theta)
    \immediate\write16{*** Unknown attribute: \BS@ figset proj(..., #1=...).}%
    \fi\fi\fi\fi\fi\fi\fi\fi}
\ctr@ld@f\def\Figsetd@p#1{\ifnum\curr@ntproj=\z@\figsetview(\v@lPsi,#1)\else\Figset@rr\fi}
\ctr@ld@f\def\Figsetth@#1{\ifnum\curr@ntproj=\z@\Figset@rr\else\figsetview(\v@lPsi,#1)\fi}
\ctr@ld@f\def\Figset@rr{\message{*** \BS@ figset proj(): Attribute "\n@mref" ignored, incompatible
    with current projection}}
\ctr@ld@f\def\initb@undb@xTD{\wd\BminTD@=\maxdimen\ht\BminTD@=\maxdimen\dp\BminTD@=\maxdimen%
    \wd\BmaxTD@=-\maxdimen\ht\BmaxTD@=-\maxdimen\dp\BmaxTD@=-\maxdimen}
\ctr@ln@w{newbox}\Gb@x      % boite a tout faire
\ctr@ln@w{newbox}\Gb@xSC    % boite qui contient le point marker
\ctr@ln@w{newtoks}\c@nsymb  % the point marker
\ctr@ln@w{newif}\ifr@undcoord\ctr@ln@w{newif}\ifunitpr@sent
\ctr@ld@f\def\unssqrttw@{0.707106 }
\ctr@ld@f\def\figAst{\raise-1.15ex\hbox{$\ast$}}
\ctr@ld@f\def\figBullet{\raise-1.15ex\hbox{$\bullet$}}
\ctr@ld@f\def\figCirc{\raise-1.15ex\hbox{$\circ$}}
\ctr@ld@f\def\figDiamond{\raise-1.15ex\hbox{$\diamond$}}%
\ctr@ld@f\def\boxit#1#2{\leavevmode\hbox{\vrule\vbox{\hrule\vglue#1%
    \vtop{\hbox{\kern#1{#2}\kern#1}\vglue#1\hrule}}\vrule}}
\ctr@ld@f
\ctr@ld@f
\ctr@ld@f\def\c@nterpt{\ignorespaces%
    \kern-.5\wd\Gb@xSC%
    \raise-.5\ht\Gb@xSC\rlap{\hbox{\raise.5\dp\Gb@xSC\hbox{\copy\Gb@xSC}}}%
    \kern .5\wd\Gb@xSC\ignorespaces}
\ctr@ld@f\def\b@undb@xSC#1#2{{\v@lXa=#1\v@lYa=#2%
    \v@leur=\ht\Gb@xSC\advance\v@leur\dp\Gb@xSC%
    \advance\v@lXa-.5\wd\Gb@xSC\advance\v@lYa-.5\v@leur\b@undb@x{\v@lXa}{\v@lYa}%
    \advance\v@lXa\wd\Gb@xSC\advance\v@lYa\v@leur\b@undb@x{\v@lXa}{\v@lYa}}}
\ctr@ln@m\Dist@n
\ctr@ln@m\l@suite
\ctr@ld@f\def\@keldist#1#2{\edef\Dist@n{#2}\y@tiunit{\Dist@n}%
    \ifunitpr@sent#1=\Dist@n\else#1=\Dist@n\unit@\fi}
\ctr@ld@f\def\y@tiunit#1{\unitpr@sentfalse\expandafter\y@tiunit@#1:}
\ctr@ld@f\def\y@tiunit@#1#2:{\ifcat#1a\unitpr@senttrue\else\def\l@suite{#2}%
    \ifx\l@suite\empty\else\y@tiunit@#2:\fi\fi}
\ctr@ln@m\figcoord
\ctr@ld@f\def\figcoordDD#1{{\v@lX=\ptT@unit@\v@lX\v@lY=\ptT@unit@\v@lY%
    \ifr@undcoord\ifcase#1\v@leur=0.5pt\or\v@leur=0.05pt\or\v@leur=0.005pt%
    \or\v@leur=0.0005pt\else\v@leur=\z@\fi%
    \ifdim\v@lX<\z@\advance\v@lX-\v@leur\else\advance\v@lX\v@leur\fi%
    \ifdim\v@lY<\z@\advance\v@lY-\v@leur\else\advance\v@lY\v@leur\fi\fi%
    (\@ffichnb{#1}{\repdecn@mb{\v@lX}},\ifmmode\else\thinspace\fi%
    \@ffichnb{#1}{\repdecn@mb{\v@lY}})}}
\ctr@ld@f\def\@ffichnb#1#2{{\def\@@ffich{\@ffich#1(}\edef\n@mbre{#2}%
    \expandafter\@@ffich\n@mbre)}}
\ctr@ld@f\def\@ffich#1(#2.#3){{#2\ifnum#1>\z@.\fi\def\dig@ts{#3}\s@mme=\z@%
    \loop\ifnum\s@mme<#1\expandafter\@ffichdec\dig@ts:\advance\s@mme\@ne\repeat}}
\ctr@ld@f\def\@ffichdec#1#2:{\relax#1\def\dig@ts{#20}}
\ctr@ld@f\def\figcoordTD#1{{\v@lX=\ptT@unit@\v@lX\v@lY=\ptT@unit@\v@lY\v@lZ=\ptT@unit@\v@lZ%
    \ifr@undcoord\ifcase#1\v@leur=0.5pt\or\v@leur=0.05pt\or\v@leur=0.005pt%
    \or\v@leur=0.0005pt\else\v@leur=\z@\fi%
    \ifdim\v@lX<\z@\advance\v@lX-\v@leur\else\advance\v@lX\v@leur\fi%
    \ifdim\v@lY<\z@\advance\v@lY-\v@leur\else\advance\v@lY\v@leur\fi%
    \ifdim\v@lZ<\z@\advance\v@lZ-\v@leur\else\advance\v@lZ\v@leur\fi\fi%
    (\@ffichnb{#1}{\repdecn@mb{\v@lX}},\ifmmode\else\thinspace\fi%
     \@ffichnb{#1}{\repdecn@mb{\v@lY}},\ifmmode\else\thinspace\fi%
     \@ffichnb{#1}{\repdecn@mb{\v@lZ}})}}
\ctr@ld@f\def\figsetroundcoord#1{\expandafter\Figsetr@undcoord#1:\ignorespaces}
\ctr@ld@f\def\Figsetr@undcoord#1#2:{\if#1n\r@undcoordfalse\else\r@undcoordtrue\fi}
\ctr@ld@f\def\Figsetwr@te#1=#2|{\keln@mun#1|%
    \def\n@mref{m}\ifx\l@debut\n@mref\figsetmark{#2}\else% mark
    \immediate\write16{*** Unknown attribute: \BS@ figset (..., #1=...)}%
    \fi}
\ctr@ld@f\def\figsetmark#1{\c@nsymb={#1}\setbox\Gb@xSC=\hbox{\the\c@nsymb}\ignorespaces}
\ctr@ln@m\ptn@me
\ctr@ld@f\def\figsetptname#1{\def\ptn@me##1{#1}\ignorespaces}
\ctr@ld@f\def\FigWrit@L#1:#2(#3,#4){\ignorespaces\@keldist\v@leur{#3}\@keldist\delt@{#4}%
    \C@rp@r@m\def\list@num{#1}\@ecfor\p@int:=\list@num\do{\FigWrit@pt{\p@int}{#2}}}
\ctr@ld@f\def\FigWrit@pt#1#2{\FigWp@r@m{#1}{#2}\Vc@rrect\figWp@si%
    \ifdim\wd\Gb@xSC>\z@\b@undb@xSC{\v@lX}{\v@lY}\fi\figWBB@x}
\ctr@ld@f\def\FigWp@r@m#1#2{\Figg@tXY{#1}%
    \setbox\Gb@x=\hbox{\def\t@xt@{#2}\ifx\t@xt@\empty\Figg@tT{#1}\else#2\fi}\c@lprojSP}
\ctr@ld@f\let\Vc@rrect=\relax
\ctr@ld@f\let\C@rp@r@m=\relax
\ctr@ld@f\def\figwrite[#1]#2{{\ignorespaces\def\list@num{#1}\@ecfor\p@int:=\list@num\do{%
    \setbox\Gb@x=\hbox{\def\t@xt@{#2}\ifx\t@xt@\empty\Figg@tT{\p@int}\else#2\fi}%
    \Figwrit@{\p@int}}}\ignorespaces}
\ctr@ld@f\def\Figwrit@#1{\Figg@tXY{#1}\c@lprojSP%
    \rlap{\kern\v@lX\raise\v@lY\hbox{\unhcopy\Gb@x}}\v@leur=\v@lY%
    \advance\v@lY\ht\Gb@x\b@undb@x{\v@lX}{\v@lY}\advance\v@lX\wd\Gb@x%
    \v@lY=\v@leur\advance\v@lY-\dp\Gb@x\b@undb@x{\v@lX}{\v@lY}}
\ctr@ld@f\def\figwritec[#1]#2{{\ignorespaces\def\list@num{#1}%
    \@ecfor\p@int:=\list@num\do{\Figwrit@c{\p@int}{#2}}}\ignorespaces}
\ctr@ld@f\def\Figwrit@c#1#2{\FigWp@r@m{#1}{#2}%
    \rlap{\kern\v@lX\raise\v@lY\hbox{\rlap{\kern-.5\wd\Gb@x%
    \raise-.5\ht\Gb@x\hbox{\raise.5\dp\Gb@x\hbox{\unhcopy\Gb@x}}}}}%
    \v@leur=\ht\Gb@x\advance\v@leur\dp\Gb@x%
    \advance\v@lX-.5\wd\Gb@x\advance\v@lY-.5\v@leur\b@undb@x{\v@lX}{\v@lY}%
    \advance\v@lX\wd\Gb@x\advance\v@lY\v@leur\b@undb@x{\v@lX}{\v@lY}}
\ctr@ld@f\def\figwritep[#1]{{\ignorespaces\def\list@num{#1}\setbox\Gb@x=\hbox{\c@nterpt}%
    \@ecfor\p@int:=\list@num\do{\Figwrit@{\p@int}}}\ignorespaces}
\ctr@ld@f\def\figwritew#1:#2(#3){\figwritegcw#1:{#2}(#3,0pt)}
\ctr@ld@f\def\figwritee#1:#2(#3){\figwritegce#1:{#2}(#3,0pt)}
\ctr@ld@f\def\figwriten#1:#2(#3){{\def\Vc@rrect{\v@lZ=\v@leur\advance\v@lZ\dp\Gb@x}%
    \Figwrit@NS#1:{#2}(#3)}\ignorespaces}
\ctr@ld@f\def\figwrites#1:#2(#3){{\def\Vc@rrect{\v@lZ=-\v@leur\advance\v@lZ-\ht\Gb@x}%
    \Figwrit@NS#1:{#2}(#3)}\ignorespaces}
\ctr@ld@f\def\Figwrit@NS#1:#2(#3){\let\figWp@si=\FigWp@siNS\let\figWBB@x=\FigWBB@xNS%
    \FigWrit@L#1:{#2}(#3,0pt)}
\ctr@ld@f\def\FigWp@siNS{\rlap{\kern\v@lX\raise\v@lY\hbox{\rlap{\kern-.5\wd\Gb@x%
    \raise\v@lZ\hbox{\unhcopy\Gb@x}}\c@nterpt}}}
\ctr@ld@f\def\FigWBB@xNS{\advance\v@lY\v@lZ%
    \advance\v@lY-\dp\Gb@x\advance\v@lX-.5\wd\Gb@x\b@undb@x{\v@lX}{\v@lY}%
    \advance\v@lY\ht\Gb@x\advance\v@lY\dp\Gb@x%
    \advance\v@lX\wd\Gb@x\b@undb@x{\v@lX}{\v@lY}}
\ctr@ld@f\def\figwritenw#1:#2(#3){{\let\figWp@si=\FigWp@sigW\let\figWBB@x=\FigWBB@xgWE%
    \def\C@rp@r@m{\v@leur=\unssqrttw@\v@leur\delt@=\v@leur%
    \ifdim\delt@=\z@\delt@=\epsil@n\fi}\let@xte={-}\FigWrit@L#1:{#2}(#3,0pt)}\ignorespaces}
\ctr@ld@f\def\figwritesw#1:#2(#3){{\let\figWp@si=\FigWp@sigW\let\figWBB@x=\FigWBB@xgWE%
    \def\C@rp@r@m{\v@leur=\unssqrttw@\v@leur\delt@=-\v@leur%
    \ifdim\delt@=\z@\delt@=-\epsil@n\fi}\let@xte={-}\FigWrit@L#1:{#2}(#3,0pt)}\ignorespaces}
\ctr@ld@f\def\figwritene#1:#2(#3){{\let\figWp@si=\FigWp@sigE\let\figWBB@x=\FigWBB@xgWE%
    \def\C@rp@r@m{\v@leur=\unssqrttw@\v@leur\delt@=\v@leur%
    \ifdim\delt@=\z@\delt@=\epsil@n\fi}\let@xte={}\FigWrit@L#1:{#2}(#3,0pt)}\ignorespaces}
\ctr@ld@f\def\figwritese#1:#2(#3){{\let\figWp@si=\FigWp@sigE\let\figWBB@x=\FigWBB@xgWE%
    \def\C@rp@r@m{\v@leur=\unssqrttw@\v@leur\delt@=-\v@leur%
    \ifdim\delt@=\z@\delt@=-\epsil@n\fi}\let@xte={}\FigWrit@L#1:{#2}(#3,0pt)}\ignorespaces}
\ctr@ld@f\def\figwritegw#1:#2(#3,#4){{\let\figWp@si=\FigWp@sigW\let\figWBB@x=\FigWBB@xgWE%
    \let@xte={-}\FigWrit@L#1:{#2}(#3,#4)}\ignorespaces}
\ctr@ld@f\def\figwritege#1:#2(#3,#4){{\let\figWp@si=\FigWp@sigE\let\figWBB@x=\FigWBB@xgWE%
    \let@xte={}\FigWrit@L#1:{#2}(#3,#4)}\ignorespaces}
\ctr@ld@f\def\FigWp@sigW{\v@lXa=\z@\v@lYa=\ht\Gb@x\advance\v@lYa\dp\Gb@x%
    \ifdim\delt@>\z@\relax%
    \rlap{\kern\v@lX\raise\v@lY\hbox{\rlap{\kern-\wd\Gb@x\kern-\v@leur%
          \raise\delt@\hbox{\raise\dp\Gb@x\hbox{\unhcopy\Gb@x}}}\c@nterpt}}%
    \else\ifdim\delt@<\z@\relax\v@lYa=-\v@lYa%
    \rlap{\kern\v@lX\raise\v@lY\hbox{\rlap{\kern-\wd\Gb@x\kern-\v@leur%
          \raise\delt@\hbox{\raise-\ht\Gb@x\hbox{\unhcopy\Gb@x}}}\c@nterpt}}%
    \else\v@lXa=-.5\v@lYa%
    \rlap{\kern\v@lX\raise\v@lY\hbox{\rlap{\kern-\wd\Gb@x\kern-\v@leur%
          \raise-.5\ht\Gb@x\hbox{\raise.5\dp\Gb@x\hbox{\unhcopy\Gb@x}}}\c@nterpt}}%
    \fi\fi}
\ctr@ld@f\def\FigWp@sigE{\v@lXa=\z@\v@lYa=\ht\Gb@x\advance\v@lYa\dp\Gb@x%
    \ifdim\delt@>\z@\relax%
    \rlap{\kern\v@lX\raise\v@lY\hbox{\c@nterpt\kern\v@leur%
          \raise\delt@\hbox{\raise\dp\Gb@x\hbox{\unhcopy\Gb@x}}}}%
    \else\ifdim\delt@<\z@\relax\v@lYa=-\v@lYa%
    \rlap{\kern\v@lX\raise\v@lY\hbox{\c@nterpt\kern\v@leur%
          \raise\delt@\hbox{\raise-\ht\Gb@x\hbox{\unhcopy\Gb@x}}}}%
    \else\v@lXa=-.5\v@lYa%
    \rlap{\kern\v@lX\raise\v@lY\hbox{\c@nterpt\kern\v@leur%
          \raise-.5\ht\Gb@x\hbox{\raise.5\dp\Gb@x\hbox{\unhcopy\Gb@x}}}}%
    \fi\fi}
\ctr@ld@f\def\FigWBB@xgWE{\advance\v@lY\delt@%
    \advance\v@lX\the\let@xte\v@leur\advance\v@lY\v@lXa\b@undb@x{\v@lX}{\v@lY}%
    \advance\v@lX\the\let@xte\wd\Gb@x\advance\v@lY\v@lYa\b@undb@x{\v@lX}{\v@lY}}
\ctr@ld@f\def\figwritegcw#1:#2(#3,#4){{\let\figWp@si=\FigWp@sigcW\let\figWBB@x=\FigWBB@xgcWE%
    \let@xte={-}\FigWrit@L#1:{#2}(#3,#4)}\ignorespaces}
\ctr@ld@f\def\figwritegce#1:#2(#3,#4){{\let\figWp@si=\FigWp@sigcE\let\figWBB@x=\FigWBB@xgcWE%
    \let@xte={}\FigWrit@L#1:{#2}(#3,#4)}\ignorespaces}
\ctr@ld@f\def\FigWp@sigcW{\rlap{\kern\v@lX\raise\v@lY\hbox{\rlap{\kern-\wd\Gb@x\kern-\v@leur%
     \raise-.5\ht\Gb@x\hbox{\raise\delt@\hbox{\raise.5\dp\Gb@x\hbox{\unhcopy\Gb@x}}}}%
     \c@nterpt}}}
\ctr@ld@f\def\FigWp@sigcE{\rlap{\kern\v@lX\raise\v@lY\hbox{\c@nterpt\kern\v@leur%
    \raise-.5\ht\Gb@x\hbox{\raise\delt@\hbox{\raise.5\dp\Gb@x\hbox{\unhcopy\Gb@x}}}}}}
\ctr@ld@f\def\FigWBB@xgcWE{\v@lZ=\ht\Gb@x\advance\v@lZ\dp\Gb@x%
    \advance\v@lX\the\let@xte\v@leur\advance\v@lY\delt@\advance\v@lY.5\v@lZ%
    \b@undb@x{\v@lX}{\v@lY}%
    \advance\v@lX\the\let@xte\wd\Gb@x\advance\v@lY-\v@lZ\b@undb@x{\v@lX}{\v@lY}}
\ctr@ld@f\def\figwritebn#1:#2(#3){{\def\Vc@rrect{\v@lZ=\v@leur}\Figwrit@NS#1:{#2}(#3)}\ignorespaces}
\ctr@ld@f\def\figwritebs#1:#2(#3){{\def\Vc@rrect{\v@lZ=-\v@leur}\Figwrit@NS#1:{#2}(#3)}\ignorespaces}
\ctr@ld@f\def\figwritebw#1:#2(#3){{\let\figWp@si=\FigWp@sibW\let\figWBB@x=\FigWBB@xbWE%
    \let@xte={-}\FigWrit@L#1:{#2}(#3,0pt)}\ignorespaces}
\ctr@ld@f\def\figwritebe#1:#2(#3){{\let\figWp@si=\FigWp@sibE\let\figWBB@x=\FigWBB@xbWE%
    \let@xte={}\FigWrit@L#1:{#2}(#3,0pt)}\ignorespaces}
\ctr@ld@f\def\FigWp@sibW{\rlap{\kern\v@lX\raise\v@lY\hbox{\rlap{\kern-\wd\Gb@x\kern-\v@leur%
          \hbox{\unhcopy\Gb@x}}\c@nterpt}}}
\ctr@ld@f\def\FigWp@sibE{\rlap{\kern\v@lX\raise\v@lY\hbox{\c@nterpt\kern\v@leur%
          \hbox{\unhcopy\Gb@x}}}}
\ctr@ld@f\def\FigWBB@xbWE{\v@lZ=\ht\Gb@x\advance\v@lZ\dp\Gb@x%
    \advance\v@lX\the\let@xte\v@leur\advance\v@lY\ht\Gb@x\b@undb@x{\v@lX}{\v@lY}%
    \advance\v@lX\the\let@xte\wd\Gb@x\advance\v@lY-\v@lZ\b@undb@x{\v@lX}{\v@lY}}
\ctr@ln@w{newread}\frf@g  \ctr@ln@w{newwrite}\fwf@g
\ctr@ln@w{newif}\ifcurr@ntPS
\ctr@ln@w{newif}\ifps@cri
\ctr@ln@w{newif}\ifUse@llipse
\ctr@ln@w{newif}\ifpsdebugmode \psdebugmodefalse 
\ctr@ln@w{newif}\ifPDFm@ke
\ifx\pdfliteral\undefined\else\ifnum\pdfoutput>\z@\PDFm@ketrue\fi\fi
\ctr@ld@f\def\initPDF@rDVI{%
\ifPDFm@ke
 \let\figscan=\figscan@E
 \let\newGr@FN=\newGr@FNPDF
 \ctr@ld@f\def\c@mcurveto{c}
 \ctr@ld@f\def\c@mfill{f}
 \ctr@ld@f\def\c@mgsave{q}
 \ctr@ld@f\def\c@mgrestore{Q}
 \ctr@ld@f\def\c@mlineto{l}
 \ctr@ld@f\def\c@mmoveto{m}
 \ctr@ld@f\def\c@msetgray{g}     \ctr@ld@f\def\c@msetgrayStroke{G}
 \ctr@ld@f\def\c@msetcmykcolor{k}\ctr@ld@f\def\c@msetcmykcolorStroke{K}
 \ctr@ld@f\def\c@msetrgbcolor{rg}\ctr@ld@f\def\c@msetrgbcolorStroke{RG}
 \ctr@ld@f\def\d@fprimarC@lor{\curr@ntcolor\space\curr@ntcolorc@md%
               \space\curr@ntcolor\space\curr@ntcolorc@mdStroke}
 \ctr@ld@f\def\d@fsecondC@lor{\sec@ndcolor\space\sec@ndcolorc@md%
               \space\sec@ndcolor\space\sec@ndcolorc@mdStroke}
 \ctr@ld@f\def\d@fthirdC@lor{\th@rdcolor\space\th@rdcolorc@md%
              \space\th@rdcolor\space\th@rdcolorc@mdStroke}
 \ctr@ld@f\def\c@msetdash{d}
 \ctr@ld@f\def\c@msetlinejoin{j}
 \ctr@ld@f\def\c@msetlinewidth{w}
 \ctr@ld@f\def\f@gclosestroke{\immediate\write\fwf@g{s}}
 \ctr@ld@f\def\f@gfill{\immediate\write\fwf@g{\fillc@md}}% Voir la def de \fillc@md ****
 \ctr@ld@f\def\f@gnewpath{}
 \ctr@ld@f\def\f@gstroke{\immediate\write\fwf@g{S}}
\else
 \let\figinsertE=\figinsert
 \let\newGr@FN=\newGr@FNDVI
 \ctr@ld@f\def\c@mcurveto{curveto}
 \ctr@ld@f\def\c@mfill{fill}
 \ctr@ld@f\def\c@mgsave{gsave}
 \ctr@ld@f\def\c@mgrestore{grestore}
 \ctr@ld@f\def\c@mlineto{lineto}
 \ctr@ld@f\def\c@mmoveto{moveto}
 \ctr@ld@f\def\c@msetgray{setgray}          \ctr@ld@f\def\c@msetgrayStroke{}
 \ctr@ld@f\def\c@msetcmykcolor{setcmykcolor}\ctr@ld@f\def\c@msetcmykcolorStroke{}
 \ctr@ld@f\def\c@msetrgbcolor{setrgbcolor}  \ctr@ld@f\def\c@msetrgbcolorStroke{}
 \ctr@ld@f\def\d@fprimarC@lor{\curr@ntcolor\space\curr@ntcolorc@md}
 \ctr@ld@f\def\d@fsecondC@lor{\sec@ndcolor\space\sec@ndcolorc@md}
 \ctr@ld@f\def\d@fthirdC@lor{\th@rdcolor\space\th@rdcolorc@md}
 \ctr@ld@f\def\c@msetdash{setdash}
 \ctr@ld@f\def\c@msetlinejoin{setlinejoin}
 \ctr@ld@f\def\c@msetlinewidth{setlinewidth}
 \ctr@ld@f\def\f@gclosestroke{\immediate\write\fwf@g{closepath\space stroke}}
 \ctr@ld@f\def\f@gfill{\immediate\write\fwf@g{\fillc@md}}
 \ctr@ld@f\def\f@gnewpath{\immediate\write\fwf@g{newpath}}
 \ctr@ld@f\def\f@gstroke{\immediate\write\fwf@g{stroke}}
\fi}
\ctr@ld@f\def\c@pypsfile#1#2{\c@pyfil@{\immediate\write#1}{#2}}
\ctr@ld@f\def\Figinclud@PDF#1#2{\openin\frf@g=#1\pdfliteral{q #2 0 0 #2 0 0 cm}%
    \c@pyfil@{\pdfliteral}{\frf@g}\pdfliteral{Q}\closein\frf@g}
\ctr@ln@w{newif}\ifmored@ta
\ctr@ln@m\bl@nkline
\ctr@ld@f\def\c@pyfil@#1#2{\def\bl@nkline{\par}{\catcode`\%=12
    \loop\ifeof#2\mored@tafalse\else\mored@tatrue\immediate\read#2 to\tr@c
    \ifx\tr@c\bl@nkline\else#1{\tr@c}\fi\fi\ifmored@ta\repeat}}
\ctr@ld@f\def\keln@mun#1#2|{\def\l@debut{#1}\def\l@suite{#2}}
\ctr@ld@f\def\keln@mde#1#2#3|{\def\l@debut{#1#2}\def\l@suite{#3}}
\ctr@ld@f\def\keln@mtr#1#2#3#4|{\def\l@debut{#1#2#3}\def\l@suite{#4}}
\ctr@ld@f\def\keln@mqu#1#2#3#4#5|{\def\l@debut{#1#2#3#4}\def\l@suite{#5}}
\ctr@ld@f\let\@psffilein=\frf@g % file to \read
\ctr@ln@w{newif}\if@psffileok    % continue looking for the bounding box?
\ctr@ln@w{newif}\if@psfbbfound   % success?
\ctr@ln@w{newif}\if@psfverbose   % report what you're making?
\@psfverbosetrue
\ctr@ln@m\@psfllx \ctr@ln@m\@psflly
\ctr@ln@m\@psfurx \ctr@ln@m\@psfury
\ctr@ln@m\resetcolonc@tcode
\ctr@ld@f\def\@psfgetbb#1{\global\@psfbbfoundfalse%
\global\def\@psfllx{0}\global\def\@psflly{0}%
\global\def\@psfurx{30}\global\def\@psfury{30}%
\openin\@psffilein=#1\relax
\ifeof\@psffilein\errmessage{I couldn't open #1, will ignore it}\else
   \edef\resetcolonc@tcode{\catcode`\noexpand\:\the\catcode`\:\relax}%
   {\@psffileoktrue \chardef\other=12
    \def\do##1{\catcode`##1=\other}\dospecials \catcode`\ =10 \resetcolonc@tcode
    \loop
       \read\@psffilein to \@psffileline
       \ifeof\@psffilein\@psffileokfalse\else
          \expandafter\@psfaux\@psffileline:. \\%
       \fi
   \if@psffileok\repeat
   \if@psfbbfound\else
    \if@psfverbose\message{No bounding box comment in #1; using defaults}\fi\fi
   }\closein\@psffilein\fi}%
\ctr@ln@m\@psfbblit
\ctr@ln@m\@psfpercent
{\catcode`\%=12 \global\let\@psfpercent=%\global\def\@psfbblit{%BoundingBox}}%
\ctr@ln@m\@psfaux
\long\def\@psfaux#1#2:#3\\{\ifx#1\@psfpercent
   \def\testit{#2}\ifx\testit\@psfbblit
      \@psfgrab #3 . . . \\%
      \@psffileokfalse
      \global\@psfbbfoundtrue
   \fi\else\ifx#1\par\else\@psffileokfalse\fi\fi}%
\ctr@ld@f\def\@psfempty{}%
\ctr@ld@f\def\@psfgrab #1 #2 #3 #4 #5\\{%
\global\def\@psfllx{#1}\ifx\@psfllx\@psfempty
      \@psfgrab #2 #3 #4 #5 .\\\else
   \global\def\@psflly{#2}%
   \global\def\@psfurx{#3}\global\def\@psfury{#4}\fi}%
\ctr@ld@f\def\PSwrit@cmd#1#2#3{{\Figg@tXY{#1}\c@lprojSP\b@undb@x{\v@lX}{\v@lY}%
    \v@lX=\ptT@ptps\v@lX\v@lY=\ptT@ptps\v@lY%
    \immediate\write#3{\repdecn@mb{\v@lX}\space\repdecn@mb{\v@lY}\space#2}}}
\ctr@ld@f\def\PSwrit@cmdS#1#2#3#4#5{{\Figg@tXY{#1}\c@lprojSP\b@undb@x{\v@lX}{\v@lY}%
    \global\result@t=\v@lX\global\result@@t=\v@lY%
    \v@lX=\ptT@ptps\v@lX\v@lY=\ptT@ptps\v@lY%
    \immediate\write#3{\repdecn@mb{\v@lX}\space\repdecn@mb{\v@lY}\space#2}}%
    \edef#4{\the\result@t}\edef#5{\the\result@@t}}
\ctr@ld@f\def\psaltitude#1[#2,#3,#4]{{\ifcurr@ntPS\ifps@cri%
    \PSc@mment{psaltitude Square Dim=#1, Triangle=[#2 / #3,#4]}%
    \s@uvc@ntr@l\et@tpsaltitude\resetc@ntr@l{2}\figptorthoprojline-5:=#2/#3,#4/%
    \figvectP -1[#3,#4]\n@rminf{\v@leur}{-1}\vecunit@{-3}{-1}%
    \figvectP -1[-5,#3]\n@rminf{\v@lmin}{-1}\figvectP -2[-5,#4]\n@rminf{\v@lmax}{-2}%
    \ifdim\v@lmin<\v@lmax\s@mme=#3\else\v@lmax=\v@lmin\s@mme=#4\fi%
    \figvectP -4[-5,#2]\vecunit@{-4}{-4}\delt@=#1\unit@%
    \edef\t@ille{\repdecn@mb{\delt@}}\figpttra-1:=-5/\t@ille,-3/%
    \figptstra-3=-5,-1/\t@ille,-4/\psline[#2,-5]\s@uvdash{\typ@dash}%
    \pssetdash{\defaultdash}\psline[-1,-2,-3]\pssetdash{\typ@dash}%
    \ifdim\v@leur<\v@lmax\Pss@tsecondSt\psline[-5,\the\s@mme]\Psrest@reSt\fi%
    \PSc@mment{End psaltitude}\resetc@ntr@l\et@tpsaltitude\fi\fi}}
\ctr@ld@f\def\Ps@rcerc#1;#2(#3,#4){\ellBB@x#1;#2,#2(#3,#4,0)%
    \f@gnewpath{\delt@=#2\unit@\delt@=\ptT@ptps\delt@%
    \BdingB@xfalse%
    \PSwrit@cmd{#1}{\repdecn@mb{\delt@}\space #3\space #4\space arc}{\fwf@g}}}
\ctr@ln@m\psarccirc
\ctr@ld@f\def\psarccircDD#1;#2(#3,#4){\ifcurr@ntPS\ifps@cri%
    \PSc@mment{psarccircDD Center=#1 ; Radius=#2 (Ang1=#3, Ang2=#4)}%
    \iffillm@de\Ps@rcerc#1;#2(#3,#4)%
    \f@gfill%
    \else\Ps@rcerc#1;#2(#3,#4)\f@gstroke\fi%
    \PSc@mment{End psarccircDD}\fi\fi}
\ctr@ld@f\def\psarccircTD#1,#2,#3;#4(#5,#6){{\ifcurr@ntPS\ifps@cri\s@uvc@ntr@l\et@tpsarccircTD%
    \PSc@mment{psarccircTD Center=#1,P1=#2,P2=#3 ; Radius=#4 (Ang1=#5, Ang2=#6)}%
    \setc@ntr@l{2}\c@lExtAxes#1,#2,#3(#4)\psarcellPATD#1,-4,-5(#5,#6)%
    \PSc@mment{End psarccircTD}\resetc@ntr@l\et@tpsarccircTD\fi\fi}}
\ctr@ld@f\def\c@lExtAxes#1,#2,#3(#4){%
    \figvectPTD-5[#1,#2]\vecunit@{-5}{-5}\figvectNTD-4[#1,#2,#3]\vecunit@{-4}{-4}%
    \figvectNVTD-3[-4,-5]\delt@=#4\unit@\edef\r@yon{\repdecn@mb{\delt@}}%
    \figpttra-4:=#1/\r@yon,-5/\figpttra-5:=#1/\r@yon,-3/}
\ctr@ln@m\psarccircP
\ctr@ld@f\def\psarccircPDD#1;#2[#3,#4]{{\ifcurr@ntPS\ifps@cri\s@uvc@ntr@l\et@tpsarccircPDD%
    \PSc@mment{psarccircPDD Center=#1; Radius=#2, [P1=#3, P2=#4]}%
    \Ps@ngleparam#1;#2[#3,#4]\ifdim\v@lmin>\v@lmax\advance\v@lmax\DePI@deg\fi%
    \edef\@ngdeb{\repdecn@mb{\v@lmin}}\edef\@ngfin{\repdecn@mb{\v@lmax}}%
    \psarccirc#1;\r@dius(\@ngdeb,\@ngfin)%
    \PSc@mment{End psarccircPDD}\resetc@ntr@l\et@tpsarccircPDD\fi\fi}}
\ctr@ld@f\def\psarccircPTD#1;#2[#3,#4,#5]{{\ifcurr@ntPS\ifps@cri\s@uvc@ntr@l\et@tpsarccircPTD%
    \PSc@mment{psarccircPTD Center=#1; Radius=#2, [P1=#3, P2=#4, P3=#5]}%
    \setc@ntr@l{2}\c@lExtAxes#1,#3,#5(#2)\psarcellPP#1,-4,-5[#3,#4]%
    \PSc@mment{End psarccircPTD}\resetc@ntr@l\et@tpsarccircPTD\fi\fi}}
\ctr@ld@f\def\Ps@ngleparam#1;#2[#3,#4]{\setc@ntr@l{2}%
    \figvectPDD-1[#1,#3]\vecunit@{-1}{-1}\Figg@tXY{-1}\arct@n\v@lmin(\v@lX,\v@lY)%
    \figvectPDD-2[#1,#4]\vecunit@{-2}{-2}\Figg@tXY{-2}\arct@n\v@lmax(\v@lX,\v@lY)%
    \v@lmin=\rdT@deg\v@lmin\v@lmax=\rdT@deg\v@lmax%
    \v@leur=#2pt\maxim@m{\mili@u}{-\v@leur}{\v@leur}%
    \edef\r@dius{\repdecn@mb{\mili@u}}}
\ctr@ld@f\def\Ps@rcercBz#1;#2(#3,#4){\Ps@rellBz#1;#2,#2(#3,#4,0)}
\ctr@ld@f\def\Ps@rellBz#1;#2,#3(#4,#5,#6){%
    \ellBB@x#1;#2,#3(#4,#5,#6)\BdingB@xfalse%
    \c@lNbarcs{#4}{#5}\v@leur=#4pt\setc@ntr@l{2}\figptell-13::#1;#2,#3(#4,#6)%
    \f@gnewpath\PSwrit@cmd{-13}{\c@mmoveto}{\fwf@g}%
    \s@mme=\z@\bcl@rellBz#1;#2,#3(#6)\BdingB@xtrue}
\ctr@ld@f\def\bcl@rellBz#1;#2,#3(#4){\relax%
    \ifnum\s@mme<\p@rtent\advance\s@mme\@ne%
    \advance\v@leur\delt@\edef\@ngle{\repdecn@mb\v@leur}\figptell-14::#1;#2,#3(\@ngle,#4)%
    \advance\v@leur\delt@\edef\@ngle{\repdecn@mb\v@leur}\figptell-15::#1;#2,#3(\@ngle,#4)%
    \advance\v@leur\delt@\edef\@ngle{\repdecn@mb\v@leur}\figptell-16::#1;#2,#3(\@ngle,#4)%
    \figptscontrolDD-18[-13,-14,-15,-16]%
    \PSwrit@cmd{-18}{}{\fwf@g}\PSwrit@cmd{-17}{}{\fwf@g}%
    \PSwrit@cmd{-16}{\c@mcurveto}{\fwf@g}%
    \figptcopyDD-13:/-16/\bcl@rellBz#1;#2,#3(#4)\fi}
\ctr@ld@f\def\Ps@rell#1;#2,#3(#4,#5,#6){\ellBB@x#1;#2,#3(#4,#5,#6)%
    \f@gnewpath{\v@lmin=#2\unit@\v@lmin=\ptT@ptps\v@lmin%
    \v@lmax=#3\unit@\v@lmax=\ptT@ptps\v@lmax\BdingB@xfalse%
    \PSwrit@cmd{#1}%
    {#6\space\repdecn@mb{\v@lmin}\space\repdecn@mb{\v@lmax}\space #4\space #5\space ellipse}{\fwf@g}}%
    \global\Use@llipsetrue}
\ctr@ln@m\psarcell
\ctr@ld@f\def\psarcellDD#1;#2,#3(#4,#5,#6){{\ifcurr@ntPS\ifps@cri%
    \PSc@mment{psarcellDD Center=#1 ; XRad=#2, YRad=#3 (Ang1=#4, Ang2=#5, Inclination=#6)}%
    \iffillm@de\Ps@rell#1;#2,#3(#4,#5,#6)%
    \f@gfill%
    \else\Ps@rell#1;#2,#3(#4,#5,#6)\f@gstroke\fi%
    \PSc@mment{End psarcellDD}\fi\fi}}
\ctr@ld@f\def\psarcellTD#1;#2,#3(#4,#5,#6){{\ifcurr@ntPS\ifps@cri\s@uvc@ntr@l\et@tpsarcellTD%
    \PSc@mment{psarcellTD Center=#1 ; XRad=#2, YRad=#3 (Ang1=#4, Ang2=#5, Inclination=#6)}%
    \setc@ntr@l{2}\figpttraC -8:=#1/#2,0,0/\figpttraC -7:=#1/0,#3,0/%
    \figvectC -4(0,0,1)\figptsrot -8=-8,-7/#1,#6,-4/\psarcellPATD#1,-8,-7(#4,#5)%
    \PSc@mment{End psarcellTD}\resetc@ntr@l\et@tpsarcellTD\fi\fi}}
\ctr@ln@m\psarcellPA
\ctr@ld@f\def\psarcellPADD#1,#2,#3(#4,#5){{\ifcurr@ntPS\ifps@cri\s@uvc@ntr@l\et@tpsarcellPADD%
    \PSc@mment{psarcellPADD Center=#1,PtAxis1=#2,PtAxis2=#3 (Ang1=#4, Ang2=#5)}%
    \setc@ntr@l{2}\figvectPDD-1[#1,#2]\vecunit@DD{-1}{-1}\v@lX=\ptT@unit@\result@t%
    \edef\XR@d{\repdecn@mb{\v@lX}}\Figg@tXY{-1}\arct@n\v@lmin(\v@lX,\v@lY)%
    \v@lmin=\rdT@deg\v@lmin\edef\Inclin@{\repdecn@mb{\v@lmin}}%
    \figgetdist\YR@d[#1,#3]\psarcellDD#1;\XR@d,\YR@d(#4,#5,\Inclin@)%
    \PSc@mment{End psarcellPADD}\resetc@ntr@l\et@tpsarcellPADD\fi\fi}}
\ctr@ld@f\def\psarcellPATD#1,#2,#3(#4,#5){{\ifcurr@ntPS\ifps@cri\s@uvc@ntr@l\et@tpsarcellPATD%
    \PSc@mment{psarcellPATD Center=#1,PtAxis1=#2,PtAxis2=#3 (Ang1=#4, Ang2=#5)}%
    \iffillm@de\Ps@rellPATD#1,#2,#3(#4,#5)%
    \f@gfill%
    \else\Ps@rellPATD#1,#2,#3(#4,#5)\f@gstroke\fi%
    \PSc@mment{End psarcellPATD}\resetc@ntr@l\et@tpsarcellPATD\fi\fi}}
\ctr@ld@f\def\Ps@rellPATD#1,#2,#3(#4,#5){\let\c@lprojSP=\relax%
    \setc@ntr@l{2}\figvectPTD-1[#1,#2]\figvectPTD-2[#1,#3]\c@lNbarcs{#4}{#5}%
    \v@leur=#4pt\c@lptellP{#1}{-1}{-2}\Figptpr@j-5:/-3/%
    \f@gnewpath\PSwrit@cmdS{-5}{\c@mmoveto}{\fwf@g}{\X@un}{\Y@un}%
    \edef\C@nt@r{#1}\s@mme=\z@\bcl@rellPATD}
\ctr@ld@f\def\bcl@rellPATD{\relax%
    \ifnum\s@mme<\p@rtent\advance\s@mme\@ne%
    \advance\v@leur\delt@\c@lptellP{\C@nt@r}{-1}{-2}\Figptpr@j-4:/-3/%
    \advance\v@leur\delt@\c@lptellP{\C@nt@r}{-1}{-2}\Figptpr@j-6:/-3/%
    \advance\v@leur\delt@\c@lptellP{\C@nt@r}{-1}{-2}\Figptpr@j-3:/-3/%
    \v@lX=\z@\v@lY=\z@\Figtr@nptDD{-5}{-5}\Figtr@nptDD{2}{-3}%
    \divide\v@lX\@vi\divide\v@lY\@vi%
    \Figtr@nptDD{3}{-4}\Figtr@nptDD{-1.5}{-6}\v@lmin=\v@lX\v@lmax=\v@lY%
    \v@lX=\z@\v@lY=\z@\Figtr@nptDD{2}{-5}\Figtr@nptDD{-5}{-3}%
    \divide\v@lX\@vi\divide\v@lY\@vi\Figtr@nptDD{-1.5}{-4}\Figtr@nptDD{3}{-6}%
    \BdingB@xfalse%
    \Figp@intregDD-4:(\v@lmin,\v@lmax)\PSwrit@cmdS{-4}{}{\fwf@g}{\X@de}{\Y@de}%
    \Figp@intregDD-4:(\v@lX,\v@lY)\PSwrit@cmdS{-4}{}{\fwf@g}{\X@tr}{\Y@tr}%
    \BdingB@xtrue\PSwrit@cmdS{-3}{\c@mcurveto}{\fwf@g}{\X@qu}{\Y@qu}%
    \B@zierBB@x{1}{\Y@un}(\X@un,\X@de,\X@tr,\X@qu)%
    \B@zierBB@x{2}{\X@un}(\Y@un,\Y@de,\Y@tr,\Y@qu)%
    \edef\X@un{\X@qu}\edef\Y@un{\Y@qu}\figptcopyDD-5:/-3/\bcl@rellPATD\fi}
\ctr@ld@f\def\c@lNbarcs#1#2{%
    \delt@=#2pt\advance\delt@-#1pt\maxim@m{\v@lmax}{\delt@}{-\delt@}%
    \v@leur=\v@lmax\divide\v@leur45 \p@rtentiere{\p@rtent}{\v@leur}\advance\p@rtent\@ne%
    \s@mme=\p@rtent\multiply\s@mme\thr@@\divide\delt@\s@mme}
\ctr@ld@f\def\psarcellPP#1,#2,#3[#4,#5]{{\ifcurr@ntPS\ifps@cri\s@uvc@ntr@l\et@tpsarcellPP%
    \PSc@mment{psarcellPP Center=#1,PtAxis1=#2,PtAxis2=#3 [Point1=#4, Point2=#5]}%
    \setc@ntr@l{2}\figvectP-2[#1,#3]\vecunit@{-2}{-2}\v@lmin=\result@t%
    \invers@{\v@lmax}{\v@lmin}%
    \figvectP-1[#1,#2]\vecunit@{-1}{-1}\v@leur=\result@t%
    \v@leur=\repdecn@mb{\v@lmax}\v@leur\edef\AsB@{\repdecn@mb{\v@leur}}% a/b
    \c@lAngle{#1}{#4}{\v@lmin}\edef\@ngdeb{\repdecn@mb{\v@lmin}}%
    \c@lAngle{#1}{#5}{\v@lmax}\ifdim\v@lmin>\v@lmax\advance\v@lmax\DePI@deg\fi%
    \edef\@ngfin{\repdecn@mb{\v@lmax}}\psarcellPA#1,#2,#3(\@ngdeb,\@ngfin)%
    \PSc@mment{End psarcellPP}\resetc@ntr@l\et@tpsarcellPP\fi\fi}}
\ctr@ld@f\def\c@lAngle#1#2#3{\figvectP-3[#1,#2]%
    \c@lproscal\delt@[-3,-1]\c@lproscal\v@leur[-3,-2]%
    \v@leur=\AsB@\v@leur\arct@n#3(\delt@,\v@leur)#3=\rdT@deg#3}
\ctr@ln@w{newif}\if@rrowratio\@rrowratiotrue
\ctr@ln@w{newif}\if@rrowhfill
\ctr@ln@w{newif}\if@rrowhout
\ctr@ld@f\def\Psset@rrowhe@d#1=#2|{\keln@mun#1|%
    \def\n@mref{a}\ifx\l@debut\n@mref\pssetarrowheadangle{#2}\else% angle
    \def\n@mref{f}\ifx\l@debut\n@mref\pssetarrowheadfill{#2}\else% fillmode
    \def\n@mref{l}\ifx\l@debut\n@mref\pssetarrowheadlength{#2}\else% length
    \def\n@mref{o}\ifx\l@debut\n@mref\pssetarrowheadout{#2}\else% out
    \def\n@mref{r}\ifx\l@debut\n@mref\pssetarrowheadratio{#2}\else% ratio
    \immediate\write16{*** Unknown attribute: \BS@ psset arrowhead(..., #1=...)}%
    \fi\fi\fi\fi\fi}
\ctr@ln@m\@rrowheadangle
\ctr@ln@m\C@AHANG \ctr@ln@m\S@AHANG \ctr@ln@m\UNSS@N
\ctr@ld@f\def\pssetarrowheadangle#1{\edef\@rrowheadangle{#1}{\c@ssin{\C@}{\S@}{#1}%
    \xdef\C@AHANG{\C@}\xdef\S@AHANG{\S@}\v@lmax=\S@ pt%
    \invers@{\v@leur}{\v@lmax}\maxim@m{\v@leur}{\v@leur}{-\v@leur}%
    \xdef\UNSS@N{\the\v@leur}}}
\ctr@ld@f\def\pssetarrowheadfill#1{\expandafter\set@rrowhfill#1:}
\ctr@ld@f\def\set@rrowhfill#1#2:{\if#1n\@rrowhfillfalse\else\@rrowhfilltrue\fi}
\ctr@ld@f\def\pssetarrowheadout#1{\expandafter\set@rrowhout#1:}
\ctr@ld@f\def\set@rrowhout#1#2:{\if#1n\@rrowhoutfalse\else\@rrowhouttrue\fi}
\ctr@ln@m\@rrowheadlength
\ctr@ld@f\def\pssetarrowheadlength#1{\edef\@rrowheadlength{#1}\@rrowratiofalse}
\ctr@ln@m\@rrowheadratio
\ctr@ld@f\def\pssetarrowheadratio#1{\edef\@rrowheadratio{#1}\@rrowratiotrue}
\ctr@ln@m\defaultarrowheadlength
\ctr@ld@f\def\psresetarrowhead{%
    \pssetarrowheadangle{\defaultarrowheadangle}%
    \pssetarrowheadfill{\defaultarrowheadfill}%
    \pssetarrowheadout{\defaultarrowheadout}%
    \pssetarrowheadratio{\defaultarrowheadratio}%
    \d@fm@cdim\defaultarrowheadlength{\defaulth@rdahlength}% Valeur par defaut...
    \pssetarrowheadlength{\defaultarrowheadlength}}
\ctr@ld@f\def\defaultarrowheadratio{0.1}
\ctr@ld@f\def\defaultarrowheadangle{20}
\ctr@ld@f\def\defaultarrowheadfill{no}
\ctr@ld@f\def\defaultarrowheadout{no}
\ctr@ld@f\def\defaulth@rdahlength{8pt}
\ctr@ln@m\psarrow
\ctr@ld@f\def\psarrowDD[#1,#2]{{\ifcurr@ntPS\ifps@cri\s@uvc@ntr@l\et@tpsarrow%
    \PSc@mment{psarrowDD [Pt1,Pt2]=[#1,#2]}\pssetfillmode{no}%
    \psarrowheadDD[#1,#2]\setc@ntr@l{2}\psline[#1,-3]%
    \PSc@mment{End psarrowDD}\resetc@ntr@l\et@tpsarrow\fi\fi}}
\ctr@ld@f\def\psarrowTD[#1,#2]{{\ifcurr@ntPS\ifps@cri\s@uvc@ntr@l\et@tpsarrowTD%
    \PSc@mment{psarrowTD [Pt1,Pt2]=[#1,#2]}\resetc@ntr@l{2}%
    \Figptpr@j-5:/#1/\Figptpr@j-6:/#2/\let\c@lprojSP=\relax\psarrowDD[-5,-6]%
    \PSc@mment{End psarrowTD}\resetc@ntr@l\et@tpsarrowTD\fi\fi}}
\ctr@ln@m\psarrowhead
\ctr@ld@f\def\psarrowheadDD[#1,#2]{{\ifcurr@ntPS\ifps@cri\s@uvc@ntr@l\et@tpsarrowheadDD%
    \if@rrowhfill\def\@hangle{-\@rrowheadangle}\else\def\@hangle{\@rrowheadangle}\fi%
    \if@rrowratio%
    \if@rrowhout\def\@hratio{-\@rrowheadratio}\else\def\@hratio{\@rrowheadratio}\fi%
    \PSc@mment{psarrowheadDD Ratio=\@hratio, Angle=\@hangle, [Pt1,Pt2]=[#1,#2]}%
    \Ps@rrowhead\@hratio,\@hangle[#1,#2]%
    \else%
    \if@rrowhout\def\@hlength{-\@rrowheadlength}\else\def\@hlength{\@rrowheadlength}\fi%
    \PSc@mment{psarrowheadDD Length=\@hlength, Angle=\@hangle, [Pt1,Pt2]=[#1,#2]}%
    \Ps@rrowheadfd\@hlength,\@hangle[#1,#2]%
    \fi%
    \PSc@mment{End psarrowheadDD}\resetc@ntr@l\et@tpsarrowheadDD\fi\fi}}
\ctr@ld@f\def\psarrowheadTD[#1,#2]{{\ifcurr@ntPS\ifps@cri\s@uvc@ntr@l\et@tpsarrowheadTD%
    \PSc@mment{psarrowheadTD [Pt1,Pt2]=[#1,#2]}\resetc@ntr@l{2}%
    \Figptpr@j-5:/#1/\Figptpr@j-6:/#2/\let\c@lprojSP=\relax\psarrowheadDD[-5,-6]%
    \PSc@mment{End psarrowheadTD}\resetc@ntr@l\et@tpsarrowheadTD\fi\fi}}
\ctr@ld@f\def\Ps@rrowhead#1,#2[#3,#4]{\v@leur=#1\p@\maxim@m{\v@leur}{\v@leur}{-\v@leur}%
    \ifdim\v@leur>\Cepsil@n{% Arrow is not degenerated
    \PSc@mment{ps@rrowhead Ratio=#1, Angle=#2, [Pt1,Pt2]=[#3,#4]}\v@leur=\UNSS@N%
    \v@leur=\curr@ntwidth\v@leur\v@leur=\ptpsT@pt\v@leur\delt@=.5\v@leur% = width / (2 sin(Angle))
    \setc@ntr@l{2}\figvectPDD-3[#4,#3]%
    \Figg@tXY{-3}\v@lX=#1\v@lX\v@lY=#1\v@lY\Figv@ctCreg-3(\v@lX,\v@lY)%
    \vecunit@{-4}{-3}\mili@u=\result@t%
    \ifdim#2pt>\z@\v@lXa=-\C@AHANG\delt@%
     \edef\c@ef{\repdecn@mb{\v@lXa}}\figpttraDD-3:=-3/\c@ef,-4/\fi%
    \edef\c@ef{\repdecn@mb{\delt@}}%
    \v@lXa=\mili@u\v@lXa=\C@AHANG\v@lXa%
    \v@lYa=\ptpsT@pt\p@\v@lYa=\curr@ntwidth\v@lYa\v@lYa=\sDcc@ngle\v@lYa%
    \advance\v@lXa-\v@lYa\gdef\sDcc@ngle{0}%
    \ifdim\v@lXa>\v@leur\edef\c@efendpt{\repdecn@mb{\v@leur}}%
    \else\edef\c@efendpt{\repdecn@mb{\v@lXa}}\fi%
    \Figg@tXY{-3}\v@lmin=\v@lX\v@lmax=\v@lY%
    \v@lXa=\C@AHANG\v@lmin\v@lYa=\S@AHANG\v@lmax\advance\v@lXa\v@lYa%
    \v@lYa=-\S@AHANG\v@lmin\v@lX=\C@AHANG\v@lmax\advance\v@lYa\v@lX%
    \setc@ntr@l{1}\Figg@tXY{#4}\advance\v@lX\v@lXa\advance\v@lY\v@lYa%
    \setc@ntr@l{2}\Figp@intregDD-2:(\v@lX,\v@lY)%
    \v@lXa=\C@AHANG\v@lmin\v@lYa=-\S@AHANG\v@lmax\advance\v@lXa\v@lYa%
    \v@lYa=\S@AHANG\v@lmin\v@lX=\C@AHANG\v@lmax\advance\v@lYa\v@lX%
    \setc@ntr@l{1}\Figg@tXY{#4}\advance\v@lX\v@lXa\advance\v@lY\v@lYa%
    \setc@ntr@l{2}\Figp@intregDD-1:(\v@lX,\v@lY)%
    \ifdim#2pt<\z@\fillm@detrue\psline[-2,#4,-1]% fill
    \else\figptstraDD-3=#4,-2,-1/\c@ef,-4/\psline[-2,-3,-1]\fi% no fill
    \ifdim#1pt>\z@\figpttraDD-3:=#4/\c@efendpt,-4/\else\figptcopyDD-3:/#4/\fi%
    \PSc@mment{End ps@rrowhead}}\fi}
\ctr@ld@f\def\sDcc@ngle{0}% Initialisation
\ctr@ld@f\def\Ps@rrowheadfd#1,#2[#3,#4]{{%
    \PSc@mment{ps@rrowheadfd Length=#1, Angle=#2, [Pt1,Pt2]=[#3,#4]}%
    \setc@ntr@l{2}\figvectPDD-1[#3,#4]\n@rmeucDD{\v@leur}{-1}\v@leur=\ptT@unit@\v@leur%
    \invers@{\v@leur}{\v@leur}\v@leur=#1\v@leur\edef\R@tio{\repdecn@mb{\v@leur}}%
    \Ps@rrowhead\R@tio,#2[#3,#4]\PSc@mment{End ps@rrowheadfd}}}
\ctr@ln@m\psarrowBezier
\ctr@ld@f\def\psarrowBezierDD[#1,#2,#3,#4]{{\ifcurr@ntPS\ifps@cri\s@uvc@ntr@l\et@tpsarrowBezierDD%
    \PSc@mment{psarrowBezierDD Control points=#1,#2,#3,#4}\setc@ntr@l{2}%
    \if@rrowratio\c@larclengthDD\v@leur,10[#1,#2,#3,#4]\else\v@leur=\z@\fi%
    \Ps@rrowB@zDD\v@leur[#1,#2,#3,#4]%
    \PSc@mment{End psarrowBezierDD}\resetc@ntr@l\et@tpsarrowBezierDD\fi\fi}}
\ctr@ld@f\def\psarrowBezierTD[#1,#2,#3,#4]{{\ifcurr@ntPS\ifps@cri\s@uvc@ntr@l\et@tpsarrowBezierTD%
    \PSc@mment{psarrowBezierTD Control points=#1,#2,#3,#4}\resetc@ntr@l{2}%
    \Figptpr@j-7:/#1/\Figptpr@j-8:/#2/\Figptpr@j-9:/#3/\Figptpr@j-10:/#4/%
    \let\c@lprojSP=\relax\ifnum\curr@ntproj<\tw@\psarrowBezierDD[-7,-8,-9,-10]%
    \else\f@gnewpath\PSwrit@cmd{-7}{\c@mmoveto}{\fwf@g}%
    \if@rrowratio\c@larclengthDD\mili@u,10[-7,-8,-9,-10]\else\mili@u=\z@\fi%
    \p@rtent=\NBz@rcs\advance\p@rtent\m@ne\subB@zierTD\p@rtent[#1,#2,#3,#4]%
    \f@gstroke%
    \advance\v@lmin\p@rtent\delt@% Initialized in \subB@zierTD
    \v@leur=\v@lmin\advance\v@leur0.33333 \delt@\edef\unti@rs{\repdecn@mb{\v@leur}}%
    \v@leur=\v@lmin\advance\v@leur0.66666 \delt@\edef\deti@rs{\repdecn@mb{\v@leur}}%
    \figptcopyDD-8:/-10/\c@lsubBzarc\unti@rs,\deti@rs[#1,#2,#3,#4]%
    \figptcopyDD-8:/-4/\figptcopyDD-9:/-3/\Ps@rrowB@zDD\mili@u[-7,-8,-9,-10]\fi%
    \PSc@mment{End psarrowBezierTD}\resetc@ntr@l\et@tpsarrowBezierTD\fi\fi}}
\ctr@ld@f\def\c@larclengthDD#1,#2[#3,#4,#5,#6]{{\p@rtent=#2\figptcopyDD-5:/#3/%
    \delt@=\p@\divide\delt@\p@rtent\c@rre=\z@\v@leur=\z@\s@mme=\z@%
    \loop\ifnum\s@mme<\p@rtent\advance\s@mme\@ne\advance\v@leur\delt@%
    \edef\T@{\repdecn@mb{\v@leur}}\figptBezierDD-6::\T@[#3,#4,#5,#6]%
    \figvectPDD-1[-5,-6]\n@rmeucDD{\mili@u}{-1}\advance\c@rre\mili@u%
    \figptcopyDD-5:/-6/\repeat\global\result@t=\ptT@unit@\c@rre}#1=\result@t}
\ctr@ld@f\def\Ps@rrowB@zDD#1[#2,#3,#4,#5]{{\pssetfillmode{no}%
    \if@rrowratio\delt@=\@rrowheadratio#1\else\delt@=\@rrowheadlength pt\fi%
    \v@leur=\C@AHANG\delt@\edef\R@dius{\repdecn@mb{\v@leur}}%
    \FigptintercircB@zDD-5::0,\R@dius[#5,#4,#3,#2]%
    \pssetarrowheadlength{\repdecn@mb{\delt@}}\psarrowheadDD[-5,#5]%
    \let\n@rmeuc=\n@rmeucDD\figgetdist\R@dius[#5,-3]%
    \FigptintercircB@zDD-6::0,\R@dius[#5,#4,#3,#2]%
    \figptBezierDD-5::0.33333[#5,#4,#3,#2]\figptBezierDD-3::0.66666[#5,#4,#3,#2]%
    \figptscontrolDD-5[-6,-5,-3,#2]\psBezierDD1[-6,-5,-4,#2]}}
\ctr@ln@m\psarrowcirc
\ctr@ld@f\def\psarrowcircDD#1;#2(#3,#4){{\ifcurr@ntPS\ifps@cri\s@uvc@ntr@l\et@tpsarrowcircDD%
    \PSc@mment{psarrowcircDD Center=#1 ; Radius=#2 (Ang1=#3,Ang2=#4)}%
    \pssetfillmode{no}\Pscirc@rrowhead#1;#2(#3,#4)%
    \setc@ntr@l{2}\figvectPDD -4[#1,-3]\vecunit@{-4}{-4}%
    \Figg@tXY{-4}\arct@n\v@lmin(\v@lX,\v@lY)%
    \v@lmin=\rdT@deg\v@lmin\v@leur=#4pt\advance\v@leur-\v@lmin%
    \maxim@m{\v@leur}{\v@leur}{-\v@leur}%
    \ifdim\v@leur>\DemiPI@deg\relax\ifdim\v@lmin<#4pt\advance\v@lmin\DePI@deg%
    \else\advance\v@lmin-\DePI@deg\fi\fi\edef\ar@ngle{\repdecn@mb{\v@lmin}}%
    \ifdim#3pt<#4pt\psarccirc#1;#2(#3,\ar@ngle)\else\psarccirc#1;#2(\ar@ngle,#3)\fi%
    \PSc@mment{End psarrowcircDD}\resetc@ntr@l\et@tpsarrowcircDD\fi\fi}}
\ctr@ld@f\def\psarrowcircTD#1,#2,#3;#4(#5,#6){{\ifcurr@ntPS\ifps@cri\s@uvc@ntr@l\et@tpsarrowcircTD%
    \PSc@mment{psarrowcircTD Center=#1,P1=#2,P2=#3 ; Radius=#4 (Ang1=#5, Ang2=#6)}%
    \resetc@ntr@l{2}\c@lExtAxes#1,#2,#3(#4)\let\c@lprojSP=\relax%
    \figvectPTD-11[#1,-4]\figvectPTD-12[#1,-5]\c@lNbarcs{#5}{#6}%
    \if@rrowratio\v@lmax=\degT@rd\v@lmax\edef\D@lpha{\repdecn@mb{\v@lmax}}\fi%
    \advance\p@rtent\m@ne\mili@u=\z@%
    \v@leur=#5pt\c@lptellP{#1}{-11}{-12}\Figptpr@j-9:/-3/%
    \f@gnewpath\PSwrit@cmdS{-9}{\c@mmoveto}{\fwf@g}{\X@un}{\Y@un}%
    \edef\C@nt@r{#1}\s@mme=\z@\bcl@rcircTD\f@gstroke%
    \advance\v@leur\delt@\c@lptellP{#1}{-11}{-12}\Figptpr@j-5:/-3/%
    \advance\v@leur\delt@\c@lptellP{#1}{-11}{-12}\Figptpr@j-6:/-3/%
    \advance\v@leur\delt@\c@lptellP{#1}{-11}{-12}\Figptpr@j-10:/-3/%
    \figptscontrolDD-8[-9,-5,-6,-10]%
    \if@rrowratio\c@lcurvradDD0.5[-9,-8,-7,-10]\advance\mili@u\result@t%
    \maxim@m{\mili@u}{\mili@u}{-\mili@u}\mili@u=\ptT@unit@\mili@u%
    \mili@u=\D@lpha\mili@u\advance\p@rtent\@ne\divide\mili@u\p@rtent\fi%
    \Ps@rrowB@zDD\mili@u[-9,-8,-7,-10]%
    \PSc@mment{End psarrowcircTD}\resetc@ntr@l\et@tpsarrowcircTD\fi\fi}}
\ctr@ld@f\def\bcl@rcircTD{\relax%
    \ifnum\s@mme<\p@rtent\advance\s@mme\@ne%
    \advance\v@leur\delt@\c@lptellP{\C@nt@r}{-11}{-12}\Figptpr@j-5:/-3/%
    \advance\v@leur\delt@\c@lptellP{\C@nt@r}{-11}{-12}\Figptpr@j-6:/-3/%
    \advance\v@leur\delt@\c@lptellP{\C@nt@r}{-11}{-12}\Figptpr@j-10:/-3/%
    \figptscontrolDD-8[-9,-5,-6,-10]\BdingB@xfalse%
    \PSwrit@cmdS{-8}{}{\fwf@g}{\X@de}{\Y@de}\PSwrit@cmdS{-7}{}{\fwf@g}{\X@tr}{\Y@tr}%
    \BdingB@xtrue\PSwrit@cmdS{-10}{\c@mcurveto}{\fwf@g}{\X@qu}{\Y@qu}%
    \if@rrowratio\c@lcurvradDD0.5[-9,-8,-7,-10]\advance\mili@u\result@t\fi%
    \B@zierBB@x{1}{\Y@un}(\X@un,\X@de,\X@tr,\X@qu)%
    \B@zierBB@x{2}{\X@un}(\Y@un,\Y@de,\Y@tr,\Y@qu)%
    \edef\X@un{\X@qu}\edef\Y@un{\Y@qu}\figptcopyDD-9:/-10/\bcl@rcircTD\fi}
\ctr@ld@f\def\Pscirc@rrowhead#1;#2(#3,#4){{%
    \PSc@mment{pscirc@rrowhead Center=#1 ; Radius=#2 (Ang1=#3,Ang2=#4)}%
    \v@leur=#2\unit@\edef\s@glen{\repdecn@mb{\v@leur}}\v@lY=\z@\v@lX=\v@leur%
    \resetc@ntr@l{2}\Figv@ctCreg-3(\v@lX,\v@lY)\figpttraDD-5:=#1/1,-3/%
    \figptrotDD-5:=-5/#1,#4/%
    \figvectPDD-3[#1,-5]\Figg@tXY{-3}\v@leur=\v@lX%
    \ifdim#3pt<#4pt\v@lX=\v@lY\v@lY=-\v@leur\else\v@lX=-\v@lY\v@lY=\v@leur\fi%
    \Figv@ctCreg-3(\v@lX,\v@lY)\vecunit@{-3}{-3}%
    \if@rrowratio\v@leur=#4pt\advance\v@leur-#3pt\maxim@m{\mili@u}{-\v@leur}{\v@leur}%
    \mili@u=\degT@rd\mili@u\v@leur=\s@glen\mili@u\edef\s@glen{\repdecn@mb{\v@leur}}%
    \mili@u=#2\mili@u\mili@u=\@rrowheadratio\mili@u\else\mili@u=\@rrowheadlength pt\fi%
    \figpttraDD-6:=-5/\s@glen,-3/\v@leur=#2pt\v@leur=2\v@leur%
    \invers@{\v@leur}{\v@leur}\c@rre=\repdecn@mb{\v@leur}\mili@u% = sin = L/(2R)
    \mili@u=\c@rre\mili@u=\repdecn@mb{\c@rre}\mili@u%
    \v@leur=\p@\advance\v@leur-\mili@u% \v@leur = cos*cos
    \invers@{\mili@u}{2\v@leur}\delt@=\c@rre\delt@=\repdecn@mb{\mili@u}\delt@%
    \xdef\sDcc@ngle{\repdecn@mb{\delt@}}% sin/(2*cos*cos) used in \Ps@rrowhead
    \sqrt@{\mili@u}{\v@leur}\arct@n\v@leur(\mili@u,\c@rre)%
    \v@leur=\rdT@deg\v@leur% \cor@ngle = atan(L/sqrt(4R*R-L*L))
    \ifdim#3pt<#4pt\v@leur=-\v@leur\fi%
    \if@rrowhout\v@leur=-\v@leur\fi\edef\cor@ngle{\repdecn@mb{\v@leur}}%
    \figptrotDD-6:=-6/-5,\cor@ngle/\psarrowheadDD[-6,-5]%
    \PSc@mment{End pscirc@rrowhead}}}
\ctr@ln@m\psarrowcircP
\ctr@ld@f\def\psarrowcircPDD#1;#2[#3,#4]{{\ifcurr@ntPS\ifps@cri%
    \PSc@mment{psarrowcircPDD Center=#1; Radius=#2, [P1=#3,P2=#4]}%
    \s@uvc@ntr@l\et@tpsarrowcircPDD\Ps@ngleparam#1;#2[#3,#4]%
    \ifdim\v@leur>\z@\ifdim\v@lmin>\v@lmax\advance\v@lmax\DePI@deg\fi%
    \else\ifdim\v@lmin<\v@lmax\advance\v@lmin\DePI@deg\fi\fi%
    \edef\@ngdeb{\repdecn@mb{\v@lmin}}\edef\@ngfin{\repdecn@mb{\v@lmax}}%
    \psarrowcirc#1;\r@dius(\@ngdeb,\@ngfin)%
    \PSc@mment{End psarrowcircPDD}\resetc@ntr@l\et@tpsarrowcircPDD\fi\fi}}
\ctr@ld@f\def\psarrowcircPTD#1;#2[#3,#4,#5]{{\ifcurr@ntPS\ifps@cri\s@uvc@ntr@l\et@tpsarrowcircPTD%
    \PSc@mment{psarrowcircPTD Center=#1; Radius=#2, [P1=#3,P2=#4,P3=#5]}%
    \figgetangleTD\@ngfin[#1,#3,#4,#5]\v@leur=#2pt%
    \maxim@m{\mili@u}{-\v@leur}{\v@leur}\edef\r@dius{\repdecn@mb{\mili@u}}%
    \ifdim\v@leur<\z@\v@lmax=\@ngfin pt\advance\v@lmax-\DePI@deg%
    \edef\@ngfin{\repdecn@mb{\v@lmax}}\fi\psarrowcircTD#1,#3,#5;\r@dius(0,\@ngfin)%
    \PSc@mment{End psarrowcircPTD}\resetc@ntr@l\et@tpsarrowcircPTD\fi\fi}}
\ctr@ld@f\def\psaxes#1(#2){{\ifcurr@ntPS\ifps@cri\s@uvc@ntr@l\et@tpsaxes%
    \PSc@mment{psaxes Origin=#1 Range=(#2)}\an@lys@xes#2,:\resetc@ntr@l{2}%
    \ifx\t@xt@\empty\ifTr@isDim\ps@xes#1(0,#2,0,#2,0,#2)\else\ps@xes#1(0,#2,0,#2)\fi%
    \else\ps@xes#1(#2)\fi\PSc@mment{End psaxes}\resetc@ntr@l\et@tpsaxes\fi\fi}}
\ctr@ld@f\def\an@lys@xes#1,#2:{\def\t@xt@{#2}}
\ctr@ln@m\ps@xes
\ctr@ld@f\def\ps@xesDD#1(#2,#3,#4,#5){%
    \figpttraC-5:=#1/#2,0/\figpttraC-6:=#1/#3,0/\psarrowDD[-5,-6]%
    \figpttraC-5:=#1/0,#4/\figpttraC-6:=#1/0,#5/\psarrowDD[-5,-6]}
\ctr@ld@f\def\ps@xesTD#1(#2,#3,#4,#5,#6,#7){%
    \figpttraC-7:=#1/#2,0,0/\figpttraC-8:=#1/#3,0,0/\psarrowTD[-7,-8]%
    \figpttraC-7:=#1/0,#4,0/\figpttraC-8:=#1/0,#5,0/\psarrowTD[-7,-8]%
    \figpttraC-7:=#1/0,0,#6/\figpttraC-8:=#1/0,0,#7/\psarrowTD[-7,-8]}
\ctr@ln@m\newGr@FN
\ctr@ld@f\def\newGr@FNPDF#1{\s@mme=\Gr@FNb\advance\s@mme\@ne\xdef\Gr@FNb{\number\s@mme}}
\ctr@ld@f\def\newGr@FNDVI#1{\newGr@FNPDF{}\xdef#1{\jobname GI\Gr@FNb.anx}}
\ctr@ld@f\def\psbeginfig#1{\newGr@FN\DefGIfilen@me\gdef\@utoFN{0}%
    \def\t@xt@{#1}\relax\ifx\t@xt@\empty\psupdatem@detrue%
    \gdef\@utoFN{1}\Psb@ginfig\DefGIfilen@me\else\expandafter\Psb@ginfigNu@#1 :\fi}
\ctr@ld@f\def\Psb@ginfigNu@#1 #2:{\def\t@xt@{#1}\relax\ifx\t@xt@\empty\def\t@xt@{#2}%
    \ifx\t@xt@\empty\psupdatem@detrue\gdef\@utoFN{1}\Psb@ginfig\DefGIfilen@me%
    \else\Psb@ginfigNu@#2:\fi\else\Psb@ginfig{#1}\fi}
\ctr@ln@m\PSfilen@me \ctr@ln@m\auxfilen@me
\ctr@ld@f\def\Psb@ginfig#1{\ifcurr@ntPS\else%
    \edef\PSfilen@me{#1}\edef\auxfilen@me{\jobname.anx}%
    \ifpsupdatem@de\ps@critrue\else\openin\frf@g=\PSfilen@me\relax%
    \ifeof\frf@g\ps@critrue\else\ps@crifalse\fi\closein\frf@g\fi%
    \curr@ntPStrue\c@ldefproj\expandafter\setupd@te\defaultupdate:%
    \ifps@cri\initb@undb@x%
    \immediate\openout\fwf@g=\auxfilen@me\initpss@ttings\fi%
    \fi}
\ctr@ld@f\def\Gr@FNb{0}
\ctr@ld@f\def\figforTeXFileno{\Gr@FNb}
\ctr@ld@f\def\figforTeXFigno{0 }
\ctr@ld@f\def\figforTeXnextFigno{1 }
\ctr@ld@f\edef\DefGIfilen@me{\jobname GI.anx}
\ctr@ld@f\def\initpss@ttings{\psreset{arrowhead,curve,first,flowchart,mesh,second,third}%
    \Use@llipsefalse}
\ctr@ld@f\def\B@zierBB@x#1#2(#3,#4,#5,#6){{\c@rre=\t@n\epsil@n% Do not reduce this value
    \v@lmax=#4\advance\v@lmax-#5\v@lmax=\thr@@\v@lmax\advance\v@lmax#6\advance\v@lmax-#3%
    \mili@u=#4\mili@u=-\tw@\mili@u\advance\mili@u#3\advance\mili@u#5%
    \v@lmin=#4\advance\v@lmin-#3\maxim@m{\v@leur}{-\v@lmax}{\v@lmax}%
    \maxim@m{\delt@}{-\mili@u}{\mili@u}\maxim@m{\v@leur}{\v@leur}{\delt@}%
    \maxim@m{\delt@}{-\v@lmin}{\v@lmin}\maxim@m{\v@leur}{\v@leur}{\delt@}%
    \ifdim\v@leur>\c@rre\invers@{\v@leur}{\v@leur}\edef\Uns@rM@x{\repdecn@mb{\v@leur}}%
    \v@lmax=\Uns@rM@x\v@lmax\mili@u=\Uns@rM@x\mili@u\v@lmin=\Uns@rM@x\v@lmin%
    \maxim@m{\v@leur}{-\v@lmax}{\v@lmax}\ifdim\v@leur<\c@rre%
    \maxim@m{\v@leur}{-\mili@u}{\mili@u}\ifdim\v@leur<\c@rre\else%
    \invers@{\mili@u}{\mili@u}\v@leur=-0.5\v@lmin%
    \v@leur=\repdecn@mb{\mili@u}\v@leur\m@jBBB@x{\v@leur}{#1}{#2}(#3,#4,#5,#6)\fi%
    \else\delt@=\repdecn@mb{\mili@u}\mili@u\v@leur=\repdecn@mb{\v@lmax}\v@lmin%
    \advance\delt@-\v@leur\ifdim\delt@<\z@\else\invers@{\v@lmax}{\v@lmax}%
    \edef\Uns@rAp{\repdecn@mb{\v@lmax}}\sqrt@{\delt@}{\delt@}%
    \v@leur=-\mili@u\advance\v@leur\delt@\v@leur=\Uns@rAp\v@leur%
    \m@jBBB@x{\v@leur}{#1}{#2}(#3,#4,#5,#6)%
    \v@leur=-\mili@u\advance\v@leur-\delt@\v@leur=\Uns@rAp\v@leur%
    \m@jBBB@x{\v@leur}{#1}{#2}(#3,#4,#5,#6)\fi\fi\fi}}
\ctr@ld@f\def\m@jBBB@x#1#2#3(#4,#5,#6,#7){{\relax\ifdim#1>\z@\ifdim#1<\p@%
    \edef\T@{\repdecn@mb{#1}}\v@lX=\p@\advance\v@lX-#1\edef\UNmT@{\repdecn@mb{\v@lX}}%
    \v@lX=#4\v@lY=#5\v@lZ=#6\v@lXa=#7\v@lX=\UNmT@\v@lX\advance\v@lX\T@\v@lY%
    \v@lY=\UNmT@\v@lY\advance\v@lY\T@\v@lZ\v@lZ=\UNmT@\v@lZ\advance\v@lZ\T@\v@lXa%
    \v@lX=\UNmT@\v@lX\advance\v@lX\T@\v@lY\v@lY=\UNmT@\v@lY\advance\v@lY\T@\v@lZ%
    \v@lX=\UNmT@\v@lX\advance\v@lX\T@\v@lY%
    \ifcase#2\or\v@lY=#3\or\v@lY=\v@lX\v@lX=#3\fi\b@undb@x{\v@lX}{\v@lY}\fi\fi}}
\ctr@ld@f\def\PsB@zier#1[#2]{{\f@gnewpath%
    \s@mme=\z@\def\list@num{#2,0}\extrairelepremi@r\p@int\de\list@num%
    \PSwrit@cmdS{\p@int}{\c@mmoveto}{\fwf@g}{\X@un}{\Y@un}\p@rtent=#1\bclB@zier}}
\ctr@ld@f\def\bclB@zier{\relax%
    \ifnum\s@mme<\p@rtent\advance\s@mme\@ne\BdingB@xfalse%
    \extrairelepremi@r\p@int\de\list@num\PSwrit@cmdS{\p@int}{}{\fwf@g}{\X@de}{\Y@de}%
    \extrairelepremi@r\p@int\de\list@num\PSwrit@cmdS{\p@int}{}{\fwf@g}{\X@tr}{\Y@tr}%
    \BdingB@xtrue%
    \extrairelepremi@r\p@int\de\list@num\PSwrit@cmdS{\p@int}{\c@mcurveto}{\fwf@g}{\X@qu}{\Y@qu}%
    \B@zierBB@x{1}{\Y@un}(\X@un,\X@de,\X@tr,\X@qu)%
    \B@zierBB@x{2}{\X@un}(\Y@un,\Y@de,\Y@tr,\Y@qu)%
    \edef\X@un{\X@qu}\edef\Y@un{\Y@qu}\bclB@zier\fi}
\ctr@ln@m\psBezier
\ctr@ld@f\def\psBezierDD#1[#2]{\ifcurr@ntPS\ifps@cri%
    \PSc@mment{psBezierDD N arcs=#1, Control points=#2}%
    \iffillm@de\PsB@zier#1[#2]%
    \f@gfill%
    \else\PsB@zier#1[#2]\f@gstroke\fi%
    \PSc@mment{End psBezierDD}\fi\fi}
\ctr@ln@m\et@tpsBezierTD% ou doubler les {}
\ctr@ld@f\def\psBezierTD#1[#2]{\ifcurr@ntPS\ifps@cri\s@uvc@ntr@l\et@tpsBezierTD%
    \PSc@mment{psBezierTD N arcs=#1, Control points=#2}%
    \iffillm@de\PsB@zierTD#1[#2]%
    \f@gfill%
    \else\PsB@zierTD#1[#2]\f@gstroke\fi%
    \PSc@mment{End psBezierTD}\resetc@ntr@l\et@tpsBezierTD\fi\fi}
\ctr@ld@f\def\PsB@zierTD#1[#2]{\ifnum\curr@ntproj<\tw@\PsB@zier#1[#2]\else\PsB@zier@TD#1[#2]\fi}
\ctr@ld@f\def\PsB@zier@TD#1[#2]{{\f@gnewpath%
    \s@mme=\z@\def\list@num{#2,0}\extrairelepremi@r\p@int\de\list@num%
    \let\c@lprojSP=\relax\setc@ntr@l{2}\Figptpr@j-7:/\p@int/%
    \PSwrit@cmd{-7}{\c@mmoveto}{\fwf@g}%
    \loop\ifnum\s@mme<#1\advance\s@mme\@ne\extrairelepremi@r\p@intun\de\list@num%
    \extrairelepremi@r\p@intde\de\list@num\extrairelepremi@r\p@inttr\de\list@num%
    \subB@zierTD\NBz@rcs[\p@int,\p@intun,\p@intde,\p@inttr]\edef\p@int{\p@inttr}\repeat}}
\ctr@ld@f\def\subB@zierTD#1[#2,#3,#4,#5]{\delt@=\p@\divide\delt@\NBz@rcs\v@lmin=\z@%
    {\Figg@tXY{-7}\edef\X@un{\the\v@lX}\edef\Y@un{\the\v@lY}%
    \s@mme=\z@\loop\ifnum\s@mme<#1\advance\s@mme\@ne%
    \v@leur=\v@lmin\advance\v@leur0.33333 \delt@\edef\unti@rs{\repdecn@mb{\v@leur}}%
    \v@leur=\v@lmin\advance\v@leur0.66666 \delt@\edef\deti@rs{\repdecn@mb{\v@leur}}%
    \advance\v@lmin\delt@\edef\trti@rs{\repdecn@mb{\v@lmin}}%
    \figptBezierTD-8::\trti@rs[#2,#3,#4,#5]\Figptpr@j-8:/-8/%
    \c@lsubBzarc\unti@rs,\deti@rs[#2,#3,#4,#5]\BdingB@xfalse%
    \PSwrit@cmdS{-4}{}{\fwf@g}{\X@de}{\Y@de}\PSwrit@cmdS{-3}{}{\fwf@g}{\X@tr}{\Y@tr}%
    \BdingB@xtrue\PSwrit@cmdS{-8}{\c@mcurveto}{\fwf@g}{\X@qu}{\Y@qu}%
    \B@zierBB@x{1}{\Y@un}(\X@un,\X@de,\X@tr,\X@qu)%
    \B@zierBB@x{2}{\X@un}(\Y@un,\Y@de,\Y@tr,\Y@qu)%
    \edef\X@un{\X@qu}\edef\Y@un{\Y@qu}\figptcopyDD-7:/-8/\repeat}}
\ctr@ld@f\def\NBz@rcs{2}
\ctr@ld@f\def\c@lsubBzarc#1,#2[#3,#4,#5,#6]{\figptBezierTD-5::#1[#3,#4,#5,#6]%
    \figptBezierTD-6::#2[#3,#4,#5,#6]\Figptpr@j-4:/-5/\Figptpr@j-5:/-6/%
    \figptscontrolDD-4[-7,-4,-5,-8]}
\ctr@ln@m\pscirc
\ctr@ld@f\def\pscircDD#1(#2){\ifcurr@ntPS\ifps@cri\PSc@mment{pscircDD Center=#1 (Radius=#2)}%
    \psarccircDD#1;#2(0,360)\PSc@mment{End pscircDD}\fi\fi}
\ctr@ld@f\def\pscircTD#1,#2,#3(#4){\ifcurr@ntPS\ifps@cri%
    \PSc@mment{pscircTD Center=#1,P1=#2,P2=#3 (Radius=#4)}%
    \psarccircTD#1,#2,#3;#4(0,360)\PSc@mment{End pscircTD}\fi\fi}
\ctr@ln@m\p@urcent
{\catcode`\%=12\gdef\p@urcent{%}}
\ctr@ld@f\def\PSc@mment#1{\ifpsdebugmode\immediate\write\fwf@g{\p@urcent\space#1}\fi}
\ctr@ln@m\acc@louv \ctr@ln@m\acc@lfer
{\catcode`\[=1\catcode`\{=12\gdef\acc@louv[{}}
{\catcode`\]=2\catcode`\}=12\gdef\acc@lfer{}]]
\ctr@ld@f\def\PSdict@{\ifUse@llipse%
    \immediate\write\fwf@g{/ellipsedict 9 dict def ellipsedict /mtrx matrix put}%
    \immediate\write\fwf@g{/ellipse \acc@louv ellipsedict begin}%
    \immediate\write\fwf@g{ /endangle exch def /startangle exch def}%
    \immediate\write\fwf@g{ /yrad exch def /xrad exch def}%
    \immediate\write\fwf@g{ /rotangle exch def /y exch def /x exch def}%
    \immediate\write\fwf@g{ /savematrix mtrx currentmatrix def}%
    \immediate\write\fwf@g{ x y translate rotangle rotate xrad yrad scale}%
    \immediate\write\fwf@g{ 0 0 1 startangle endangle arc}%
    \immediate\write\fwf@g{ savematrix setmatrix end\acc@lfer def}%
    \fi\PShe@der{EndProlog}}
\ctr@ld@f\def\Pssetc@rve#1=#2|{\keln@mun#1|%
    \def\n@mref{r}\ifx\l@debut\n@mref\pssetroundness{#2}\else% roundness
    \immediate\write16{*** Unknown attribute: \BS@ psset curve(..., #1=...)}%
    \fi}
\ctr@ln@m\curv@roundness
\ctr@ld@f\def\pssetroundness#1{\edef\curv@roundness{#1}}
\ctr@ld@f\def\defaultroundness{0.2} % Valeur par defaut
\ctr@ln@m\pscurve
\ctr@ld@f\def\pscurveDD[#1]{{\ifcurr@ntPS\ifps@cri\PSc@mment{pscurveDD Points=#1}%
    \s@uvc@ntr@l\et@tpscurveDD%
    \iffillm@de\Psc@rveDD\curv@roundness[#1]%
    \f@gfill%
    \else\Psc@rveDD\curv@roundness[#1]\f@gstroke\fi%
    \PSc@mment{End pscurveDD}\resetc@ntr@l\et@tpscurveDD\fi\fi}}
\ctr@ld@f\def\pscurveTD[#1]{{\ifcurr@ntPS\ifps@cri%
    \PSc@mment{pscurveTD Points=#1}\s@uvc@ntr@l\et@tpscurveTD\let\c@lprojSP=\relax%
    \iffillm@de\Psc@rveTD\curv@roundness[#1]%
    \f@gfill%
    \else\Psc@rveTD\curv@roundness[#1]\f@gstroke\fi%
    \PSc@mment{End pscurveTD}\resetc@ntr@l\et@tpscurveTD\fi\fi}}
\ctr@ld@f\def\Psc@rveDD#1[#2]{%
    \def\list@num{#2}\extrairelepremi@r\Ak@\de\list@num%
    \extrairelepremi@r\Ai@\de\list@num\extrairelepremi@r\Aj@\de\list@num%
    \f@gnewpath\PSwrit@cmdS{\Ai@}{\c@mmoveto}{\fwf@g}{\X@un}{\Y@un}%
    \setc@ntr@l{2}\figvectPDD -1[\Ak@,\Aj@]%
    \@ecfor\Ak@:=\list@num\do{\figpttraDD-2:=\Ai@/#1,-1/\BdingB@xfalse%
       \PSwrit@cmdS{-2}{}{\fwf@g}{\X@de}{\Y@de}%
       \figvectPDD -1[\Ai@,\Ak@]\figpttraDD-2:=\Aj@/-#1,-1/%
       \PSwrit@cmdS{-2}{}{\fwf@g}{\X@tr}{\Y@tr}\BdingB@xtrue%
       \PSwrit@cmdS{\Aj@}{\c@mcurveto}{\fwf@g}{\X@qu}{\Y@qu}%
       \B@zierBB@x{1}{\Y@un}(\X@un,\X@de,\X@tr,\X@qu)%
       \B@zierBB@x{2}{\X@un}(\Y@un,\Y@de,\Y@tr,\Y@qu)%
       \edef\X@un{\X@qu}\edef\Y@un{\Y@qu}\edef\Ai@{\Aj@}\edef\Aj@{\Ak@}}}
\ctr@ld@f\def\Psc@rveTD#1[#2]{\ifnum\curr@ntproj<\tw@\Psc@rvePPTD#1[#2]\else\Psc@rveCPTD#1[#2]\fi}
\ctr@ld@f\def\Psc@rvePPTD#1[#2]{\setc@ntr@l{2}%
    \def\list@num{#2}\extrairelepremi@r\Ak@\de\list@num\Figptpr@j-5:/\Ak@/%
    \extrairelepremi@r\Ai@\de\list@num\Figptpr@j-3:/\Ai@/%
    \extrairelepremi@r\Aj@\de\list@num\Figptpr@j-4:/\Aj@/%
    \f@gnewpath\PSwrit@cmdS{-3}{\c@mmoveto}{\fwf@g}{\X@un}{\Y@un}%
    \figvectPDD -1[-5,-4]%
    \@ecfor\Ak@:=\list@num\do{\Figptpr@j-5:/\Ak@/\figpttraDD-2:=-3/#1,-1/%
       \BdingB@xfalse\PSwrit@cmdS{-2}{}{\fwf@g}{\X@de}{\Y@de}%
       \figvectPDD -1[-3,-5]\figpttraDD-2:=-4/-#1,-1/%
       \PSwrit@cmdS{-2}{}{\fwf@g}{\X@tr}{\Y@tr}\BdingB@xtrue%
       \PSwrit@cmdS{-4}{\c@mcurveto}{\fwf@g}{\X@qu}{\Y@qu}%
       \B@zierBB@x{1}{\Y@un}(\X@un,\X@de,\X@tr,\X@qu)%
       \B@zierBB@x{2}{\X@un}(\Y@un,\Y@de,\Y@tr,\Y@qu)%
       \edef\X@un{\X@qu}\edef\Y@un{\Y@qu}\figptcopyDD-3:/-4/\figptcopyDD-4:/-5/}}
\ctr@ld@f\def\Psc@rveCPTD#1[#2]{\setc@ntr@l{2}%
    \def\list@num{#2}\extrairelepremi@r\Ak@\de\list@num%
    \extrairelepremi@r\Ai@\de\list@num\extrairelepremi@r\Aj@\de\list@num%
    \Figptpr@j-7:/\Ai@/%
    \f@gnewpath\PSwrit@cmd{-7}{\c@mmoveto}{\fwf@g}%
    \figvectPTD -9[\Ak@,\Aj@]%
    \@ecfor\Ak@:=\list@num\do{\figpttraTD-10:=\Ai@/#1,-9/%
       \figvectPTD -9[\Ai@,\Ak@]\figpttraTD-11:=\Aj@/-#1,-9/%
       \subB@zierTD\NBz@rcs[\Ai@,-10,-11,\Aj@]\edef\Ai@{\Aj@}\edef\Aj@{\Ak@}}}
\ctr@ld@f\def\psendfig{\ifcurr@ntPS\ifps@cri\immediate\closeout\fwf@g%
    \immediate\openout\fwf@g=\PSfilen@me\relax%
    \ifPDFm@ke\PSBdingB@x\else%
    \immediate\write\fwf@g{\p@urcent\string!PS-Adobe-2.0 EPSF-2.0}%
    \PShe@der{Creator\string: TeX (fig4tex.tex)}%
    \PShe@der{Title\string: \PSfilen@me}%
    \PShe@der{CreationDate\string: \the\day/\the\month/\the\year}%
    \PSBdingB@x%
    \PShe@der{EndComments}\PSdict@\fi%
    \immediate\write\fwf@g{\c@mgsave}%
    \openin\frf@g=\auxfilen@me\c@pypsfile\fwf@g\frf@g\closein\frf@g%
    \immediate\write\fwf@g{\c@mgrestore}%
    \PSc@mment{End of file.}\immediate\closeout\fwf@g%
    \immediate\openout\fwf@g=\auxfilen@me\immediate\closeout\fwf@g%
    \immediate\write16{File \PSfilen@me\space created.}\fi\fi\curr@ntPSfalse\ps@critrue}
\ctr@ld@f\def\PShe@der#1{\immediate\write\fwf@g{\p@urcent\p@urcent#1}}
\ctr@ld@f\def\PSBdingB@x{{\v@lX=\ptT@ptps\c@@rdXmin\v@lY=\ptT@ptps\c@@rdYmin%
     \v@lXa=\ptT@ptps\c@@rdXmax\v@lYa=\ptT@ptps\c@@rdYmax%
     \PShe@der{BoundingBox\string: \repdecn@mb{\v@lX}\space\repdecn@mb{\v@lY}%
     \space\repdecn@mb{\v@lXa}\space\repdecn@mb{\v@lYa}}}}
\ctr@ld@f\def\psfcconnect[#1]{{\ifcurr@ntPS\ifps@cri\PSc@mment{psfcconnect Points=#1}%
    \pssetfillmode{no}\s@uvc@ntr@l\et@tpsfcconnect\resetc@ntr@l{2}%
    \fcc@nnect@[#1]\resetc@ntr@l\et@tpsfcconnect\PSc@mment{End psfcconnect}\fi\fi}}
\ctr@ld@f\def\fcc@nnect@[#1]{\let\N@rm=\n@rmeucDD\def\list@num{#1}%
    \extrairelepremi@r\Ai@\de\list@num\edef\pr@m{\Ai@}\v@leur=\z@\p@rtent=\@ne\c@llgtot%
    \ifcase\fclin@typ@\edef\list@num{[\pr@m,#1,\Ai@}\expandafter\pscurve\list@num]%
    \else\ifdim\fclin@r@d\p@>\z@\Pslin@conge[#1]\else\psline[#1]\fi\fi%
    \v@leur=\@rrowp@s\v@leur\edef\list@num{#1,\Ai@,0}%
    \extrairelepremi@r\Ai@\de\list@num\mili@u=\epsil@n\c@llgpart%
    \advance\mili@u-\epsil@n\advance\mili@u-\delt@\advance\v@leur-\mili@u%
    \ifcase\fclin@typ@\invers@\mili@u\delt@%
    \ifnum\@rrowr@fpt>\z@\advance\delt@-\v@leur\v@leur=\delt@\fi%
    \v@leur=\repdecn@mb\v@leur\mili@u\edef\v@lt{\repdecn@mb\v@leur}%
    \extrairelepremi@r\Ak@\de\list@num%
    \figvectPDD-1[\pr@m,\Aj@]\figpttraDD-6:=\Ai@/\curv@roundness,-1/%
    \figvectPDD-1[\Ak@,\Ai@]\figpttraDD-7:=\Aj@/\curv@roundness,-1/%
    \delt@=\@rrowheadlength\p@\delt@=\C@AHANG\delt@\edef\R@dius{\repdecn@mb{\delt@}}%
    \ifcase\@rrowr@fpt%
    \FigptintercircB@zDD-8::\v@lt,\R@dius[\Ai@,-6,-7,\Aj@]\psarrowheadDD[-5,-8]\else%
    \FigptintercircB@zDD-8::\v@lt,\R@dius[\Aj@,-7,-6,\Ai@]\psarrowheadDD[-8,-5]\fi%
    \else\advance\delt@-\v@leur%
    \p@rtentiere{\p@rtent}{\delt@}\edef\C@efun{\the\p@rtent}%
    \p@rtentiere{\p@rtent}{\v@leur}\edef\C@efde{\the\p@rtent}%
    \figptbaryDD-5:[\Ai@,\Aj@;\C@efun,\C@efde]\ifcase\@rrowr@fpt%
    \delt@=\@rrowheadlength\unit@\delt@=\C@AHANG\delt@\edef\t@ille{\repdecn@mb{\delt@}}%
    \figvectPDD-2[\Ai@,\Aj@]\vecunit@{-2}{-2}\figpttraDD-5:=-5/\t@ille,-2/\fi%
    \psarrowheadDD[\Ai@,-5]\fi}
\ctr@ld@f\def\c@llgtot{\@ecfor\Aj@:=\list@num\do{\figvectP-1[\Ai@,\Aj@]\N@rm\delt@{-1}%
    \advance\v@leur\delt@\advance\p@rtent\@ne\edef\Ai@{\Aj@}}}
\ctr@ld@f\def\c@llgpart{\extrairelepremi@r\Aj@\de\list@num\figvectP-1[\Ai@,\Aj@]\N@rm\delt@{-1}%
    \advance\mili@u\delt@\ifdim\mili@u<\v@leur\edef\pr@m{\Ai@}\edef\Ai@{\Aj@}\c@llgpart\fi}
\ctr@ld@f\def\Pslin@conge[#1]{\ifnum\p@rtent>\tw@{\def\list@num{#1}%
    \extrairelepremi@r\Ai@\de\list@num\extrairelepremi@r\Aj@\de\list@num%
    \figptcopy-6:/\Ai@/\figvectP-3[\Ai@,\Aj@]\vecunit@{-3}{-3}\v@lmax=\result@t%
    \@ecfor\Ak@:=\list@num\do{\figvectP-4[\Aj@,\Ak@]\vecunit@{-4}{-4}%
    \minim@m\v@lmin\v@lmax\result@t\v@lmax=\result@t%
    \det@rm\delt@[-3,-4]\maxim@m\mili@u{\delt@}{-\delt@}\ifdim\mili@u>\Cepsil@n%
    \ifdim\delt@>\z@\figgetangleDD\Angl@[\Aj@,\Ak@,\Ai@]\else%
    \figgetangleDD\Angl@[\Aj@,\Ai@,\Ak@]\fi%
    \v@leur=\PI@deg\advance\v@leur-\Angl@\p@\divide\v@leur\tw@%
    \edef\Angl@{\repdecn@mb\v@leur}\c@ssin{\C@}{\S@}{\Angl@}\v@leur=\fclin@r@d\unit@%
    \v@leur=\S@\v@leur\mili@u=\C@\p@\invers@\mili@u\mili@u%
    \v@leur=\repdecn@mb{\mili@u}\v@leur%
    \minim@m\v@leur\v@leur\v@lmin\edef\t@ille{\repdecn@mb{\v@leur}}%
    \figpttra-5:=\Aj@/-\t@ille,-3/\psline[-6,-5]\figpttra-6:=\Aj@/\t@ille,-4/%
    \figvectNVDD-3[-3]\figvectNVDD-8[-4]\inters@cDD-7:[-5,-3;-6,-8]%
    \ifdim\delt@>\z@\psarccircP-7;\fclin@r@d[-5,-6]\else\psarccircP-7;\fclin@r@d[-6,-5]\fi%
    \else\psline[-6,\Aj@]\figptcopy-6:/\Aj@/\fi% Points alignes
    \edef\Ai@{\Aj@}\edef\Aj@{\Ak@}\figptcopy-3:/-4/}\psline[-6,\Aj@]}\else\psline[#1]\fi}
\ctr@ld@f\def\psfcnode[#1]#2{{\ifcurr@ntPS\ifps@cri\PSc@mment{psfcnode Points=#1}%
    \s@uvc@ntr@l\et@tpsfcnode\resetc@ntr@l{2}%
    \def\t@xt@{#2}\ifx\t@xt@\empty\def\g@tt@xt{\setbox\Gb@x=\hbox{\Figg@tT{\p@int}}}%
    \else\def\g@tt@xt{\setbox\Gb@x=\hbox{#2}}\fi%
    \v@lmin=\h@rdfcXp@dd\advance\v@lmin\Xp@dd\unit@\multiply\v@lmin\tw@%
    \v@lmax=\h@rdfcYp@dd\advance\v@lmax\Yp@dd\unit@\multiply\v@lmax\tw@%
    \Figv@ctCreg-8(\unit@,-\unit@)\def\list@num{#1}%
    \delt@=\curr@ntwidth bp\divide\delt@\tw@%
    \fcn@de\PSc@mment{End psfcnode}\resetc@ntr@l\et@tpsfcnode\fi\fi}}
\ctr@ld@f\def\d@butn@de{\g@tt@xt\v@lX=\wd\Gb@x%
    \v@lY=\ht\Gb@x\advance\v@lY\dp\Gb@x\advance\v@lX\v@lmin\advance\v@lY\v@lmax}
\ctr@ld@f\def\fcn@deE{%
    \@ecfor\p@int:=\list@num\do{\d@butn@de\v@lX=\unssqrttw@\v@lX\v@lY=\unssqrttw@\v@lY%
    \ifdim\thickn@ss\p@>\z@% Shadow
    \v@lXa=\v@lX\advance\v@lXa\delt@\v@lXa=\ptT@unit@\v@lXa\edef\XR@d{\repdecn@mb\v@lXa}%
    \v@lYa=\v@lY\advance\v@lYa\delt@\v@lYa=\ptT@unit@\v@lYa\edef\YR@d{\repdecn@mb\v@lYa}%
    \arct@n\v@leur(\v@lXa,\v@lYa)\v@leur=\rdT@deg\v@leur\edef\@nglde{\repdecn@mb\v@leur}%
    {\c@lptellDD-2::\p@int;\XR@d,\YR@d(\@nglde)}% \v@lmin & \v@lmax modified in \c@lptellDD
    \advance\v@leur-\PI@deg\edef\@nglun{\repdecn@mb\v@leur}%
    {\c@lptellDD-3::\p@int;\XR@d,\YR@d(\@nglun)}%
    \figptstra-6=-3,-2,\p@int/\thickn@ss,-8/\pssetfillmode{yes}\us@secondC@lor%
    \psline[-2,-3,-6,-5]\psarcell-4;\XR@d,\YR@d(\@nglun,\@nglde,0)\fi% End shadow
    \v@lX=\ptT@unit@\v@lX\v@lY=\ptT@unit@\v@lY%
    \edef\XR@d{\repdecn@mb\v@lX}\edef\YR@d{\repdecn@mb\v@lY}%
    \pssetfillmode{yes}\us@thirdC@lor\psarcell\p@int;\XR@d,\YR@d(0,360,0)%
    \pssetfillmode{no}\us@primarC@lor\psarcell\p@int;\XR@d,\YR@d(0,360,0)}}
\ctr@ld@f\def\fcn@deL{\delt@=\ptT@unit@\delt@\edef\t@ille{\repdecn@mb\delt@}%
    \@ecfor\p@int:=\list@num\do{\Figg@tXYa{\p@int}\d@butn@de%
    \ifdim\v@lX>\v@lY\itis@Ktrue\else\itis@Kfalse\fi%
    \advance\v@lXa-\v@lX\Figp@intreg-1:(\v@lXa,\v@lYa)%
    \advance\v@lXa\v@lX\advance\v@lYa-\v@lY\Figp@intreg-2:(\v@lXa,\v@lYa)%
    \advance\v@lXa\v@lX\advance\v@lYa\v@lY\Figp@intreg-3:(\v@lXa,\v@lYa)%
    \advance\v@lXa-\v@lX\advance\v@lYa\v@lY\Figp@intreg-4:(\v@lXa,\v@lYa)%
    \ifdim\thickn@ss\p@>\z@\Figg@tXYa{\p@int}\pssetfillmode{yes}\us@secondC@lor% Shadow
    \c@lpt@xt{-1}{-4}\c@lpt@xt@\v@lXa\v@lYa\v@lX\v@lY\c@rre\delt@%
    \Figp@intregDD-9:(\v@lZ,\v@lYa)\Figp@intregDD-11:(\v@lZa,\v@lYa)%
    \c@lpt@xt{-4}{-3}\c@lpt@xt@\v@lYa\v@lXa\v@lY\v@lX\delt@\c@rre%
    \Figp@intregDD-12:(\v@lXa,\v@lZ)\Figp@intregDD-10:(\v@lXa,\v@lZa)%
    \ifitis@K\figptstra-7=-9,-10,-11/\thickn@ss,-8/\psline[-9,-11,-5,-6,-7]\else%
    \figptstra-7=-10,-11,-12/\thickn@ss,-8/\psline[-10,-12,-5,-6,-7]\fi\fi% End shadow
    \pssetfillmode{yes}\us@thirdC@lor\psline[-1,-2,-3,-4]%
    \pssetfillmode{no}\us@primarC@lor\psline[-1,-2,-3,-4,-1]}}
\ctr@ld@f\def\c@lpt@xt#1#2{\figvectN-7[#1,#2]\vecunit@{-7}{-7}\figpttra-5:=#1/\t@ille,-7/%
    \figvectP-7[#1,#2]\Figg@tXY{-7}\c@rre=\v@lX\delt@=\v@lY\Figg@tXY{-5}}
\ctr@ld@f\def\c@lpt@xt@#1#2#3#4#5#6{\v@lZ=#6\invers@{\v@lZ}{\v@lZ}\v@leur=\repdecn@mb{#5}\v@lZ%
    \v@lZ=#2\advance\v@lZ-#4\mili@u=\repdecn@mb{\v@leur}\v@lZ%
    \v@lZ=#3\advance\v@lZ\mili@u\v@lZa=-\v@lZ\advance\v@lZa\tw@#1}
\ctr@ld@f\def\fcn@deR{\@ecfor\p@int:=\list@num\do{\Figg@tXYa{\p@int}\d@butn@de%
    \advance\v@lXa-0.5\v@lX\advance\v@lYa-0.5\v@lY\Figp@intreg-1:(\v@lXa,\v@lYa)%
    \advance\v@lXa\v@lX\Figp@intreg-2:(\v@lXa,\v@lYa)%
    \advance\v@lYa\v@lY\Figp@intreg-3:(\v@lXa,\v@lYa)%
    \advance\v@lXa-\v@lX\Figp@intreg-4:(\v@lXa,\v@lYa)%
    \ifdim\thickn@ss\p@>\z@\pssetfillmode{yes}\us@secondC@lor% Shadow
    \Figv@ctCreg-5(-\delt@,-\delt@)\figpttra-9:=-1/1,-5/%
    \Figv@ctCreg-5(\delt@,-\delt@)\figpttra-10:=-2/1,-5/%
    \Figv@ctCreg-5(\delt@,\delt@)\figpttra-11:=-3/1,-5/%
    \figptstra-7=-9,-10,-11/\thickn@ss,-8/\psline[-9,-11,-5,-6,-7]\fi% End shadow
    \pssetfillmode{yes}\us@thirdC@lor\psline[-1,-2,-3,-4]%
    \pssetfillmode{no}\us@primarC@lor\psline[-1,-2,-3,-4,-1]}}
\ctr@ln@m\@rrowp@s
\ctr@ln@m\Xp@dd     \ctr@ln@m\Yp@dd
\ctr@ln@m\fclin@r@d \ctr@ln@m\thickn@ss
\ctr@ld@f\def\Pssetfl@wchart#1=#2|{\keln@mtr#1|%
    \def\n@mref{arr}\ifx\l@debut\n@mref\expandafter\keln@mtr\l@suite|%
     \def\n@mref{owp}\ifx\l@debut\n@mref\edef\@rrowp@s{#2}\else% arrowposition
     \def\n@mref{owr}\ifx\l@debut\n@mref\setfcr@fpt#2|\else% arrowrefpt
     \immediate\write16{*** Unknown attribute: \BS@ psset flowchart(..., #1=...)}%
     \fi\fi\else%
    \def\n@mref{lin}\ifx\l@debut\n@mref\setfccurv@#2|\else% line
    \def\n@mref{pad}\ifx\l@debut\n@mref\edef\Xp@dd{#2}\edef\Yp@dd{#2}\else% padding
    \def\n@mref{rad}\ifx\l@debut\n@mref\edef\fclin@r@d{#2}\else% connection radius
    \def\n@mref{sha}\ifx\l@debut\n@mref\setfcshap@#2|\else% shape
    \def\n@mref{thi}\ifx\l@debut\n@mref\edef\thickn@ss{#2}\else% thickness
    \def\n@mref{xpa}\ifx\l@debut\n@mref\edef\Xp@dd{#2}\else% xpadding
    \def\n@mref{ypa}\ifx\l@debut\n@mref\edef\Yp@dd{#2}\else% ypadding
    \immediate\write16{*** Unknown attribute: \BS@ psset flowchart(..., #1=...)}%
    \fi\fi\fi\fi\fi\fi\fi\fi}
\ctr@ln@m\@rrowr@fpt \ctr@ln@m\fclin@typ@
\ctr@ld@f\def\setfcr@fpt#1#2|{\if#1e\def\@rrowr@fpt{1}\else\def\@rrowr@fpt{0}\fi}
\ctr@ld@f\def\setfccurv@#1#2|{\if#1c\def\fclin@typ@{0}\else\def\fclin@typ@{1}\fi}
\ctr@ln@m\h@rdfcXp@dd \ctr@ln@m\h@rdfcYp@dd
\ctr@ln@m\fcn@de \ctr@ln@m\fcsh@pe
\ctr@ld@f\def\setfcshap@#1#2|{%
    \if#1e\let\fcn@de=\fcn@deE\def\h@rdfcXp@dd{4pt}\def\h@rdfcYp@dd{4pt}%
     \edef\fcsh@pe{ellipse}\else%
    \if#1l\let\fcn@de=\fcn@deL\def\h@rdfcXp@dd{4pt}\def\h@rdfcYp@dd{4pt}%
     \edef\fcsh@pe{lozenge}\else%
          \let\fcn@de=\fcn@deR\def\h@rdfcXp@dd{6pt}\def\h@rdfcYp@dd{6pt}%
     \edef\fcsh@pe{rectangle}\fi\fi}
\ctr@ld@f\def\psline[#1]{{\ifcurr@ntPS\ifps@cri\PSc@mment{psline Points=#1}%
    \let\pslign@=\Pslign@P\Pslin@{#1}\PSc@mment{End psline}\fi\fi}}
\ctr@ld@f\def\pslineF#1{{\ifcurr@ntPS\ifps@cri\PSc@mment{pslineF Filename=#1}%
    \let\pslign@=\Pslign@F\Pslin@{#1}\PSc@mment{End pslineF}\fi\fi}}
\ctr@ld@f\def\pslineC(#1){{\ifcurr@ntPS\ifps@cri\PSc@mment{pslineC}%
    \let\pslign@=\Pslign@C\Pslin@{#1}\PSc@mment{End pslineC}\fi\fi}}
\ctr@ld@f\def\Pslin@#1{\iffillm@de\pslign@{#1}%
    \f@gfill%
    \else\pslign@{#1}\ifx\derp@int\premp@int%
    \f@gclosestroke%
    \else\f@gstroke\fi\fi}
\ctr@ld@f\def\Pslign@P#1{\def\list@num{#1}\extrairelepremi@r\p@int\de\list@num%
    \edef\premp@int{\p@int}\f@gnewpath%
    \PSwrit@cmd{\p@int}{\c@mmoveto}{\fwf@g}%
    \@ecfor\p@int:=\list@num\do{\PSwrit@cmd{\p@int}{\c@mlineto}{\fwf@g}%
    \edef\derp@int{\p@int}}}
\ctr@ld@f\def\Pslign@F#1{\s@uvc@ntr@l\et@tPslign@F\setc@ntr@l{2}\openin\frf@g=#1\relax%
    \ifeof\frf@g\message{*** File #1 not found !}\end\else%
    \read\frf@g to\tr@c\edef\premp@int{\tr@c}\expandafter\extr@ctCF\tr@c:%
    \f@gnewpath\PSwrit@cmd{-1}{\c@mmoveto}{\fwf@g}%
    \loop\read\frf@g to\tr@c\ifeof\frf@g\mored@tafalse\else\mored@tatrue\fi%
    \ifmored@ta\expandafter\extr@ctCF\tr@c:\PSwrit@cmd{-1}{\c@mlineto}{\fwf@g}%
    \edef\derp@int{\tr@c}\repeat\fi\closein\frf@g\resetc@ntr@l\et@tPslign@F}
\ctr@ln@m\extr@ctCF
\ctr@ld@f\def\extr@ctCFDD#1 #2:{\v@lX=#1\unit@\v@lY=#2\unit@\Figp@intregDD-1:(\v@lX,\v@lY)}
\ctr@ld@f\def\extr@ctCFTD#1 #2 #3:{\v@lX=#1\unit@\v@lY=#2\unit@\v@lZ=#3\unit@%
    \Figp@intregTD-1:(\v@lX,\v@lY,\v@lZ)}
\ctr@ld@f\def\Pslign@C#1{\s@uvc@ntr@l\et@tPslign@C\setc@ntr@l{2}%
    \def\list@num{#1}\extrairelepremi@r\p@int\de\list@num%
    \edef\premp@int{\p@int}\f@gnewpath%
    \expandafter\Pslign@C@\p@int:\PSwrit@cmd{-1}{\c@mmoveto}{\fwf@g}%
    \@ecfor\p@int:=\list@num\do{\expandafter\Pslign@C@\p@int:%
    \PSwrit@cmd{-1}{\c@mlineto}{\fwf@g}\edef\derp@int{\p@int}}%
    \resetc@ntr@l\et@tPslign@C}
\ctr@ld@f\def\Pslign@C@#1 #2:{{\def\t@xt@{#1}\ifx\t@xt@\empty\Pslign@C@#2:% Discard leading spaces
    \else\extr@ctCF#1 #2:\fi}}
\ctr@ln@m\c@ntrolmesh
\ctr@ld@f\def\Pssetm@sh#1=#2|{\keln@mun#1|%
    \def\n@mref{d}\ifx\l@debut\n@mref\pssetmeshdiag{#2}\else% diag
    \immediate\write16{*** Unknown attribute: \BS@ psset mesh(..., #1=...)}%
    \fi}
\ctr@ld@f\def\pssetmeshdiag#1{\edef\c@ntrolmesh{#1}}
\ctr@ld@f\def\defaultmeshdiag{0}    % Valeur par defaut
\ctr@ld@f\def\psmesh#1,#2[#3,#4,#5,#6]{{\ifcurr@ntPS\ifps@cri%
    \PSc@mment{psmesh N1=#1, N2=#2, Quadrangle=[#3,#4,#5,#6]}%
    \s@uvc@ntr@l\et@tpsmesh\Pss@tsecondSt\setc@ntr@l{2}%
    \ifnum#1>\@ne\Psmeshp@rt#1[#3,#4,#5,#6]\fi%
    \ifnum#2>\@ne\Psmeshp@rt#2[#4,#5,#6,#3]\fi%
    \ifnum\c@ntrolmesh>\z@\Psmeshdi@g#1,#2[#3,#4,#5,#6]\fi%
    \ifnum\c@ntrolmesh<\z@\Psmeshdi@g#2,#1[#4,#5,#6,#3]\fi\Psrest@reSt%
    \psline[#3,#4,#5,#6,#3]\PSc@mment{End psmesh}\resetc@ntr@l\et@tpsmesh\fi\fi}}
\ctr@ld@f\def\Psmeshp@rt#1[#2,#3,#4,#5]{{\l@mbd@un=\@ne\l@mbd@de=#1\loop%
    \ifnum\l@mbd@un<#1\advance\l@mbd@de\m@ne\figptbary-1:[#2,#3;\l@mbd@de,\l@mbd@un]%
    \figptbary-2:[#5,#4;\l@mbd@de,\l@mbd@un]\psline[-1,-2]\advance\l@mbd@un\@ne\repeat}}
\ctr@ld@f\def\Psmeshdi@g#1,#2[#3,#4,#5,#6]{\figptcopy-2:/#3/\figptcopy-3:/#6/%
    \l@mbd@un=\z@\l@mbd@de=#1\loop\ifnum\l@mbd@un<#1%
    \advance\l@mbd@un\@ne\advance\l@mbd@de\m@ne\figptcopy-1:/-2/\figptcopy-4:/-3/%
    \figptbary-2:[#3,#4;\l@mbd@de,\l@mbd@un]%
    \figptbary-3:[#6,#5;\l@mbd@de,\l@mbd@un]\Psmeshdi@gp@rt#2[-1,-2,-3,-4]\repeat}
\ctr@ld@f\def\Psmeshdi@gp@rt#1[#2,#3,#4,#5]{{\l@mbd@un=\z@\l@mbd@de=#1\loop%
    \ifnum\l@mbd@un<#1\figptbary-5:[#2,#5;\l@mbd@de,\l@mbd@un]%
    \advance\l@mbd@de\m@ne\advance\l@mbd@un\@ne%
    \figptbary-6:[#3,#4;\l@mbd@de,\l@mbd@un]\psline[-5,-6]\repeat}}
\ctr@ln@m\psnormal
\ctr@ld@f\def\psnormalDD#1,#2[#3,#4]{{\ifcurr@ntPS\ifps@cri%
    \PSc@mment{psnormal Length=#1, Lambda=#2 [Pt1,Pt2]=[#3,#4]}%
    \s@uvc@ntr@l\et@tpsnormal\resetc@ntr@l{2}\figptendnormal-6::#1,#2[#3,#4]%
    \figptcopyDD-5:/-1/\psarrow[-5,-6]%
    \PSc@mment{End psnormal}\resetc@ntr@l\et@tpsnormal\fi\fi}}
\ctr@ld@f\def\psreset#1{\trtlis@rg{#1}{\Psreset@}}
\ctr@ld@f\def\Psreset@#1|{\keln@mde#1|%
    \def\n@mref{ar}\ifx\l@debut\n@mref\psresetarrowhead\else% arrowhead
    \def\n@mref{cu}\ifx\l@debut\n@mref\psset curve(roundness=\defaultroundness)\else% curve
    \def\n@mref{fi}\ifx\l@debut\n@mref\psset (color=\defaultcolor,dash=\defaultdash,%
         fill=\defaultfill,join=\defaultjoin,width=\defaultwidth)\else% primary settings
    \def\n@mref{fl}\ifx\l@debut\n@mref\psset flowchart(arrowp=\defaultfcarrowposition,%
	arrowr=\defaultfcarrowrefpt,line=\defaultfcline,xpadd=\defaultfcxpadding,%
	ypadd=\defaultfcypadding,radius=\defaultfcradius,shape=\defaultfcshape,%
	thick=\defaultfcthickness)\else% flow chart
    \def\n@mref{me}\ifx\l@debut\n@mref\psset mesh(diag=\defaultmeshdiag)\else% mesh
    \def\n@mref{se}\ifx\l@debut\n@mref\psresetsecondsettings\else% secondary
    \def\n@mref{th}\ifx\l@debut\n@mref\psset third(color=\defaultthirdcolor)\else% ternary
    \immediate\write16{*** Unknown keyword #1 (\BS@ psreset).}%
    \fi\fi\fi\fi\fi\fi\fi}
\ctr@ld@f\def\psset#1(#2){\def\t@xt@{#1}\ifx\t@xt@\empty\trtlis@rg{#2}{\Pssetf@rst}% primary settings
    \else\keln@mde#1|%
    \def\n@mref{ar}\ifx\l@debut\n@mref\trtlis@rg{#2}{\Psset@rrowhe@d}\else% arrow-head
    \def\n@mref{cu}\ifx\l@debut\n@mref\trtlis@rg{#2}{\Pssetc@rve}\else% curve
    \def\n@mref{fi}\ifx\l@debut\n@mref\trtlis@rg{#2}{\Pssetf@rst}\else% primary settings
    \def\n@mref{fl}\ifx\l@debut\n@mref\trtlis@rg{#2}{\Pssetfl@wchart}\else% flow chart
    \def\n@mref{me}\ifx\l@debut\n@mref\trtlis@rg{#2}{\Pssetm@sh}\else% mesh
    \def\n@mref{se}\ifx\l@debut\n@mref\trtlis@rg{#2}{\Pssets@cond}\else% secondary settings
    \def\n@mref{th}\ifx\l@debut\n@mref\trtlis@rg{#2}{\Pssetth@rd}\else% ternary settings
    \immediate\write16{*** Unknown keyword: \BS@ psset #1(...)}%
    \fi\fi\fi\fi\fi\fi\fi\fi}
\ctr@ld@f\def\pssetdefault#1(#2){\ifcurr@ntPS\immediate\write16{*** \BS@ pssetdefault is ignored
    inside a \BS@ psbeginfig-\BS@ psendfig block.}%
    \immediate\write16{*** It must be called before \BS@ psbeginfig.}\else%
    \def\t@xt@{#1}\ifx\t@xt@\empty\trtlis@rg{#2}{\Pssd@f@rst}\else\keln@mde#1|%
    \def\n@mref{ar}\ifx\l@debut\n@mref\trtlis@rg{#2}{\Pssd@@rrowhe@d}\else% arrow-head
    \def\n@mref{cu}\ifx\l@debut\n@mref\trtlis@rg{#2}{\Pssd@c@rve}\else% curve
    \def\n@mref{fi}\ifx\l@debut\n@mref\trtlis@rg{#2}{\Pssd@f@rst}\else% primary settings
    \def\n@mref{fl}\ifx\l@debut\n@mref\trtlis@rg{#2}{\Pssd@fl@wchart}\else% flow chart
    \def\n@mref{me}\ifx\l@debut\n@mref\trtlis@rg{#2}{\Pssd@m@sh}\else% mesh
    \def\n@mref{se}\ifx\l@debut\n@mref\trtlis@rg{#2}{\Pssd@s@cond}\else% secondary settings
    \def\n@mref{th}\ifx\l@debut\n@mref\trtlis@rg{#2}{\Pssd@th@rd}\else% ternary settings
    \immediate\write16{*** Unknown keyword: \BS@ pssetdefault #1(...)}%
    \fi\fi\fi\fi\fi\fi\fi\fi\initpss@ttings\fi}
\ctr@ld@f\def\Pssd@f@rst#1=#2|{\keln@mun#1|%
    \def\n@mref{c}\ifx\l@debut\n@mref\edef\defaultcolor{#2}\else% color
    \def\n@mref{d}\ifx\l@debut\n@mref\edef\defaultdash{#2}\else% dash
    \def\n@mref{f}\ifx\l@debut\n@mref\edef\defaultfill{#2}\else% fillmode
    \def\n@mref{j}\ifx\l@debut\n@mref\edef\defaultjoin{#2}\else% line join
    \def\n@mref{u}\ifx\l@debut\n@mref\edef\defaultupdate{#2}\pssetupdate{#2}\else% update
    \def\n@mref{w}\ifx\l@debut\n@mref\edef\defaultwidth{#2}\else% line width
    \immediate\write16{*** Unknown attribute: \BS@ pssetdefault (..., #1=...)}%
    \fi\fi\fi\fi\fi\fi}
\ctr@ld@f\def\Pssd@@rrowhe@d#1=#2|{\keln@mun#1|%
    \def\n@mref{a}\ifx\l@debut\n@mref\edef\defaultarrowheadangle{#2}\else% angle
    \def\n@mref{f}\ifx\l@debut\n@mref\edef\defaultarrowheadangle{#2}\else% fillmode
    \def\n@mref{l}\ifx\l@debut\n@mref\y@tiunit{#2}\ifunitpr@sent%
     \edef\defaulth@rdahlength{#2}\else\edef\defaulth@rdahlength{#2pt}%
     \message{*** \BS@ pssetdefault (..., #1=#2, ...) : unit is missing, pt is assumed.}%
     \fi\else% length
    \def\n@mref{o}\ifx\l@debut\n@mref\edef\defaultarrowheadout{#2}\else% out
    \def\n@mref{r}\ifx\l@debut\n@mref\edef\defaultarrowheadratio{#2}\else% ratio
    \immediate\write16{*** Unknown attribute: \BS@ pssetdefault arrowhead(..., #1=...)}%
    \fi\fi\fi\fi\fi}
\ctr@ld@f\def\Pssd@c@rve#1=#2|{\keln@mun#1|%
    \def\n@mref{r}\ifx\l@debut\n@mref\edef\defaultroundness{#2}\else%
    \immediate\write16{*** Unknown attribute: \BS@ pssetdefault curve(..., #1=...)}%
    \fi}
\ctr@ld@f\def\Pssd@fl@wchart#1=#2|{\keln@mtr#1|%
    \def\n@mref{arr}\ifx\l@debut\n@mref\expandafter\keln@mtr\l@suite|%
     \def\n@mref{owp}\ifx\l@debut\n@mref\edef\defaultfcarrowposition{#2}\else% arrowposition
     \def\n@mref{owr}\ifx\l@debut\n@mref\edef\defaultfcarrowrefpt{#2}\else% arrowrefpt
     \immediate\write16{*** Unknown attribute: \BS@ pssetdefault flowchart(..., #1=...)}%
     \fi\fi\else%
    \def\n@mref{lin}\ifx\l@debut\n@mref\edef\defaultfcline{#2}\else% line
    \def\n@mref{pad}\ifx\l@debut\n@mref\edef\defaultfcxpadding{#2}%
                    \edef\defaultfcypadding{#2}\else% padding
    \def\n@mref{rad}\ifx\l@debut\n@mref\edef\defaultfcradius{#2}\else% connection radius
    \def\n@mref{sha}\ifx\l@debut\n@mref\edef\defaultfcshape{#2}\else% shape
    \def\n@mref{thi}\ifx\l@debut\n@mref\edef\defaultfcthickness{#2}\else% thickness
    \def\n@mref{xpa}\ifx\l@debut\n@mref\edef\defaultfcxpadding{#2}\else% xpadding
    \def\n@mref{ypa}\ifx\l@debut\n@mref\edef\defaultfcypadding{#2}\else% ypadding
    \immediate\write16{*** Unknown attribute: \BS@ pssetdefault flowchart(..., #1=...)}%
    \fi\fi\fi\fi\fi\fi\fi\fi}
\ctr@ld@f\def\defaultfcarrowposition{0.5}%\ctr@ld@f\let\defaultfcarrowpos=\defaultfcarrowposition
\ctr@ld@f\def\defaultfcarrowrefpt{start}
\ctr@ld@f\def\defaultfcline{polygon}
\ctr@ld@f\def\defaultfcradius{0}
\ctr@ld@f\def\defaultfcshape{rectangle}
\ctr@ld@f\def\defaultfcthickness{0}%\ctr@ld@f\let\defaultfcthick=\defaultfcthickness
\ctr@ld@f\def\defaultfcxpadding{0}%\ctr@ld@f\let\defaultfcxpad=\defaultfcxpadding
\ctr@ld@f\def\defaultfcypadding{0}%\ctr@ld@f\let\defaultfcypad=\defaultfcypadding
\ctr@ld@f\def\Pssd@m@sh#1=#2|{\keln@mun#1|%
    \def\n@mref{d}\ifx\l@debut\n@mref\edef\defaultmeshdiag{#2}\else%
    \immediate\write16{*** Unknown attribute: \BS@ pssetdefault mesh(..., #1=...)}%
    \fi}
\ctr@ld@f\def\Pssd@s@cond#1=#2|{\keln@mun#1|%
    \def\n@mref{c}\ifx\l@debut\n@mref\edef\defaultsecondcolor{#2}\else%
    \def\n@mref{d}\ifx\l@debut\n@mref\edef\defaultseconddash{#2}\else%
    \def\n@mref{w}\ifx\l@debut\n@mref\edef\defaultsecondwidth{#2}\else%
    \immediate\write16{*** Unknown attribute: \BS@ pssetdefault second(..., #1=...)}%
    \fi\fi\fi}
\ctr@ld@f\def\Pssd@th@rd#1=#2|{\keln@mun#1|%
    \def\n@mref{c}\ifx\l@debut\n@mref\edef\defaultthirdcolor{#2}\else%
    \immediate\write16{*** Unknown attribute: \BS@ pssetdefault third(..., #1=...)}%
    \fi}
\ctr@ln@w{newif}\iffillm@de
\ctr@ld@f\def\pssetfillmode#1{\expandafter\setfillm@de#1:}
\ctr@ld@f\def\setfillm@de#1#2:{\if#1n\fillm@defalse\else\fillm@detrue\fi}
\ctr@ld@f\def\defaultfill{no}     % Valeur par defaut
\ctr@ln@w{newif}\ifpsupdatem@de
\ctr@ld@f\def\pssetupdate#1{\ifcurr@ntPS\immediate\write16{*** \BS@ pssetupdate is ignored inside a
     \BS@ psbeginfig-\BS@ psendfig block.}%
    \immediate\write16{*** It must be called before \BS@ psbeginfig.}%
    \else\expandafter\setupd@te#1:\fi}
\ctr@ld@f\def\setupd@te#1#2:{\if#1n\psupdatem@defalse\else\psupdatem@detrue\fi}
\ctr@ld@f\def\defaultupdate{no}     % Valeur par defaut
\ctr@ln@m\curr@ntcolor \ctr@ln@m\curr@ntcolorc@md
\ctr@ld@f\def\Pssetc@lor#1{\ifps@cri\result@tent=\@ne\expandafter\c@lnbV@l#1 :%
    \def\curr@ntcolor{}\def\curr@ntcolorc@md{}%
    \ifcase\result@tent\or\pssetgray{#1}\or\or\pssetrgb{#1}\or\pssetcmyk{#1}\fi\fi}
\ctr@ln@m\curr@ntcolorc@mdStroke
\ctr@ld@f\def\pssetcmyk#1{\ifps@cri\def\curr@ntcolor{#1}\def\curr@ntcolorc@md{\c@msetcmykcolor}%
    \def\curr@ntcolorc@mdStroke{\c@msetcmykcolorStroke}%
    \ifcurr@ntPS\PSc@mment{pssetcmyk Color=#1}\us@primarC@lor\fi\fi}
\ctr@ld@f\def\pssetrgb#1{\ifps@cri\def\curr@ntcolor{#1}\def\curr@ntcolorc@md{\c@msetrgbcolor}%
    \def\curr@ntcolorc@mdStroke{\c@msetrgbcolorStroke}%
    \ifcurr@ntPS\PSc@mment{pssetrgb Color=#1}\us@primarC@lor\fi\fi}
\ctr@ld@f\def\pssetgray#1{\ifps@cri\def\curr@ntcolor{#1}\def\curr@ntcolorc@md{\c@msetgray}%
    \def\curr@ntcolorc@mdStroke{\c@msetgrayStroke}%
    \ifcurr@ntPS\PSc@mment{pssetgray Gray level=#1}\us@primarC@lor\fi\fi}
\ctr@ln@m\fillc@md
\ctr@ld@f\def\us@primarC@lor{\immediate\write\fwf@g{\d@fprimarC@lor}%
    \let\fillc@md=\prfillc@md}
\ctr@ld@f\def\prfillc@md{\d@fprimarC@lor\space\c@mfill}
\ctr@ld@f\def\defaultcolor{0}       % Valeur par defaut
\ctr@ld@f\def\c@lnbV@l#1 #2:{\def\t@xt@{#1}\relax\ifx\t@xt@\empty\c@lnbV@l#2:% Discard leading spaces
    \else\c@lnbV@l@#1 #2:\fi}
\ctr@ld@f\def\c@lnbV@l@#1 #2:{\def\t@xt@{#2}\ifx\t@xt@\empty%
    \def\t@xt@{#1}\ifx\t@xt@\empty\advance\result@tent\m@ne\fi% Discard trailing spaces
    \else\advance\result@tent\@ne\c@lnbV@l@#2:\fi}
\ctr@ld@f\def\Blackcmyk{0 0 0 1}
\ctr@ld@f\def\Whitecmyk{0 0 0 0}
\ctr@ld@f\def\Cyancmyk{1 0 0 0}
\ctr@ld@f\def\Magentacmyk{0 1 0 0}
\ctr@ld@f\def\Yellowcmyk{0 0 1 0}
\ctr@ld@f\def\Redcmyk{0 1 1 0}
\ctr@ld@f\def\Greencmyk{1 0 1 0}
\ctr@ld@f\def\Bluecmyk{1 1 0 0}
\ctr@ld@f\def\Graycmyk{0 0 0 0.50}
\ctr@ld@f\def\BrickRedcmyk{0 0.89 0.94 0.28} % PANTONE 1805
\ctr@ld@f\def\Browncmyk{0 0.81 1 0.60} % PANTONE 1615
\ctr@ld@f\def\ForestGreencmyk{0.91 0 0.88 0.12} % PANTONE 349
\ctr@ld@f\def\Goldenrodcmyk{ 0 0.10 0.84 0} % PANTONE 109
\ctr@ld@f\def\Marooncmyk{0 0.87 0.68 0.32} % PANTONE 201
\ctr@ld@f\def\Orangecmyk{0 0.61 0.87 0} % PANTONE ORANGE-021
\ctr@ld@f\def\Purplecmyk{0.45 0.86 0 0} % PANTONE PURPLE
\ctr@ld@f\def\RoyalBluecmyk{1. 0.50 0 0} % No PANTONE match
\ctr@ld@f\def\Violetcmyk{0.79 0.88 0 0} % PANTONE VIOLET
\ctr@ld@f\def\Blackrgb{0 0 0}
\ctr@ld@f\def\Whitergb{1 1 1}
\ctr@ld@f\def\Redrgb{1 0 0}
\ctr@ld@f\def\Greenrgb{0 1 0}
\ctr@ld@f\def\Bluergb{0 0 1}
\ctr@ld@f\def\Cyanrgb{0 1 1}
\ctr@ld@f\def\Magentargb{1 0 1}
\ctr@ld@f\def\Yellowrgb{1 1 0}
\ctr@ld@f\def\Grayrgb{0.5 0.5 0.5}
\ctr@ld@f\def\Chocolatergb{0.824 0.412 0.118}
\ctr@ld@f\def\DarkGoldenrodrgb{0.722 0.525 0.043}
\ctr@ld@f\def\DarkOrangergb{1 0.549 0}
\ctr@ld@f\def\Firebrickrgb{0.698 0.133 0.133}
\ctr@ld@f\def\ForestGreenrgb{0.133 0.545 0.133}
\ctr@ld@f\def\Goldrgb{1 0.843 0}
\ctr@ld@f\def\HotPinkrgb{1 0.412 0.706}
\ctr@ld@f\def\Maroonrgb{0.690 0.188 0.376}
\ctr@ld@f\def\Pinkrgb{1 0.753 0.796}
\ctr@ld@f\def\RoyalBluergb{0.255 0.412 0.882}
\ctr@ld@f\def\Pssetf@rst#1=#2|{\keln@mun#1|%
    \def\n@mref{c}\ifx\l@debut\n@mref\Pssetc@lor{#2}\else% color
    \def\n@mref{d}\ifx\l@debut\n@mref\pssetdash{#2}\else% dash
    \def\n@mref{f}\ifx\l@debut\n@mref\pssetfillmode{#2}\else% fillmode
    \def\n@mref{j}\ifx\l@debut\n@mref\pssetjoin{#2}\else% line join
    \def\n@mref{u}\ifx\l@debut\n@mref\pssetupdate{#2}\else% update
    \def\n@mref{w}\ifx\l@debut\n@mref\pssetwidth{#2}\else% line width
    \immediate\write16{*** Unknown attribute: \BS@ psset (..., #1=...)}%
    \fi\fi\fi\fi\fi\fi}
\ctr@ln@m\curr@ntdash
\ctr@ld@f\def\s@uvdash#1{\edef#1{\curr@ntdash}}
\ctr@ld@f\def\defaultdash{1}        % Valeur par defaut (numero sans espace)
\ctr@ld@f\def\pssetdash#1{\ifps@cri\edef\curr@ntdash{#1}\ifcurr@ntPS\expandafter\Pssetd@sh#1 :\fi\fi}
\ctr@ld@f\def\Pssetd@shI#1{\PSc@mment{pssetdash Index=#1}\ifcase#1%
    \or\immediate\write\fwf@g{[] 0 \c@msetdash}%         Index=1
    \or\immediate\write\fwf@g{[6 2] 0 \c@msetdash}%      Index=2
    \or\immediate\write\fwf@g{[4 2] 0 \c@msetdash}%      Index=3
    \or\immediate\write\fwf@g{[2 2] 0 \c@msetdash}%      Index=4
    \or\immediate\write\fwf@g{[1 2] 0 \c@msetdash}%      Index=5
    \or\immediate\write\fwf@g{[2 4] 0 \c@msetdash}%      Index=6
    \or\immediate\write\fwf@g{[3 5] 0 \c@msetdash}%      Index=7
    \or\immediate\write\fwf@g{[3 3] 0 \c@msetdash}%      Index=8
    \or\immediate\write\fwf@g{[3 5 1 5] 0 \c@msetdash}%  Index=9
    \or\immediate\write\fwf@g{[6 4 2 4] 0 \c@msetdash}%  Index=10
    \fi}
\ctr@ld@f\def\Pssetd@sh#1 #2:{{\def\t@xt@{#1}\ifx\t@xt@\empty\Pssetd@sh#2:% Discard leading spaces
    \else\def\t@xt@{#2}\ifx\t@xt@\empty\Pssetd@shI{#1}\else\s@mme=\@ne\def\debutp@t{#1}%
    \an@lysd@sh#2:\ifodd\s@mme\edef\debutp@t{\debutp@t\space\finp@t}\def\finp@t{0}\fi%
    \PSc@mment{pssetdash Pattern=#1 #2}%
    \immediate\write\fwf@g{[\debutp@t] \finp@t\space\c@msetdash}\fi\fi}}
\ctr@ld@f\def\an@lysd@sh#1 #2:{\def\t@xt@{#2}\ifx\t@xt@\empty\def\finp@t{#1}\else%
    \edef\debutp@t{\debutp@t\space#1}\advance\s@mme\@ne\an@lysd@sh#2:\fi}
\ctr@ln@m\curr@ntwidth
\ctr@ld@f\def\s@uvwidth#1{\edef#1{\curr@ntwidth}}
\ctr@ld@f\def\defaultwidth{0.4}     % Valeur par defaut
\ctr@ld@f\def\pssetwidth#1{\ifps@cri\edef\curr@ntwidth{#1}\ifcurr@ntPS%
    \PSc@mment{pssetwidth Width=#1}\immediate\write\fwf@g{#1 \c@msetlinewidth}\fi\fi}
\ctr@ln@m\curr@ntjoin
\ctr@ld@f\def\pssetjoin#1{\ifps@cri\edef\curr@ntjoin{#1}\ifcurr@ntPS\expandafter\Pssetj@in#1:\fi\fi}
\ctr@ld@f\def\Pssetj@in#1#2:{\PSc@mment{pssetjoin join=#1}%
    \if#1r\def\t@xt@{1}\else\if#1b\def\t@xt@{2}\else\def\t@xt@{0}\fi\fi%
    \immediate\write\fwf@g{\t@xt@\space\c@msetlinejoin}}
\ctr@ld@f\def\defaultjoin{miter}   % Valeur par defaut
\ctr@ld@f\def\Pssets@cond#1=#2|{\keln@mun#1|%
    \def\n@mref{c}\ifx\l@debut\n@mref\Pssets@condcolor{#2}\else%
    \def\n@mref{d}\ifx\l@debut\n@mref\pssetseconddash{#2}\else%
    \def\n@mref{w}\ifx\l@debut\n@mref\pssetsecondwidth{#2}\else%
    \immediate\write16{*** Unknown attribute: \BS@ psset second(..., #1=...)}%
    \fi\fi\fi}
\ctr@ln@m\curr@ntseconddash
\ctr@ld@f\def\pssetseconddash#1{\edef\curr@ntseconddash{#1}}
\ctr@ld@f\def\defaultseconddash{4}  % Valeur par defaut (numero sans espace)
\ctr@ln@m\curr@ntsecondwidth
\ctr@ld@f\def\pssetsecondwidth#1{\edef\curr@ntsecondwidth{#1}}
\ctr@ld@f\edef\defaultsecondwidth{\defaultwidth} % Valeur par defaut
\ctr@ld@f\def\psresetsecondsettings{%
    \pssetseconddash{\defaultseconddash}\pssetsecondwidth{\defaultsecondwidth}%
    \Pssets@condcolor{\defaultsecondcolor}}
\ctr@ln@m\sec@ndcolor \ctr@ln@m\sec@ndcolorc@md
\ctr@ld@f\def\Pssets@condcolor#1{\ifps@cri\result@tent=\@ne\expandafter\c@lnbV@l#1 :%
    \def\sec@ndcolor{}\def\sec@ndcolorc@md{}%
    \ifcase\result@tent\or\pssetsecondgray{#1}\or\or\pssetsecondrgb{#1}%
    \or\pssetsecondcmyk{#1}\fi\fi}
\ctr@ln@m\sec@ndcolorc@mdStroke
\ctr@ld@f\def\pssetsecondcmyk#1{\def\sec@ndcolor{#1}\def\sec@ndcolorc@md{\c@msetcmykcolor}%
    \def\sec@ndcolorc@mdStroke{\c@msetcmykcolorStroke}}
\ctr@ld@f\def\pssetsecondrgb#1{\def\sec@ndcolor{#1}\def\sec@ndcolorc@md{\c@msetrgbcolor}%
    \def\sec@ndcolorc@mdStroke{\c@msetrgbcolorStroke}}
\ctr@ld@f\def\pssetsecondgray#1{\def\sec@ndcolor{#1}\def\sec@ndcolorc@md{\c@msetgray}%
    \def\sec@ndcolorc@mdStroke{\c@msetgrayStroke}}
\ctr@ld@f\def\us@secondC@lor{\immediate\write\fwf@g{\d@fsecondC@lor}%
    \let\fillc@md=\sdfillc@md}
\ctr@ld@f\def\sdfillc@md{\d@fsecondC@lor\space\c@mfill}
\ctr@ld@f\edef\defaultsecondcolor{\defaultcolor} % Valeur par defaut
\ctr@ld@f\def\Pss@tsecondSt{%
    \s@uvdash{\typ@dash}\pssetdash{\curr@ntseconddash}%
    \s@uvwidth{\typ@width}\pssetwidth{\curr@ntsecondwidth}\us@secondC@lor}
\ctr@ld@f\def\Psrest@reSt{\pssetwidth{\typ@width}\pssetdash{\typ@dash}\us@primarC@lor}
\ctr@ld@f\def\Pssetth@rd#1=#2|{\keln@mun#1|%
    \def\n@mref{c}\ifx\l@debut\n@mref\Pssetth@rdcolor{#2}\else%
    \immediate\write16{*** Unknown attribute: \BS@ psset third(..., #1=...)}%
    \fi}
\ctr@ln@m\th@rdcolor \ctr@ln@m\th@rdcolorc@md
\ctr@ld@f\def\Pssetth@rdcolor#1{\ifps@cri\result@tent=\@ne\expandafter\c@lnbV@l#1 :%
    \def\th@rdcolor{}\def\th@rdcolorc@md{}%
    \ifcase\result@tent\or\Pssetth@rdgray{#1}\or\or\Pssetth@rdrgb{#1}%
    \or\Pssetth@rdcmyk{#1}\fi\fi}
\ctr@ln@m\th@rdcolorc@mdStroke
\ctr@ld@f\def\Pssetth@rdcmyk#1{\def\th@rdcolor{#1}\def\th@rdcolorc@md{\c@msetcmykcolor}%
    \def\th@rdcolorc@mdStroke{\c@msetcmykcolorStroke}}
\ctr@ld@f\def\Pssetth@rdrgb#1{\def\th@rdcolor{#1}\def\th@rdcolorc@md{\c@msetrgbcolor}%
    \def\th@rdcolorc@mdStroke{\c@msetrgbcolorStroke}}
\ctr@ld@f\def\Pssetth@rdgray#1{\def\th@rdcolor{#1}\def\th@rdcolorc@md{\c@msetgray}%
    \def\th@rdcolorc@mdStroke{\c@msetgrayStroke}}
\ctr@ld@f\def\us@thirdC@lor{\immediate\write\fwf@g{\d@fthirdC@lor}%
    \let\fillc@md=\thfillc@md}
\ctr@ld@f\def\thfillc@md{\d@fthirdC@lor\space\c@mfill}
\ctr@ld@f\def\defaultthirdcolor{1}  % Valeur par defaut
\ctr@ld@f\def\pstrimesh#1[#2,#3,#4]{{\ifcurr@ntPS\ifps@cri%
    \PSc@mment{pstrimesh Type=#1, Triangle=[#2,#3,#4]}%
    \s@uvc@ntr@l\et@tpstrimesh\ifnum#1>\@ne\Pss@tsecondSt\setc@ntr@l{2}%
    \Pstrimeshp@rt#1[#2,#3,#4]\Pstrimeshp@rt#1[#3,#4,#2]%
    \Pstrimeshp@rt#1[#4,#2,#3]\Psrest@reSt\fi\psline[#2,#3,#4,#2]%
    \PSc@mment{End pstrimesh}\resetc@ntr@l\et@tpstrimesh\fi\fi}}
\ctr@ld@f\def\Pstrimeshp@rt#1[#2,#3,#4]{{\l@mbd@un=\@ne\l@mbd@de=#1\loop\ifnum\l@mbd@de>\@ne%
    \advance\l@mbd@de\m@ne\figptbary-1:[#2,#3;\l@mbd@de,\l@mbd@un]%
    \figptbary-2:[#2,#4;\l@mbd@de,\l@mbd@un]\psline[-1,-2]%
    \advance\l@mbd@un\@ne\repeat}}
\initpr@lim\initpss@ttings\initPDF@rDVI% Initialisation preliminaire
\ctr@ln@w{newbox}\figBoxA
\ctr@ln@w{newbox}\figBoxB
\ctr@ln@w{newbox}\figBoxC
\catcode`\@=12

\pssetdefault(update=yes)
\newbox\figbox

\usepackage{color}
\definecolor{gr}{rgb}   {0.,   0.69,   0.23 } 
\definecolor{mg}{rgb}   {0.85,  0.,    0.85} 
\newcommand{\Gr}{\color{gr}}
\newcommand{\Mg}{\color{mg}}
\newcommand{\Bk}{\color{black}}
\newcommand{\Rd}{\color{red}}                                                                                                  
\newcommand{\Bl}{\color{blue}}

\definecolor{marin}{rgb}   {0.,   0.85,   0.25}
\definecolor{rouge}{rgb}   {0.8,   0.,   0.}

\newtheorem{theorem}{Theorem}[section]
\newtheorem{lemma}[theorem]{Lemma}
\newtheorem{proposition}[theorem]{Proposition}
\newtheorem{corollary}[theorem]{Corollary}

\theoremstyle{definition}
\newtheorem{definition}[theorem]{Definition}

\theoremstyle{remark}
\newtheorem{remark}[theorem]{Remark}

\numberwithin{equation}{section}

\newcommand{\ee}{\hskip0.15ex}
\newcommand{\me}{\hskip-0.15ex}
\newcommand{\dd}[1]{_{\raise-0.3ex\hbox{$\scriptstyle #1$}}}

\newcommand{\vpha}{\left.\vphantom{T^{j_0}_{j_0}}\!\!\right.}

\newcommand {\Norm}[2]{ \mathchoice 
    {|\ee #1\ee|\dd{#2}}
    {| #1 |_{#2}}
    {| #1 |_{#2}}
    {| #1 |_{#2}} }
\newcommand {\DNorm}[2]{ \mathchoice 
    {\|\ee #1\ee\|\dd{#2}}
    {\| #1 \|_{#2}}
    {\| #1 \|_{#2}}
    {\| #1 \|_{#2}} }
\newcommand {\Normc}[2]{ \mathchoice 
    {|\ee #1\ee|\dd{#2}^2}
    {| #1 |_{#2}^2}
    {| #1 |_{#2}^2}
    {| #1 |_{#2}^2} }
\newcommand {\DNormc}[2]{ \mathchoice 
    {\|\ee #1\ee\|\dd{#2}^2}
    {\| #1 \|_{#2}^2}
    {\| #1 \|_{#2}^2}
    {\| #1 \|_{#2}^2} }

\newcommand\rd{{\mathrm d}}
\let\d=\rd
\newcommand\re{{\mathrm e}}

\renewcommand{\div}{\operatorname{\rm div}}
\newcommand{\grad}{\operatorname{\textbf{grad}}}
\newcommand{\curl}{\operatorname{\textbf{curl}}}
\newcommand{\curls}{\operatorname{\rm curl}}
\newcommand{\Cinf}{\mathscr{C}^\infty}

\renewcommand\P{{\mathbb P}}
\newcommand\Sbb{{\mathbb S}}
\newcommand\C{{\mathbb C}}
\newcommand\R{{\mathbb R}}
\newcommand\N{{\mathbb N}}
\newcommand\Q{{\mathbb Q}}
\newcommand\T{\mathbb{T}}
\newcommand\Z{{\mathbb Z}}

\newcommand{\Rin}{R_{\mathrm{in}}}
\newcommand{\Rout}{R_{\mathrm{out}}}

\newcommand\A{{\mathscr A}}
\newcommand\B{{\mathcal B}}
\newcommand\BB{{\mathscr B}}
\newcommand\sC{{\mathscr C}}
\newcommand\CC{{\mathcal C}}
\newcommand\D{{\mathcal D}}
\newcommand\DD{{\mathscr D}}
\newcommand\E{{\mathcal E}}
\newcommand\F{{\mathcal F}}
\newcommand\G{{\mathcal G}}
\newcommand\GG{{\mathcal G}}
\newcommand\I{{\mathcal I}}
\newcommand\J{{\mathcal J}}
\newcommand\K{{\mathcal K}}
\renewcommand\L{{\mathscr L}}
\newcommand{\M}{\mathcal{M}}
\newcommand\OO{{\mathcal O}}
\newcommand\PP{{\mathscr P}}
\newcommand\QQ{{\mathcal Q}}
\newcommand\RR{{\mathcal R}}
\newcommand\U{{\mathcal U}}
\renewcommand\SS{{\mathcal S}}
\newcommand\TT{{\mathcal T}}
\newcommand\VV{{\boldsymbol V}}
\newcommand\WW{{\boldsymbol W}}
\newcommand\XX{{\boldsymbol X}}

\newcommand{\ba}{\boldsymbol{a}}
\newcommand{\bb}{\boldsymbol{b}}
\newcommand{\bc}{{\boldsymbol{c}}}
\newcommand{\be}{{\boldsymbol{e}}}
\newcommand{\bff}{\boldsymbol{f}}
\newcommand{\bn}{\boldsymbol{n}}
\newcommand{\bu}{\boldsymbol{u}}
\newcommand{\bv}{\boldsymbol{v}}
\newcommand{\bw}{\boldsymbol{w}}
\newcommand{\bx}{{\boldsymbol{x}}}
\let\x=\bx
\newcommand{\by}{{\boldsymbol{y}}}

\renewcommand\H{{\boldsymbol H}}

\newcommand\ttau{\boldsymbol{\tau}}
 
\newcommand{\sS}{{\mathscr{S}}}

\newcommand\sep{\hskip-0.2ex,\hskip0.15ex}
 
\newcommand {\gA}{\mathfrak{A}}
\newcommand {\gC}{\mathfrak{C}}
\newcommand {\gE}{\mathfrak{E}}
\newcommand {\gF}{\mathfrak{F}}
\newcommand {\gK}{{\mathfrak K}}
\newcommand {\gM}{{\mathfrak M}}
\newcommand {\gP}{{\mathfrak P}}
\newcommand {\gR}{{\mathfrak R}}
\newcommand {\gS}{{\mathfrak S}}
\newcommand {\gU}{{\mathfrak U}}
\newcommand {\gW}{{\mathfrak W}}
\newcommand {\gZ}{{\mathfrak Z}}
\newcommand{\CCC}{\mathfrak{C}}
\newcommand{\FFF}{\mathfrak{F}}

\newcommand{\inff}{\mathop{\operatorname{\vphantom{p}inf}}}
\newcommand{\xix}{\boldsymbol{\xi}}
\newcommand{\ii}{\mathrm{i}}
\newcommand {\Id}{\mathrm{Id}}
\newcommand{\comp}{\mathrm{comp}}
\newcommand{\diam}{\mathrm{diam}}
\newcommand{\dist}{\mathrm{dist}}
\newcommand{\meas}{\mathrm{meas}}
\newcommand{\HP}{\mathsf{HP}}

\newcommand{\DU}{\Delta^{-1}_U}

\newcommand{\di}{\displaystyle}

\renewcommand{\Re}{\operatorname{\rm Re}}
\renewcommand{\Im}{\operatorname{\rm Im}}

\begin{document}

\title[Continuity properties of the inf-sup constant]{\bf Continuity properties of the \\inf-sup constant for the divergence}

\author{Christine Bernardi, Martin Costabel, Monique Dauge and Vivette Girault}

\address{IRMAR UMR 6625 du CNRS, Universit\'{e} de Rennes 1}

\address{\vglue -4ex Campus de Beaulieu,
35042 Rennes Cedex, France \bigskip}

\email{martin.costabel@univ-rennes1.fr}
\urladdr{http://perso.univ-rennes1.fr/martin.costabel/}

\email{monique.dauge@univ-rennes1.fr}
\urladdr{http://perso.univ-rennes1.fr/monique.dauge/}

\address{Laboratoire Jacques-Louis Lions,
  CNRS et Universit\'e Pierre et Marie Curie}

\address{\vglue -4ex 4 place Jussieu, 75252 PARIS Cedex 05, France\bigskip}

\email{bernardi@ann.jussieu.fr}

\email{girault@ann.jussieu.fr}

\date{}  

\keywords{inf-sup constant, LBB condition}

\subjclass[2010]{35Q35,65M60,76D07}

\begin{abstract}
The inf-sup constant for the divergence, or LBB constant, is explicitly known for only few domains. For other domains, upper and lower estimates are known. If more precise values are required, one can try to compute a numerical approximation. This involves, in general, approximation of the domain and then the computation of a discrete LBB constant that can be obtained from the numerical solution of an eigenvalue problem for the Stokes system. This eigenvalue problem does not fall into a class for which standard results about numerical approximations can be applied. Indeed, many reasonable finite element methods do  not yield a convergent approximation. In this article, we show that under fairly weak conditions on the approximation of the domain, the LBB constant is an upper semi-continuous shape functional, and we give more restrictive sufficient conditions for its continuity with respect to the domain. 
For numerical approximations based on variational formulations of the Stokes eigenvalue problem, we also show upper semi-continuity under weak approximation properties, and we give stronger conditions that are sufficient for convergence of the discrete LBB constant towards the continuous LBB constant. 
Numerical examples show that our conditions are, while not quite optimal, not very far from necessary.     
\end{abstract}

\maketitle

%%%%%%%%%%%%%
%% Section %%
%%%%%%%%%%%%%
\section{Introduction}

We define the usual inf-sup constant for the divergence or Ladyzhenskaya-Babu\v{s}ka-Brezzi (LBB) constant $\beta(\Omega)$  as
\begin{equation}
\label{E:infsup}
    \beta(\Omega) = \ {\inff_{q\ee\in\ee L^2_\circ(\Omega)}} \ \ 
   {\sup_{\bv\ee\in\ee
   H^1_0(\Omega)^d}} \ \ 
   \frac{\big\langle\! \div\bv, q \big\rangle_\Omega}{\Norm{\bv}{1,\Omega}\,\DNorm{q}{0,\Omega}} \ .
\end{equation}
The constant thus depends on a domain (connected open set) $\Omega\subset\R^{d}$ and on two function spaces, the space (``pressure space'')
$L^2_\circ(\Omega)$ of square integrable functions with mean value zero,  with norm $\DNorm{\cdot}{0,\Omega}$ and scalar product $\big\langle \cdot, \cdot \big\rangle_\Omega$, and the ``velocity space'', which is the vector-valued version of the standard Sobolev space
$H^1_0(\Omega)$, closure of the space of smooth functions of compact support in $\Omega$ with respect to the $H^{1}$ seminorm, which for vector-valued functions is defined by
\[
   \Norm{\bv}{1,\Omega} = \DNorm{\grad \bv}{0,\Omega} =
   \Big(\sum_{k=1}^d \sum_{j=1}^d \DNormc{\partial_{x_j}v_k}{0,\Omega} \Big)^{1/2} .
\]
The present article is devoted to the study of the variation of the LBB constant with respect to changes of the domain and also of the two function spaces.

The inf-sup condition $\beta(\Omega)>0$ plays an important role in the existence and stability of solutions of incompressible fluid models such as the Stokes and Navier--Stokes equations, see for example \cite[Chap.~I, Theorem 4.1]{GiraultRaviart}. 
The value of the LBB constant influences the convergence rate of iterative algorithms such as the Uzawa algorithm, see~\cite{Crouzeix1997}.
The inf-sup condition has been known for a long time to be true for bounded  domains satisfying a uniform cone condition, that is, Lipschitz domains \cite{Bogovskii79} and has more recently been shown for the larger class of John domains defined by a ``twisted cone'' condition \cite{AcostaDuranMuschietti2006}. It is, however, not satisfied for domains having an exterior cusp \cite{Friedrichs1937,Tartar_NS2006}. 

If one defines  an analogous  ``discrete LBB constant'' $\beta_{n}$ by
\begin{equation}
\label{E:infsupN}
    \beta_{n} = \ {\inff_{q\ee\in\ee M_{n}}} \ \ 
   {\sup_{\bv\ee\in\ee
   X_{n}}} \ \ 
   \frac{\big\langle\! \div\bv, q \big\rangle_\Omega}{\Norm{\bv}{1,\Omega}\,\DNorm{q}{0,\Omega}} \  , 
\end{equation}
where 
\[
 M_{n}\subset L^2_\circ(\Omega) \quad\mbox{ and }\quad  X_{n}\subset H^1_0(\Omega)^d
\]
are subspaces, then it is known that the discrete LBB condition
\begin{equation}
\label{E:discreteLBB}
  \forall\, n, \quad \beta_{n}\ge \beta_\star>0
\end{equation}
plays an equally important role for the stability of Galerkin approximation methods for the Stokes system defined by a sequence of finite-dimensional subspaces $(X_{n},M_{n})_{n}$ \cite{BabuskaAziz,Brezzi1974}, \cite[Chap.~II, Theorem 1.1]{GiraultRaviart}, as well as for the convergence of solution methods for the corresponding linear systems.  
There exists a large body of finite element literature proving discrete LBB conditions \cref{E:discreteLBB} for many pairs of velocity/pressure spaces $(X_{n},M_{n})$, often referred to as \emph{inf-sup stable} elements.

Much less is known about the question whether $\beta_{n}$ is an approximation of $\beta(\Omega)$, namely
\begin{equation}
\label{E:betaNtobeta}
 \lim_{n\to\infty}\beta_{n}=\beta(\Omega)\,.
\end{equation}
It turns out that there is a large class of finite element methods for which a uniform lower bound \cref{E:discreteLBB} has been shown, but the convergence \cref{E:betaNtobeta} has not been established, and sometimes is even false.
One of the main results of this work is to give conditions on the approximation properties of the spaces $X_{n}$ and $M_{n}$ so that the convergence \cref{E:betaNtobeta} is true, see \cref{T:bNconv}. 

Under much weaker approximation properties, concerning only the pressure spaces $M_{n}$, we show upper semi-continuity (see \cref{P:USC}), which implies in particular that the uniform lower bound $\beta_\star$ in the discrete LBB condition \cref{E:discreteLBB} can never be better than the (continuous) LBB constant $\beta(\Omega)$:
\begin{equation}
\label{E:beta0<=beta}
  \beta_\star\le\beta(\Omega)\,.
\end{equation}
The latter result has been shown under more restrictive hypotheses  on the regularity of the domain  in 
\cite[Theorem 1.2]{ChizhOlsh00}.

Regarding the approximation of the domain $\Omega$, we prove two general results. 
For inner approximations, we show upper semi-continuity of $\beta(\Omega)$ under the weak condition that the difference of the domains tends to zero in measure, see \cref{T:semiconv}. We describe several examples where one does not have convergence of $\beta(\Omega)$ in this situation.

The second main result of this paper is that under convergence with respect to Lipschitz deformations of a bounded Lipschitz domain $\Omega$, one does have continuity of the constant $\beta(\Omega)$, see \cref{T:encadrement}. As a corollary, we find that the LBB constant converges if a smooth plane domain is approximated by polygonal interpolation.
This has an important application in finite element discretizations, where the first step often consists in replacing a curved domain by a union of triangles or polyhedra. Our result shows that this replacement essentially  does not deteriorate the stability of the continuous problem. A second application is in the numerical approximation of the inf-sup constant itself, where one might want to approximate both the function spaces and the domain. 

The outline of this article is as follows: 

$\bullet$ We first prove in \Cref{S:USC} upper semi-continuity results of the inf-sup constant both with respect to the function spaces and with respect to interior approximations of the domain. 

$\bullet$ \Cref{S:ExContDom} is devoted to several examples exhibiting various types of converging domains, resulting in convergence or non-convergence of the inf-sup constant.

$\bullet$ Then in  \Cref{S:ContDom} we prove the continuity with respect to Lipschitz deformations of the domain and, as a corollary, with respect to polygonal approximations of smooth plane domains.

$\bullet$ Finally, in  \Cref{S:ContFun} we prove that under certain conditions, the discrete inf-sup constants converge to the continuous inf-sup constant. We give several numerical examples exhibiting convergence or non-convergence. 

%%%%%%%%%%%%%
%% Section %%
%%%%%%%%%%%%%
\section{Upper semi-continuity}\label{S:USC}

We first prove  a general result about the upper semi-continuity of the inf-sup constants $\beta_{n}$ defined in \cref{E:infsupN} as $n$ tends to infinity. We then show that this result can be used in two ways: First as a statement on the behavior of the sequence of the inf-sup constants 
$\beta(\Omega_{n})$ for subdomains $\Omega_{n}\subset\Omega$ converging to $\Omega$, and second as a statement of the relation between the discrete LBB condition \cref{E:discreteLBB} and the LBB constant $\beta(\Omega)$. 

%%%%%%%%%%%%
\subsection{Upper semi-continuity with respect to the function spaces}\label{Ss:USC}
%%%%%%%%%%%%

We now state and prove a key result.

\begin{theorem}
\label{P:USC}
Let $\Omega\subset\R^{d}$ be a bounded domain. 
Let $M_{n}\subset L^2_\circ(\Omega)$ and $X_{n}\subset H^1_0(\Omega)^d$, $n\in\N$, be a sequence of closed subspaces. Let $\beta_{n}$ be the inf-sup constant defined by these spaces according to \cref{E:infsupN}. We assume that the sequence $M_{n}$ is asymptotically dense in the sense that for every $q\in L^2_\circ(\Omega)$ there exists a sequence $(q_{n})_{n\in\N}$ with $q_{n}\in M_{n}$ converging to $q$ in $L^{2}(\Omega)$. Then
\begin{equation}
\label{E:limsup}
 \limsup_{n\to\infty} \beta_{n} \,\le\, \beta(\Omega)\,.
\end{equation}
\end{theorem}

\begin{proof}
 For $q\in L^2_\circ(\Omega)$ define 
\begin{equation}
\label{E:J}
 J(q) = \sup_{\bv\ee\in\ee H^1_0(\Omega)^d}
   \frac{\big\langle\! \div\bv, q \big\rangle_\Omega}{\Norm{\bv}{1,\Omega}}
  \quad\mbox{ and }\quad 
 J_{n}(q) = \sup_{\bv\ee\in\ee X_{n}}
   \frac{\big\langle\! \div\bv, q \big\rangle_\Omega}{\Norm{\bv}{1,\Omega}}.
\end{equation}
Owing to the inequality 
$$\DNorm{\div\bv}{0,\Omega} \le \Norm{\bv}{1,\Omega},$$
the functionals $J$ and $J_n$ satisfy for all $q\in L^2_\circ(\Omega)$
$$|J(q)| \le \DNorm{q}{0,\Omega}\,,\qquad |J_n(q)| \le \DNorm{q}{0,\Omega}.
$$
Moreover,  the functional $J$ is continuous in $L^2_\circ(\Omega)$, as can be seen by introducing the function $\bw(q)\in H^1_0(\Omega)^d$ defined as the solution of the Dirichlet problem 
$\Delta\bw(q)=\grad q$, or in variational form
\begin{equation}
\label{E:w(q)}
 \forall \,\bv\in H^1_0(\Omega)^d,\quad
 \big\langle\grad\bw(q), \grad\bv\big\rangle_\Omega 
  = \big\langle \div\bv,q \big\rangle_\Omega \,.
\end{equation}
It is easy to see that the supremum in $J(q)$ is attained by the function $\bv=\bw(q)$. 
The mapping $q\mapsto\bw(q)$ being continuous from $L^2_\circ(\Omega)$ to $H^1_0(\Omega)^d$, the continuity of the functional $J$ follows from the formula
\[
  J(q) = \frac{\big\langle\! \div\bw(q), q \big\rangle_\Omega}{\Norm{\bw(q)}{1,\Omega}}
       = \Norm{\bw(q)}{1,\Omega}.
\]
We assume now that, possibly after passing to a subsequence, the sequence $(\beta_{n})_n$ converges to some $\beta_{\infty}$. Let a nontrivial function $q\in L^2_\circ(\Omega)$ be given and choose $q_{n}\in M_{n}$ such that $q_{n}\to q$ in $L^2(\Omega)$. We have
\[ 
  \frac{J(q)}{\DNorm{q}{0,\Omega}} 
  = \lim_{n\to\infty}\frac{J(q_{n})}{\DNorm{q_{n}}{0,\Omega}}
\]
and 
\[
 J(q_{n})\ge J_{n}(q_{n}) \ge \beta_{n} \DNorm{q_{n}}{0,\Omega}\,,
\]
which together imply
\[
  \frac{J(q)}{\DNorm{q}{0,\Omega}} \ge\beta_{\infty}\,.
\]
From the definition
\[
\beta(\Omega) = \inff_{q\ee\in\ee L^2_\circ(\Omega)} \frac{J(q)}{\DNorm{q}{0,\Omega}}
\]
we therefore get the desired inequality
$\beta(\Omega)\ge\beta_{\infty}$.
\end{proof}

%%%%%%%%%%%%
\subsection{Inner approximations of the domain}\label{Ss:InnerApprox}
%%%%%%%%%%%%
A first corollary of \cref{P:USC} is obtained by defining the subspaces $X_{n}$ and $M_{n}$ via the natural inclusion of the spaces $H^1_0(\Omega_{n})^d$ and $L^2_\circ(\Omega_{n})$, respectively, where $\Omega_{n}$ is a subdomain of $\Omega$.

\begin{theorem}
\label{T:semiconv}
Let $\Omega\subset\R^{d}$ be a bounded domain. Assume that the sequence of domains $\big(\Omega_n\big)_{n\in\N}$ converges to the limiting domain $\Omega$ in the following sense
\begin{equation}
\label{A:semiconv}
   \forall n\in\N, \; \Omega_n\subset\Omega\ \quad\mbox{and}\quad
   \meas(\Omega\setminus\Omega_n) \underset{n\to\infty}{\longrightarrow} 0\,,
\end{equation}
where $\meas$ denotes the Lebesgue measure. 
Then there holds
\begin{equation}
\label{E:semiconv}
   \limsup_{n\to\infty} \beta(\Omega_n) \le \beta(\Omega)\,.
\end{equation}
\end{theorem}

\begin{proof}
For a function $q$ defined on $\Omega_{n}$, let $\tilde q$ be its extension by $0$ to $\Omega$. This mapping $\mathcal{Z}_{n}:q\mapsto\tilde q$ defines the natural inclusions of 
$L^2(\Omega_{n})$ into $L^2(\Omega)$ and of $H^1_0(\Omega_{n})$ into $H^1_0(\Omega)$. 
We define 
\[
  M_{n} = \mathcal{Z}_{n} L^2_\circ(\Omega_{n})\subset L^2_\circ(\Omega)
  \quad\mbox{ and }\quad
  X_{n} = \mathcal{Z}_{n} H^1_0(\Omega_{n})^d \subset H^1_0(\Omega)^d
\]
and check that $\beta_{n}$ as defined in \cref{E:infsupN} then coincides with $\beta(\Omega_{n})$. 
We have to verify the hypothesis on the approximation property of $M_{n}$:
Choose $q\in L^2_\circ(\Omega)$,   
let us set on $\Omega_n$:
\[
   q_n = q\big|_{\Omega_n} - \frac{1}{\meas(\Omega_n)} \int_{\Omega_n} q(\x)\,\d\x,
\]
and let $\tilde q_n$ be the extension by $0$ of $q_n$ to $\Omega$. Then it is easy to see that
\[
   q_n\in L^2_\circ(\Omega_n),\quad
   \tilde q_n\in M_{n}\subset L^2_\circ(\Omega),\quad\mbox{ and }\quad
  \tilde q_n \underset{n\to\infty}{\longrightarrow} q
   \quad\mbox{in}\quad L^2(\Omega). 
\]
\cref{P:USC} can now be applied and gives the inequality \cref{E:semiconv}.
\end{proof}

We  discuss in \Cref{S:ExContDom} several examples of domains $\Omega_{n}$ tending to a 
domain $\Omega$ and observe whether or not $\beta(\Omega_{n})$ converges to $\beta(\Omega)$.

%%%%%%%%%%%%
\subsection{Discrete and continuous LBB conditions}\label{Ss:DiscLBB}
%%%%%%%%%%%%

In a conforming finite element discretization of the Stokes or Navier--Stokes equations, the trial space of the pressure variable is a finite-dimensional subspace $M_{n}$ of $L^2_\circ(\Omega)$, and the trial space for the velocity variable is a finite-dimensional subspace $X_{n}$ of $H^1_0(\Omega)^d$. The index $n$ may be representative for the sum of the dimensions of these spaces. The discrete inf-sup constant is defined as in \cref{E:infsupN}, and the uniform positivity of $\beta_{n}$, as expressed by the discrete LBB condition \cref{E:discreteLBB} plays an important role in the analysis of the method.

Proving uniform lower bounds for $\beta_{n}$ for various pairs of finite element spaces $(X_{n},M_{n})$ is an important subject of many papers in numerical fluid dynamics. A standard procedure in such proofs is the construction of a Fortin operator, see \cite{GiraultScott2003}. This is a projection operator 
\[
  \Pi_{n}: H^1_0(\Omega)^d \to X_{n}
\]
satisfying for each $q\in M_{n}$
\[\forall \bv\in H^1_0(\Omega)^d, \quad
  \big\langle \div\Pi_{n}\bv, q  \big \rangle_\Omega = \big\langle \div\bv, q \big \rangle_\Omega. 
\]
The Fortin operator provides a lower bound for the discrete inf-sup constant  via the inequality
\[
\begin{aligned}
  \beta(\Omega) \DNorm{q}{0,\Omega} \le J(q)
 &=
 \sup_{\bv\ee\in\ee H^1_0(\Omega)^d}
 \frac{\big\langle\! \div\bv, q \big\rangle_\Omega}{\Norm{\bv}{1,\Omega}} 
 =
 \sup_{\bv\ee\in\ee H^1_0(\Omega)^d}
 \frac{\big\langle\! \div\Pi_{n}\bv, q \big\rangle_\Omega}{\Norm{\Pi_{n}\bv}{1,\Omega}}
 \frac{\Norm{\Pi_{n}\bv}{1,\Omega}}{\Norm{\bv}{1,\Omega}}\\
 &\le
 \sup_{\bv\ee\in\ee X_{n}}
   \frac{\big\langle\! \div\bv, q \big\rangle_\Omega}{\Norm{\bv}{1,\Omega}}
   \|\Pi_{n}\|
  = J_{n}(q)\|\Pi_{n}\|
\end{aligned}
\]
valid for all $q\in M_{n}$. Here $\|\Pi_{n}\|$ is the operator norm of $\Pi_{n}$ in $H^1_0(\Omega)^d$ associated with the norm $\Norm{\cdot}{1,\Omega}$.
Dividing by $\DNorm{q}{0,\Omega}$ and taking the infimum over $q\in M_{n}$, one finds 
\[
 \beta(\Omega)\le \beta_{n}\|\Pi_{n}\|\,.
\]
Uniformly bounding $\|\Pi_{n}\|$ is then the main work in the proof of the discrete LBB condition.

The resulting lower bound  $\beta_{n}\ge \frac{\beta(\Omega)}{\|\Pi_{n}\|}$ lies always below $\beta(\Omega)$, because the norm of the projection operator $\Pi_{n}$ is always at least $1$.

The available theoretical estimates of the discrete LBB conditions may be very pessimistic compared to what is really observed in the numerical algorithms, but the fact that the uniform discrete LBB bound is always below the inf-sup constant of the domain is a corollary to the upper semi-continuity shown in \cref{P:USC}.

\begin{theorem}
\label{T:discreteUSC}
Suppose that the sequence of spaces $(X_{n},M_{n})_{n\in\N}$ satisfies the discrete LBB condition
\[
 \inf_{n\in\N}\beta_{n}=\beta_\star>0
\]
and that the sequence $M_{n}$ is asymptotically dense in $L^2_\circ(\Omega)$.
Then 
\[
  \beta_\star \le \beta(\Omega)\,.
\]
\end{theorem} 
\begin{proof}
This is an immediate consequence of \cref{P:USC}. 
\end{proof}

%%%%%%%%%%%%%
%% Section %%
%%%%%%%%%%%%%
\section{Examples of dependence with respect to the domain}\label{S:ExContDom}
We present several examples of sequences of domains $\Omega_n$ that converge to a limiting domain $\Omega$ in various senses. All examples satisfy the assumption of \cref{T:semiconv}, after a possible rescaling to achieve $\Omega_{n}\subset\Omega$. Some of them satisfy a stronger assumption (namely, the transformations tend to the identity in Lipschitz norm) and convergence occurs. As was already mentioned in \cite{Zsuppan2011} for conformal mappings, the absence of such a condition results in an absence of convergence.

%=======================
\subsection{Cusp domains tending to a Lipschitz domain}
\label{Ss:CusptoLips}
%=======================
For each integer $n \ge 1$, let $\Omega_n$ be the set of points $(x,y)$ such that
$$0< x <1, \quad -x^{\frac1n+1} < y < x^{\frac1n+ 1}.$$
For any $n$, the domain $\Omega_n$ has a cusp at the origin. By similar arguments as in \cite{Tartar_NS2006}, we derive that $\beta(\Omega_n)=0$. In the limit $n\to\infty$, we obtain the isosceles right triangle
\[
   \Omega = \{(x,y)\in\R^2,\quad 0< x <1, \quad -x < y < x\}.
\]
So $\Omega$ is a Lipschitz domain and its inf-sup constant is positive. We note that for any $n\ge1$
$$\Omega_n \subset \Omega,$$
and the distance of $\partial\Omega_n$ to $ \partial\Omega$ is smaller than
$\sup_{0\le x \le 1} x(1-x^{\frac1n})$. This implies that the measure of $\Omega\setminus\Omega_n$ tends to zero when $n$ tends to infinity. So we are in the framework of \cref{P:USC}. We have the upper semi-continuity without the continuity:
\[
   \beta(\Omega_n)=0<\beta(\Omega).
\]

%======================
\subsection{A self-similar excrescence}
\label{Ss:sefsim}
%======================
Another example where the convergence does not hold is obtained if we add to a disk $\Omega$ a small equilateral triangle whose surface tends to $0$. Let us say that $\Omega_n$ is the union of the unit disk centered at the origin and the equilateral triangle with side length $1/n$ and top vertex at $(0,1+\frac{1}{2n})$, see \cref{F:DiskEpi}.
The inf-sup constants of $\Omega_n$ do not converge to that of the disk as $n\to\infty$ because the presence of the angle $\frac\pi3$ on the boundary of $\Omega_n$ implies by virtue of \cite[Theorem 3.3]{CoCrDaLa15} that we have the upper bound for all  $n\ge1$ 
\[
   \beta(\Omega_n) \le \sqrt{\frac12-\frac{\sin\omega}{2\omega}}\quad
   \mbox{with}\quad \omega=\frac{\pi}{3}.
\]
This upper bound is smaller than $\beta(\Omega)$:
\[
   \beta(\Omega_n) \le \sqrt{\frac12-\frac{3\sqrt3}{4\pi}} < \sqrt{\frac12}=\beta(\Omega)\,.
\]

%-----------------------------------------------
\begin{figure}[ht]
\begin{minipage}{0.48\textwidth}
% 1. Definition of characteristic points
    \figinit{1.3mm}
    \figpt 1:( 0,  0)
    \figpt 3:( 0, 15.)
    \figpt 4:( 5, 6.34)
    \figpt 5:( -5, 6.34)
    \figptcirc 6 :: 1;10 (71.40962211)
    \figptcirc 7 :: 1;10 (108.59037789)
    \figpt 21:( 23,  0)
    \figpt 23:( 23, 11.)
    \figpt 24:( 24, 9.27)
    \figpt 25:( 22, 9.27)
    \figptcirc 26 :: 21;10 (86.632987031)
    \figptcirc 27 :: 21;10 (93.367012969)

% 2. Creation of the postscript file
    \def\MyPSfile{}
    \psbeginfig{\MyPSfile}
    \psline[6,3,7]
    \psarccircP 1 ; 10 [7,6]
    \psline[26,23,27]
    \psarccircP 21 ; 10 [27,26]
  \psendfig

% 3. Writing text on the figure
    \figvisu{\figbox}{}{%
    \figinsert{\MyPSfile}
    \figwrites 1 :{$n=1$}(12)
    \figwrites 21 :{$n=5$}(12)
    }
\centerline{\box\figbox}
\end{minipage}
\;\;
\begin{minipage}{0.48\textwidth}
\begin{tabular}{@{}c@{\,\,}c@{}}
\ &\ \\
\includegraphics[width=0.44\textwidth]{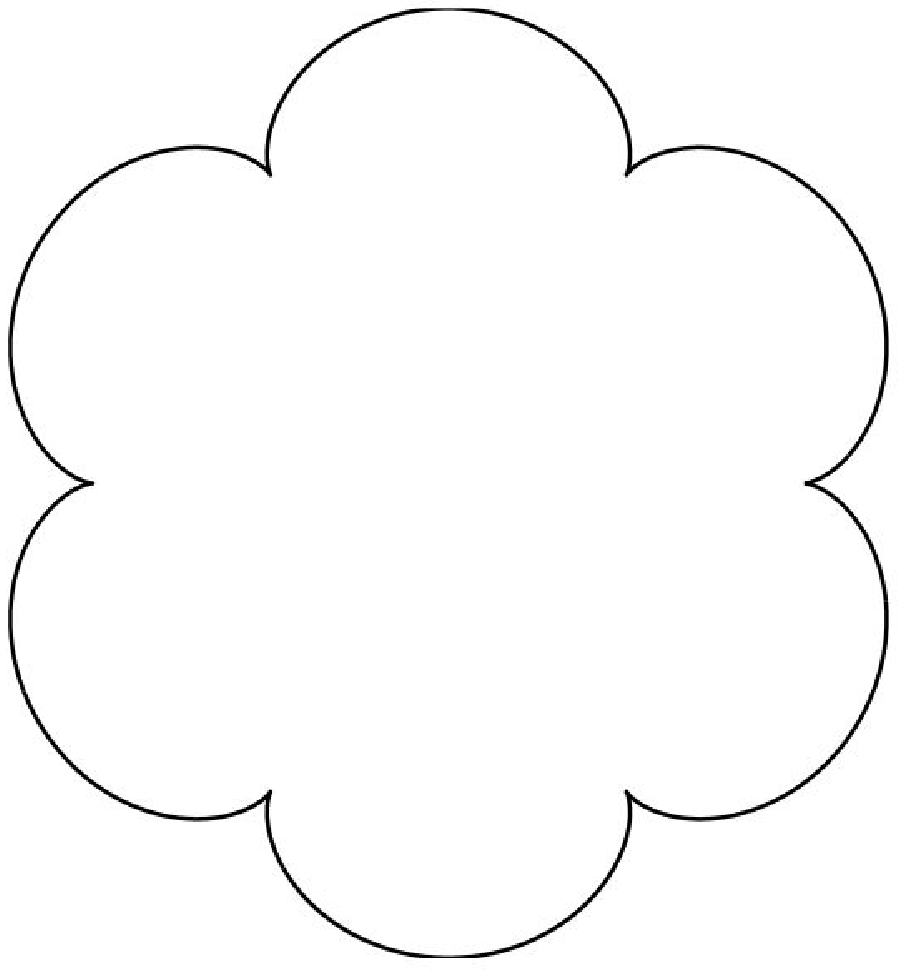}% 
&
\includegraphics[width=0.47\textwidth]{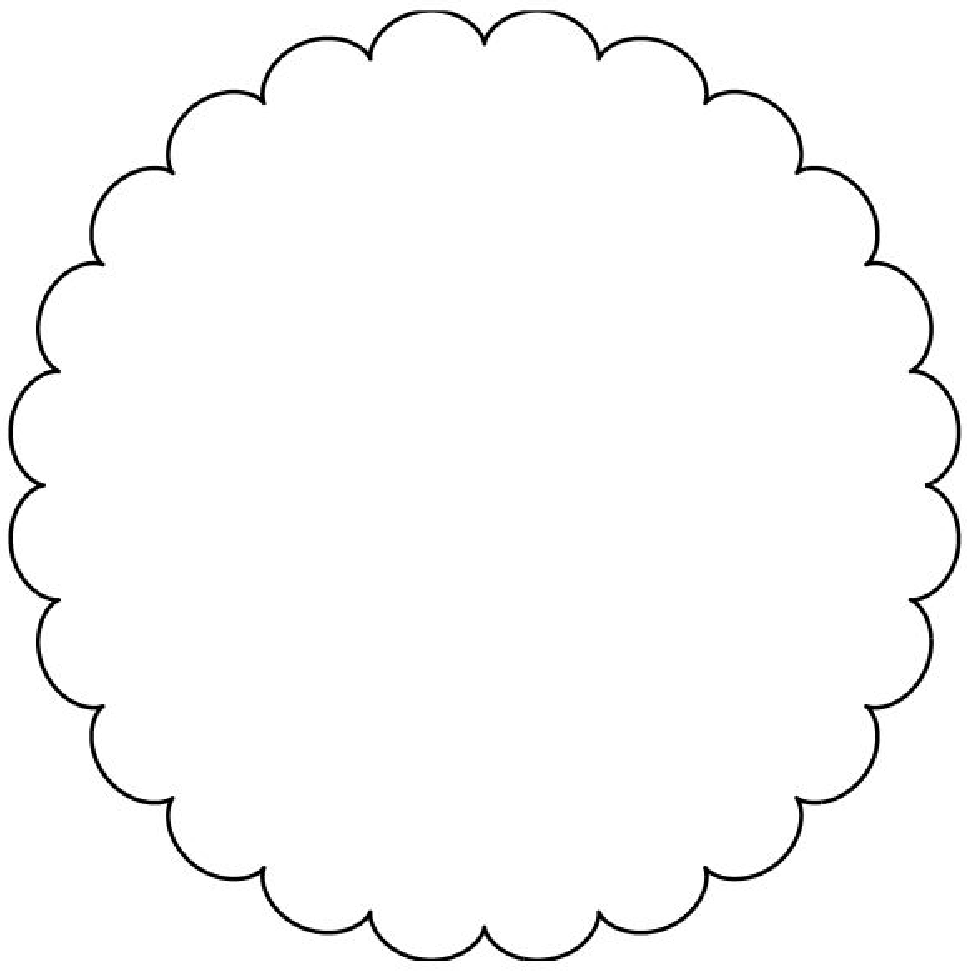} \\[0.1ex]
$n=7$ & $n=25$
\end{tabular}
\end{minipage}
\caption{Disks with hats from \S\,\ref{Ss:sefsim} (left) \quad Epitrochoids from \S\,\ref{Ss:Epitro} (right)}
\label{F:DiskEpi}
\end{figure}

%-----------------------------------------------

%=====================
\subsection{Regular polygons tending to a disk}
\label{Ss:Regpoltodisc}
%=====================
Let $\Omega$ be the unit disk in $\R^2$ and let $\Omega_n$  be a regular convex polygon with $n$ edges which is inscribed in $\Omega$. The Horgan--Payne angle (cf. \cite{Stoyan2001}) of $\Omega_n$ with respect to the center of the ball is  $\frac{\pi}{2}- \frac{\pi}{n}$. Thus, it follows from \cite{HorganPayne1983} and \cite[Theorem 5.1]{CoDa2015} that 
$$
   \sin (\frac{\pi}{4}- \frac{\pi}{2n}) \le 
   \beta(\Omega_n) \le \frac{1}{\sqrt{2}} = \sin \frac{\pi}{4}.
$$
Since the inf-sup constant of the disk $\Omega$ is $\frac{1}{\sqrt{2}}$, we find
\[
   0 \le \beta(\Omega) - \beta(\Omega_n) \le 
   \sin \frac{\pi}{4} - \sin \big(\frac{\pi}{4} - \frac{\pi}{2n}\big) \le 
   \frac{\pi}{2n} \,.
\]
So we have (at least) a convergence of order $1$. This example pertains typically to the framework of polygonal approximation of a regular domain. We prove generally the convergence of the inf-sup constant in this case, see \cref{T:Omegah}  further  on.

%======================
\subsection{Conformal mappings and epitrochoids}
\label{Ss:Epitro}
%======================
The situation where plane domains are transformed by conformal mappings has been investigated by Zsupp\'an \cite{Zsuppan2011}: Corollary 3.9 {\em loc.cit.} can be stated as follows

\begin{theorem}[Zsupp\'an]
\label{T:Zs}
Let $\Omega$ and $\widetilde\Omega$ be two simply connected plane domains with piecewise smooth   boundaries.  Let $g$ denote the bijective conformal mapping of $\Omega$ onto $\widetilde\Omega$. If $|g'-1| \le \varepsilon<1$ in the closure of $\Omega$, then we have
\[
   \big| \beta(\Omega) - \beta(\widetilde\Omega) \big| \le \frac{\sqrt{2}\, \varepsilon}{1-\varepsilon}\,.
\]
\end{theorem}

We will generalize this statement to Lipschitz diffeomorphisms in \cref{T:encadrement} below. Still in \cite{Zsuppan2011}, an explicit example of a family of conformal mappings is investigated. The domain $\Omega$ is chosen as the unit disk and the mapping $g$ depends on two parameters: an integer $n\ge2$ and a real number $c>0$:
$$
 g_{n;c}(z) = 
    z-\frac {c}{n} z^{n}, \quad z\in\C.
$$
For any $c\le1$, the transformation $g_{n;c}$ is bijective on the unit disk $\Omega$.
Then $\Omega_{n}$ is defined as the image of $\Omega$ under $g_{n;c}$. This is an epitrochoid, see examples with $c=1$ in \cref{F:DiskEpi}. From \cite{Zsuppan2011} 
\[
 \beta(\Omega_n)^{2} =
 \begin{cases}
   \frac12-\frac{c}{4}(1+\frac1n)\,, & \mbox{ for $n$ odd,} \\
   \frac12-\frac{c}{4}\sqrt{1+\frac2n}\,, & \mbox{ for $n$ even.}
\end{cases} 
\]
As $n\to\infty$, $\Omega_{n}$ tends to $\Omega$, because $g_{n;c}(z)$ tends to $z$ uniformly in $z$ for $\vert z\vert \le 1$. Nevertheless, the assumption of \cref{T:Zs} is not satisfied, and 
\[
   \beta(\Omega_n) \to \beta_{\infty}=\sqrt{\tfrac12-\tfrac c4} < \sqrt{\tfrac12} = \beta(\Omega)\,.
\]
In contrast, if we make $c$ depend on $n$ via the law $c=n^{-\alpha}$ for some positive $\alpha$, these $\Omega_n$ enter  the framework of \cref{T:Zs} and the convergence of $\beta(\Omega_{n})$ to $\beta(\Omega)$ occurs.

%%%%%%%%%%%%%
%% Section %%
%%%%%%%%%%%%%
\section{Continuity with respect to the domain}\label{S:ContDom}

%%%%%%%%%%%%
\subsection{The continuity theorem}\label{Ss:Cont}
%%%%%%%%%%%%

We recall that $d$ is the space dimension. We denote by $\I_d$ the identity matrix in $\R^d$ and by $\E_d$ the space of endomorphisms of $\R^d$ equipped with the norm subordinated to the Euclidean norm on $\R^d$. In this section $c(d)$ refers to various constants {\em depending only on} $d$ (and not on the domains or functions under consideration).

\begin{definition}
\label{D:perturbation}
Let $\Omega$ and $\widetilde\Omega$ be two bounded Lipschitz domains in $\R^d$, and let $\varepsilon$ be a positive number. We say that $\Omega$ and $\widetilde\Omega$ are {\em $\varepsilon$-close in Lipschitz norm} if there exists a diffeomorphism $\F=(\F_1,\ldots,\F_d)$ from $\widetilde\Omega$ onto $\Omega$ that satisfies
\begin{equation}
\label{E:perturbation}
\begin{cases}
   \F\in W^{1,\infty}(\widetilde\Omega)^d \quad\mbox{and}\quad
   \DNorm{ D\F -\I_d}{L^\infty(\widetilde\Omega; \E_d)} \le\varepsilon\,,\\
   \F^{-1}\in W^{1,\infty}(\Omega)^d \quad\mbox{and}\quad
   \DNorm{ D\F^{-1} -\I_d}{L^\infty(\Omega; \E_d)} \le \varepsilon\,.
\end{cases}
\end{equation}
\end{definition}

\begin{remark}
\label{R:perturbation}
Condition \cref{E:perturbation} is slightly redundant in the following sense: By an expansion of $D\F^{-1}$ as a Neumann series of $D\F-\I_d$  we find that if the first line of \cref{E:perturbation} is satisfied with some value $\varepsilon_0<1$ of $\varepsilon$, then the second line is satisfied for $\varepsilon =\varepsilon_0/(1-\varepsilon_0)$.
\end{remark} 

If condition \cref{E:perturbation} holds, the Jacobian determinant $\J_{\F}$ of $\F$ satisfies the estimate
\begin{equation}
\label{E:Jac}
   \DNorm{ 1- \J_{\F}}{L^\infty(\widetilde\Omega)} \le c(d)\varepsilon.
\end{equation}
Thus, the change of variables $\F$ in integrals and partial derivatives yields:

\begin{lemma} 
\label{L:Pio10}
We assume that the diffeomorphism $\F$ satisfies property \cref{E:perturbation}. 
Then there holds  
\\
(i) For $q\in L^2(\Omega)$, let $\tilde q$ denote $q\circ\F$. Then $\tilde q$ belongs to $L^2(\widetilde\Omega)$ and 
\begin{equation}
\label{E:Fq}
   \frac{1}{1+c(d)\,\varepsilon}\,  \DNorm{q }{0,\Omega} \le 
    \DNorm{\tilde q }{0,\widetilde\Omega}  \le 
   \big(1+c(d)\,\varepsilon\big)\DNorm{q }{0,\Omega}. 
\end{equation}
(ii) For $v\in H^1_0(\Omega)$, let $\tilde v$ denote $v\circ\F$. Then $\tilde v$ belongs to $H^1_0(\widetilde\Omega)$ and satisfies the estimates
\begin{gather}
\label{E:Fvj}
   \forall j=1,\ldots,d,\qquad
    \DNorm{\partial_{\tilde x_j}\tilde v - (\partial_{x_j}v) \circ\F }{0,\widetilde\Omega} 
   \le c(d)\,\varepsilon\, \Norm{v }{1,\Omega}\,,  \\
\label{E:Fv}
   \frac{1}{1+c(d)\varepsilon} \Norm{v }{1,\Omega}  \le 
    \Norm{\tilde v }{1,\widetilde\Omega}  \le 
   (1+c(d)\varepsilon) \Norm{v }{1,\Omega} \,.
\end{gather}
\end{lemma}

Thanks to \cref{L:Pio10}, we are in a position to compare $\beta(\widetilde\Omega)$ with $\beta(\Omega)$. Recall that both $\beta(\widetilde\Omega)$ and $\beta(\Omega)$ are positive since $\widetilde\Omega$ and $\Omega$ are bounded and Lipschitz.

\begin{theorem}
\label{T:encadrement}
There exists a constant $c(d)$ depending only on the dimension $d$ such that, if $\Omega$ and $\widetilde\Omega$ are $\varepsilon$-close in Lipschitz norm, 
the following estimate holds
\begin{equation}
\label{E:encadr}
   \big|\beta(\Omega) - \beta(\widetilde\Omega) \big|
   \le c(d)\,\varepsilon.
\end{equation}
\end{theorem}

\begin{proof} 
Since the property \cref{E:perturbation} defining the $\varepsilon$-closeness is symmetric, it suffices to prove one inequality
\begin{equation}
\label{E:encadr2}
   \beta(\widetilde\Omega) - c(d)\varepsilon \le \beta(\Omega)
\end{equation}
to prove estimate \cref{E:encadr}.

Thus, let $q$ be any function in $L^2_\circ(\Omega)$. 
Setting $\tilde q = q \circ \F$, since the mean value of $\tilde q$ is not necessarily zero, we  consider  instead
\[
   \tilde q_0 = \tilde q\, \J_{\F}
\] 
that belongs to $L^2_\circ(\widetilde\Omega)$.
By \cref{E:Jac}, it satisfies
\begin{equation}
\label{E:Fq0}
   \DNorm{\tilde q - \tilde q_0 }{0,\widetilde\Omega}  \le 
   c(d)\,\varepsilon\,  \DNorm{\tilde q }{0,\widetilde\Omega} \,.
\end{equation}
Combining this with estimates \cref{E:Fq}, we find
\begin{equation}
\label{E:Fq0b}
   \frac{1}{1+c(d)\,\varepsilon}\,  \DNorm{q }{0,\Omega}   \le 
    \DNorm{\tilde q_0 }{0,\widetilde\Omega}  \le 
   \big(1+c(d)\,\varepsilon\big)  \DNorm{q }{0,\Omega}.
\end{equation}
Now it is well known (see for example \cite[Chap.~I, Lemma 4.1]{GiraultRaviart}) that the inf-sup condition on $\widetilde\Omega$ yields the existence of a vector function $\tilde\bv$ in $H^1_0(\widetilde\Omega)^d$  such that
\begin{equation}
\label{E:infsupwO}
   \div\tilde \bv=\tilde q_0 \quad\mbox{and}\quad
   \beta(\widetilde\Omega)  \Norm{\tilde \bv }{1,\widetilde\Omega}   \le
    \DNorm{\tilde q_0 }{0,\widetilde\Omega}. 
\end{equation}
We now set:
$$\bv = \tilde \bv \circ \F^{-1}.$$
The idea is to bound the quantity $\int_{\Omega} (\div\bv)(\x)q(\x)\,d\x$  from below  as follows.  Using \cref{E:Fvj} we immediately obtain
\begin{equation}
\label{E:divtv}
    \DNorm{\div\tilde\bv - (\div\bv) \circ\F }{0,\widetilde\Omega} 
   \le c(d)\,\varepsilon\,  \Norm{\bv }{1,\Omega}\,, 
\end{equation}
which allows to write
\begin{equation}
\label{E:I1I2I3}
   \int_{\Omega} (\div\bv)(\x)q(\x)\,\d\x =
   \int_{\widetilde\Omega} (\div\tilde\bv)(\tilde\x)\tilde q_0(\tilde\x)\,\d\tilde\x +
   \E
\end{equation}
where the error term $\E$ is given by
\[
   \E = \int_{\widetilde\Omega} \big((\div\bv)\circ\F - \div\tilde\bv\big)(\tilde\x)
   \  \tilde q_0(\tilde\x)\,\d\tilde\x.
\]
We bound this term using \cref{E:divtv} first, 
\[
   |\E| \le c(d)\,\varepsilon\,  \Norm{\bv }{1,\Omega}  
    \DNorm{\tilde q_0 }{0,\widetilde\Omega} 
\]
and next \cref{E:Fq0b}
\begin{equation}
\label{E:I1I2I3b}
   |\E| \le  
   c(d)\,\varepsilon\,  \Norm{\bv }{1,\Omega} 
    \DNorm{q }{0,\Omega} . 
\end{equation}
By substituting \cref{E:I1I2I3b} and the choice \cref{E:infsupwO} of $\tilde\bv$ into \cref{E:I1I2I3}, we obtain successively
\begin{align*}
   \int_{\Omega} (\div\bv)(\x)q(\x)\,\d\x &\ge
   \int_{\widetilde\Omega} \tilde q_0(\tilde\x)^2 \,\d\tilde\x -
   c(d)\,\varepsilon\,  \Norm{\bv }{1,\Omega} 
    \DNorm{q }{0,\Omega} \\
&\ge
   \beta(\widetilde\Omega)  \Norm{\tilde \bv }{1,\widetilde\Omega}  
    \DNorm{\tilde q_0 }{0,\widetilde\Omega}   -
   c(d)\,\varepsilon\,  \Norm{\bv }{1,\Omega} 
    \DNorm{q }{0,\Omega}  \\[1ex]
&\ge
   \big(\beta(\widetilde\Omega) -
   c(d)\,\varepsilon\big)  \Norm{\bv }{1,\Omega} 
    \DNorm{q }{0,\Omega} , 
\end{align*}
where we have used \cref{E:Fv} and \cref{E:Fq0b} for the last line (recall that $c(d)$ denotes a generic constant that may change between lines). This proves that $\beta(\Omega)$ is larger than $\beta(\widetilde\Omega) - c(d)\,\varepsilon$, which is our aim.
\end{proof}

%%%%%%%%%%%%
\subsection{Polygonal approximation of plane curved domains}\label{Ss:Polyg}
%%%%%%%%%%%%

An important application of \cref{T:encadrement} is the finite element approximation of curved domains.
Let $\Omega$ be a two-dimensional curved polygon with a Lipschitz-continuous and piecewise
$\CC^2$ boundary. This means that the boundary of $\Omega$ is a finite union of $\CC^2$-arcs $\Gamma_j$ that touch at corners $\bc_k$ and determine opening angles distinct from $0$ and $2\pi$, thus excluding outward and inward cusps.

\begin{definition}
\label{D:happrox}
Let $\Omega$ be a two-dimensional curved Lipschitz polygon with a piecewise $\CC^2$-boundary. 
Let $h$ be a positive number. 
A {\em polygonal $h$-approximation} of $\Omega$ denotes a polygon $\Omega_h$ with straight sides such that
\begin{itemize}
\item[(i)] Its  set of  corners  contains   the set $\{\bc_k\}$ of corners of $\Omega$,
\item[(ii)] Its corners belong to the boundary $\partial\Omega$ of $\Omega$,
\item[(iii)] The length of each side is less than $h$.
\end{itemize}
\end{definition}

We now state the main result of this section.

\begin{theorem}
\label{T:Omegah}
Let $\Omega$ be a two-dimensional curved Lipschitz polygon with a piecewise $\CC^2$-boundary. There exists a constant $c(\Omega)$ such that for all $h$-approximation $\Omega_h$ of $\Omega$
\begin{equation}
\label{E:encadrh}
   \big|\beta(\Omega) - \beta(\Omega_h) \big|
   \le c(\Omega)\,h.
\end{equation}
\end{theorem}

\begin{proof}
In view of \cref{T:encadrement}, it would suffice to prove that $\Omega$ and $\Omega_h$ are $\varepsilon$-close for an $\varepsilon=c(\Omega)h$. In fact, we are going to construct a finite number of intermediate domains $\Omega^k_h$, $k=1,\ldots,K-1$,  such that
\begin{itemize}
\item The number of these domains is independent of $h$,
\item Setting $\Omega^0_h=\Omega$ and $\Omega^K_h=\Omega_h$, each pair of consecutive domains $(\Omega^{k-1}_h,\Omega^k_h)$ are $\varepsilon$-close for an $\varepsilon=c(\Omega)h$, for $k=1,\ldots,K$.
\end{itemize}
Taking \cref{D:perturbation} and \cref{R:perturbation} into account, this amounts to prove for each $k$ that there exists a Lipschitz mapping $\F^k_h$ in $W^{1,\infty}(\Omega^{k-1}_h)^2$ which is a diffeomorphism from $\Omega^{k-1}_h$ onto $\Omega^k_h$ and satisfies
\begin{equation}
\label{E:perturbationh}
   \F^k_h =  \Id + \G^k_h, \quad \hbox{with} \quad 
   \DNorm{ D\G^k_h }{L^\infty(\Omega^{k-1}_h; \E_2)} \le C(\Omega)h\,,
\end{equation}
where the constant $C(\Omega)$ depends only on $\Omega$. Here $\Id$ denotes the identity mapping $\bx\mapsto\bx$ in $\R^2$. To construct this mapping, we proceed in several steps.

\medskip\noindent
{\sc Step 1: Partition of the domain.}
After the possible adjunction of extra points inside the original sides of $\Omega$, that we will still denote by $\bc_k$, we may assume that each new smaller side $\Gamma_k$ is the graph of a $\CC^2$ function in some coordinate system. 
Thus we can cover the boundary of the domain $\Omega$  by the closure of  open sets $\U^1, \ldots,  \U^K$ so that for each $k$, $1 \le k \le K$, after a possible rigid motion $\M_k$,
\begin{itemize}
\item the set $\overline\U{}^k$ is a rectangle $[a_k,b_k]\times[0,r_k]$,
\item the local parts of the boundaries $\partial\Omega\cap\U^k$ and $\partial\Omega_h\cap\U^k$ are the graphs of a $\CC^2$ map $\varphi^k$ and a Lipschitz map $\varphi^k_h$, respectively, defined on $[a_k,b_k]$ with values in $[0,r_k]$
\item $(a_k,\varphi(a_k))=(a_k,\varphi^k_h(a_k))=\bc_{k-1}$ and $(b_k,\varphi(b_k))=(b_k,\varphi^k_h(b_k))=\bc_{k}$.
\end{itemize}   
%
%-----------------------------------------------
\begin{figure}[ht]
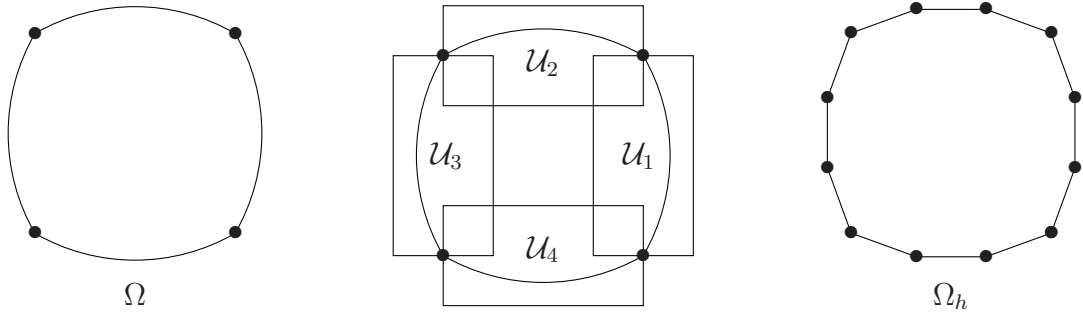

\begin{center}
\begin{minipage}{0.3\textwidth}
% 1. Definition of characteristic points
    \figinit{1.33mm}
    \figpt 1:( 0,  0)
    \figpt 2:( -7.320508075688775,0)
    \figptcirc 3 :: 2;20 (-30)
    \figptcirc 4 :: 2;20 (-10)
    \figptcirc 5 :: 2;20 ( 10)
    \figptcirc 6 :: 2;20 ( 30)
    \figptsrot 7 = 1, 2, 3, 4, 5, 6 /1, 90/
    \figptsrot 13 = 1, 2, 3, 4, 5, 6 /1, 180/
    \figptsrot 19 = 1, 2, 3, 4, 5, 6 /1, 270/
% 2. Creation of the postscript file
    \def\MyPSfile{}
    \psbeginfig{\MyPSfile}
    \psarccircP 2 ; 20 [3,6]
    \psarccircP 8 ; 20 [9,12]
    \psarccircP 14 ; 20 [15,18]
    \psarccircP 20 ; 20 [21,24]
  \psendfig
% 3. Writing text on the figure
    \figvisu{\figbox}{}{%
    \figinsert{\MyPSfile}
    \figwrites 1 :{$\Omega$}(15)
    \figwritec[3,9,15,21]{$\bullet$}
    }
\centerline{\box\figbox}
\end{minipage}
\hfill
\begin{minipage}{0.3\textwidth}
% 1. Definition of characteristic points
    \figinit{1.33mm}
    \figpt 1:( 0,  0)
    \figpt 2:( -7.320508075688775,0)
    \figptcirc 3 :: 2;20 (-30)
    \figptcirc 4 :: 2;20 ( 30)
    \figvectC 30 (5,0)
    \figptstra 5 = 3, 4 /1, 30/
    \figptstra 7 = 3, 4 /-1, 30/
    \figptsrot 9 = 1, 2, 3, 4, 5, 6, 7, 8 /1, 90/
    \figptsrot 17 = 1, 2, 3, 4, 5, 6, 7, 8 /1, 180/
    \figptsrot 25 = 1, 2, 3, 4, 5, 6, 7, 8 /1, 270/
% 2. Creation of the postscript file
    \def\MyPSfile{}
    \psbeginfig{\MyPSfile}
    \psarccircP 2 ; 20 [3,4]
    \psarccircP 10 ; 20 [11,12]
    \psarccircP 18 ; 20 [19,20]
    \psarccircP 26 ; 20 [27,28]
    \psline [5,6,8,7,5]
    \psline [13,14,16,15,13]
    \psline [21,22,24,23,21]
    \psline [29,30,32,31,29]
  \psendfig
% 3. Writing text on the figure
    \figvisu{\figbox}{}{%
    \figinsert{\MyPSfile}
    \figwrites 1 :{$\U_4$}(8)
    \figwritee 1 :{$\U_1$}(8)
    \figwriten 1 :{$\U_2$}(8)
    \figwritew 1 :{$\U_3$}(8)
    \figwritec[3,11,19,27]{$\bullet$}
    }
\centerline{\box\figbox}
\end{minipage}
\hfill
\begin{minipage}{0.3\textwidth}
% 1. Definition of characteristic points
    \figinit{1.33mm}
    \figpt 1:( 0,  0)
    \figpt 2:( -7.320508075688775,0)
    \figptcirc 3 :: 2;20 (-30)
    \figptcirc 4 :: 2;20 (-10)
    \figptcirc 5 :: 2;20 ( 10)
    \figptcirc 6 :: 2;20 ( 30)
    \figptsrot 7 = 1, 2, 3, 4, 5, 6 /1, 90/
    \figptsrot 13 = 1, 2, 3, 4, 5, 6 /1, 180/
    \figptsrot 19 = 1, 2, 3, 4, 5, 6 /1, 270/
% 2. Creation of the postscript file
    \def\MyPSfile{}
    \psbeginfig{\MyPSfile}
    \psline [3,4,5,9,10,11,15,16,17,21,22,23,24]
  \psendfig
% 3. Writing text on the figure
    \figvisu{\figbox}{}{%
    \figinsert{\MyPSfile}
    \figwrites 1 :{$\Omega_h$}(15)
    \figwritec[3,4,5,9,10,11,15,16,17,21,22,23,24]{$\bullet$}
    }
\centerline{\box\figbox}
\end{minipage}
\end{center}
\caption{A curved square, its map neighborhoods, and its polygonal approximation}
\label{F:CurvedSquare2}
\end{figure}
%-----------------------------------------------
% 
Then we introduce intermediate domains $\Omega^0_h=\Omega, \Omega^1_h, \ldots, \Omega^{K }_h=\Omega_h$ so that
$$ \Omega^{k-1}_h\cap\U^k = \Omega\cap\U^{k}, \qquad
   \Omega^{k}_h\cap\U^k = \Omega_h\cap\U^k, \qquad
   \Omega^{k-1}_h\cap\complement\,\U^k = \Omega^{k}_h\cap\complement\,\U^k.
$$
Here  $\complement\,\U$ stands for $\R^d\setminus \overline \U$.
It follows that $\Omega^{k}_h$ is a local polygonal $h$-approximation of $\Omega^{k-1}_h$ subordinate to the neighborhood $\U^k$ (in the sense that $\Omega^{k-1}_h$ and $\Omega^k_h$ coincide outside $\U^k$ and that the properties (i) to (iii) of \cref{D:happrox} hold for the parts $\partial\Omega\cap\U^k$). Indeed, we intend to construct a diffeomorphism
$\F_h^k$  from $\Omega^{k-1}_h$ onto  $\Omega^k_h$ which is equal to the identity outside of $\U^k$.

\medskip\noindent
{\sc Step 2: Construction of the diffeomorphism at step $k$.}
From now on, we drop the exponent $k$ and restrict the discussion to a local polygonal $h$-approximation in a rectangle $\U=(a,b)\times(0,r)$  such that 
\begin{align*}
   \overline\Omega\cap\U &= 
   \big\{(x_1,x_2)\in\R^2,\quad a<x_1<b,\ \ 0<x_2\le \varphi(x_1)\big\}\,, \\
   \overline\Omega_h\cap\U 
   &= \big\{(x_1,x_2)\in\R^2,\quad a<x_1<b,\ \ 0<x_2\le \varphi_h(x_1) \big\} \,.
\end{align*}
Without restriction, we assume that $\varphi$ is bounded  from below  by a constant $\eta_0>0$
\begin{equation}
\label{E:eta0}
   \exists\eta_0>0,\quad\forall x_1\in[a,b],\quad \varphi(x_1)\ge\eta_0.\end{equation}
Since, in this situation, $\Omega$ and $\Omega_h$ coincide outside $\U$, then $\varphi(a)=\varphi_h(a)$ and $\varphi(b)=\varphi_h(b)$. We define the diffeomorphism $\F_h$ by
\[
  \F_h\x =
\begin{cases}
  \x & \mbox{if}\quad \x\in\Omega\cap\complement\,\U\\
	(x_1, \frac{\varphi_h(x_1)}{\varphi(x_1)}\, x_2) & \mbox{if}\quad \x\in\Omega\cap\U\,.
\end{cases}
\]
So the mapping $\F_h$  has a continuous extension to  $\overline\Omega$ since it is the identity on $\partial\U\cap\Omega$ (i.e., on the three segments $\{ x_1\in[a,b],\ x_2=0\}$, $\{x_1=a,\ x_2\in[0,\varphi(a)]\}$, and $\{x_1=b,\ x_2\in[0,\varphi(b)]\}$). We write
\[
  \F_h = \Id + \G_h \quad\mbox{with}\quad \G_h\x =
\begin{cases}
  0 & \mbox{if}\quad \x\in\Omega\cap\complement\,\U\\
	(0, \frac{\varphi_h(x_1)-\varphi(x_1)}{\varphi(x_1)}\, x_2) 
	& \mbox{if}\quad \x\in\Omega\cap\U\,.
\end{cases}
\]
Therefore, we have the bound for $D\G_h$, with a constant $c(\Omega)$ depending only on $\Omega$:
$$\DNorm{ D\G_h }{L^\infty(\Omega; \E_2)} \le c(\Omega) \Big\Vert 
   \frac{\varphi_h-\varphi}{\varphi} \Big\Vert_{W^{1,\infty}(a,b)} .$$
Since $\varphi'$ is bounded from above and $\varphi$ satisfies \cref{E:eta0},  we obtain, for another constant $c(\Omega)$,
$$  \DNorm{ D\G_h }{L^\infty(\Omega; \E_2)} \le c(\Omega)
 \DNorm{ \varphi_h-\varphi }{W^{1,\infty}(a,b)} .$$
By definition of the polygonal $h$-approximation, the interval $[a,b]$ is the union of smaller intervals $[a',b']$ of length less than $h$ and such that
$$ \varphi_h(a')-\varphi(a') = 0, \qquad
 \varphi_h(b')-\varphi(b') = 0, \qquad\hbox{and}\qquad
 \forall x_1\in[a',b'],\; \varphi''_h(x_1) = 0.$$
  Thus, $\varphi_h$ is a piecewise affine interpolate of $\varphi$ (see \cref{D:happrox}).
Using that $\varphi$ is of class $\CC^2$, one obtains immediately
$$\DNorm{ \varphi_h-\varphi }{W^{1,\infty}(a,b)} \le \DNorm{\varphi''}{L^{\infty}(a,b)}\,(h+h^2).$$
Then, reintroducing the exponent $k$, we find, as a direct consequence of \cref{T:encadrement}:
\begin{equation}
\label{E:betak}
  \big|\beta(\Omega_h^{k-1}) - \beta(\Omega_h^k) \big|
   \le c(\Omega)\,h, \quad k=1,\ldots, K.
\end{equation}

\noindent{\sc Step 3: Conclusion}.
Finally, the bound for $\big|\beta(\Omega) - \beta(\Omega_h) \big|$
follows from the triangle inequality
$$\big|\beta(\Omega) - \beta(\Omega_h) \big| \le \sum_{k = 1}^K \big|\beta(\Omega_h^{k-1}) - \beta(\Omega_h^k) \big|$$
and estimates \cref{E:betak}, since the number $K$ does not depend on $h$. 
\end{proof}

\begin{remark}
Along the same lines as for the two-dimensional case, one can prove convergence of the inf-sup constant for polyhedral approximations of  certain classes of three-dimensional domains.

(1) Suppose  that $\Omega\subset\R^3$ has a $\CC^2$ boundary $\partial\Omega$. Then one may consider a family of polyhedral approximations $\Omega_h$ defined by regular triangulations of $\partial\Omega$. This means that the boundary $\partial\Omega_h$ of $\Omega_h$ is defined by planar triangles with vertices on $\partial\Omega$ whose diameter is bounded from above by  $Ch$ and whose inner radius is bounded from below  by $ch$, where $c$ and $C$ are constants independent of $h$.  In this situation, we can prove  an error estimate of order $h$ for the approximation of $\beta(\Omega)$ by $\beta(\Omega_h)$ as in \cref{E:encadrh}.  The proof has two steps: (i) construction of a global diffeomorphism $\F_h$, (ii) error estimates. For (i), we use the regularity of the inner unit normal field $\bn$ on the $\CC^2$ surface $\sS:=\partial\Omega$ and note that for a sufficiently small positive $\varepsilon_0$, the mapping
\[
\begin{array}{ccc}
   \Psi: \sS\times[0,2\varepsilon_0] &\longrightarrow &\overline\U \\
   (\by,\rho) & \longmapsto & \bx=\by + (\varepsilon_0-\rho)\bn(\by)
\end{array}
\]
is a diffeomorphism onto a tubular neighborhood $\overline\U$ of $\sS$ in $\R^3$, of width $2 \varepsilon_0$, that sends $\sS\times\{\varepsilon_0\}$ onto $\partial\Omega$. Then for $h$ small enough, the boundary of $\Omega_h$ is represented by the graph $\rho=\varphi_h(\by)$ of a function $\varphi_h:\sS\to[\frac12\varepsilon_0,\frac32\varepsilon_0]$:
\[
    \sS \ni \by \longmapsto \bx = \by + (\varepsilon_0-\varphi_h(\by))\bn(\by)\in\partial\Omega_h.
\]
The diffeomorphism $\F_h$ is then defined as
\[
  \F_h\x =
\begin{cases}
  \x & \mbox{if}\quad \x\in\Omega\cap\complement\,\U\\
	\by + (\varepsilon_0 - \frac{\rho}{\varepsilon_0}\varphi_h(\by))\, \bn(\by) & \mbox{if}\quad \x\in\Omega\cap\U
	\quad\mbox{with}\quad (\by,\rho)=\Psi^{-1}(\bx)\,.
\end{cases}
\]
(ii) Estimates for $\F_h-\Id$ in the Lipschitz norm then follow from standard $W^{1,\infty}$ estimates for interpolation with two-dimensional finite elements (see \cite[Theorem 16.1]{Ciarlet1991}, for example). 

(2) We can extend the proof above in the spirit of \cref{T:Omegah} if $\Omega$ is a piecewise $\CC^2$ domain with {\em straight edges} that form the boundaries of its faces $\mathsf{f}_j$ which are polygonal subdomains of $\CC^2$ surfaces $\sS_j$. Just as we imposed in 2D that corners of $\Omega$ be corners of $\Omega_h$, we impose now that the edges of $\Omega$ be edges of $\Omega_h$. 
\end{remark}

%%%%%%%%%%%%%
%% Section %%
%%%%%%%%%%%%%
\section{Continuity with respect to the function spaces}\label{S:ContFun}

 Whereas the uniform discrete LBB condition~\cref{E:discreteLBB} is regularly addressed in the literature about numerical approximation of the solutions of incompressible fluid models, the convergence of the discrete inf-sup constants  to the exact inf-sup constant is rarely (if ever) examined.
 In this section, we intend to fill this gap and  give conditions on the sequence of finite-dimensional function spaces 
$X_{n}\subset H^1_0(\Omega)^d$, $M_{n}\subset L^2_\circ(\Omega)$, $n\in\N$ that guarantee that the discrete inf-sup constants $\beta_{n}$ converge to the inf-sup constant $\beta(\Omega)$ of the domain.
The convergence proof uses arguments in the spirit of the proof of the discrete LBB condition in the paper \cite{Verfuerth1984}, namely a combination of inverse estimates for the pressure spaces and approximation properties of the velocity spaces.

%%%%%%%%%%%%
\subsection{The continuity theorem}\label{Ss:ContFun}
%%%%%%%%%%%%

For a regularity index $s>0$, we introduce two characteristic constants associated with the function spaces $X_{n}$ and $M_{n}$.
The first constant is defined by comparing the Sobolev norm of order $s$ on $M_{n}$ with the equivalent $L^2$ norm \ (recall that $M_n$ is finite-dimensional)\begin{equation}
\label{E:etaN}
 \eta_{n,s} = \sup_{q\in M_{n}} \frac{\DNorm{q}{s,\Omega}}{\DNorm{q}{0,\Omega}}.
\end{equation} 
The second constant is defined by the approximation property of the spaces $X_{n}$
\begin{equation}
\label{E:epsN}
 \varepsilon_{n,s}= \sup_{\bu\ee\in\ee H^{1+s}(\Omega)\cap H^1_0(\Omega)}\ \ 
                 \inff_{\bv\ee\in\ee X_{n}}
                 \frac{\Norm{\bu-\bv}{1,\Omega}}{\DNorm{\bu}{1+s,\Omega}}\,.
\end{equation}

\begin{theorem}
\label{T:bNconv}
Let $\Omega\subset\R^{d}$ be a bounded Lipschitz domain. We assume that $X_{n}\subset H^1_0(\Omega)^d$, $M_{n}\subset L^2_\circ(\Omega)$, $n\in\N$, are finite-dimensional subspaces and that for some $0<s<\frac12$, $M_{n}$ is contained in the Sobolev space $H^{s}(\Omega)$. 
Then there is a constant $C_{s}$ depending on the domain $\Omega$ such that 
with the constants $\eta_{n,s}$ and $\varepsilon_{n,s}$ defined above,
we have 
\begin{equation}
\label{E:bNb}
  \beta_{n} \ge \beta(\Omega) - C_{s}\,\eta_{n,s}\,\varepsilon_{n,s}\,.
\end{equation}
In particular, if the $M_{n}$ are asymptotically dense in $L^2_\circ(\Omega)$ and 
\begin{equation}
\label{E:suff}
 \mbox{ if }\quad 
  \lim_{n\to\infty}\eta_{n,s}\,\varepsilon_{n,s}=0 
  \quad\mbox{ then }\quad
 \lim_{n\to\infty}\beta_{n} = \beta(\Omega)\,.
\end{equation}
\end{theorem}
\begin{proof}
Define the functionals $J$ and $J_{n}$ as in \cref{E:J}. Recall that for $J$ we have the relation
\[
 J(q)=\Norm{\bw(q)}{1,\Omega}
\]
where $\bw(q)=\Delta^{-1}\grad q$ is the solution of the harmonic Dirichlet problem with $\grad q$ as right-hand  side. Likewise, the sup in $J_{n}$ is attained at $\bw_{n}(q)\in X_{n}$, which is the Galerkin solution 
$\bw_{n}(q)=\Delta_{n}^{-1}\grad q$
of the Dirichlet problem with the same right-hand side:
\begin{equation}
\label{E:wN(q)}
 \forall \,\bv\in X_{n},\quad
 \big\langle\grad\bw_{n}(q), \grad\bv\big\rangle_\Omega 
  = \big\langle \div\bv,q \big\rangle_\Omega \,.
\end{equation}
Taking $\bv=\bw_{n}(q)$ it follows that
\[
   J_{n}(q)=\Norm{\bw_{n}(q)}{1,\Omega}\,.
\]
Now for each $n$, we solve in $M_{n}$ the finite-dimensional eigenvalue problem for the Schur complement of the discretized Stokes system
\begin{equation}
\label{E:SchurEV}
  \SS_{n}q = \sigma q \quad \mbox{ with } \quad
    \big\langle \SS_{n}q,p \big \rangle_\Omega  = -\big\langle \bw_{n}(q),\grad p\big\rangle
    \;\;\forall\,p\in M_{n}\,.
\end{equation}
Let $\sigma_{n}$ be the smallest eigenvalue and $q_{n}$ a corresponding normalized eigenfunction. Thus $q_{n}$ is a minimizer in $M_{n}$  of the Rayleigh quotient
\[
  \frac{\big\langle \SS_{n}q,q \big\rangle_\Omega }{\DNormc{q}{0,\Omega}}
  =
  \frac{\big\langle \div \bw_{n}(q),q  \big\rangle_\Omega }{\DNormc{q}{0,\Omega}}
  = 
  \frac{\Normc{\bw_{n}(q)}{1,\Omega}}{\DNormc{q}{0,\Omega}}
  =
  \Big( \frac{J_{n}(q)}{\DNorm{q}{0,\Omega}} \Big)^{2}.
\]
We see that $q_{n}$ realizes the inf-sup condition, and if $\DNorm{q_{n}}{0,\Omega}=1$, we simply have 
\[
  \beta_{n}=J_{n}(q_{n})\,.
\]
Incidentally, we also have shown that $\sigma_{n}=\beta_{n}^{2}$.

Now we can write
\[
\begin{aligned}
 \beta(\Omega) \le J(q_{n}) &= J_{n}(q_{n}) + \big(J(q_{n})-J_n(q_{n})\big)\\
  & = \beta_{n} + \big(\Norm{\bw(q_{n})}{1,\Omega} - \Norm{\bw_{n}(q_{n})}{1,\Omega}\big)\\
  &\le \beta_{n} + \Norm{\bw(q_{n})-\bw_{n}(q_{n})}{1,\Omega}\,.
\end{aligned}
\]
Thus we have to estimate $\Norm{\bw(q_{n})-\bw_{n}(q_{n})}{1,\Omega}$, which is the Galerkin error in the solution of the Dirichlet problem with $\grad q_{n}$ as right-hand  side. We have
\[
 \Norm{\bw(q_{n})-\bw_{n}(q_{n})}{1,\Omega} = 
   \inff_{\bv_{n}\in X_{n}} \Norm{\bw(q_{n})-\bv_{n}}{1,\Omega}
   \le \varepsilon_{n,s}\DNorm{\bw(q_{n})}{1+s,\Omega}
\]
if $\bw(q_{n})\in H^{1+s}(\Omega)^{d}$. 
Now we use the fact that on a Lipschitz domain we have $H^{1+s}$ regularity for the solution of the Dirichlet problem of the Laplacian, because $s\in(0,\frac12)$, see \cite{CoLip}:
\[
  \Norm{\bw(q_{n})}{1+s,\Omega}\le C_{s}\DNorm{\grad q_{n}}{-1+s,\Omega}
   \le C_{s} \DNorm{q_{n}}{s,\Omega}\,.
\]
With the definition of $\eta_{n,s}$ we can further estimate
\[
 \DNorm{q_{n}}{s,\Omega} \le \eta_{n,s} \DNorm{q_{n}}{0,\Omega}=\eta_{n,s}\,.
\]
Altogether we have shown 
\[
  \beta(\Omega) \le \beta_{n} +\varepsilon_{n,s}\,C_{s}\,\eta_{n,s}
\]
as claimed in \cref{E:bNb}.
Finally, if $\lim_{n\to\infty}\eta_{n,s}\varepsilon_{n,s}=0$, then \cref{E:bNb} implies
\[
  \liminf_{n\to\infty} \beta_{n}\ge \beta(\Omega),
\]
which together with inequality \cref{E:limsup} from \cref{P:USC} proves that
$\di\lim_{n\to\infty}\beta_{n} = \beta(\Omega)$.
\end{proof}

\begin{remark}
\label{R:exactvelocity}
 The main tools in the proof, namely introduction of Sobolev spaces of fractional order, the approximation property \cref{E:epsN}, the inverse estimate \cref{E:etaN} and the $H^{1+s}$ regularity for the Dirichlet problem, were in the end only used to prove the error estimate for the Galerkin approximation of the Dirichlet problem \cref{E:wN(q)} with $\grad q_{n}$ as right-hand side, namely
\begin{equation}
\label{E:Galerr}
 \Norm{\bw(q_{n})-\bw_{n}(q_{n})}{1,\Omega}
 \le \delta_{n} \DNorm{q_{n}}{0,\Omega} \,,
\end{equation}
where $\delta_{n}= \varepsilon_{n,s}\,C_{s}\,\eta_{n,s}$. The assumptions of  \cref{T:bNconv}  implied that $\lim_{n\to\infty}\delta_{n}=0$, which allowed to complete the proof of the convergence of $\beta_{n}$.
 A weaker hypothesis that would still be sufficient for the convergence result would therefore be to assume directly that 
\[
 \mbox{ for all } q_{n}\in M_{n} \mbox{ , the estimate \cref{E:Galerr} holds with }
 \lim_{n\to\infty}\delta_{n}=0\,.
\]
There is one situation where this hypothesis obviously holds, namely when $\bw(q_{n})\in X_{n}$ for all $q_{n}\in M_{n}$, that is, the Dirichlet problem 
\[
 \Delta \bw = \grad q_{n}
\] 
can be solved exactly in $X_{n}$, because then $\bw(q_{n})=\bw_{n}(q_{n})$. 
This can be made to happen in very particular cases, if the Laplacian with Dirichlet conditions can be inverted explicitly for  right-hand  sides of the form $\grad q_{n}$ where $q_{n}$ runs through a basis of $M_{n}$. The space $X_{n}$ can then simply be chosen as the space generated by these solutions of the Dirichlet problem.
Such a special situation is analyzed in \cite{GaultierLezaun1996}, where the convergence $\beta_{n}\to\beta(\Omega)$ is then proved.
\end{remark}

%%%%%%%%%%%%
\subsection{Convergence of the first eigenfunction}\label{Ss:convq1}
%%%%%%%%%%%%

For the approximation of the inf-sup constant $\beta(\Omega)$ in \cref{T:bNconv} we computed the discrete inf-sup constant $\beta_{n}$ by solving an eigenvalue problem in $M_{n}$. In \cref{E:SchurEV}, this eigenvalue problem was described by the Schur complement of the discretized Stokes system. It is easy to see that the eigenvalue problem \cref{E:SchurEV} is equivalent to the following discretized Stokes eigenvalue problem in variational form.\\
Find $(\bw_{n},q)\in X_{n}\times M_{n}$ and $\sigma\in\C$ such that $q\ne0$ and
\begin{equation}
 \label{E:StokesEV}
\begin{aligned}
 &\forall\,\bv\in X_{n}, \quad   
 &\big\langle \grad\bw_{n},\grad\bv  \big\rangle_\Omega  - \big\langle \div\bv,q  \big\rangle_\Omega  &=0 , \\
  &\forall\,p\in M_{n},\quad 
  &\big\langle \div\bw_{n}, p \big\rangle_\Omega 
   &= \sigma \big\langle q,p  \big\rangle_\Omega . 
 \end{aligned}
\end{equation}
This is an eigenvalue problem in mixed form of the second type studied in \cite[Section 4]{BoffiBrezziGastaldi1997}. In the case of the Stokes system, however, the convergence analysis of \cite{BoffiBrezziGastaldi1997} does not apply, because the basic assumption is not satisfied, namely the compactness of the solution operator $T_{\Xi}:g\mapsto p$ in $L^2_\circ(\Omega)$ of the Stokes system for 
$(\bu,p)\in H^1_0(\Omega)^d \times L^2_\circ(\Omega)$
\[
 \begin{aligned}
  -\Delta \bu + \grad p &= 0\\
  \div\bu  &= g \,.
 \end{aligned}
\] 
Indeed, $T_{\Xi}$ is the inverse of the Schur complement operator $\SS=\div\Delta^{-1}\grad$, and it has long been known from the analysis of the Cosserat eigenvalue problem  that $\SS$ has a non-trivial essential spectrum: 
Namely, if $\Omega$ has a smooth boundary, then as shown in \cite{Mikhlin1973}, the essential spectrum is the set $\{\frac12,1\}$ with an eigenspace of infinite dimension at $\sigma=1$ and an accumulation point of eigenvalues at $\sigma=\frac12$. 
In \cite{CoCrDaLa15}, it is shown that if $\Omega$ is a polygonal plane domain with corner angles $\omega_j$, then the essential spectrum is the union of $\{1\}$ and of the closed intervals
\[
   \Big[\frac12-\frac{\sin\omega_j}{2\omega_j},\frac12+\frac{\sin\omega_j}{2\omega_j}\Big]\,,
\] 
whereas for three-dimensional piecewise smooth domains, the essential spectrum contains at least such intervals associated with the opening angles of the edges.
In many cases of non-smooth domains, for example a square in two dimensions, it is unknown whether $\sigma(\Omega)$ is the bottom of the essential spectrum of $\SS$ or an isolated eigenvalue below the essential spectrum. 

Due to the presence of the essential spectrum, the eigenvalue problem \cref{E:StokesEV} does not fit into any known convergence analysis of eigenvalue approximation methods (and this is the main reason why we had to develop the analysis presented in this section). In particular, the discrete eigenfunctions $q_{n}$ will, in general, not converge to an eigenfunction of the corresponding continuous eigenvalue problem. This is true even when our sufficient conditions for the convergence of the eigenvalues $\sigma_{n}\to\beta(\Omega)^{2}$ are satisfied. We can, however, prove convergence of $q_{n}$ if the continuous eigenvalue $\sigma(\Omega)=\beta(\Omega)^{2}$ lies below the essential spectrum. Due to the inherent non-uniqueness of eigenfunctions, the convergence has to be formulated in a somewhat convoluted fashion.
\begin{theorem}
\label{T:convqN}
Let the sequence of finite-dimensional subspaces $X_{n}$ and $M_{n}$ be such that the sufficient conditions of \cref{T:bNconv} for the validity of 
$\lim_{n\to\infty}\beta_{n} = \beta(\Omega)$ are satisfied. Assume further that $\sigma_{0}=\beta(\Omega)^{2}$ is an isolated eigenvalue of the Schur complement operator $\SS=\div\Delta^{-1}\grad$ and let $\mathcal{P}$ be the orthogonal projection operator in $L^2_\circ(\Omega)$ onto the eigenspace of $\SS$ with eigenvalue $\sigma_{0}$. Define $q_{n}$ to be a normalized eigenfunction of the operator $\SS_{n}$ defined in \cref{E:SchurEV} associated with its smallest eigenvalue $\sigma_{n}=\beta_{n}^{2}$. \\
Then $q_{n}$ converges to an eigenfunction of $\SS$ in the following sense:
\begin{equation}
\label{E:q(N)}
 \lim_{n\to\infty}\DNorm{q_{n}- \mathcal{P}q_{n}}{0,\Omega}=0\,.
\end{equation}
\end{theorem}

\begin{proof}
It follows from the selfadjointness of $\SS$  that $\mathcal P$ and $\SS$ commute. 
Since the smallest eigenvalue $\sigma_{0}$ is isolated, there exists $\delta>0$ such that for all   $q\in\ker\mathcal{P}$ there holds
$\big\langle \SS q,q  \big\rangle_\Omega  \ge (\sigma_{0}+\delta)\DNormc{q}{0,\Omega}.$

We will now use the notation from the proof of \cref{P:USC}. In particular, we have for any $q$
\[
  \big\langle \SS q,q  \big\rangle_\Omega  = \big\langle \div\bw(q),q  \big\rangle_\Omega 
  = \Normc{\bw(q)}{1,\Omega} = J(q)^{2}\,.
\]
We can therefore estimate
\[
 \begin{aligned}
 \sigma_{0} 
  &= \sigma_{0}\big(\DNormc{\mathcal{P}q_{n}}{0,\Omega}
    + \DNormc{q_{n}-\mathcal{P}q_{n}}{0,\Omega} \big)\\
  &\le \sigma_{0}\DNormc{\mathcal{P}q_{n}}{0,\Omega}
    + (\sigma_{0}+\delta)\DNormc{q_{n}-\mathcal{P}q_{n}}{0,\Omega}\\
  &\le \big\langle \SS\mathcal{P}q_{n},q_{n}  \big\rangle_\Omega 
    + \big\langle \SS(q_{n}-\mathcal{P}q_n),q_{n}  \big\rangle_\Omega \\
  & = \big\langle \SS q_{n},q_{n}  \big\rangle_\Omega  = J(q_{n})^{2} \,.    
 \end{aligned}
\]
We also know from the proof of \cref{T:bNconv} that $J(q_{n})\to\beta(\Omega)$ as $n\to\infty$. It follows that 
$\delta\DNormc{q_{n}-\mathcal{P}q_{n}}{0,\Omega}\to0$, hence \cref{E:q(N)}.
\end{proof}
\begin{remark}
\label{R:minim}
Under the sufficient conditions of \cref{T:bNconv} for the convergence of $\beta_{n} \to\beta(\Omega)$, the discrete eigenfunctions  $q_{n}$ define a minimizing sequence both for the functional $J$ and for the Rayleigh quotient of the operator $\SS$. Therefore if there exists a convergent subsequence, it will converge to an eigenfunction of the Schur complement operator $\SS$. If $\sigma_{0}=\beta(\Omega)^{2}$ is below the essential spectrum of $\SS$, hence of finite multiplicity, the previous theorem shows that every subsequence of $q_{n}$ has a subsequence converging to an eigenfunction of $\SS$ associated with $\sigma_{0}$. If, however, $\sigma_{0}$ is not an eigenvalue, then the sequence $q_{n}$ has no subsequence that converges in $L^2(\Omega)$.
\end{remark}

%%%%%%%%%%%%
\subsection{Consequences for the convergence of the discrete inf-sup constant in the $p$ and $h$ versions of the finite element method}\label{Ss:pandh}
%%%%%%%%%%%%

The sufficient conditions \cref{E:suff} for the convergence $\beta_{n}\to\beta(\Omega)$ given in \cref{T:bNconv} rely on the estimates of the quantities defined by \cref{E:etaN} and \cref{E:epsN}. Such estimates are known to hold in many finite element methods, and we can therefore concretize the convergence theorem for these methods.

Let us assume first that the spaces $X_{n}$ and $M_{n}$ correspond to an $h$ version finite element method. That is, they consist of piecewise polynomials of a fixed degree on meshes with meshsize tending to zero. We consider the case where velocity and pressure variables are discretized on two different families of meshes indexed by $n\in\N$. The inclusions $X_{n}\subset H^1_0(\Omega)^d$ and 
$M_{n}\subset L^2_\circ(\Omega)$ mean that $X_{n}$ consists of globally continuous functions, whereas the elements of $M_{n}$ may be discontinuous.
For the mesh of the velocity space $X_{n}$, we define $\overline h_{X_n}$ as the diameter of the largest element, and for the mesh of the pressure space $M_{n}$ we define $\underline h_{M_n}$ as the smallest radius of the inscribed sphere of any element.
Both $\overline h_{X_n}$ and $\underline h_{M_n}$ tend to zero as $n\to\infty$.

Under very mild conditions on local uniformity and shape regularity of the meshes
one has standard finite element approximation properties and inverse estimates that imply the existence of a constant $C$ independent of $n$ such that,  for some $s\in(0,\frac12)$,
\begin{equation}
\label{E:approx_h}
  \forall\, \bu\in (H^{1+s}(\Omega)\cap H^{1}_{0}(\Omega))^{d},\quad
  \inf_{\bv\in X_{n}}\Norm{\bu-\bv}{1,\Omega} 
    \le C\, (\overline h_{X_n})^{s} \DNorm{\bu}{1+s,\Omega}
\end{equation}
and
\begin{equation}
\label{E:inverse_h}
 \forall\, q\in M_{n}, \quad
 \DNorm{q}{s,\Omega} \le C\, (\underline h_{M_n})^{-s}\, \DNorm{q}{0,\Omega} \,.
\end{equation}
The approximation estimate \cref{E:approx_h} can be deduced by Hilbert space interpolation from the 
$O(\overline h_{X_n})$ approximation property for $\bu\in H^{2}(\Omega)^{d}$ that can be found in many textbooks on the mathematical theory of finite elements, see for example \cite{BrennerScott,Ciarlet1991}.
The inverse estimate \cref{E:inverse_h}, which we need for discontinuous elements, is proved in \cite{GrahamHS2005} under very general hypotheses on the meshes that include highly anisotropic elements.

Thus if the estimates \cref{E:approx_h} and \cref{E:inverse_h} are satisfied, we compare \cref{E:approx_h} with \cref{E:epsN} and \cref{E:inverse_h} with \cref{E:etaN} and conclude $\varepsilon_{n,s}\le C\overline h_{X_n}^{\,s}$ and 
$\eta_{n,s}\le C\underline h_{M_n}^{-s}$. In this case,  
\cref{T:bNconv} has the following corollary.
\begin{corollary}
\label{C:hversion}
 In the $h$ version of the finite element method,
\begin{equation}
\label{E:hversion}
 \mbox{ if }\quad
 \lim_{n\to\infty} \frac{\overline h_{X_n}}{\underline h_{M_{n}}} = 0
 \quad\mbox{ then }\quad
 \lim_{n\to\infty} \beta_{n} = \beta(\Omega)\,.
\end{equation}
\end{corollary}

\begin{remark}
\label{R:Taylor-Hood}
For many popular inf-sup stable elements, such as the Taylor-Hood pair of elements (see \cite[Chap.~II, Section 4.2]{GiraultRaviart}, \cite{HoodTaylor1973}), the velocity and pressure are discretized on the same mesh, in which case \cref{E:hversion} is not satisfied and \cref{C:hversion} gives no information on the convergence (or divergence) of $\beta_{n}$. 

The asymptotic condition \cref{E:hversion} may seem unusual, 
but it has been known for a long time that when the mesh size for the pressure is sufficiently larger than that for the velocity, then the corresponding pair of elements is inf-sup stable. An example, where the velocity mesh is a straightforward refinement of the pressure mesh, is given at the end of \cite[Chap.~II, Section 4.2]{GiraultRaviart} and in \cite[Chap.~4, Section 3.2]{Pironneau1989}.
\end{remark}

Second, we consider spectral methods or the $p$ version of the finite element method. Here the meshes are fixed, and the spaces $X_{n}$ and $M_{n}$ consist of (piecewise) polynomials of degrees $p_{X_n}$ and $p_{M_n}$, respectively, tending to infinity as $n\to\infty$. 
Then the typical approximation properties and inverse estimates are
\begin{equation}
\label{E:approx_p}
  \forall\, \bu\in  (H^{1+s}(\Omega)\cap H^{1}_{0}(\Omega))^{d},  \quad
  \inf_{\bv\in X_{n}}\Norm{\bu-\bv}{1,\Omega} 
    \le C\, (p_{X_n})^{-s}\, \DNorm{\bu}{1+s,\Omega}
\end{equation}
and
\begin{equation}
\label{E:inverse_p}
 \forall\, q\in M_{n}, \quad
 \DNorm{q}{s,\Omega} \le C\, (p_{M_n})^{2s}\, \DNorm{q}{0,\Omega} \,.
\end{equation}
Again, the approximation property for continuous elements \cref{E:approx_p} is standard, see \cite{BabuskaSuri94}. The inverse estimate \cref{E:inverse_p} under very general assumptions on the meshes is proved in \cite{Georgoulis2008}, where one can even find corresponding estimates for the $hp$ version of the finite element method. 

Now if the estimates \cref{E:approx_p} and \cref{E:inverse_p} are satisfied, we find $\varepsilon_{n,s}\le Cp_{X_n}^{-s}$ and 
$\eta_{n,s}\le Cp_{M_n}^{2s}$, and in this case
we obtain the second corollary of \cref{T:bNconv}.
\begin{corollary}
\label{C:pversion}
 In the $p$ version of the finite element method and in spectral methods,
\begin{equation}
\label{E:pversion}
 \mbox{ if }\quad
 \lim_{n\to\infty} \frac{(p_{M_n})^{2}}{p_{X_n}} = 0
 \quad\mbox{ then }\quad
 \lim_{n\to\infty} \beta_{n} = \beta(\Omega)\,.
\end{equation}
\end{corollary}

\begin{remark}
\label{R:CNS?}
While the conditions \cref{E:hversion} and \cref{E:pversion} are sufficient for the convergence of the discrete inf-sup constants to the exact inf-sup constant of the domain, they might not be necessary. In  Section~\ref{Ss:NumEx}  this question is studied via concrete examples of finite element approximations. It turns out that for the $h$ version there are examples where \cref{E:hversion} is not satisfied and the $\beta_{n}$ converge to arbitrarily given small positive values. On the other hand, for the $p$ version we only have numerical observations, strongly indicating that the condition \cref{E:pversion} on the degrees is too restrictive: We only found situations where either the $\beta_{n}$ converge to $0$ (no discrete LBB condition) or else they converge to $\beta{(\Omega)}$. The optimal condition instead of \cref{E:pversion} may be conjectured to be
\[
  \limsup_{n\to\infty} \frac{p_{M_n}}{p_{X_n}} < 1\,.
\]
\end{remark}

%%%%%%%%%%%%
\subsection{Examples of finite element approximations}\label{Ss:NumEx}
%%%%%%%%%%%%
We discuss two examples: the first one is devoted to the Scott-Vogelius triangular element in the $h$-version, and the second one to the $p$-version (or spectral discretization) on rectangular elements.
 The examples address the points made in \cref{R:CNS?}. The first example illustrates the situation of a finite element method that satisfies a discrete LBB condition, but the discrete LBB constants converge to a limit that may be arbitrarily small. The second example reports on numerical observations concerning the degree condition \cref{E:pversion}.

\subsubsection{Scott-Vogelius $\P_{4}$--${\P_{3}}^{\!{\rm dc}}$ triangular element}
Let $\TT$ be a triangulation of the Lipschitz polygonal domain $\Omega$. We denote by $\overline h_{\TT}$ the largest diameter of its elements. For any node $\bx$ of $\TT$, let $K_1,\ldots,K_J$ be the triangles of $\TT$ that have $\bx$ as vertex, ordered so that $K_j$ and $K_{j+1}$ have a common edge for all $j=1,\ldots,J-1$. Let $\theta_j$ be the opening of $K_j$ at its vertex $\bx$.
The regularity index at $\bx$ is defined as
\begin{equation}
\label{eq:SVRegx}
   \RR(\bx) = \max_{j=1,\ldots,J-1} |\theta_j+\theta_{j+1}-\pi|\,.
\end{equation}
If $\RR(\bx)=0$, $\bx$ is said singular. Note that any interior singular point satisfies $J=4$ and the edges to which $\bx$ belongs are contained in two lines. We also denote
\begin{equation}
\label{eq:SVRegT}
   \RR(\TT) = \min_{\bx\text{\ node of\ }\TT} \RR(\bx).
\end{equation}

Let us pick a quasiuniform family of triangulations $(\TT_n)_n$ and choose the discrete $\P_{4}$--${\P_{3}}^{\!{\rm dc}}$ spaces (continuous piecewise polynomials of (total) degree $4$ for velocities, discontinuous piecewise polynomials of degree $3$ for pressures):
\begin{equation}
\label{eq:SVXM}
\begin{cases}
   X_n = \{\bv\in H^1_0(\Omega)^2,\quad \bv\big|_K\in\P_4(K)^2 \mbox{\ \ for all \ } K\in\TT_n\} ,\\
   M_n = \{q\in L^2_\circ(\Omega),\quad q\big|_K\in\P_3(K)\mbox{\ \ for all \ } K\in\TT_n\}.
\end{cases}
\end{equation}
The result of \cite[Th.5.2]{ScottVogelius1985} provides the following implication
\begin{equation}
\label{eq:SVresult} 
   \big(\exists\delta>0 \mbox{\ such that\ } \forall n,\ \RR(\TT_n)\ge\delta\big)
   \quad\Longrightarrow\quad
   \big(\exists\beta_\star>0 \mbox{\ such that\ } \forall n,\ \beta_n\ge\beta_\star\big).
\end{equation}
However this implication does not guarantee the convergence of $\beta_n$ to $\beta(\Omega)$. We now show that convergence to a value lower than $\beta(\Omega)$ may occur.

\begin{proposition}
\label{prop:SV}
Let $\Omega$ be a Lipschitz polygonal domain. There exists $\beta_0>0$ such that for all $\beta_\infty\in(0,\beta_0]$ the following property holds: There exists a sequence of quasiuniform triangulations $(\TT_n)_n$ such that $\overline h_{\TT_n}$ tends to $0$ and the inf-sup constant $\beta_n$ associated with the spaces $X_n$ and $M_n$ given by \cref{eq:SVXM} converges to $\beta_\infty$ as $n$ tends to $\infty$.
\end{proposition}

\begin{proof}
We mesh $\Omega$ by a finite number of  quadrilaterals  $Q$ that are images of the unit square $\widehat{Q} = (0,1)^2$ under bi-affine diffeomorphisms  and denote by $\QQ_1$ this first mesh. For $n\ge2$ we define the quadrangular mesh $\QQ_n$ as the refinement of $\QQ_1$ obtained by the  images   of the meshes associated with the grid $\{0,\frac1n,\frac2n,\ldots,1\}^2$ of $\widehat{Q}$. Each element $Q$ of $\QQ_n$ is itself  bi-affinely  diffeomorphic to $\widehat{Q}$. For a parameter $b\in[0,\frac12)$ and each element $Q$ of $\QQ_n$ we define the interior point $\ba_Q(b)$ as the  image   of $(\frac12,\frac12+b)$. If $b=0$, $\ba_Q(b)$ is the center of $Q$. Then we associate with the four edges $\be$ of $Q$ the four triangles having $\be$ as an edge and $\ba_Q(b)$ as a vertex. Choosing $b=\frac14$, for instance, we define in this way the triangular mesh $\TT^0_n$ and we can check that the  left-hand  condition of \cref{eq:SVresult} for the sequence $(\TT^0_n)_n$ is satisfied. Then the  right-hand  condition of \cref{eq:SVresult} defines $\beta_\star$ and we take $\beta_0=\beta_\star$. Let us choose $\beta_\infty\in(0,\beta_0]$.

In a second step, we pick one element $Q_n$ in each quadrangular mesh $\QQ_n$. We consider another parameter $a\in(-\frac12,\frac12)$ and the interior point $\ba_{Q_n}(a)$. The triangular mesh $\TT_n(a)$ is obtained by the method above with the interior point $\ba_{Q_n}(a)$ for $Q_n$, and $\ba_{Q}(\frac14)$ for the other  quadrilaterals  $Q$ of $\QQ_n$. Let $\beta_n(a)$ be the associated inf-sup constant. It is easy to see that it is a continuous function of $a$. Since $\beta_n(0)=0$ (because the node $\ba_{Q_n}(0)$ is singular) and $\beta_n(\frac14)\ge\beta_0$, there exists $a_n\in(0,\frac14)$ such that $\beta_n(a_n)=\beta_\infty$. Choosing $\TT_n=\TT_n(a_n)$ proves the proposition.
\end{proof}

\begin{figure}[h]
\begin{minipage}{0.42\textwidth}
\includegraphics[width=1\textwidth]{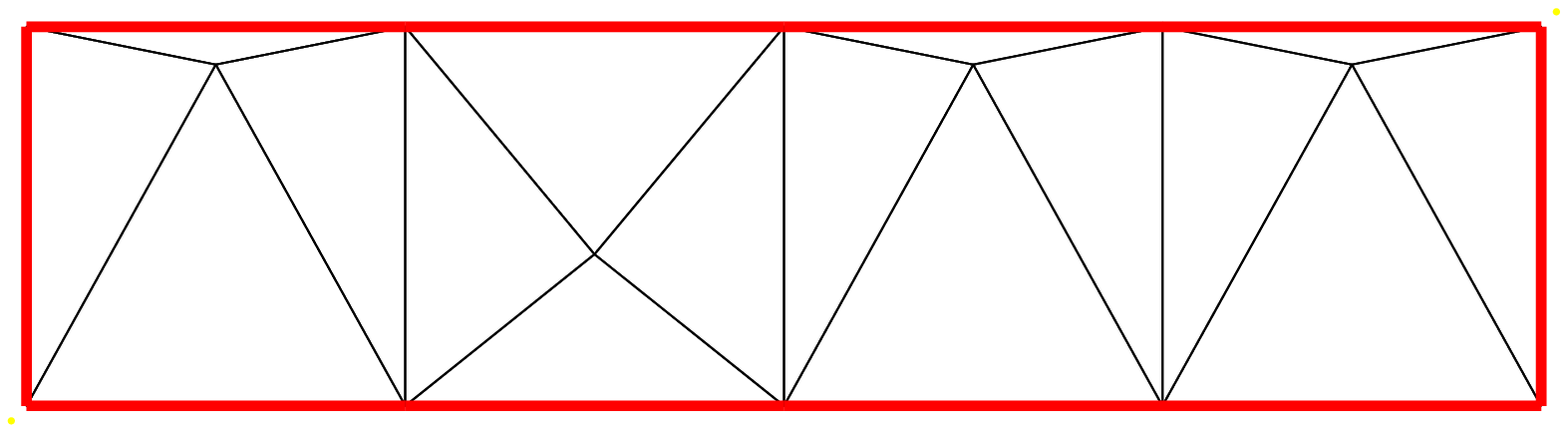}\\[0.5ex]
\includegraphics[width=1\textwidth]{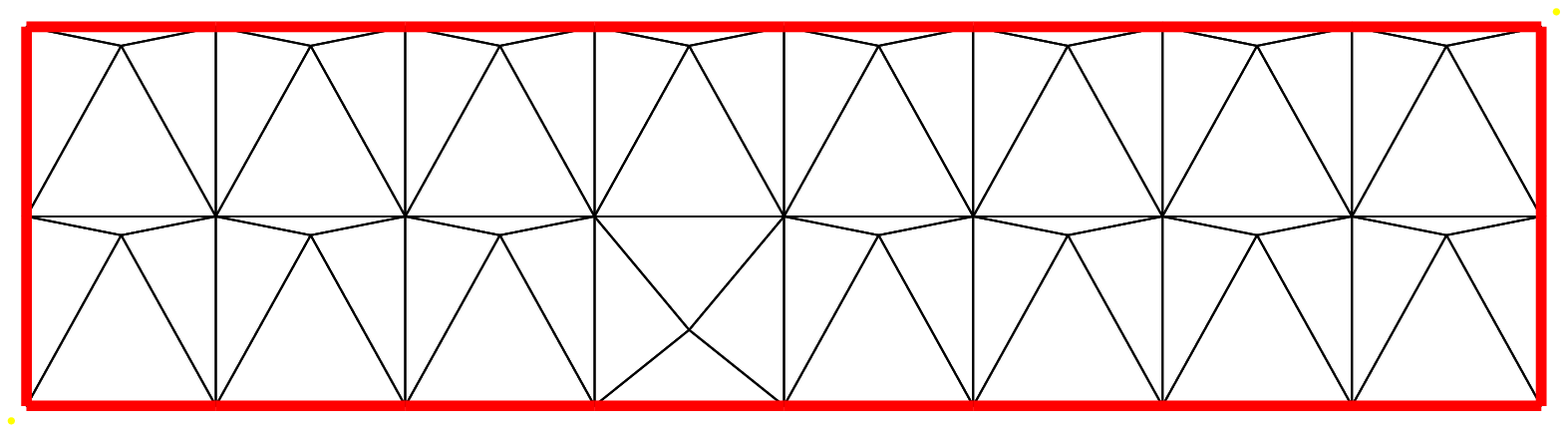}\\[0.5ex]
\includegraphics[width=1\textwidth]{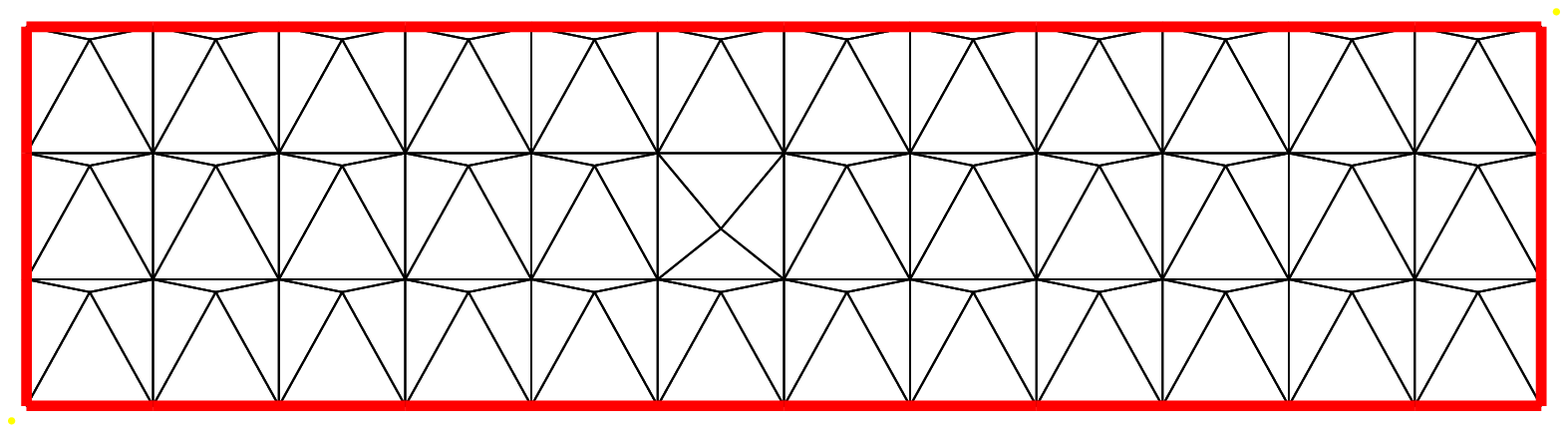}\\[0.5ex]
\includegraphics[width=1\textwidth]{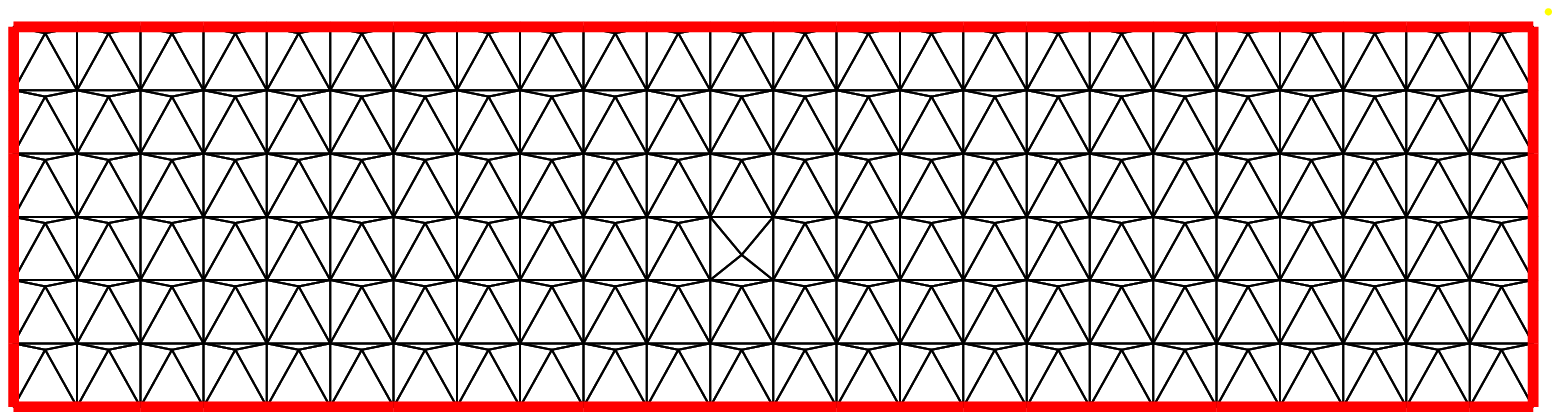}
\end{minipage}
\hfil
\begin{minipage}{0.56\textwidth}
\includegraphics[width=\textwidth]{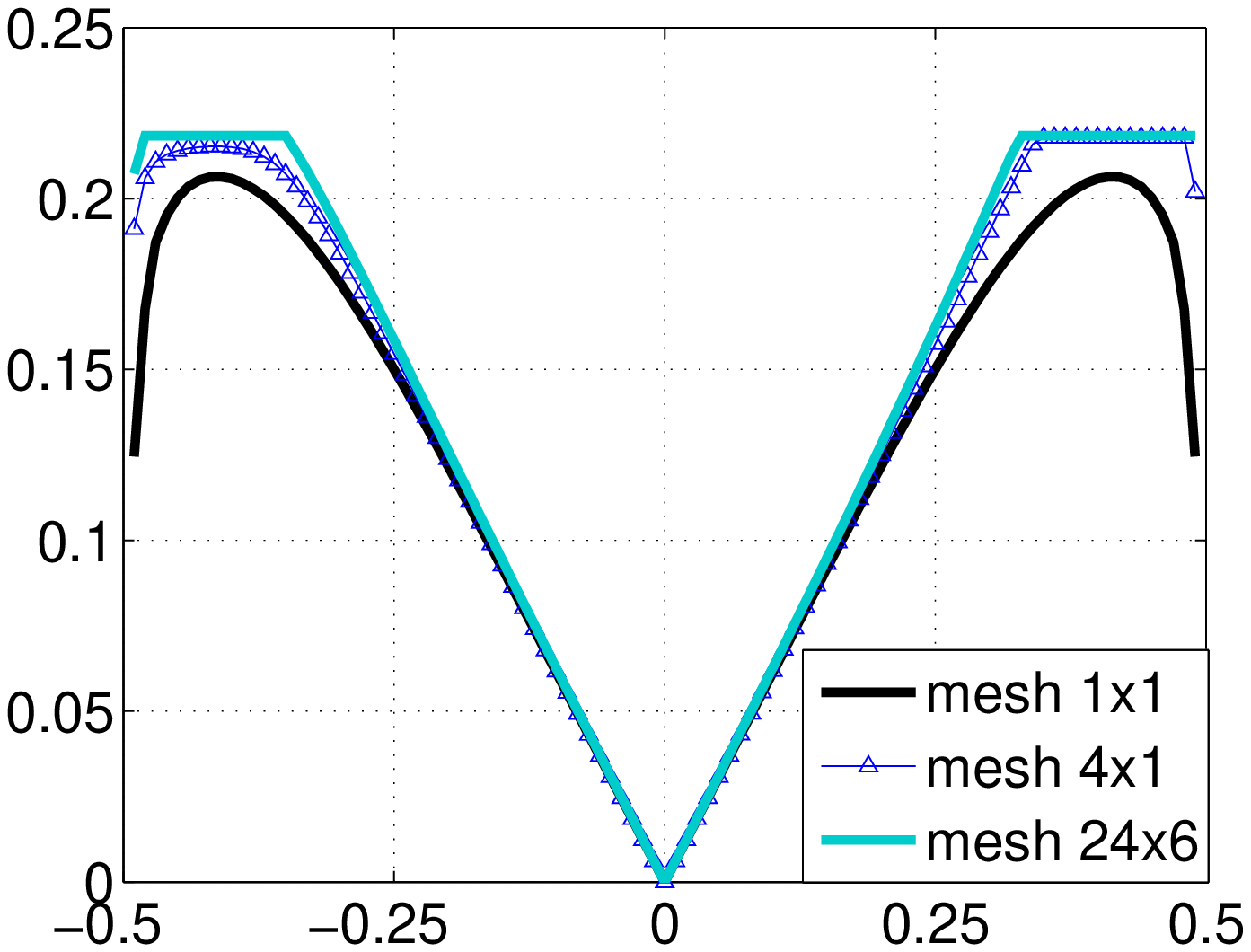}\\
\end{minipage}
\caption{Meshes of $\Omega=(0,4)\times(0,1)$ for $a=-0.1$ and $b=0.4$ (left). $\beta_{n}(a)$ for discrete spaces \cref{eq:SVXM} versus $a\in(-0.5,0.5)$ on a one-element mesh of the square and on two meshes of the rectangle (right)}
\label{Fig:SV1}
\end{figure}
We illustrate this proof by the example of $\Omega$ as the rectangle $(0,4)\times(0,1)$.
In this case, we know \cite[Section~5.1]{CoCrDaLa15} that $\beta(\Omega)^{2}$ is an isolated eigenvalue of the Cosserat operator, and the most precise computations give a value of $\beta(\Omega)\simeq0.218444$. 
 In our computations with the Scott-Vogelius $\P_{4}$--${\P_{3}}^{\!{\rm dc}}$ element, the quadrangular meshes consist of squares, specifically $4\times1$, $8\times2$, $12\times3$, and $24\times6$, see \cref{Fig:SV1} (left). The general ``decentering'' parameter $b$ is chosen as $0.4$ and the special parameter $a$ is running from $-0.49$ to $0.49$ by steps of $0.01$. We have computed with the code \texttt{FreeFem++} \cite{FreeFem} the lowest eigenvalues $\sigma_j$, $j=0,1,\ldots$ of the discrete Stokes system \cref{E:StokesEV}.
The first one $\sigma_0$ is always $0$ and, as already mentioned in Section \ref{Ss:convq1}, the square root of $\sigma_1$ equals the inf-sup constant $\beta_{n}$ of the discretization.

In \cref{Fig:SV1} (right) we observe a quasi-constant value for $\beta_{n}(a)$ when $|a|$ lies between $0.35$ and $0.48$. This value is a good approximation of the inf-sup constant $\beta(\Omega)$. Our computations approach the exact value from below, see \cref{Fig:SV1} (right). For $|a|$ less that $0.3$, we observe a linear behavior of $\beta_{n}(a)$ as the central node $\ba_{Q_n}(a)$ approaches the center (singular point), that is $\beta_{n}(a)$ behaves as a multiple of the regularity index \cref{eq:SVRegx} at $\ba_{Q_n}(a)$.
Interestingly, the same behavior with a very similar proportionality constant can already be observed on the one-element square mesh. 

\begin{figure}[h]
\hglue-0.02\textwidth
\begin{minipage}{0.5\textwidth}
\includegraphics[width=\textwidth]{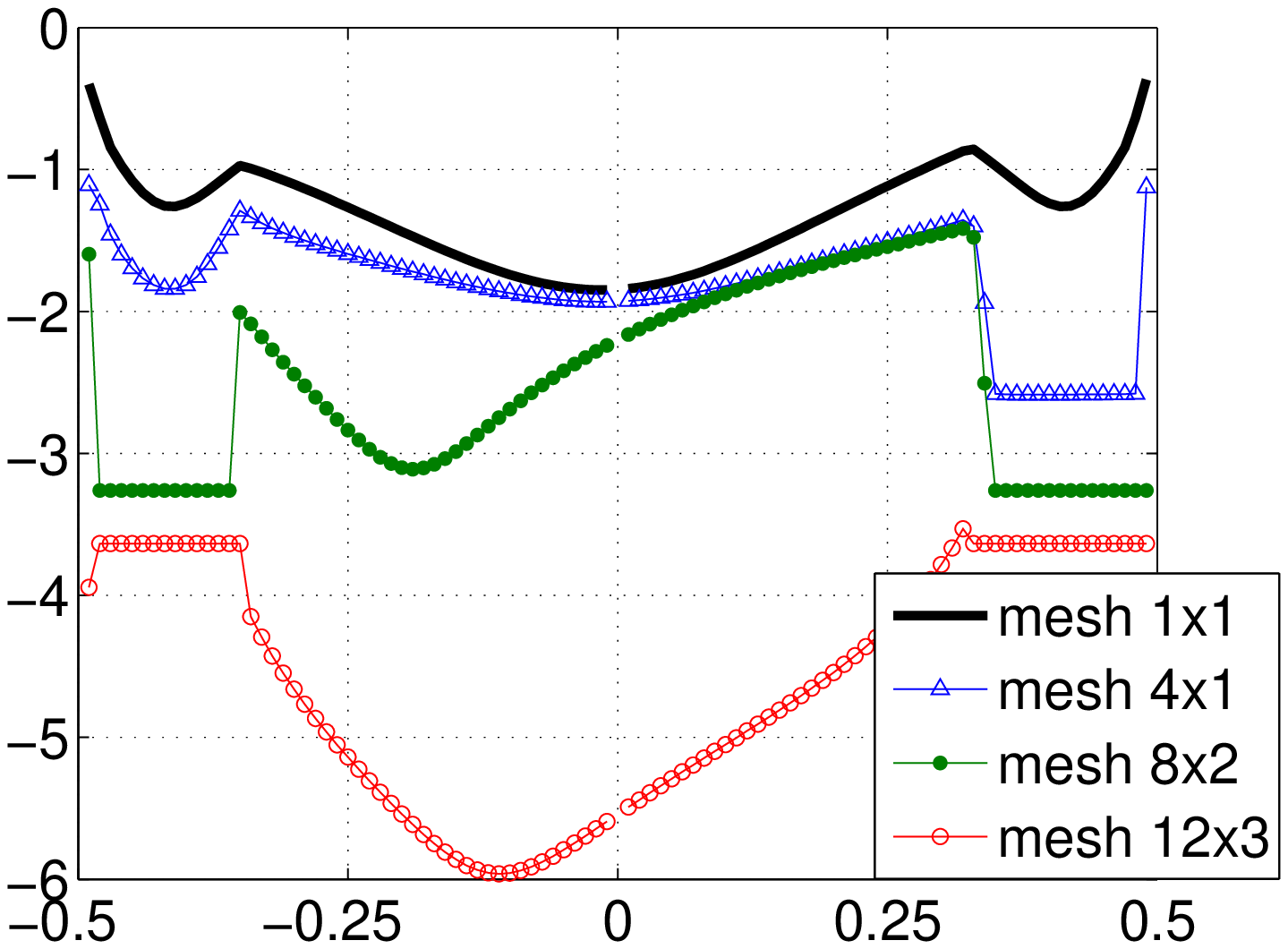}\\[0.5ex]
\end{minipage}
\hfil
\begin{minipage}{0.5\textwidth}
\includegraphics[width=\textwidth]{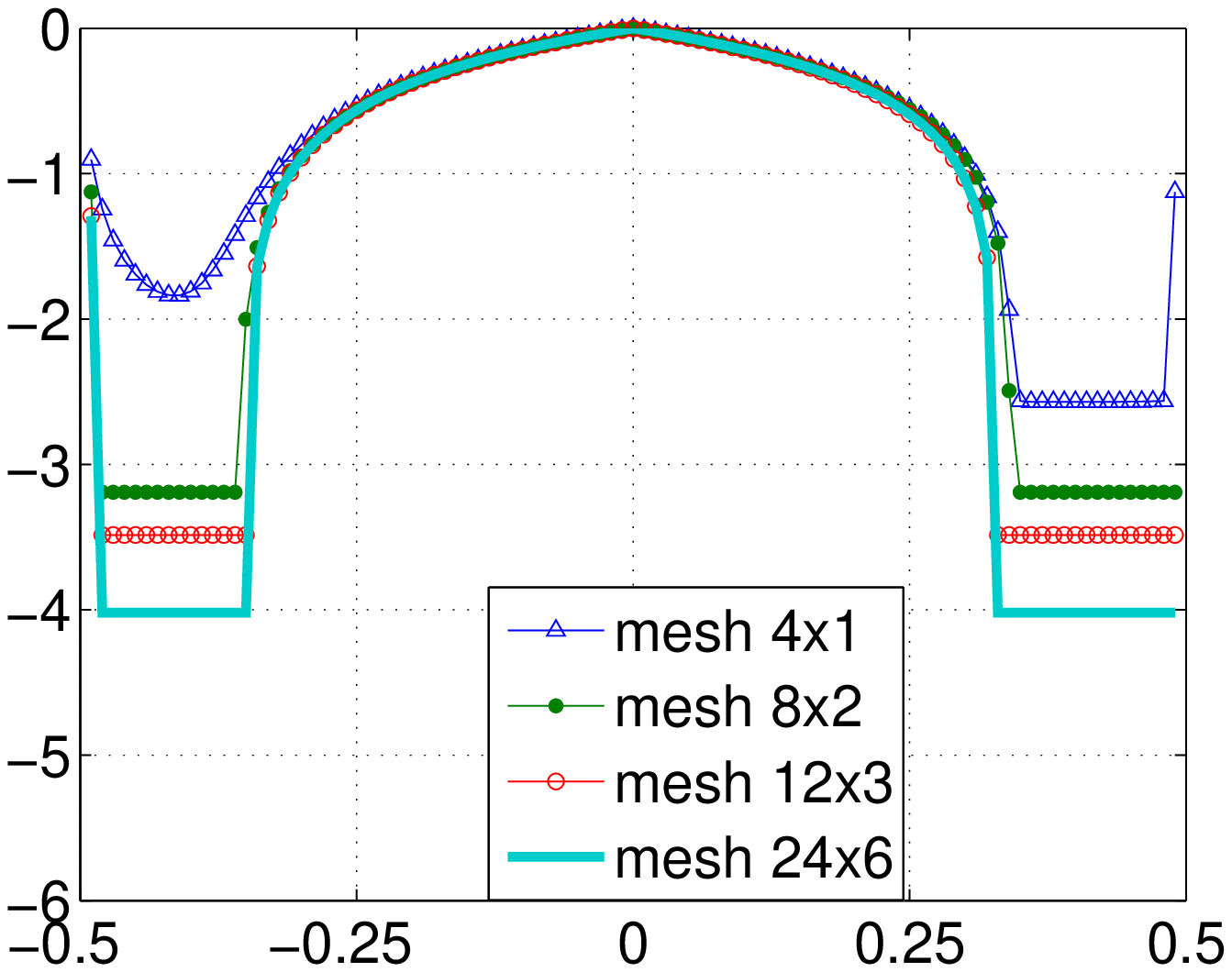}\\
\end{minipage}
\caption{$\beta_{n}(a)$ for discrete spaces \cref{eq:SVXM} versus $a\in(-0.5,0.5)$ on four meshes of the rectangle: Difference with the finer mesh (left) and difference with $\beta(\Omega)$ (right) in log10 scale.}
\label{Fig:SV2}
\end{figure}

In \cref{Fig:SV2} we plot the (relative) difference between the  constants $\beta$  when the mesh grows finer (left) and the difference with $\beta(\Omega)\simeq0.218444$ (right). More precisely, denoting by $\beta_n$ the inf-sup constant associated with the mesh $4n\times n$, $n=1,2,3,6$, and $\beta_\Box$  the inf-sup constant of  the one-element square mesh, we plot $a\mapsto \log_{10}\{(\beta_6(a)-\beta_n(a))/\beta_6(a)\}$, $ n=\Box,1,2,3$ on the left part of \cref{Fig:SV2}, and $a\mapsto \log_{10}(\beta(\Omega)-\beta_n(a))$, $n=1,2,3,6$, on the right part of \cref{Fig:SV2}.

\subsubsection{$p$-version on rectangular elements}
For one spectral element in dimension $2$, i.e. for the square $\Omega=(-1,1)^2$ discretized by means of the space $\Q_p$ of polynomials with  degree $\le p$ with respect to each variable, we have the following results, quoted or proved in \cite{BernardiMaday1999}: As discrete space for velocities let us take
\[
   X_n = H^1_0(\Omega)^2\cap\Q_n^2,\quad n\ge2\,.
\]
Then, according to the choice of the pressure space $M_n$ there holds
\begin{enumerate}
\item If $M_n=\Q_{n-1}$, then $\beta_n=0$.
\item If $M_n=\Q_{n-d}$ with a chosen $d\ge2$, then $\beta_n=\OO(n^{-1})$ as $n\to\infty$.
\item If $M_n=\Q_{n-\lambda n}$ with a chosen $\lambda\in(0,1)$, then $\beta_n\ge\beta_\star>0$.
\end{enumerate}
In this paper, we have proved convergence of $\beta_n$ to $\beta(\Omega)$ if $M_n=\Q_{\lambda_n\sqrt{n}}$ for any sequence of positive numbers $\lambda_n$  such that $\lambda_n\to0$ and $\lambda_n\sqrt{n}\to\infty$ as $n\to\infty$.

We have performed computations on rectangles of different aspect ratios with two meshes $\TT$ with $1\times1$ or $2\times2$ isometric rectangular elements. Here we use the finite element library \texttt{M\'elina++}. Now we set
\begin{equation}
\label{eq:pversionXM}
\begin{cases}
   X_n = \{\bv\in H^1_0(\Omega)^2,\quad \bv\big|_K\in\Q_n(K)^2 \mbox{\ \ for all \ } K\in\TT\} ,\\
   M_n = \{q\in L^2_\circ(\Omega),\quad q\big|_K\in\Q_k(K)\mbox{\ \ for all \ } K\in\TT\}.
\end{cases}
\end{equation}
We observe
\begin{enumerate}
\item If $k=n-1$, then $\beta_n=0$ on the mesh $1\times1$ and $\beta_n$ behaves as in the next point on the mesh $2\times2$.
\item If $k=n-d$ with $d=2,3$, then $\beta_n$ may display a preasymptotic convergence to $\beta(\Omega)$, and eventually tends to $0$ as $n\to\infty$.
\item If $k=n/2$, then $\beta_n$ converges to $\beta(\Omega)$.
\end{enumerate}

\begin{figure}[h]
\hglue-0.02\textwidth
\begin{minipage}{0.5\textwidth}
\includegraphics[width=\textwidth]{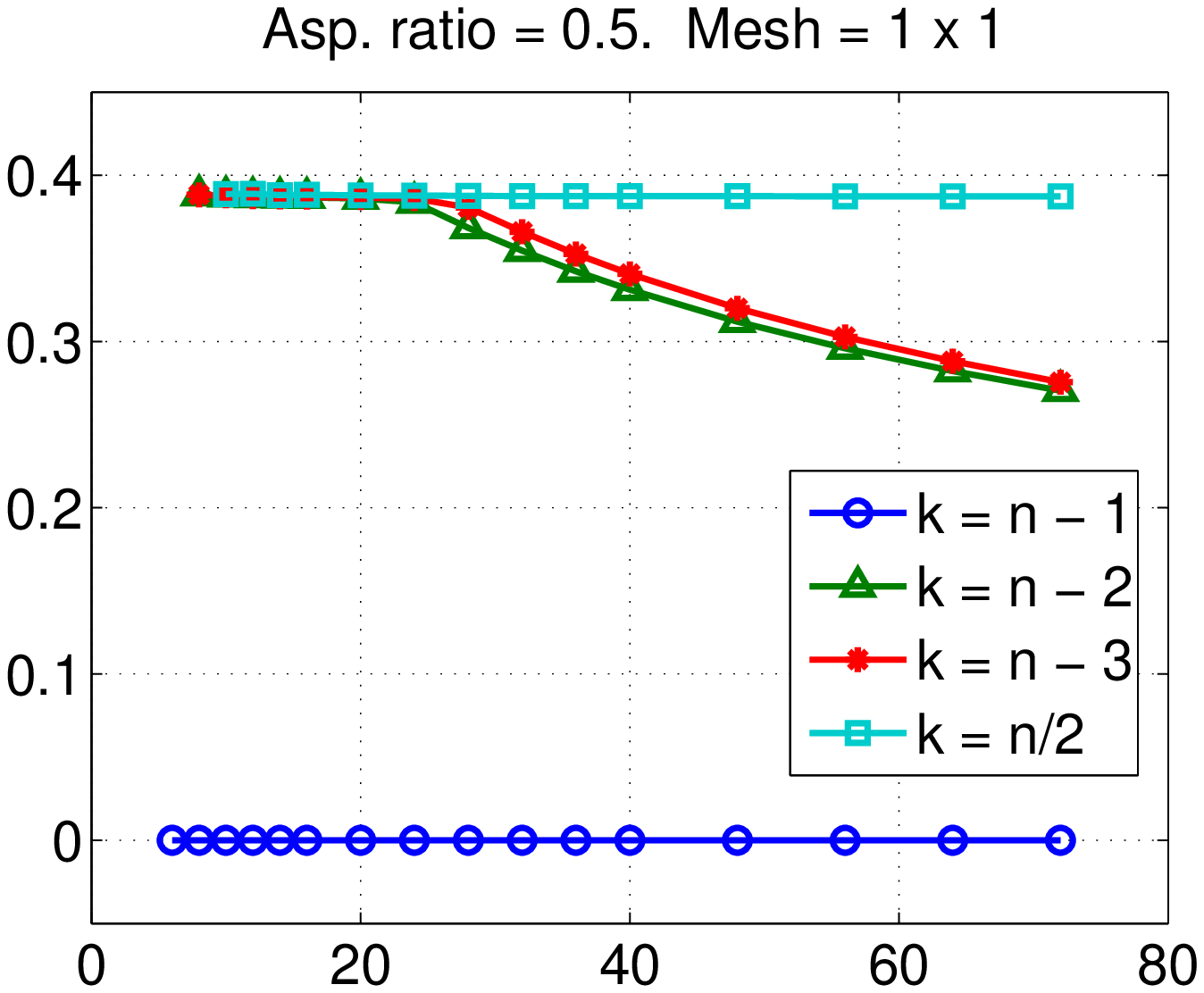}
\end{minipage}
\hfill
\begin{minipage}{0.5\textwidth}
\includegraphics[width=\textwidth]{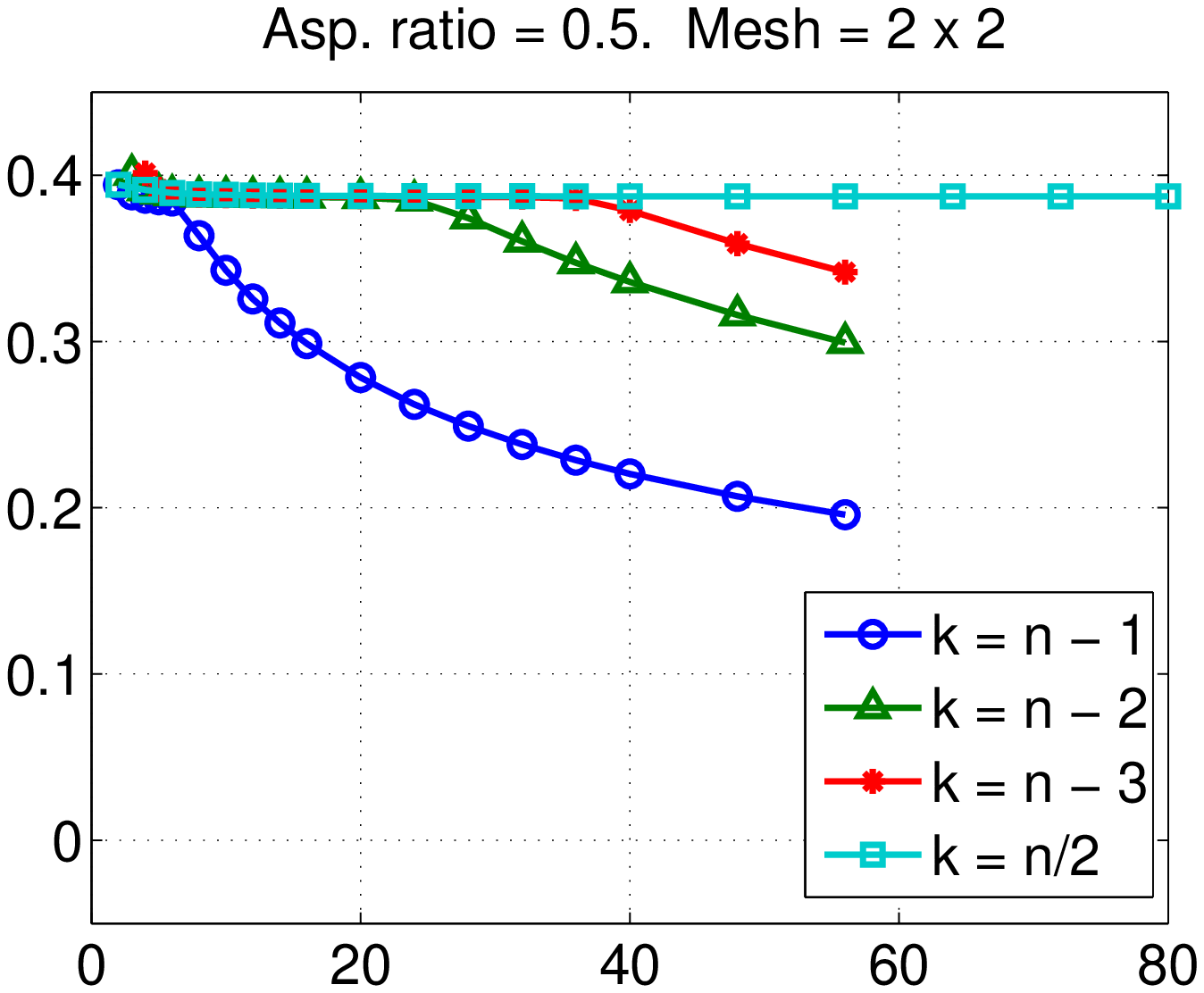}
\end{minipage}
\caption{$\beta_{n}$ for discrete spaces \cref{eq:pversionXM} versus $n$ on two meshes of the rectangle $(0,2)\times(0,1)$.}
\label{Fig:pv1}
\end{figure}

In \cref{Fig:pv1}, we present numerical results for the rectangle $\Omega=(0,2)\times(0,1)$ for which we know that $\beta(\Omega)^2$ is an isolated eigenvalue of the Cosserat operator, and our most precise computations yield the approximate value $0.387262$ for $\beta(\Omega)$.

\bibliographystyle{siam}
\bibliography{infsup}
\bigskip
\bigskip

\end{document}